\documentclass[reqno]{memo-l}
\usepackage{float,rotating,amsfonts, amsbsy, amsmath, amsthm, amssymb, latexsym, verbatim, enumerate}
\usepackage{mathrsfs}
\usepackage{pdfsync}
\usepackage{stmaryrd}
\usepackage{bm}
\usepackage{fancyhdr}
\usepackage{tikz}

\makeatletter
\newcommand{\imod}[1]{\allowbreak\mkern4mu({\operator@font mod}\,\,#1)}
\makeatother

\renewcommand{\a}{\alpha}
\renewcommand{\b}{\beta}

 \newcommand{\e}{\epsilon}
 \renewcommand{\L}{\Lambda}
\renewcommand{\l}{\lambda} 
 
 \renewcommand{\to}{\rightarrow}

 \newcommand{\C}{\mathcal{C}}

\newcommand{\leqs}{\leqslant}

 \newcommand{\vs}{\vspace{3mm}}
\newcommand{\la}{\langle}
\newcommand{\ra}{\rangle}

\newcommand{\Z}{\mathbb{Z}}

\newtheorem{theorem}{Theorem}
\newtheorem{hyp}{Hypothesis}

\newtheorem{remark}{Remark}

\newtheorem{thm}{Theorem}[section]
\newtheorem{prop}[thm]{Proposition}
\newtheorem{lem}[thm]{Lemma}
\newtheorem{cor}[thm]{Corollary}

\theoremstyle{definition}

\theoremstyle{remark}
\newtheorem{rmk}[thm]{Remark}

\numberwithin{section}{chapter}
\numberwithin{equation}{chapter}

\begin{document}

\frontmatter

\title{Irreducible almost simple subgroups \\ of classical algebraic groups}

\author{Timothy C. Burness}
\address{T.C. Burness, School of Mathematics, University of Bristol, Bristol BS8 1TW, United Kingdom}
\email{\texttt{t.burness@bristol.ac.uk}}

\author{Souma\"{i}a Ghandour}
\address{S. Ghandour, Facult\'{e} des Sciences, Section V, Universit\'{e} Libanaise, Nabatieh, Lebanon}
\email{\texttt{soumaia.ghandour@gmail.com}} 

\author{Claude Marion}
\address{C. Marion, D\'{e}partement de Math\'{e}matiques, Universit\'{e} de Fribourg, CH-1700 Fribourg, Switzerland}
\email{claude.marion@unifr.ch}

\author{Donna M. Testerman}
\address{D.M. Testerman, Section de Math\'ematiques, Station 8, \'{E}cole Polytechnique F\'{e}d\'{e}rale de Lausanne, CH-1015 Lausanne, Switzerland}
\email{\texttt{donna.testerman@epfl.ch}}

\date{November 27th, 2012}



\thanks{To appear in \emph{Memoirs of the American Mathematical Society}}

\begin{abstract}
Let $G$ be a simple classical algebraic group over an algebraically closed field $K$ of characteristic $p\geq 0$ with natural module $W$. Let $H$ be a closed subgroup of $G$ and let $V$ be a nontrivial $p$-restricted irreducible tensor indecomposable rational $KG$-module such that the restriction of $V$ to $H$ is irreducible. In this paper we classify the triples $(G,H,V)$ of this form, where $V \neq W,W^{*}$ and $H$ is a disconnected almost simple positive-dimensional closed subgroup of $G$ acting irreducibly on $W$. Moreover, by combining this result with earlier work, we 
complete the classification of the irreducible triples $(G,H,V)$ where $G$ is a simple algebraic group over $K$, and $H$ is a maximal closed subgroup of positive dimension.
\end{abstract}

\maketitle

\tableofcontents

\mainmatter

\chapter{Introduction}\label{s:intro}

In this paper  we study triples $(G,H,V)$, where $G$ is a simple algebraic group over an algebraically closed field $K$, $H$ is a closed positive-dimensional subgroup of $G$ and $V$ is a rational irreducible $KG$-module such that $V$ is irreducible as a $KH$-module.  We will refer to such a triple $(G,H,V)$ as an \emph{irreducible triple}.

The study of irreducible triples 
has a long history, dating back to fundamental work of Dynkin in the 1950s. In \cite{Dynkin1}, Dynkin determined the maximal closed connected subgroups of the classical matrix groups over $\mathbb{C}$. One of the most difficult parts of the analysis concerns irreducible simple subgroups; here Dynkin lists all the triples $(G,H,V)$ where $G$ is a simple closed irreducible subgroup of ${\rm SL}(V)$ different from ${\rm SL}(V)$, ${\rm SO}(V)$, ${\rm Sp}(V)$  and $H$ is a positive-dimensional closed connected subgroup of $G$ such that $V$ is an irreducible module for $H$. 

Not surprisingly, the analogous problem in the positive characteristic setting is much more difficult. For example, complete reducibility may fail for rational modules for simple groups, and we have no general formula for the dimensions of irreducible modules. In the 1980s, Seitz \cite{Seitz2} initiated the investigation of irreducible triples over fields of positive characteristic as part of a wider study of the subgroup structure of finite and algebraic simple groups. By introducing several new tools and techniques, which differed greatly from those employed by Dynkin, Seitz determined all irreducible triples $(G,H,V)$, where $G$ is a simply connected simple  algebraic group of classical type  defined over any algebraically closed field $K$ and $H$ is a closed connected subgroup of $G$. This was extended by Testerman \cite{Test1} to exceptional algebraic groups $G$ (of type $E_8$, $E_7$, $E_6$, $F_4$ or $G_2$), again for $H$ a closed connected subgroup. In all cases $H$ is semisimple, and in view of Steinberg's tensor product theorem, one may assume that $V$ is $p$-restricted as a $KG$-module (where $p \ge 0$ denotes the characteristic of $K$, and one adopts the convention that every $KG$-module is $p$-restricted if $p=0$). In both papers, the irreducible triples $(G,H,V)$ are presented in tables, giving the highest weights of the modules $V|_{G}$ and $V|_{H}$. 

The work of Seitz and Testerman provides a complete classification of the triples $(G,H,V)$ with $H$  connected, so it is natural to consider the analogous problem for disconnected positive-dimensional subgroups. A recent paper of Ghandour \cite{g_paper} handles the case where $G$ is exceptional, so let us assume $G$ is a classical group. In \cite{Ford1, Ford2}, Ford studies irreducible triples in the special case  where $G$ is a simple classical algebraic group over an algebraically closed field of characteristic $p \ge 0$, and $H$ is a closed disconnected subgroup such that the connected component $H^0$ is simple and the restriction $V|_{H^0}$ has $p$-restricted composition factors. These extra assumptions help to simplify the analysis. Nevertheless, under these hypotheses Ford discovered a very interesting family of triples $(G,H,V)$ with $G=B_n$ and $H=D_n.2$ (see \cite[Section 3]{Ford1}). Furthermore, these examples were found to have applications to the representation theory of the symmetric groups, and led to a proof of the Mullineux conjecture (see \cite{FK}). However, for future applications it is desirable to study the general problem for classical groups, without any extra conditions on the composition factors of $V|_{H^0}$.

In this paper we treat the case of irreducible triples $(G,H,V)$ where $G$ is of classical type, $H$ is maximal among  closed positive-dimensional subgroups of $G$ and $V$ is a $p$-restricted irreducible tensor indecomposable $KG$-module.
Using Steinberg's tensor product theorem, one can obtain all irreducible triples $(G,H,V)$ as above, without requiring the $p$-restricted condition on $V$. We now explain precisely the content of this paper. 
  
Let $G$ be a simple classical algebraic group over an algebraically closed field $K$ of characteristic $p \ge 0$ with natural module $W$. More precisely,
let $G = {\rm Isom}(W)'$, where ${\rm Isom}(W)$ is the full isometry 
group of a suitable form $f$ on $W$, namely, the zero bilinear form, a 
symplectic form, or a non-degenerate quadratic form. 
We write $G=Cl(W)$ to denote the respective simple classical groups ${\rm SL}(W)$, ${\rm Sp}(W)$ and ${\rm SO}(W)$ defined in this way. Note that
$G = {\rm Isom}(W)\cap {\rm SL}(W)$, with the exception
that if $p=2$, $f$ is quadratic and $\dim W$ is even, then $G$ has index 2 in 
${\rm Isom}(W)\cap {\rm SL}(W)$. 

A key theorem on the subgroup structure of $G$ is due to Liebeck and Seitz, which provides an algebraic group analogue of Aschbacher's well known subgroup structure theorem for finite classical groups. In \cite{LS}, six natural (or \emph{geometric}) families of subgroups of $G$ are defined in terms of the underlying geometry of $W$, labelled $\C_i$ for $1 \le i \le 6$. For instance, these collections include the stabilizers of subspaces of $W$, and the stabilizers of direct sum and tensor product decompositions of $W$. The main theorem of \cite{LS} states that if $H$ is a positive-dimensional closed subgroup of $G$ then either $H$ is contained in a subgroup in one of the $\C_i$ collections, or roughly speaking, $H^0$ is simple (modulo scalars) and $H^0$ acts irreducibly on $W$. (More precisely, modulo scalars, $H$ is \emph{almost simple} in the sense that it is contained in the group of algebraic automorphisms of the simple group $H^0$.) We write $\mathcal{S}$ to denote this additional collection of `non-geometric' subgroups of $G$. 

Let $H$ be a maximal closed positive-dimensional disconnected subgroup of a simple classical algebraic group $G = Cl(W)$ as above, let $V$ be a rational irreducible $KG$-module and assume $V \neq W$ or $W^*$, where $W^*$ denotes the dual of $W$. 
The irreducible triples $(G,H,V)$ such that $V|_{H^0}$ is irreducible are easily deduced from the aforementioned work of Seitz \cite{Seitz2}, so we focus on the situation where $V|_{H}$ is irreducible, but $V|_{H^0}$ is reducible. By Clifford theory, the highest weights of $KH^0$-composition factors of $V$ are $H$-conjugate and we can exploit this to restrict the possibilities for $V$. This approach, based on a combinatorial analysis of weights, is effective when $H$ belongs to one of the geometric $\C_i$ families since we have an explicit description of the embedding of $H$ in $G$. In this way, all the irreducible triples $(G,H,V)$ where $H$ is a disconnected positive-dimensional maximal geometric subgroup of $G$ have been determined in \cite{BGT}. However, in general, an explicit description of the embeddings of the subgroups in the family 
$\mathcal{S}$ is not available, so in this situation an entirely different analysis is required. 

The main aim of this paper is to determine the irreducible triples $(G,H,V)$ in the case where  
$H$ is a positive-dimensional disconnected subgroup in the collection $\mathcal{S}$. More precisely, we will assume $(G,H,V)$ satisfies the precise conditions recorded in Hypothesis \ref{h:our} below.
By definition (see \cite[Theorem 1]{LS}), every positive-dimensional subgroup $H \in \mathcal{S}$ has the following three properties:
\begin{itemize}\addtolength{\itemsep}{0.3\baselineskip}
\item[S1.]  $H^0$ is a simple algebraic group, $H^0 \neq G$;
\item[S2.] $H^0$ acts irreducibly and tensor indecomposably on $W$;
\item[S3.] If $G={\rm SL}(W)$ then $H^0$ does not fix a non-degenerate form on $W$.
\end{itemize}

In particular, since we are assuming $H$ is disconnected and $H^0$ is irreducible, it follows that $H \leqs {\rm Aut}(H^0)$ modulo scalars, that is, 
$$HZ(G)/Z(G) \leqs {\rm Aut}(H^0Z(G)/Z(G)),$$ 
so  
$$H \in \{A_{m}.2, D_{m}.2, D_{4}.3, D_{4}.S_{3}, E_{6}.2\}.$$
(Here ${\rm Aut}(H^0)$ denotes the group of algebraic automorphisms of $H^0$, rather than the abstract automorphism group.)  We are interested in the case where $H$ is maximal in $G$ but we will not invoke this condition, in general. However, there is one situation where we do assume maximality. Indeed, if $G={\rm Sp}(W)$, $p=2$ and $H$ fixes a non-degenerate quadratic form on $W$ then $H \leqs {\rm GO}(W)<G$. Since the special case $(G,H) = (C_n,D_n.2)$ with $p=2$  is handled in \cite{BGT} (see \cite[Lemma 3.2.7]{BGT}), this leads naturally to the following  additional hypothesis:

\begin{itemize}\addtolength{\itemsep}{0.3\baselineskip}
\item[S4.]  If $G={\rm Sp}(W)$ and $p=2$ then $H^0$ does not fix a non-degenerate quadratic form on $W$.
\end{itemize}

Note that if $H^0$ is a classical group then conditions S3 and S4 imply that $W$ is not the natural $KH^0$-module. Finally since $W$ is a tensor indecomposable irreducible $KH^0$-module, we will assume:
\begin{itemize}\addtolength{\itemsep}{0.3\baselineskip}
\item[S5.] $W$ is a $p$-restricted irreducible $KH^0$-module.
\end{itemize}

Our methods apply in a slightly more general setup; namely, we take $H \leqs {\rm GL}(W)$ with $H^0 \leqs G \leqs {\rm GL}(W)$ satisfying S1 -- S5. To summarise, our main aim is to determine the triples $(G,H,V)$ satisfying the conditions given in Hypothesis \ref{h:our}. 

\begin{remark}\label{e:neww}
\emph{As previously noted, if $p=0$ we adopt the convention that all irreducible $KG$-modules  are $p$-restricted. In addition, to ensure that the weight lattice of the underlying root system $\Phi$ of $G$  coincides with the character group of a maximal torus of $G$, in Hypothesis \ref{h:our} we replace $G$ by a simply connected cover also having root system $\Phi$.}
\end{remark}

\begin{hyp}\label{h:our}
$G$ is a simply connected cover of a simple classical algebraic group $Cl(W)$ defined over an algebraically closed field $K$ of characteristic $p \geq 0$, $H \leqs {\rm GL}(W)$ is a closed disconnected positive-dimensional subgroup of ${\rm Aut}(G)$ satisfying S1 -- S5 above, and $V$ is a rational tensor indecomposable $p$-restricted irreducible $KG$-module such that $V|_H$ is irreducible, but $V|_{H^0}$ is reducible.  
\end{hyp}

Given a triple $(G,H,V)$ satisfying Hypothesis \ref{h:our}, let $\l$ and $\delta$ denote the highest weights of the $KG$-module $V$ and the $KH^0$-module $W$, respectively (see Section \ref{ss:nota} for further details). Note that $W$ is self-dual as a $KH^0$-module (the condition $H \leqs {\rm GL}(W)$ implies that the relevant graph automorphism acts on $W$), so $H^0$ fixes a non-degenerate form on $W$ and thus $G$ is either a symplectic or orthogonal group (by condition S3 above). We write $\l|_{H^0}$ to denote the restriction of the highest weight $\l$ to a suitable maximal torus of $H^0$.

\begin{remark}\label{r:new}
\emph{If $(G,H,V)$ is a triple satisfying the conditions in Hypothesis \ref{h:our} then either $H \leqs G$, or $G=D_n$ and $H \leqs D_n.2 = {\rm GO}(W)$.  
In Theorem \ref{t:new} below we describe all the irreducible triples $(G,H,V)$ satisfying Hypothesis \ref{h:our}. By inspecting the list of examples with $G=D_n$ we can determine the cases with $H \leqs G$, and this provides a complete classification of the relevant triples with $H$ in the $\mathcal{S}$ collection of subgroups of $G$.}
\end{remark}

\begin{remark}\label{r:new22}
\emph{If $(G,p)=(B_n,2)$ then $G$ is reducible on the natural $KG$-module $W$ (the corresponding symmetric form on $W$ has a $1$-dimensional radical). In particular, no positive-dimensional subgroup $H$ of $G$ satisfies condition S$2$ above, so in this paper we will always assume $p \neq 2$ when $G=B_n$. (The only exception to this rule arises in the statement of Theorems \ref{t:main2} and \ref{T:MAIN3}, where we allow $(G,p)=(B_n,2)$.)}
\end{remark}

\begin{theorem}\label{t:main}
A triple $(G,H,V)$ with $H \leqs G$ satisfies Hypothesis \ref{h:our} if and only if 
$(G,H,\l,\delta)=(C_{10},A_5.2, \l_3, \delta_3)$, $p \neq 2,3$ and $\l|_{H^0} = \delta_1+2\delta_4$ or $2\delta_2+\delta_5$. 
\end{theorem}

\begin{remark}\label{r:main}
\emph{Note that in the one example that arises here, $V|_{H^0}$ has $p$-restricted composition factors. However, this is a new example, which is missing from Ford's tables in \cite{Ford1}. See Remark \ref{r:ford} for further details. 
It is also important to note that the proof of Theorem \ref{t:main} is independent of Ford's work \cite{Ford1, Ford2}; our analysis provides an alternative proof (and correction), without imposing any conditions on the composition factors of $V|_{H^0}$.}
\end{remark}

More generally, using \cite[Theorem 2]{Seitz2}, we can determine the triples $(G,H,V)$ satisfying the following weaker hypothesis:

\begin{hyp}\label{h:ourprime}
$G$ and $H$ are given as in Hypothesis \ref{h:our}, $V$ is a rational
tensor indecomposable $p$-restricted irreducible $KG$-module such that $V|_{H}$
is irreducible, and $V$ is not the natural $KG$-module, nor its dual.
\end{hyp}

\begin{theorem}\label{t:main11}
The triples $(G,H,V)$ with $H \leqs G$ satisfying Hypothesis \ref{h:ourprime} are listed in Table \ref{tab:main2}.  
\end{theorem}

\renewcommand{\arraystretch}{1.2}
\begin{table}[h]
$$\begin{array}{lllllll} \hline
G & H & \l & \delta & \l|_{H^0} & \kappa & \mbox{Conditions} \\ \hline
C_{10} & A_5.2 & \l_3 & \delta_3 & \mbox{{\rm $\delta_1+2\delta_4$ or $2\delta_2+\delta_5$}} & 2 & p \neq 2,3 \\
  
C_{10} & A_5.2 & \l_2 & \delta_3 & \delta_2+\delta_4 & 1 & p \neq 2  \\

B_3 & A_2.2 & 2\l_1 & \delta_1+\delta_2 & 2\delta_1+2\delta_2 & 1 & p = 3  \\

D_7 & A_3.2 & \l_6,\,\l_7 &  \delta_1+\delta_3 & \delta_1+\delta_2+\delta_3 & 1 & p = 2 \\

D_{13} & D_{4}.Y & \l_{12},\,\l_{13} & \delta_2 & \delta_1+\delta_2+\delta_3 +\delta_4 & 1 & p = 2,\; 1 \neq Y \leqs S_3 \\ \hline
\end{array}$$
\caption{The triples $(G,H,V)$ with $H \leqs G$ satisfying Hypothesis \ref{h:ourprime}}
\label{tab:main2}
\end{table}
\renewcommand{\arraystretch}{1}

\begin{remark}\label{r:main2}
\emph{Let us make some remarks on the statement of Theorem \ref{t:main11}:
\begin{itemize}\addtolength{\itemsep}{0.3\baselineskip}
\item[{\rm (a)}] In Table \ref{tab:main2}, $\kappa$ denotes the number of $KH^0$-composition factors of $V|_{H^0}$. 
\item[{\rm (b)}] Note that $V|_{H^0}$ is irreducible in each of the cases listed in the final four rows of Table \ref{tab:main2}; these are the cases labelled ${\rm II}_{1}$, ${\rm S}_{1}$, ${\rm S}_{7}$ and ${\rm S}_{8}$, respectively, in \cite[Table 1]{Seitz2}.
\item[{\rm (c)}] Note that $A_3.2<D_7$ and $D_4.2<D_{13}$ (see Theorem \ref{t:seitz1}), so the cases listed in the final two rows of Table \ref{tab:main2} give rise to genuine examples with $H \leqs G$. Moreover, for the case appearing in the final row we can take $H = D_4.2, D_4.3$ or $D_4.S_3$.
\end{itemize}}
\end{remark}

\begin{theorem}\label{t:new}
A triple $(G,H,V)$ satisfies Hypothesis \ref{h:ourprime} if and only if one of the following holds:
\begin{itemize}\addtolength{\itemsep}{0.3\baselineskip}
\item[{\rm (i)}] $H \leqs G$ and $(G,H,V)$ is one of the cases in Table \ref{tab:main2}; or
\item[{\rm (ii)}] $H \not \leqs G$  and $(G,H,\l) = (D_{10},A_3.2,\mbox{$\l_{9}$ or $\l_{10}$})$, $\delta = 2\delta_2$, $p \neq 2,3,5,7$ and $\l|_{H^0} = 3\delta_1+\delta_2+\delta_3$ or $\delta_1+\delta_2+3\delta_3$.
\end{itemize}
In particular, the triples satisfying Hypothesis \ref{h:our} are listed in Table \ref{tab:main}.  
\end{theorem}

\renewcommand{\arraystretch}{1.2}
\begin{table}[h]
$$\begin{array}{lllllll} \hline
G &  H & \l & \delta & \l|_{H^0} & \kappa & \mbox{Conditions} \\ \hline
C_{10} & A_5.2 & \l_3 & \delta_3 & \mbox{{\rm $\delta_1+2\delta_4$ or $2\delta_2+\delta_5$}} & 2 & p \neq 2,3 \\
D_{10} & A_3.2 & \l_9,\l_{10} & 2\delta_2 & \mbox{{\rm $3\delta_1+\delta_2+\delta_3$ or $\delta_1+\delta_2+3\delta_3$}} & 2 & p \neq 2,3,5,7 \\ \hline
 \end{array}$$
 \caption{The triples $(G,H,V)$ satisfying Hypothesis \ref{h:our}}
 \label{tab:main}
\end{table} 
\renewcommand{\arraystretch}{1}

\begin{remark}\label{r:new2}
\emph{Note that in case (ii) of Theorem \ref{t:new}, we have $A_3.2<D_{10}.2$, but the graph automorphism of $A_3$ is not contained in the simple group $D_{10}$ (see Lemma \ref{l:a3d10}). This is the only triple $(G,H,V)$  satisfying Hypothesis \ref{h:ourprime} with $H \not\leqs G$.}
\end{remark}

By combining Theorem \ref{t:main11}  with the main theorems of \cite{BGT} and \cite{Seitz2}, we obtain the following result. 

\begin{theorem}\label{t:main2}
Let $G$ be a simple classical algebraic group over an algebraically closed field $K$ of characteristic $p \ge 0$, $H$ a maximal positive-dimensional closed subgroup of $G$, and let $V$ be a tensor indecomposable $p$-restricted irreducible $KG$-module. If $H$ acts irreducibly on $V$ then  one of the following holds:
\begin{itemize}\addtolength{\itemsep}{0.3\baselineskip}
\item[{\rm (i)}] $H$ is connected and either $V$ is the natural $KG$-module (or its dual), or $(G,H,V)$ is described by \cite[Theorem 2]{Seitz2} of Seitz;
\item[{\rm (ii)}] $H$ is a disconnected geometric subgroup of $G$ and $(G,H,V)$ appears in \cite[Table 1]{BGT};
\item[{\rm (iii)}] $H$ is a disconnected almost simple subgroup of $G$ and either $V$ is the natural $KG$-module (or its dual), or $(G,H,V)$ is one of the cases in Table \ref{tab:main2}. 
\end{itemize}
Moreover, for $(G,H,V)$ in \cite[Theorem 2]{Seitz2}, \cite[Table 1]{BGT} or Table \ref{tab:main2}, $H$ acts irreducibly on $V$. 
\end{theorem}

Similar problems have been studied recently by various authors. For instance, in \cite{GT},
Guralnick and Tiep consider irreducible triples $(G,H,V)$ in the special case $G={\rm SL}(W)$ with $V=S^k(W)$, the $k$-th symmetric power of the natural module $W$ for $G$, and $H$ is \emph{any} (possibly finite) closed subgroup of $G$. A similar analysis of the exterior powers $\L^k(W)$ is in progress. These results have found interesting applications in the study of holonomy groups of stable vector bundles on  smooth projective varieties (see \cite{BK}). We also refer the reader to \cite{GT2} for related results on the irreducibility of subgroups acting on small tensor powers of the natural module. At the level of finite groups, a similar problem for subgroups of ${\rm GL}_{n}(q)$ is studied by Kleshchev and Tiep in \cite{KT}.

As a special case of Theorem \ref{t:main2}, we determine the maximal positive-dimensional closed subgroups of a simple classical algebraic group $G$ acting irreducibly on all $KG$-composition factors of a  symmetric or exterior power of the natural $KG$-module. (The proof of Theorem \ref{T:MAIN3} is given in Section \ref{s:thm5}.)

\begin{theorem}\label{T:MAIN3}
Let $G$ be a simple classical algebraic group over an algebraically closed field $K$ of characteristic $p \ge 0$ and let $H$ be a closed positive-dimensional maximal subgroup of $G$. Let $W$ be the natural $KG$-module and let $n$ denote the rank of $G$.  
\begin{itemize}\addtolength{\itemsep}{0.3\baselineskip}
\item[{\rm (i)}] Suppose $1<k< n$. Then $H$  acts irreducibly on all $KG$-composition factors of $\L^k(W)$ if and only if  $(G,H,K)$ is one of the cases in Table \ref{t:ext}. 
\item[{\rm (ii)}] Suppose $1< k$, and $k<p$ if $p \neq 0$. Then  $H$  acts irreducibly on all $KG$-composition factors of $S^k(W)$ if and only if  $(G,H,K)$ is one of the cases in Table \ref{t:sym}. 
\end{itemize}
\end{theorem}

\begin{remark}\label{r:tabs}
\emph{In the third column of Tables \ref{t:ext} and \ref{t:sym} we describe the embedding of $H$ in $G$ in terms of a suitable set of fundamental dominant weights for $H^0$. In addition, $T_n$ denotes a maximal torus of dimension $n$ in the third line of Table \ref{t:ext}.}
\end{remark}

\renewcommand{\arraystretch}{1.2}
\begin{table}[h]\small
$$\begin{array}{lllll} \hline
G & H & W|_{H^0} & k & \mbox{Conditions} \\ \hline

A_{n} & B_l & \omega_{1} & 1< k< 2l& n=2l,\, p \neq 2 \\
  
A_{n}& D_l.2 & \omega_{1} & 1 < k< 2l-1 & n=2l-1,\, p \neq 2  \\

A_{n} & N_G(T_n) & - & 1< k < n &  \\

A_n& A_l^2.2 & \omega_{1} \otimes \omega_{1} & 2,n-1& n+1 = (l+1)^2, \, p\ne 2  \\

A_{n} & A_l& \omega_{2} & 2,n-1& n=(l^2+l-2)/2,\, l\geq 3, \, p \ne 2  \\

A_{n} & A_l& 2\omega_{1} & 2,n-1&n=(l^2+3l)/2,\, p \ne 2  \\

A_{26}& E_{6} & \omega_{1} & 2,3,4,23,24,25 & p\ne 2, \, (\mbox{$p \neq 2,3$ if $k=3,4,23,24$})  \\

A_{15} & D_5 & \omega_{5} & 2,3, 13,14 &p \ne 2, \, (k,p) \neq (3,3),(13,3)  \\ 

C_n&D_n.2& \omega_{1} & 1< k < n &p=2  \\

C_{28}&E_7&\omega_{7} & 2&p\ne 2  \\

&& &3,4,5& p\ne 3,\, (k,p) \neq (5,5)  \\

C_{16}&D_6& \omega_{6} & 2,3& p\ne 2,\, (k,p) \neq (3,3)  \\

C_{10} & A_5.2 & \omega_{3} & 2,3&p\ne 2  \\

C_7&C_3& \omega_{3} & 2,3& p\ne 2, \, (k,p) \neq (2,3), (3,7)  \\

C_4&C_1^3.S_3 & \omega_{1} \otimes \omega_{1} \otimes \omega_{1} &2,3&p\ne 2,\, (k,p) \neq (3,3)  \\

C_3&C_1^3.S_3& \omega_{1} \oplus \omega_{1} \oplus \omega_{1} & 2&p=3  \\

C_3&G_2& \omega_{1} & 2&p=2  \\ 

D_4 & C_1^3.S_3 & \omega_{1} \otimes \omega_{1} \otimes \omega_{1} & 3 & p=2  \\ \hline
\end{array}$$
\caption{$H$ irreducible on all $KG$-composition factors of $\L^k(W)$}
\label{t:ext}
\end{table}
\renewcommand{\arraystretch}{1}

\renewcommand{\arraystretch}{1.2}
\begin{table}[h]
$$\begin{array}{lllll} \hline
G & H & W|_{H^0} & k & \mbox{Conditions} \\  \hline

A_{n}& C_l& \omega_{1} & \mbox{all} & n=2l-1 \\

B_3&G_2& \omega_{1} & 2&p\ne2\\

B_3&A_2.2& \omega_{1}+\omega_2 & 2&p=3\\

B_6&C_3& \omega_2 & 2&p=3\\

B_{12}&F_4& \omega_4 & 2&p=3\\ \hline
\end{array}$$
\caption{$H$ irreducible on all $KG$-composition factors of $S^k(W)$}
\label{t:sym}
\end{table}
\renewcommand{\arraystretch}{1}

To close this introductory section, we would like to comment briefly on our methods, and how they relate to those used by Ford in \cite{Ford1, Ford2}. As in the work of Seitz \cite{Seitz2}
on irreducibly acting connected subgroups, we use a result of Smith \cite{Smith},
which states that if $P = QL$ is a parabolic subgroup of a semisimple
algebraic group $G$, and $V$ is an irreducible $KG$-module, then $L$ acts 
irreducibly on the commutator quotient $V/[V,Q]$. In our set-up, this can
be applied to the irreducible $KG$-module $V=V_G(\l)$, as well as to the irreducible summands of 
$V|_{X}$, where $X=H^0$. In order to exploit this property, for any parabolic subgroup $P_X = Q_XL_X$ of $X$ we will
construct a canonical parabolic subgroup $P=QL$ of $G$
as the stabilizer of the sequence of subspaces 
\begin{equation}\label{ee:flag}
W > [W,Q_X]>[[W,Q_X],Q_X]> \cdots > 0
\end{equation}
It turns out that the embedding $P_X<P$ has several important properties (see Lemma \ref{l:flag}). For instance, $L_X$ is contained in a Levi factor $L$ of $P$.  

We can study the weights occurring in each of 
the subspaces in the above flag of $W$ to obtain a lower bound on the dimensions of the quotients, which then leads to structural information on  
$L'$. Moreover, we can use this to impose conditions on the highest weight 
$\lambda$ of $V$. For example, consider the generic case $H=X\la t \ra$, where $t$ is an involutory graph automorphism of $X$. If we take $P_X$ to be a $t$-stable 
Borel subgroup of $X$ then we find that the Levi factor $L$ has 
an $A_1$ factor, and the restriction of $\lambda$ to a suitable maximal torus of $L'$ affords  the natural $2$-dimensional module for this $A_1$ factor (see Lemma \ref{l:main}). By combining this observation with our lower bounds on the dimensions of the quotients arising in \eqref{ee:flag}, we can impose severe restrictions on the highest weight of $W$ (viewed as an irreducible $KX$-module). As noted in Lemma \ref{l:main}, even if our analysis of the quotients in \eqref{ee:flag} does not rule out the existence of an $A_1$ factor of $L'$, we still obtain very useful restrictions on the coefficients $a_i$ when we express $\l=\sum_{i=1}^{n}a_i\l_i$ as a linear combination of fundamental dominant weights. Further restrictions on the $a_i$ can be obtained by considering the flag \eqref{ee:flag} with respect to different parabolic subgroups $P_X=Q_XL_X$ of $X$. This is how we proceed.

This general set-up, based on the embeddings $P_X<P$, can also be found in Ford's work.
The main difference comes in the consideration of the second commutator quotient. In \cite[2.14(iv)]{Seitz2}, Seitz establishes an upper bound for the dimension of $[V,Q]/[[V,Q],Q]$ in terms of  the dimensions of $V/[V,Q_X]$ and a certain quotient of $Q_X$. However, this upper bound is only valid if the highest weights of the composition factors of $V|_X$ are $p$-restricted.  
 
To replace this technique, we carefully analyse the action of 
$X$ on $W$ in order to obtain information on the restrictions of 
weights and roots
for $G$ to a suitable maximal torus of $X$. In addition, we consider the action 
of certain $(A_m)^t$ Levi factors of $X$ on $V$, and we use the fact that every weight space of an irreducible $K(A_1)^t$-module is $1$-dimensional. Various considerations such as these enable us to reduce to a very short list of possibilities for the  highest weight $\l$.  

\vs

Finally, some comments on the organisation of this paper. In Section \ref{s:prel} we present several preliminary results which will be needed in the proof of our main theorems. In particular, we recall some standard results on weights and their multiplicities, and 
following Ford \cite{Ford1} (and Seitz \cite{Seitz2} initially) we study certain parabolic subgroups of $G$ constructed in a natural way from parabolic subgroups of $H^0$; these parabolic embeddings play a crucial role in our analysis. The remainder of the paper is dedicated to the proof of Theorems \ref{t:main} -- \ref{T:MAIN3} (with the focus on Theorem \ref{t:new}).
 In Section \ref{s:am} we assume $H^0=A_m$; the low rank cases $m=2,3$ require
special attention, and they are dealt with in Sections \ref{ss:am1} and \ref{ss:am2}, respectively, while the general situation is considered in Section \ref{ss:am3}. Next, in Section \ref{s:dm5} we assume $H^0$ is of type $D_m$ with $m \ge 5$; the special case $H^0=D_4$ is handled separately in Section \ref{s:d4}. Finally, the case $H^0=E_6$ is dealt with in Section \ref{s:e6}, and the short proof of Theorem \ref{T:MAIN3} is given in Section \ref{s:thm5}.

\section*{Acknowledgments}

The first author was supported by EPSRC grant EP/I019545/1, and he thanks the Section de Math\'{e}matiques at EPFL for their generous hospitality. The third and fourth authors were supported by the Fonds National Suisse de la Recherche Scientifique, grant number 200021-122267.

\chapter{Preliminaries}\label{s:prel}

\numberwithin{table}{chapter}

\section{Notation and terminology}\label{ss:nota}

Let us begin by introducing some notation that will be used throughout the paper. Further notation will be introduced in the subsequent sections of this chapter, and we refer the reader to p.\pageref{p:notation} for a convenient list of the main notation we will use.

As in Hypothesis \ref{h:our}, let $G$ be a simply connected cover of a simple classical algebraic group $Cl(W) = {\rm SL}(W)$, ${\rm Sp}(W)$ or ${\rm SO}(W)$, defined over an algebraically closed field $K$ of characteristic $p \ge 0$. Here $Cl(W) = {\rm Isom}(W)'$, \label{p:clw} where ${\rm Isom}(W)$ \label{p:isomw} is the full isometry 
group of a form $f$ on $W$, which is either the zero bilinear form, a 
symplectic form or a non-degenerate quadratic form. 
It is convenient to adopt the familiar Lie notation $A_n$, $B_n$, $C_n$ and $D_n$ to denote the various possibilities for $G$, where $n$ denotes the rank of $G$. Note that if $Cl(W)={\rm SO}(W)$, where $p=2$ and $W$ is odd-dimensional, then $G$ acts reducibly on $W$, so for the purpose of proving our main theorems, we may assume that $p \neq 2$ if $G$ is of type $B_n$. Also note that we will often refer to $G$ as a `classical group', by which we mean the simply connected version of $G$.

Fix a Borel subgroup $B=UT$ of $G$, where $T$ is a maximal torus of $G$ and $U$ is the unipotent radical of $B$. Let $\Delta(G)=\{\a_1, \ldots, \a_{n}\}$ be a corresponding base of the root system $\Phi(G)=\Phi^+(G) \cup \Phi^{-}(G)$ of $G$, where $\Phi^+(G)$ and $\Phi^{-}(G)$ denote the positive and negative roots of $G$, respectively. We extend the notation $\Delta,\Phi,\Phi^+$ and $\Phi^-$ to any reductive algebraic group. Let 
$X(T) \cong \Z^n$ denote the character group of $T$ and let 
$\{\l_1, \ldots, \l_{n}\}$  be the fundamental dominant weights for $T$ corresponding to our choice of base $\Delta(G)$. There is a bijection between the set of dominant weights of $G$ and the set of isomorphism classes of irreducible $KG$-modules; if $\l$ is a dominant weight then we use $V_G(\l)$  to denote the unique irreducible $KG$-module with highest weight $\l$, while $W_G(\l)$ denotes the corresponding Weyl module (recall that $W_G(\l)$ has a unique maximal submodule $M_0$ such that $W_G(\l)/M_0 \cong V_G(\l)$, and $M_0$ is trivial if $p=0$). We also recall that if $p>0$ then a dominant weight $\l = \sum_{i}a_i\l_i$  is said to be  $p$-\emph{restricted} if $a_i<p$ for all $i$. By Steinberg's tensor product theorem, every irreducible $KG$-module decomposes in a unique way as a tensor product $V_1 \otimes V_2^{F} \otimes \cdots \otimes V_r^{F^r}$, where $V_i$ is a $p$-restricted irreducible $KG$-module, $F:G \to G$ is a standard Frobenius morphism, and $V_i^{F^i}$ is the $KG$-module obtained by preceding the action of $G$ on $V_i$ by the endomorphism $F^i$. By a slight abuse of terminology, it is convenient to say that every dominant weight is $p$-restricted when $p=0$.

Suppose $G$ and $H$ satisfy the conditions in Hypothesis \ref{h:our}. Write $X=H^0$, so $X$ is a simple algebraic group of rank $m$, say. Fix a maximal torus $T_X$ of $X$ contained in $T$ and let $\{\delta_1, \ldots, \delta_m\}$ be the fundamental dominant weights for $T_X$ corresponding to a choice of base $\Delta(X) = \{\b_1, \ldots, \b_m\}$ for the root system $\Phi(X)$ of $X$. By hypothesis (see condition S2), $X$ acts irreducibly and tensor-indecomposably on $W$; let $\delta = \sum_{i}b_i\delta_i$ denote the highest weight of $W$ as a $KX$-module. Note that $\delta$ is $p$-restricted (see condition S5). In this paper we adopt the standard labelling of simple roots and fundamental weights given in Bourbaki \cite{Bou}. Finally, we will write $Y=\Phi$ to denote a semisimple algebraic group $Y$ with root system $\Phi$.
  
\section{Weights and multiplicities}\label{ss:weights}

Let $V$ be an irreducible $KG$-module with $p$-restricted highest weight $\l = \sum_{i}a_{i}\l_{i}$, that is, let $V=V_G(\lambda)$. Let $\L(V)$ denote the set of weights of $V$ and let $m_{V}(\mu)$ be the multiplicity of a weight $\mu \in \L(V)$, so $m_{V}(\mu)$ is simply the dimension of the corresponding weight space $V_{\mu}$. Recall that we can define a partial order on the set of weights for $T$ by the relation $\mu \preccurlyeq \nu$ \label{p:underr} if and only if $\mu = \nu - \sum_{i=1}^{n}c_i\a_i$ for some non-negative integers $c_i$. In this situation, if $\mu$ and $\nu$ are weights of $V$ and $\mu \preccurlyeq \nu$ then we say that $\mu$ is \emph{under} $\nu$. Note that if $\mu$ is a weight of $V$ then $\mu \preccurlyeq \lambda$.

If $J$ is a closed subgroup of $G$ and $T_J$ is a maximal torus of $J$ contained in $T$ then we abuse notation by writing $\l|_{J}$ to denote the restriction of $\l \in X(T)$ to the subtorus $T_J$. Let $e(G)$ be the maximum of the squares of the ratios of the lengths of the roots in $\Phi(G)$. Here we record some useful preliminary results on weights and their multiplicities.

\begin{lem}\label{l:t1}
If $a_{i} \neq 0$ then $\mu=\l-d\a_{i} \in \L(V)$ for all $1 \le d \le a_{i}$. Moreover $m_{V}(\mu)=1$.
\end{lem}

\begin{proof}
This follows from \cite[1.30]{Test1}.
\end{proof}

Recall that a weight $\mu = \l-\sum_{i}{c_{i}\a_{i}} \in \L(V)$ is \emph{subdominant} if $\mu$ is a dominant weight, that is, $\mu = \sum_{i}d_i\l_i$ with $d_i \ge 0$ for all $i$.

\begin{lem}\label{l:pr}
Suppose $\mu$ is a weight of the Weyl module $W_G(\l)$, and assume $p=0$ or $p>e(G)$. Then $\mu \in \L(V)$. In particular, if $\mu = \l-\sum_{i}{c_{i}\a_{i}}$ is a subdominant weight then $\mu \in \L(V)$.
\end{lem}

\begin{proof}
This follows from \cite[Theorem 1]{Pr}.
\end{proof}

\begin{cor}\label{c:sat}
Suppose $\mu \in \L(V)$, and assume $p=0$ or $p>e(G)$. Then $\mu-k\a \in \L(V)$ for all $\a \in \Phi^{+}(G)$ and integers $k$ in the range $0 \le k \le \la \mu, \a\ra$.
\end{cor}

\begin{proof}
The set of weights of the Weyl module $W_G(\l)$ is saturated (see \cite[Section 13.4]{Hu1}), so the result follows from Lemma \ref{l:pr}.
\end{proof}

\begin{lem}\label{l:s816}
Suppose $G=A_n$ and $\l=a\l_i+b\l_j$ with $i<j$ and $a \neq 0 \neq b$. If $1 \le r \le i$ and $j \le s \le n$ then $\mu=\l-(\a_r+ \cdots +\a_s) \in \L(V)$ and 
$$m_{V}(\mu)=\left\{\begin{array}{ll}
j-i & a+b+j-i \equiv 0 \imod{p} \\
j-i+1 & \mbox{otherwise.}
\end{array}\right.$$
\end{lem}

\begin{proof}
This is \cite[8.6]{Seitz2}.
\end{proof}

\begin{lem}\label{l:nat}
Let $W$ be equipped with a non-degenerate form $f$ and let $Y = Cl(W) = {\rm Isom}(W)'$ be the corresponding simple group. Consider the natural embedding $Y < {\rm SL}(W)$. Then there exists a choice of maximal tori $T_Y< Y$ and $T<{\rm SL}(W)$,
and a choice of Borel subgroups $B_Y<Y$ and $B<{\rm SL}(W)$ 
with $B_Y\leqs B$, such that 
if $\{\delta_1,\dots,\delta_\ell\}$ and
$\{\lambda_1,\dots,\lambda_n\}$ are the associated fundamental dominant 
weights, respectively, then $\lambda_i|_Y = \delta_i$ for $i\leq \ell-2$. In particular, if $\lambda = \sum_{i=1}^n a_i\lambda_i$ is a dominant weight with $a_j\ne 0$ 
for some $j\leq\ell-2$, then
$\lambda|_Y = \sum_{i=1}^\ell b_i\delta_i$, with $b_j\ne 0$. 
\end{lem}

\begin{proof} 
Let $\{\b_1, \ldots, \b_{\ell}\}$ and $\{\a_1, \ldots, \a_n\}$ be bases of the respective root systems dual to the given weights. Consider the stabilizers in $Y$ and ${\rm SL}(W)$ of a maximal totally singular subspace of $W$ (with respect to the form $f$). We may assume that these correspond to the parabolic subgroups whose root systems have bases $\{\b_1, \ldots, \b_{\ell-1}\}$ and $\{\a_1, \ldots, \a_{\ell-1}, \a_{\ell+1}, \ldots, \a_{n}\}$, respectively, and that $\a_{i}|_{Y} = \b_i$ for all $1 \le i \le \ell-1$. In particular, we may assume that $\lambda_1|_Y = \delta_1$. 
Since $\lambda_i=i\lambda_1-\sum_{k=1}^{i-1}(i-k)\alpha_k$ for $i> 1$ (see \cite[Table 1]{Hu1}), the result follows.
\end{proof}

\section{Tensor products and reducibility}\label{ss:tpr}

Suppose $V_1$ and $V_2$ are $p$-restricted irreducible $KY$-modules, where $Y$ is a simply connected simple algebraic group. Then according to \cite[1.6]{Seitz2}, the tensor product $V_1\otimes V_2$ is an irreducible $KY$-module only under some very tight constraints. In particular, if $Y$ is simply laced (that is, if $e(Y)=1$) and at least one of the $V_i$ is nontrivial then $V_1 \otimes V_2$ is reducible (see \cite[1.6]{Seitz2}).

\begin{prop}\label{p:s16}
Let $Y$ be a simple algebraic group and let $V=V_Y(\l)$ be a $p$-restricted irreducible $KY$-module. Then $V$ can be expressed as a tensor product $V = V_1 \otimes V_2$ of two nontrivial $p$-restricted irreducible $KY$-modules $V_i = V_Y(\mu_i)$ if and only if the following conditions hold:
\begin{itemize}\addtolength{\itemsep}{0.3\baselineskip}
\item[{\rm (i)}] $Y$ has type $B_n, C_n, F_4$ or $G_2$, with $p=2,2,2,3$, respectively.
\item[{\rm (ii)}] $\l = \mu_1+\mu_2$ and $\mu_1$ (respectively $\mu_2$) has support on the fundamental dominant weights corresponding to short (respectively long) simple roots.
\end{itemize}
\end{prop}

In particular, if $Y$ is simple and $e(Y)=1$ then every $p$-restricted irreducible $KY$-module is tensor indecomposable.

\section{Invariant forms}\label{ss:if}

Define $G,H$ and $W$ as in Hypothesis \ref{h:our}, and set $X=H^0$. Recall that $W$ is self-dual as a $KX$-module, so $X$ fixes a non-degenerate bilinear form $f$ on $W$. We state a result of Steinberg (see \cite[Lemma 79]{Stein}), which determines the nature of this form. To state the result, we require some additional notation. For $\b \in \Phi(X)$ and $c \in K^*$, let $x_{\b}(c)$ be an element of the corresponding $T_X$-root subgroup of $X$ and set \label{p:hbc}
\begin{equation}\label{e:hbc}
w_{\b}(c) = x_{\b}(c)x_{-\b}(-c^{-1})x_{\b}(c), \;\; h_{\b}(c) = w_{\b}(c)w_{\b}(-1) \in T_X.
\end{equation}

\begin{lem}\label{l:st}
Let $h = \prod{h_{\b}(-1)}$, the product over the positive roots of $X$. Then $f$ is symmetric if $\delta(h)=1$, and skew-symmetric if $\delta(h)=-1$, where $\delta$ is the $T_X$-highest weight of the $KX$-module $W$. In particular, if $p \neq 2$ then $G={\rm SO}(W)$ if and only if $\delta(h)=1$.
\end{lem}

\section{Connected subgroups}

Recall that in Hypothesis \ref{h:our} we are interested in the irreducible triples $(G,H,V)$ where $V|_{H}$ is irreducible but $V|_{H^0}$ is reducible. We impose the latter condition because with the aid of \cite[Table 1]{Seitz2} it is straightforward to determine all examples with $V|_{H^0}$ irreducible (assuming $V$ is not the natural $KG$-module, nor its dual; see Theorem \ref{t:seitz1} below). To state this result, we define the following hypothesis:

\begin{hyp}\label{h:ourprimeprime}
$G$ and $H$ are given as in Hypothesis \ref{h:our}, and $V$ is a rational
tensor indecomposable $p$-restricted irreducible $KG$-module such that $V|_{H^0}$ is irreducible and $V$ is not the natural $KG$-module, nor its dual.
\end{hyp}

Also, in Table \ref{t:seitztab} we adopt the labelling of cases used in \cite[Table 1]{Seitz2}.

\begin{thm}\label{t:seitz1}
The triples satisfying Hypothesis \ref{h:ourprimeprime} are listed in Table \ref{t:seitztab}. Moreover, in each of these cases we have $H \leqs G$. 
\end{thm}

\renewcommand{\arraystretch}{1.2}
\begin{table}
$$\begin{array}{lllllll} \hline
{\rm no.} & G & H  & \l & \delta & \l|_{H^0} & \mbox{Conditions} \\ \hline
{\rm II}_{1} & C_{10} & A_5.2 & \l_2 & \delta_3 & \delta_2+\delta_4 & p \neq 2  \\

{\rm S}_{1} & B_3 & A_2.2 & 2\l_1 & \delta_1+\delta_2 & 2\delta_1+2\delta_2 & p = 3 \\

{\rm S}_{7} & D_7 & A_3.2 & \l_6 & \delta_1+\delta_3 & \delta_1+\delta_2+\delta_3 & p = 2 \\

{\rm S}_{8} & D_{13} & D_4.Y & \l_{12} & \delta_2 & \delta_1+\delta_2+\delta_3 +\delta_4 & p = 2,\; 1 \neq Y \leqs S_3  \\ \hline
\end{array}$$
\caption{The triples $(G,H,V)$ satisfying Hypothesis \ref{h:ourprimeprime}}
\label{t:seitztab}
\end{table}
\renewcommand{\arraystretch}{1}

\begin{proof}
First, by inspecting \cite[Table 1]{Seitz2} we identify the cases where $G$ is a classical group, $H^0$ is of type $A_m$, $D_m$ or $E_6$, and the highest weight of $W|_{H^0}$ is invariant under a suitable graph automorphism of $H^0$. The relevant cases are labelled ${\rm I}_{4}$, ${\rm I}_{5}$, ${\rm I}_{6}$ (with $n=3$), ${\rm II}_{1}$, ${\rm S}_{1}$, ${\rm S}_{7}$, ${\rm S}_{8}$ and ${\rm MR}_{4}$ in \cite[Table 1]{Seitz2}. We can discard ${\rm I}_{4}$, ${\rm I}_{5}$ and ${\rm MR}_{4}$ since in each of these cases $W$ is the natural $KH^0$-module, which is incompatible with conditions S3 and S4 in Hypothesis \ref{h:ourprimeprime}. Similarly, the case ${\rm I}_{6}$ (with $n=3$) is also incompatible with Hypothesis \ref{h:ourprimeprime} (here $W$ is the wedge-square of the natural module for $A_3$, so $G = D_3 \cong A_3$, which is ruled out by condition S1). Each of the remaining four cases corresponds to a triple $(G,H,V)$ satisfying Hypothesis \ref{h:ourprimeprime}, and these are the cases listed in Table \ref{t:seitztab}. 

To complete the proof of the theorem, it remains to show that $H \leqs G$ in each of the cases in Table \ref{t:seitztab}. This is clear in the first two cases since $C_{10}$ and $B_{3}$ have no nontrivial outer automorphisms. Similarly, we clearly have $D_4.3<D_{13}$ in the case labelled ${\rm S}_{8}$. To see that $D_4.2<D_{13}$ in this case, first observe that $D_4.S_3<F_4$, and we may identify $W$ with the irreducible $KF_4$-module $V_{F_4}(\l_1)$. The Jordan form of every unipotent element in $F_4$ on $V_{F_4}(\l_1)$ is recorded in \cite[Table 3]{Law}. In particular, since $p=2$, we deduce that every involution in $F_4$ (and thus every involution in $D_4.S_3$) acts on $W$ with an even number of unipotent Jordan blocks of size $2$. Therefore, every involution in $D_4.S_3$ lies in the simple group ${\rm SO}(W)=D_{13}$ (this follows from \cite[Section 7]{AS}, for example), hence $D_4.S_3 < D_{13}$ as claimed.

Finally, let us consider the case labelled ${\rm S}_7$ in Table \ref{t:seitztab}. Here $G=D_{7}$, $H=A_3.2$ and $W$ is the unique nontrivial composition factor of $\mathcal{L}(X)$,  where $\mathcal{L}(X)$ denotes the Lie algebra of $X=A_3$. A straightforward calculation with $\mathcal{L}(X)$ reveals that the standard graph automorphism of $A_3$ has Jordan form $[J_2^4,J_1^6]$ on $W$ (where $J_i$ denotes a standard unipotent Jordan block of size $i$). We conclude that $A_3.2<D_{7}$ as required.
\end{proof} 

We now observe that Theorem \ref{t:main11} follows from Theorems \ref{t:main} and \ref{t:seitz1}. Moreover, since Theorem \ref{t:new} implies Theorem \ref{t:main}, and Theorem \ref{t:main2} is a corollary of Theorem \ref{t:new} and the results in \cite{Seitz2} and \cite{BGT}, for the remainder of the paper we will focus on classifying the triples $(G,H,V)$ satisfying Hypothesis \ref{h:our}. (The proof of Theorem \ref{T:MAIN3} is given in Section \ref{s:thm5}.)

\section{Clifford theory}\label{ss:ct}

By Theorem \ref{t:seitz1}, we can focus on the irreducible triples $(G,H,V)$ such that $V|_{H^0}$ is reducible. Let $(G,H,V)$ be such a triple, and set $X=H^0$. If $H=X\langle t \rangle = X.2$, where $t$ is an involutory graph automorphism of $X$, then Clifford theory implies that 
$V|_{X}=V_1 \oplus V_2$, where $V_1$ and $V_2$ are irreducible $KX$-modules interchanged by $t$. More generally, if $H= X.F$ is any extension of $X$ by a finite group $F$ then 
\begin{equation}\label{e:vx}
V|_{X}=V_1 \oplus \cdots \oplus V_m,
\end{equation}
where $m$ divides the order of $F$, and the $V_i$ are irreducible $KX$-modules that are transitively permuted under the action of $F$ (in particular, the $V_i$ are equidimensional). 
In the following results, we consider the special case where $F$ is a cyclic group (note that under our hypotheses, $F$ is cyclic unless $H=D_4.S_3$).
First, we record a trivial observation.

\begin{lem}\label{l:vxsum}
If $H=X.F$ and $F$ has prime order then $\dim V$ is divisible by $|F|$. 
In particular, $\dim V$ is even if $H=X.2$.
\end{lem}

\begin{prop}\label{p:niso}
If $H$ is a cyclic extension of $X$ then the irreducible $KX$-modules $V_i$ in \eqref{e:vx} are pairwise non-isomorphic.
\end{prop}

To prove this proposition, we require a well known preliminary lemma (for completeness, we include a proof).

\begin{lem}\label{l:lemiso}
Suppose $H$ is a cyclic extension of $X$, and let $M$ be an irreducible $KX$-module with corresponding representation $\rho : X \to {\rm GL}(M)$. Assume that the $H$-conjugates of $M$
are all isomorphic to $M$. Then $\rho$ can be extended to a homomorphism 
$\rho :H\to {\rm GL}(M)$.
\end{lem}

\begin{proof}
Let $a$ be the order of $H/X$ and fix a generator $H/X = \la Xf \ra$. By hypothesis,
there exists $g\in {\rm GL}(M)$ such that 
\begin{equation}\label{rel:1}
\rho(fxf^{-1}) = g\rho(x)g^{-1}
\end{equation}
for all $x\in X$. (Note that for any other choice $g'\in{\rm GL}(M)$
satisfying the same relation, we have $g^{-1}g'\in C_{{\rm GL}(M)}(X)$,
so $g^{-1}g'$ is a scalar by Schur's lemma.) Since $f^a\in X$, \eqref{rel:1} implies that 
$g^{-a}\rho(f^a)$ acts as a scalar on $M$, say $\xi \in K$. 
Choose $\eta\in K$ such that $\eta^a = \xi$ 
and so $\rho(f^a) = (\eta g)^a$. Now define 
$\rho:H\to {\rm GL}(M)$ by $\rho(xf^b) = \rho(x)(\eta g)^b$. 
This map is well-defined
by construction. (If $x_1f^b = x_2f^c$ then $x_2^{-1}x_1 = f^{c-b}$, so $a$ divides $c-b$ and $\rho(f^{c-b}) = \rho(f^{a\ell}) = (\eta g)^{a\ell} = 
(\eta g)^{c-b}$ as required.) It remains to check that $\rho:H\to {\rm GL}(M)$ is a homomorphism:
\begin{align*}
\rho(n_1f^b n_2f^c) &= \rho(n_1 f^b n_2 f^{-b}f^{b+c}) \\
&= \rho(n_1)\rho(f^b n_2 f^{-b})(\eta g)^{b+c}\\
& = \rho(n_1)g^b\rho(n_2)g^{-b}\eta^{b+c}g^{b+c}, \mbox{ by } (\ref{rel:1})\\
& = \rho(n_1f^b)\rho(n_2f^c) \qedhere
\end{align*} 
\end{proof}

\begin{proof}[Proof of Proposition \ref{p:niso}]
Let $H=X\la s \ra$ be a cyclic extension of $X$ and let $V$ be an irreducible $KG$-module 
on which $X$ acts reducibly. Let $V_1$ be a $KX$-composition factor of $V|_{X}$ and assume 
that $(V_1)^s \cong V_1$ as $KX$-modules. To prove the proposition, it suffices to show that $V|_{H}$ is 
reducible.

By hypothesis, $V$ is a homogeneous
$KX$-module $V\cong V_1\oplus \cdots\oplus V_1$. 
Let $J = X\langle f\rangle$ be a cyclic extension of $X$ isomorphic to $H$.
By Lemma \ref{l:lemiso}, the representation $\rho:X\to {\rm GL}(V_1)$ extends
to a representation $\rho:J \to {\rm GL}(V_1)$, so  
$\varphi:J \to {\rm GL}(V_1)\times\cdots\times {\rm GL}(V_1)\leqs {\rm GL}(V)$. 
Since $H\cong J$, we have $\varphi(f)^{-1}s\in C_{{\rm GL}(V)}(X)$. By \cite[Proposition 18.1]{MT} we have  a tensor product decomposition 
$$V = V_1\otimes U$$ where $U={\rm Hom}_{KX}(V_1,V)$ and  $$C_{{\rm GL}(V)}(X) = 1\otimes {\rm GL}(U), \;\; C_{{\rm GL}(V)}(1\otimes {\rm GL}(U)) = {\rm GL}(V_1)\otimes 1.$$
Hence $s = \varphi(f)c$ with $c\in 1\otimes {\rm GL}(U)$. Now
$\varphi(f)\in {\rm GL}(V_1)\otimes 1$ and hence $\varphi(f)$ centralizes $c$. But this implies that $s$ also centralizes $c$, so $s \in {\rm GL}(V_1)\otimes 1$. In particular, $H = X\la  s \rangle$ stabilizes
$V_1$ and so acts reducibly on $V$. 
\end{proof}

\section{Parabolic embeddings}\label{ss:parabs}

Let us continue to define $G,H$ and $W$ as in Hypothesis \ref{h:our}. Recall that we may assume $X=H^0$ is simple of type $A_m$, $D_m$ or $E_6$, so $e(X)=1$.
Let $P_X = Q_XL_X$ be a parabolic subgroup of $X$, where $Q_X=R_u(P_X)$ is the unipotent radical of $P_X$, and $L_X$ is a Levi factor containing $T_X$ with root system $\Phi(L_X')$. Set $Z=Z(L_X)^0$. We choose $P_X$ so that 
$Q_X$ is generated by the $T_X$-root subgroups of $X$ corresponding to the roots $\Phi^{-}(X) \setminus \Phi(L_X')$. Define a sequence of subspaces $[W,Q_X^i]$ of $W$ (the natural $KG$-module) by setting $[W,Q_X^0] = W$ and \label{p:wqx}
$$[W,Q_X^i]=\la qw-w \,:\, w \in [W,Q_X^{i-1}], q \in Q_X\ra$$
for all $i \ge 1$. The flag
$$W > [W,Q_X] > [W,Q_X^2] > \cdots > 0$$
is called the \emph{$Q_X$-commutator series of $W$}, and following Ford \cite{Ford1} (and initially Seitz \cite{Seitz2}), we can use this flag to 
construct a parabolic subgroup $P$ of $G$ (with Levi decomposition $P=QL$) with some desirable properties.

\begin{lem}\label{l:flag}
The $G$-stabilizer of the $Q_X$-commutator series
$$W > [W,Q_X] > [W,Q_X^2] > \cdots > 0$$
is a parabolic subgroup $P=QL$ of $G$ with the following properties:
\begin{itemize}\addtolength{\itemsep}{0.3\baselineskip}
\item[{\rm (i)}] $P_X \leqs P$ and $Q_X \leqs Q = R_u(P)$; 
\item[{\rm (ii)}] $L=C_{G}(Z)$ is a Levi factor of $P$ containing $L_X$, where $Z=Z(L_X)^0$;
\item[{\rm (iii)}] If $T$ is a maximal torus of $G$ containing $T_X$ then $T \leqs L$.
\end{itemize}
\end{lem}

\begin{proof}
This is \cite[2.7]{Ford1} (see also \cite[2.8]{Seitz2}).
\end{proof}

Let $\delta = \sum_{i}b_i\delta_i$ be the highest weight of the $KX$-module $W$, and recall that $\delta$ is $p$-restricted. 
Let $\mu$ be a $T_X$-weight of $W$, so $\mu = \delta - \sum_{i}{d_i\b_i}$ for some non-negative integers $d_i$. We define the \emph{$Q_X$-level} of $\mu$ to be the sum of the coefficients $d_i$ for $\b_i \in \Delta(X) \setminus \Delta(L_X')$, while the \emph{$i$-th $Q_X$-level of $W$}, denoted by $W_{i}$, \label{p:wi} is the sum of the weight spaces $W_{\mu}$ for weights $\mu$ of $Q_X$-level $i$. In addition, the \emph{$Q_X$-shape} of $\mu$ is the ordered $t$-tuple of integers $(d_{i_{1}}, \ldots, d_{i_{t}})$, where $i_{1} < i_{2} < \cdots < i_{t}$ and $\Delta(X) \setminus \Delta(L_X')=\{\b_{i_{1}}, \ldots, \b_{i_{t}}\}$.

\begin{lem}\label{l:com}
Let $i$ be a non-negative integer. 
Then the following hold:
\begin{itemize}\addtolength{\itemsep}{0.3\baselineskip}
\item[{\rm (i)}] $[W,Q_X^{i}] = \bigoplus_{\mu} W_{\mu}$, where the sum runs over the weights $\mu$ of $Q_X$-level at least $i$.
\item[{\rm (ii)}] $[W,Q_{X}^{i}] / [W,Q_{X}^{i+1}] \cong W_{i}$ as $KL_X$-modules.
\end{itemize}
\end{lem}

\begin{proof}
We apply \cite[2.3]{Seitz2}, given that $e(X)=1$ (since $X=A_m,D_m$ or $E_6$) and $\delta$ is $p$-restricted.
\end{proof}

\begin{lem}\label{l:vq}
With the notation established, the following hold: 
\begin{itemize}\addtolength{\itemsep}{0.3\baselineskip}
\item[{\rm (i)}] Let $V=V_G(\l)$ be an irreducible $KG$-module. Then $V/[V,Q]$ is an irreducible module for $L$ and $L'$, with respective highest weights $\l$ and $\l|_{T \cap L'}$.
\item[{\rm (ii)}] $W/[W,Q_X]$ is an irreducible module for $L_X$ and $L_X'$, with respective highest weights $\delta$ and $\delta|_{T_X \cap L_X'}$.
\end{itemize}
\end{lem}

\begin{proof}
This follows from 2.1 and 2.10 in \cite{Seitz2}.
\end{proof}

The above result shows that $W_0$ is an irreducible $KL_X'$-module. In many cases, for $i>0$, $W_i$ is a reducible $KL_X'$-module.

\begin{lem}\label{l:cl}
Suppose a $Q_X$-level $W_i$ contains two weights with distinct $Q_X$-shapes. Then $W_i$ is a reducible $KL_X'$-module.
\end{lem}

\begin{proof}
Since $L_X'$ preserves the sum of the $T_X$-weight spaces of level $i$ with a given $Q_X$-shape, the result follows.
\end{proof}

Recall that Hypothesis \ref{h:our} implies that $X$ fixes a form $f$ on $W$, so $G$ is a symplectic or orthogonal group. Let $w_0$ denote the longest word in the Weyl group of $X$, and let $\ell$ be the $Q_X$-level of the lowest weight $w_0(\delta)$, so $\ell$ is minimal such that $[W,Q_X^{\ell+1}]=0$. Note that $w_0 = -1$ or $-\tau$ (where $\tau$ is an involutory graph automorphism of $X$), and $\tau(\delta)=\delta$ since $\tau$ acts on $W$, so in all cases the lowest weight is $-\delta$. Set $\ell' = \lfloor \ell/2 \rfloor$ and let $P=QL$ be the parabolic subgroup of $G$ constructed in Lemma \ref{l:flag}.
If $\ell$ is odd then each $W_i$ is totally singular, whereas if $\ell$ is even then $W_i$ is totally singular for all $i\neq \ell/2$, and $W_{\ell/2}$ is non-degenerate (see Ford \cite[Section 3.1]{Ford2} for more details). Moreover, as $KL_X'$-modules we have $(W_i)^* \cong W_{\ell-i}$ for all $0 \le i \le \ell'$. In view of Lemma \ref{l:com}(ii), we deduce that $L = R_0 \cdots R_{\ell'}$, where $R_i = {\rm Isom}(W_i)$ is the full isometry group of the form on $W_i$ induced by $f$. 

In order to describe the derived group $L'$, set $s=1$ if $G$ is symplectic and $s=2$ otherwise. Let $\{e_1,\dots,e_r\}$ be the subsequence of $\{0, \dots, \ell' \}$ such that $\dim W_{e_i}>s$ if $\ell$ is even and  $e_i=\ell/2$, and $\dim W_{e_i}>1$ otherwise. Then
\begin{equation}\label{e:ldash}
L'=L_{e_1}\cdots L_{e_r}
\end{equation}
where each $L_{e_i} ={\rm Isom}(W_{e_i})'$ is nontrivial.

For instance, consider the specific case
$$X = A_3,\; p \neq 2, \; \delta = 2\delta_2$$
with $P_X=U_XT_X$ a Borel subgroup of $X$. 
By \cite[Table A.7]{Lubeck}, if $p= 3$  then $\dim W=19$ and so $G=B_9$, otherwise $\dim W=20$ and $G=D_{10}$ (this follows from Lemma \ref{l:st}, since $\delta(h_{\b}(-1))=1$ for all $\b \in \Phi^+(X)$). Now $\ell=8$ and the $T_X$-weights in each $U_X$-level $W_i$, $0 \le i \le 4$, are recorded in Table \ref{tab:neww}, together with $\dim W_i$. Here we are using Lemma \ref{l:pr} to see that the listed $T_X$-weights are in $\L(W)$ (they are clearly weights of the corresponding Weyl module $W_X(\delta)$). Note that each  $T_X$-weight space in $W$ is $1$-dimensional, with the possible exception of the zero weight $\delta-\beta_1-2\beta_2-\beta_3$, which has multiplicity $2-\delta_{3,p}$. We therefore obtain $L'=L_2L_3L_4$, where $\Delta(L_2)=\{\a_3,\a_4\}$, $\Delta(L_3)=\{\alpha_6,\alpha_7\}$, and either $\Delta(L_4)=\{\alpha_9,\alpha_{10}\}$, or $p=3$ and $\Delta(L_4)=\{\alpha_9\}$. 

\renewcommand{\arraystretch}{1.2}
\begin{table}
$$\begin{array}{lll} \hline
i & \mbox{$T_X$-weights in $W_i$} & \dim W_i \\ \hline
0 & \delta & 1 \\
1 & \delta-\beta_2 & 1 \\
2 & \delta-\beta_1-\beta_2, \, \delta-2\beta_2, \, \delta-\beta_2-\beta_3 & 3 \\
3 & \delta-\beta_1-\beta_2-\beta_3, \, \delta-\beta_1-2\beta_2, \, \delta-2\beta_2-\beta_3 & 3 \\
4 &  \delta-\beta_1-2\beta_2-\beta_3, \, \delta-2\beta_1-2\beta_2, \, \delta-2\beta_2-2\beta_3 & 4-\delta_{3,p} \\ \hline 
\end{array}$$
\caption{}
\label{tab:neww}
\end{table}
\renewcommand{\arraystretch}{1}

\begin{rmk}\label{r:a1factor}
In our later analysis, taking $P_X$ to be a Borel subgroup of $X$, it is important to determine whether or not the Levi factor of the associated parabolic subgroup $P=QL$ of $G$ has a factor isomorphic to $A_1$  (see Lemma \ref{l:main}, for example). Clearly, this depends on the dimension of the $Q_X$-levels; indeed, $L_{e_i}$ has an $A_1$ factor if and only if one of the following holds:
\begin{itemize}\addtolength{\itemsep}{0.3\baselineskip}
\item[(i)] $e_i\neq \ell/2$ and $\dim W_{e_i}=2$;
\item[(ii)] $\ell$ is even, $e_i=\ell/2$ and either $G$ is symplectic and $\dim W_{\ell/2}=2$, or $G$ is orthogonal and $\dim W_{\ell/2}=3$ or $4$. 
\end{itemize}
\end{rmk}

Let $V=V_G(\l)$. By Lemma \ref{l:vq}(i), $V/[V,Q]$ is an irreducible $KL'$-module so 
$$V/[V,Q] = M_1 \otimes \cdots \otimes M_r$$ 
where each $M_i$ is a $p$-restricted irreducible $KL_{e_i}$-module. Since $L_X \leqs L$ we have $L_X' \leqs L'$, so every $KL_{e_i}$-module can be regarded as a $KL_X'$-module via the $i$-th projection map $\pi_i:L_X' \to L_{e_i}$. We note that $\pi_i(L_X')$ may be trivial.

In view of Remark \ref{r:a1factor}, we will need to determine whether or not $\dim W_i>4$ for a given $Q_X$-level $W_i$. With this aim in mind, the following result will be useful in induction arguments.

\begin{lem}\label{l:indwt}
Let $B_X=U_XT_X$ be a Borel subgroup of $X$ and let $\mu=\delta-\sum_{i=1}^mk_i\beta_i$ be a dominant weight, where the $k_i$ are non-negative integers such that $\sum_{i}k_i=b$. Suppose the $KX$-module $V_X(\mu)$  has at least $r$ distinct weights at 
$U_X$-level $a$. Then $V_X(\delta)$ has at least $r$ distinct weights at $U_X$-level $a+b$. 
\end{lem}

\begin{proof}
Let $\nu$ be a weight of $V_X(\mu)$ at $U_X$-level $a$, so $\nu=\mu-\sum_{i=1}^ml_i\beta_i$ for some non-negative integers $l_i$ with $\sum_{i}l_i=a$. There exists an element $\sigma$ in the Weyl group of $X$ such that $\sigma(\nu)$ is subdominant to $\mu$. Since $\mu=\delta-\sum_{i}k_i\beta_i$, $\sigma(\nu)$ is also subdominant to $\delta$, and thus Lemma \ref{l:pr} implies that
$\nu=\delta-\sum_{i}(k_i+l_i)\beta_i$ is a weight in $V_X(\delta)$ at level $a+b$ (note that Lemma \ref{l:pr} is applicable since $X$ is of type $A_m, D_m$ or $E_6$, so $e(X)=1$). Finally, we note that distinct weights in the $a$-th level of $V_X(\mu)$ give rise to distinct weights at the $(a+b)$-th level of $V_X(\delta)$. The result follows.
\end{proof}

\begin{rmk}\label{r:ord}
For $0 \leq i< \ell/2$, let $d_i=\sum_{j=0}^{i} \dim W_j$ and set $d_{-1}=0$. The natural $KG$-module $W$ may be taken to be $V_G(\l_1)$, so we can label the $T$-weights in each $Q_X$-level $W_i$ as follows:  
$$\nu_{i,1}= \lambda_1-\sum_{j=1}^{d_{i-1}}\a_j,\; \nu_{i,2}= \lambda_1-\sum_{j=1}^{d_{i-1}+1}\a_j, \ldots, 
\nu_{i,d_{i}-d_{i-1}}=\lambda_1-\sum_{j=1}^{d_{i}-1}\a_j$$ 
if $i \geq 1$, and 
$$\nu_{0,1}= \lambda_1, \nu_{0,2} = \lambda_1-\a_1, \ldots, \nu_{0,d_0}=\lambda_1-\alpha_1-\dots-\alpha_{d_0-1} $$ 
if $i=0$.  Now $R_i'={\rm Isom}(W_i)' = A_{d_i-d_{i-1}-1}$ and so the Weyl group of $R_i'$ contains the full set of permutation matrices on $W_i$.  Therefore, by conjugating by an element of $R_i'$ if necessary, we may assume that if $\{\theta_{i,1},\dots,\theta_{i,d_{i}-d_{i-1}}\}$ is a complete ordered set of $T_X$-weights of $W_i$ (including multiplicities) then $\theta_{i,j}=\nu_{i,j}|_{X}$ for all $1\leq j \leq d_{i}-d_{i-1}$.
\end{rmk}

\begin{rmk}\label{r:ldash}
In the proofs of our main theorems, we will often consider the decomposition of $L'$ given in \eqref{e:ldash}. However, it will be convenient to renumber the factors, so that $L'=L_1 \cdots L_r$, where $L_i={\rm Isom}(W_{e_i})'$.
\end{rmk}

\section{Some remarks on the case $H=H^0.2$}\label{ss:x2}

In this section we briefly discuss the generic case $H=X\langle t \rangle$, where $t$ is an involutory graph automorphism of $X = H^0$. Note that $H$ has this form, unless $X=D_4$ and $H$ contains a triality graph automorphism (in which case $H=X.3$ or $X.S_3$). Recall that $V|_{X}=V_1\oplus V_2$ where $V_1$ and $V_2$  are  irreducible $KX$-modules permuted by $t$.

As in Section \ref{ss:parabs}, let $P_X=Q_XL_X$ be a parabolic subgroup of $X$ and let $P=QL$ be the parabolic subgroup of $G$ constructed from the $Q_X$-levels of $W$ (see Lemma \ref{l:flag}). By construction, we have $Q_X \leqs Q$ and thus $[V,Q_X] \leqs [V,Q]$, so $V/[V,Q]$ is a quotient of
$$V/[V,Q_X] = V_1/[V_1,Q_X] \oplus V_2/[V_2,Q_X],$$
where each summand is an irreducible $KL_X'$-module (by Lemma \ref{l:vq}).
Since $L_X' \leqs L'$, either $V/[V,Q]$ is irreducible as a $KL_X'$-module, or 
$$V/[V,Q] = V/[V,Q_X] = V_1/[V_1,Q_X] \oplus V_2/[V_2,Q_X]$$
as $KL_X'$-modules.
If $P_X$ is $t$-stable we can in fact say more. 

\begin{lem}\label{l:tstable}
Let $P_X=Q_XL_X$ be a $t$-stable parabolic subgroup of $X$, embedded in a parabolic subgroup $P=QL$ of $G$ as before. Then $P$ is $t$-stable and
$$V/[V,Q] = V/[V,Q_X] = V_1/[V_1,Q_X] \oplus V_2/[V_2,Q_X]$$
as $KL_X'$-modules.
\end{lem}

\begin{proof}
See Lemmas 2.8 and 2.9 in \cite{Ford1}.
\end{proof}

Suppose $P_X=B_X$ is a $t$-stable Borel subgroup of $X$ containing a $t$-stable maximal torus $T_X$, with Levi decomposition $B_X = U_XT_X$. The analysis of this special case will play an essential role in the proof of our main theorems. Note that each space $V_i/[V_i,U_X]$ is 
$1$-dimensional (it is spanned by the image of a maximal vector for the $T_X$-highest weight of $V_i$), so Lemma \ref{l:tstable} implies that $V/[V,Q]$ is a $2$-dimensional irreducible $KL'$-module. But recall that $L'=L_{e_1} \cdots L_{e_r}$ and $V/[V,Q] = M_1 \otimes \cdots \otimes M_r$, where each $M_i$ is an irreducible $KL_{e_i}$-module, so $L'$ must have a factor isomorphic to $A_1$. In turn, as noted in Remark \ref{r:a1factor}, this can be translated into a condition on the dimensions of the $U_X$-levels of $W$ (the natural $KG$-module).

As the next lemma demonstrates, we can use this observation to obtain important restrictions on the coefficients of the highest weight $\l = \sum_{i}a_i\l_i$ of $V$.

\begin{lem}\label{l:main}
Let $P=QL$ be the parabolic subgroup of $G$ constructed from a $t$-stable Borel subgroup of $X$ and write $L'=L_{1}\cdots L_{r}$ as a product of simple factors. Then 
there exists an $i$ such that $L_i = A_1$. Moreover if  $\Delta(L_i)=\{\a\}$ then $\la \l,\a\ra = 1$ and $\la \l,\b\ra = 0$ for all $\b \in \Delta(L') \setminus \{\a\}$.
\end{lem}

\begin{proof}
This is \cite[Lemma 5.2]{Ford1}. 
\end{proof}

\section{Reducible subgroups}\label{ss:red}

In the proofs of our main theorems we will frequently appeal to \cite[Theorem 5.1]{Seitz2} to eliminate certain configurations. Roughly speaking, this result states that if $Y$ is a semisimple closed subgroup of $G$ acting irreducibly on $V$, where $V$ is a nontrivial $p$-restricted irreducible $KG$-module (and $V$ is not the natural $KG$-module $W$ nor its dual) then $W|_{Y}$ is irreducible, with the exception of a small number of very specific examples. Here we record a version of \cite[Theorem 5.1]{Seitz2} which will be suitable for our purposes (in particular, we assume $(G,p) \neq (B_n,2)$ -- see Remark \ref{r:new22}, in the Introduction).

\begin{thm}\label{t:gary51}
Let $G$ be a simple classical algebraic group with natural module $W$ and let $V$ be a nontrivial $p$-restricted irreducible tensor indecomposable  $KG$-module with highest weight $\l$ such that $V \neq W, W^*$.
Let $Y$ be a semisimple closed subgroup of $G$ such that $V|_{Y}$ is irreducible. Then either $W|_{Y}$ is irreducible, or one of the following holds:
\begin{itemize}\addtolength{\itemsep}{0.3\baselineskip}
\item[{\rm (i)}] $G=C_n$, $Y \leqs C_k \times C_{n-k}$ with $0<k<n$, $\l=\l_n$, $p = 2$;
\item[{\rm (ii)}] $G=D_{n}$, $Y \leqs B_k \times B_{n-1-k}$ with $0<k<n-1$, $\l = \l_{n-1}$ or $\l_{n}$;
\item[{\rm (iii)}] $G=D_{n}$, $Y \leqs B_{n-1}$ and $\l = a\l_{n-1}+b\l_i$ or $a\l_{n}+b\l_i$, where either $b=0$, or $a,b \neq 0$ and $a+b+n-i-1 \equiv 0 \pmod{p}$. 
\end{itemize}
\end{thm}

\begin{cor}\label{c:g51}
Let $G,W$ and $V$ be as in Theorem \ref{t:gary51}. Let $Y=A_m$ be a simple closed subgroup of $G$ such that $m>1$ and $V|_{Y}$ is irreducible. Then $W|_{Y}$ is irreducible.
\end{cor}

\begin{proof}
Seeking a contradiction, let us assume $W|_{Y}$ is reducible. By Theorem \ref{t:gary51}, we are in one of the situations labelled (i)--(iii). A similar argument applies in each of these cases, so let us assume (i) holds. Write $C_k \times C_{n-k} = J_1 \times J_2$. 
Since $p=2$ and $V$ has highest weight $\l_n$ it follows that $V|_{J_1 \times J_2} = U_1 \otimes U_2$, where each $U_i$ is a $p$-restricted irreducible $KJ_i$-module of dimension $2^k$ and $2^{n-k}$, respectively (see the case labelled ${\rm MR}_{5}$ in \cite[Table 1]{Seitz2}). 
Let $\pi_i:Y \to J_i$, $i=1,2$, denote the projection maps. Without loss of generality, we may assume that $\pi_1$ is nontrivial. In particular, $\pi_1(Y)$ is a closed simple $A_m$-type subgroup of $J_1$. Let $\{\delta_1, \ldots, \delta_k\}$ be a set of fundamental dominant weights for $J_1$ (labelled in the usual way), and note that $U_1$ has highest weight $\delta_{k}$. We now have a configuration $(A_m,C_k,U_1)$ (with $p=2$), which must be one of the cases listed in \cite[Table 1]{Seitz2}. However, it is easy to see that no such example arises. This is a contradiction. 
\end{proof}

We will typically apply Corollary \ref{c:g51} in the context of the following situation. Let $P_X=Q_XL_X$ be a parabolic subgroup of $X$ and let $P=QL$ be the corresponding parabolic subgroup of $G$ constructed in Lemma \ref{l:flag}. Write $L'=L_1\cdots L_r$, where each $L_i$ is simple. By Lemma \ref{l:vq}(i), $V/[V,Q]$ is an irreducible $KL'$-module, so we may write 
$$V/[V,Q]=M_1\otimes \cdots \otimes M_r$$ 
with each $M_i$ a $p$-restricted irreducible $KL_i$-module. The natural $KL_i$-module corresponds to one of the $Q_X$-levels of $W$; if $L_X'$ acts nontrivially on the appropriate $Q_X$-level, the projection of $L_X'$ into $L_i$ is a nontrivial semisimple subgroup of $L_i$. We will therefore view $L_X'$ as a subgroup of $L_i$. Moreover, with the aid of Lemma \ref{l:cl} for example, we may be able to show that this $Q_X$-level is reducible as a $KL_X'$-module. Let us assume that this is the case. Consider the configuration $(L_X',L_i,M_i)$, and  assume that $M_{i}|_{L_X'}$ is irreducible. Then Theorem \ref{t:gary51} implies that either we are in one of the cases labelled (i)--(iii) in the theorem, or $M_i$ is either trivial or the natural module (or its dual) for $L_i$. In this way we obtain very useful information on the coefficients $a_i$ in the highest weight $\l$ corresponding to the simple roots in $\Delta(L_i)$. 

\section{Some $A_1$-restrictions}

To close this preliminary section we record some results on composition factors of $KA_1^k$-modules, which will be needed later. The first result is well known (see \cite[1.13]{Seitz2}).

\begin{lem}\label{l:wsirra1}
Let $M$ be an irreducible $KA_1$-module. Then every weight space of $M$ is 
$1$-dimensional.
\end{lem}

\begin{lem}\label{l:a1a3}
Let $G=A_3$ and $V = V_G(\l)$, where $\l =\sum_{i=1}^3 a_i\lambda_i$ is a nontrivial $p$-restricted weight. Let $W$ be the natural module for $G$ and let $Y=A_1$ be a subgroup of $G$ such that $W|_{Y} = U \oplus U$, where $U$ is the natural $2$-dimensional module for $Y$.  
\begin{itemize}\addtolength{\itemsep}{0.3\baselineskip} 
\item[{\rm (i)}] If $\{a_1,a_3\}=\{0,2\}$, or $a_1=a_3=0$, or $a_1a_2\neq 0$ then $V|_{Y}$ and $V^*|_{Y}$ each has more than three composition factors.
\item[{\rm (ii)}] If $a_1=5$ or $(a_1,a_3) = (1,2)$ then $V|_{Y}$ and $V^*|_{Y}$ each has more than six composition factors.
\end{itemize}
\end{lem}

\begin{proof}
Let   $\{\gamma_1\}$ be   a base of   the root system  of $Y$ with respect to a maximal torus $T_Y$ of $Y$.  Since $W|_{Y}$ has the given form we may assume without loss of generality that  
$$h_{\gamma_1}(c):=h_{\alpha_1}(c)h_{\alpha_3}(c)$$
for all $c \in K^*$. (See \eqref{e:hbc} for the definition of the elements $h_{\a_i}(c) \in T$ and $h_{\gamma_1}(c) \in T_Y$, where $T$ and $T_Y$ are appropriate maximal tori in $G$ and $Y$, respectively.)

First assume that  $a_1=a_3=0$ (so $a_2\neq 0$ since $\l$ is non-zero). Now $\l$, $\l-\alpha_1-\alpha_2$, $\l-\alpha_2-\alpha_3$ and $\l-\alpha_1-2\alpha_2-\alpha_3$ are weights of $V$ (by Lemma \ref{l:pr}), and we calculate that $h_{\gamma_1}(c)$ fixes (pointwise) the corresponding weight spaces. 
For example, if $v \in V$ is a vector in the weight space of $\mu = \l-\alpha_1-2\alpha_2-\alpha_3 = (a_2-2)\l_2$ then
$$h_{\gamma_1}(c)\cdot v = h_{\a_1}(c)\cdot c^{\la \a_3, \mu\ra}v = c^{\la \a_1,\mu\ra+\la \a_3, \mu\ra}v = c^{0}v=v.$$
By Lemma \ref{l:wsirra1}, every weight space of an irreducible $KY$-module is $1$-dimensional, whence  $V|_{Y}$ (and also $V^*|_{Y}$) has more than three composition factors, as required.  

Next suppose that $\{a_1,a_3\}=\{0,2\}$ or $a_1a_2 \neq 0$. In the former case we may as well assume that $(a_1,a_3)=(2,0)$ as an entirely similar argument applies if $(a_1,a_3)=(0,2)$. Now $\l-\alpha_1$, $\l-\alpha_1-\alpha_2-\alpha_3$, $\l-2\alpha_1-2\alpha_2-\alpha_3$ and $\l-2\alpha_1-\alpha_2$ are weights of $V$ (by Lemma \ref{l:pr}), and $h_{\gamma_1}(c)$ acts as multiplication by $c^{a_1+a_3-2}=1$ on the respective weight spaces. As in the previous paragraph, we conclude that $V|_{Y}$ has more than three composition factors.

The remaining cases are very similar. For example, if $a_1=5$ then  $\l-\alpha_1$, $\l-\alpha_1-\alpha_2-\alpha_3$, $\l-2\alpha_1-\alpha_2$, $\l-2\alpha_1-2\alpha_2-\alpha_3$, $\l-3\alpha_1-3\alpha_2-\alpha_3$,
$\l-4\alpha_1-4\alpha_2-\alpha_3$ and $\l-5\alpha_1-5\alpha_2-\alpha_3$  are weights of $V$, and we find that $h_{\gamma_1}(c)$ acts as multiplication by $c^{3+a_3}$ on the corresponding weight spaces. It follows that $V|_{Y}$ has at least seven composition factors. Similar reasoning applies when $(a_1,a_3)=(1,2)$.
\end{proof}

\begin{lem}\label{l:a1a5}
Let $G=A_5$ and $V = V_G(\l)$, where $\l =\sum_{i=1}^5 a_i\lambda_i$ is a nontrivial $p$-restricted weight. Let $W$ be the natural module for $G$ and let $Y=A_1$ be a subgroup of $G$ such that $W|_{Y} = U \oplus U \oplus U$, where $U$ is the natural $2$-dimensional module for $Y$.  
\begin{itemize}\addtolength{\itemsep}{0.3\baselineskip}
\item[{\rm (i)}] If $a_1\neq 0$ or $a_3 \neq 0$ then $V|_{Y}$ and $V^*|_{Y}$  each has more than two composition factors.
\item[{\rm (ii)}] If $a_1a_3\neq 0$ then $V|_{Y}$ and $V^*|_{Y}$ each has more than three composition factors.
\item[{\rm (iii)}] If either $a_1 \ge 2$, or $a_2\neq 0$, or $a_3=2$  then $V|_{Y}$ and $V^*|_{Y}$ each has more than six composition factors.
\end{itemize}
\end{lem}

\begin{proof}
We proceed as in the proof of Lemma \ref{l:a1a3}.  Let   $\{\gamma_1\}$ be   a base of   the root system  of $Y$ with respect to a maximal torus $T_Y$ of $Y$. We may assume that    
$$h_{\gamma_1}(c):=h_{\alpha_1}(c)h_{\alpha_3}(c)h_{\alpha_5}(c)$$
for all $c \in K^*$.

Suppose $a_1\neq 0$. Then  Lemma \ref{l:pr} implies that $\l-\alpha_1$, $\l-\alpha_1-\alpha_2-\alpha_3$ and $\l-\alpha_1-\alpha_2-\alpha_3-\alpha_4-\alpha_5$  are weights of $V$, and it is easy to check that $h_{\gamma_1}(c)$ acts as multiplication by $c^{a_1+a_3+a_5-2}$ on each of the corresponding weight spaces. In the usual way, by applying Lemma \ref{l:wsirra1}, we deduce that $V|_{Y}$ has at least three composition factors. The other cases are dealt with in a similar fashion and we leave the details to the reader.
\end{proof}

In the next three lemmas, $Y=Y_1 \cdots Y_k$ is a semisimple subgroup of $G={\rm SL}(W)$, where $Y_i=A_1$ for all $i$. Let $\{\gamma_1, \ldots, \gamma_k\}$ be a base of the root system of $Y$ with respect to a maximal torus $T_Y$ of $Y$, and let $\{\eta_1, \ldots, \eta_k\}$ be a corresponding set of fundamental dominant weights.

\begin{lem}\label{l:a1a1am}
Let $G=A_8$ and $V = V_G(\l)$, where $\l =\sum_{i=1}^8 a_i\lambda_i$ is a nontrivial $p$-restricted weight. Let $W$ be the natural module for $G$ and let $Y=A_1A_1$ be a subgroup of $G$. Suppose $a_2 \neq 0$, $p \neq 2$ and $W|_{Y}$ is irreducible with highest weight $\eta = 2(\eta_1+\eta_2)$. Then $V|_{Y}$ and $V^*|_{Y}$ each has more than two composition factors.
\end{lem}

\begin{proof}
Arguing as in Remark \ref{r:ord} (recall here that $W$ is equipped with the zero form, so each $U_Y$-level of $W$ is totally singular), we can order the $T$-weights in $W$ so that we obtain the root restrictions recorded in Table \ref{t:d41}. 

\renewcommand{\arraystretch}{1.2}
\begin{table}
$$\begin{array}{lll} \hline
\mbox{$T_Y$-weight} & \mbox{$T$-weight} & \mbox{Root restriction} \\ \hline
 \eta & \lambda_1 & \\
 \eta-\gamma_1 & \lambda_1-\alpha_1 & \alpha_1|_{Y}=\gamma_1 \\
 \eta-\gamma_2 & \lambda_1-\alpha_1-\alpha_2 & \alpha_2|_{Y}=\gamma_2-\gamma_1 \\
\eta-\gamma_1-\gamma_2 & \lambda_1-\alpha_1-\alpha_2-\alpha_3 & \alpha_3|_{Y}=\gamma_1 \\
\eta-2\gamma_1 & \lambda_1-\alpha_1-\alpha_2-\alpha_3-\alpha_4 & \alpha_4|_{Y}=\gamma_1-\gamma_2 \\
\eta-2\gamma_2 & \lambda_1-\alpha_1-\alpha_2-\alpha_3-\alpha_4-\alpha_5 & \alpha_5|_{Y}=2\gamma_2-2\gamma_1 \\
\eta-2\gamma_1-\gamma_2 & \lambda_1-\alpha_1-\alpha_2-\alpha_3-\alpha_4-\alpha_5-\alpha_6 & \alpha_6|_{Y}=2\gamma_1-\gamma_2 \\
\eta-\gamma_1-2\gamma_2 & \lambda_1-\alpha_1-\alpha_2-\alpha_3-\alpha_4-\alpha_5-\alpha_6-\alpha_7 & \alpha_7|_{Y}=\gamma_2-\gamma_1 \\ \hline
\end{array}$$
\caption{}
\label{t:d41}
\end{table}
\renewcommand{\arraystretch}{1}

Since $a_2\neq 0$, it follows that 
$\lambda-\alpha_1-2\alpha_2-\alpha_3$, $\lambda-\alpha_1-\alpha_2-\alpha_3-\alpha_4-\alpha_5$ and $\lambda-\alpha_2-\alpha_3-\alpha_4-\alpha_5-\alpha_6-\alpha_7$  are weights of $V$, each of which restricts to the same $T_Y$-weight $\lambda|_{Y}-2\gamma_2$.  Since every weight space of an irreducible  $KA_1$-module is $1$-dimensional (see Lemma \ref{l:wsirra1}), the same is true for irreducible $KY$-modules. We conclude that $V|_Y$ (and thus $V^*|_{Y}$ also) has more than two composition factors.
\end{proof}

\begin{lem}\label{l:a1a3_2}
Let $G=A_3$ and $V = V_G(\l)$, where $\l =\sum_{i=1}^3 a_i\lambda_i$ is a nontrivial $p$-restricted weight. Let $W$ be the natural module for $G$ and let $Y=A_1A_1$ be a subgroup of $G$. Suppose $a_1\ne 0$, $(a_2,a_3) = (1,0)$ and $W|_Y$ is irreducible with highest weight $\eta = \eta_1+\eta_2$. Then either 
$a_1 = p-2$, or $V|_{Y}$ has more than two composition factors.
\end{lem}

\begin{proof}
As explained in Remark \ref{r:ord}, we can order the $T$-weights in $W$ so 
that we obtain the root restrictions $\alpha_1|_{Y} = \gamma_2$, $\alpha_2|_{Y} = \gamma_1-\gamma_2$ and 
$ \alpha_3|_{Y} = \gamma_2$.  Suppose $a_1 \neq 0$ and $(a_2,a_3) = (1,0)$, so   
$\lambda-\alpha_1-\alpha_2$ and $\lambda-\alpha_2-\alpha_3$ are both 
weights in $\Lambda(V)$ which restrict to $T_Y$ as $\lambda|_{Y}-\gamma_1$.
Moreover, by Lemma \ref{l:s816}, if $a_1\ne p-2$ then the first of these weights has 
multiplicity $2$ in $V$. Hence Lemma \ref{l:wsirra1} implies that either $a_1=p-2$, or $V|_{Y}$ has
more than two $KY$-composition factors, as claimed.
\end{proof}

\begin{lem}\label{l:a1a1a1am}
Let $G=A_m$ and $V = V_G(\l)$, where $\l =\sum_{i=1}^m a_i\lambda_i$ is a nontrivial $p$-restricted weight. Let $W$ be the natural module for $G$ and let $Y=A_1A_1A_1$ be a subgroup of $G$. Suppose $(a_2,a_3) \neq (0,0)$ and $W|_Y$ is irreducible with highest weight $\eta = a(\eta_1+\eta_2+\eta_3)$, where 
$0<a<p$.
Then $V|_{Y}$ and $V^*|_{Y}$ each has more than three composition factors.
\end{lem}

\begin{proof}
This is similar to the proof of the previous lemmas. 
By appealing to Remark \ref{r:ord}, we can order the $T$-weights in $W$ to give the restrictions listed in Table \ref{t:d42}.

\renewcommand{\arraystretch}{1.2}
\begin{table}
$$\begin{array}{lll} \hline
\mbox{$T_Y$-weight} & \mbox{$T$-weight} & \mbox{Root restriction} \\ \hline
\eta & \lambda_1 & \\
\eta-\gamma_1 & \lambda_1-\alpha_1 & \alpha_1|_{Y}=\gamma_1\\
\eta-\gamma_2 & \lambda_1-\alpha_1-\alpha_2 & \alpha_2|_{Y}=\gamma_2-\gamma_1\\
\eta-\gamma_3 & \lambda_1-\alpha_1-\alpha_2-\alpha_3 & \alpha_3|_{Y}=\gamma_3-\gamma_2\\
\eta-\gamma_1-\gamma_2 & \lambda_1-\alpha_1-\alpha_2-\alpha_3-\alpha_4 & \alpha_4|_{Y}=\gamma_1+\gamma_2-\gamma_3\\
\eta-\gamma_1-\gamma_3 & \lambda_1-\alpha_1-\alpha_2-\alpha_3-\alpha_4-\alpha_5 & \alpha_5|_{Y}=\gamma_3-\gamma_2\\
\eta-\gamma_2-\gamma_3 & \lambda_1-\alpha_1-\alpha_2-\alpha_3-\alpha_4-\alpha_5-\alpha_6 & \alpha_6|_{Y}=\gamma_2-\gamma_1\\
\eta-\gamma_1-\gamma_2-\gamma_3 & \lambda_1-\alpha_1-\alpha_2-\alpha_3-\alpha_4-\alpha_5-\alpha_6-\alpha_7 & \alpha_7|_{Y}=\gamma_1 \\ \hline
\end{array}$$
\caption{}
\label{t:d42}
\end{table}
\renewcommand{\arraystretch}{1}

First assume $a_2\neq 0$. Then $\lambda-\alpha_1-2\alpha_2-2\alpha_3-\alpha_4$,  $\lambda-\alpha_1-2\alpha_2-\alpha_3-\alpha_4-\alpha_5$, $\lambda-\sum_{i=1}^6\alpha_i$ and 
$\lambda-\sum_{i=2}^7\alpha_i$ are all weights of $V$ with the same restriction $\lambda|_{Y}-\gamma_2-\gamma_3$.  As in the proof of the previous lemma, by applying Lemma \ref{l:wsirra1} we deduce that $V|_Y$ (and also $V^*|_{Y}$) has more than three composition factors. A similar argument applies if $a_3 \neq 0$.
\end{proof}

\chapter{The case $H^0 = A_m$}\label{s:am}

We begin our analysis by considering the triples $(G,H,V)$ satisfying Hypothesis \ref{h:our} with $X=H^0=A_m$. Here $m \ge 2$ and $H=X\langle t \rangle$, where $t$ is an involutory graph automorphism of $X$. As in Section \ref{ss:x2}, fix a $t$-stable Borel subgroup $B_X$ of $X$ containing a $t$-stable maximal torus $T_X$. We may choose $B_X$ so that $B_X \leqs B$ and $T_X \leqs T$, where $B=UT$ is the fixed Borel subgroup of $G$ corresponding to the base $\{\a_1,\ldots, \a_n\}$ of the root system $\Phi(G)$. Let $\{\b_1, \ldots, \b_m\}$ be a base of the root system $\Phi(X)$ with respect to $B_X$, and let $\{\delta_1, \ldots, \delta_m\}$ be the corresponding fundamental dominant weights. Let $\delta = \sum_{i}b_i\delta_i$ denote the highest weight of $W$ as an irreducible $KX$-module. Since $t$ acts on $W$, it follows that $\delta$ is fixed under the induced action of $t$ on the set of weights of $X$, so $b_i=b_{m+1-i}$ for all $i$. Recall that the symmetric nature of $\delta$ implies that $X$ fixes a non-degenerate form on $W$, so $G$ is symplectic or orthogonal.
In addition, since $X$ is simply laced, the set of $T_X$-weights of $W$ is the same as the set of $T_X$-weights of the Weyl module $W_X(\delta)$ (see Lemma \ref{l:pr}).

Let $V=V_G(\lambda)$ be a $p$-restricted irreducible tensor indecomposable  $KG$-module with highest weight $\l=\sum_{i}a_i\l_i$, where $\{\l_1, \ldots, \l_n\}$ are the fundamental dominant weights of $G$ dual to the given simple roots $\{\a_1, \ldots, \a_n\}$. Assume $(G,H,V)$ satisfies Hypothesis \ref{h:our}, so $V|_{H}$ is irreducible, but $V|_{X}$ is reducible, whence 
$$V|_{X}=V_1 \oplus V_2,$$
where the $V_i$ are non-isomorphic irreducible $KX$-modules interchanged by $t$ (see Proposition \ref{p:niso}).
Since $B_X \leqs B$, without loss of generality we may assume that $V_1$ has $T_X$-highest weight $\mu_1 = \l|_{X}$, so the $T_X$-highest weight $\mu_2$ of $V_2$ is the image of $\mu_1$ under the action of $t$, say
\begin{equation}\label{e:ci10}
\mu_1 = \sum_{i=1}^mc_i\delta_i, \;\; \mu_2 = \sum_{i=1}^mc_{m+1-i}\delta_i.
\end{equation}
Since $V_1$ and $V_2$ are non-isomorphic $KX$-modules, it follows that $c_{i} \neq c_{m-i+1}$ for some $1 \le i \le \lfloor m/2 \rfloor$.

\section{The main result} 

\begin{thm}\label{t:am}
Let $(G,H,V)$ be a triple satisfying Hypothesis \ref{h:our}, where $H^0=A_m$ for some $m \ge 2$. Then $(G,H,V)$ is one of the cases in Table \ref{t:amtab}. Moreover, such a triple $(G,H,V)$ has the property $H \leqs G$ if and only if $(G,H,V) = (C_{10},A_5.2,\l_3)$.
\end{thm}

\renewcommand{\arraystretch}{1.2}
\begin{table}[h]
$$\begin{array}{llllll} \hline
G & H & \l & \delta & \l|_{H^0} & \mbox{{\rm Conditions}} \\ \hline
D_{10} & A_3.2 & \l_9, \l_{10} & 2\delta_2 & \mbox{{\rm $3\delta_1+\delta_2+\delta_3$ or $\delta_1+\delta_2+3\delta_3$}} & p\neq 2,3,5,7 \\
C_{10} & A_5.2 & \l_3 & \delta_3 & \mbox{{\rm $\delta_1+2\delta_4$ or $2\delta_2+\delta_5$}} & p \neq 2,3 \\ \hline
\end{array}$$
\caption{The triples $(G,H,V)$ with $H^0=A_m$}
\label{t:amtab}
\end{table}
\renewcommand{\arraystretch}{1}

\begin{rmk} 
Note that $V|_{H^0}$ has $p$-restricted composition factors in both of the cases 
that arise in Theorem \ref{t:am}. The first case in Table \ref{t:amtab} is labelled by ${\rm U}_{5}$ in \cite[Tables I, II]{Ford1}, but the second case is a new example which is missing from Ford's tables in \cite{Ford1} (see Remark \ref{r:ford} for further details). 
\end{rmk}

The proof of Theorem \ref{t:am} proceeds by induction on $m$. The base cases are $m=2$ and $m=3$; these cases require special attention and we deal with them separately in Sections \ref{ss:am1} and \ref{ss:am2}, respectively. The general case $m \ge 4$ is handled in Section \ref{ss:am3}. 

\section{Preliminaries}

Here we record some preliminary results which we will use in the proof of Theorem \ref{t:am}. 

\begin{lem}\label{l:ammu1mu2}
Set $k=m/2$ if $m$ is even, otherwise $k=(m-1)/2$. Then
\begin{eqnarray*}
\mu_2-\mu_1& = & \frac1{m+1}\sum_{j=1}^{k-1}(\beta_j-\beta_{m-j+1})\sum_{i=1}^ji(m-2j+1)(c_{m-i+1}-c_i)\\
& & +\frac1{m+1} \sum_{j=1}^{k-1}(\beta_j-\beta_{m-j+1})\sum_{i=j+1}^k j(m-2i+1)(c_{m-i+1}-c_i)
\\
& &  +\frac1{m+1}\sum_{i=1}^k i(m-2k+1)(c_{m-i+1}-c_i)(\beta_k-\beta_{m-k+1}).
\end{eqnarray*}
\end{lem}

\begin{proof}
First observe that $\mu_2-\mu_1=\sum_{i=1}^k(c_{m+1-i}-c_i)(\delta_i-\delta_{m+1-i})$
and 
\begin{eqnarray*}
\delta_i-\delta_{m-i+1} & = & \frac{1}{m+1}\sum_{j=1}^{i-1}j(m-2i+1)(\beta_j-\beta_{m-j+1}) \\
& & + \frac{1}{m+1}\sum_{j=i}^ki(m-2j+1)(\beta_j-\beta_{m-j+1})
\end{eqnarray*}
for all $1\leq i\leq k$ (see \cite[Table 1]{Hu1}, for example). Therefore  
\begin{eqnarray*}
\mu_2-\mu_1& = & \frac1{m+1}\sum_{i=1}^{k}(c_{m-i+1}-c_i)\sum_{j=1}^{i-1}j(m-2i+1)(\beta_j-\beta_{m-j+1})\\
& & +   \frac1{m+1}\sum_{i=1}^{k}(c_{m-i+1}-c_i)\sum_{j=i}^{k}i(m-2j+1)(\beta_j-\beta_{m-j+1}),
\end{eqnarray*}
and the result follows by suitably re-ordering the summation.
\end{proof}

Recall that if $\mu$ and $\nu$ are weights of a $KX$-module $M$ then we say that $\mu$ is \emph{under} $\nu$ (denoted by 
$\mu \preccurlyeq \nu$) if and only if $\mu = \nu - \sum_{i=1}^{m}c_i\b_i$ for some non-negative integers $c_i$. 

\begin{lem}\label{l:ammu2}
Suppose $\mu \in \Lambda (V)$ and $\mu|_X=\lambda|_X -\beta_i+\beta_{m-i+1}$ for some $1\leq i \leq m$, and assume $i\neq (m+1)/2$ if $m$ is odd.  Then $\mu_2=\mu_1-\beta_i+\beta_{m-i+1}$. 
\end{lem}

\begin{proof}
Let $\nu=\mu|_{X}$. Since $\nu =\mu_1-\beta_i+\beta_{m-i+1}$, $\nu$ is not under $\mu_1$ and thus $\nu \not \in \Lambda(V_1)$. Therefore $\nu\in \Lambda(V_2)$ and $\nu=\mu_2-\sum_{j=1}^m k_j\beta_j$ for some non-negative integers $k_j$.
Hence 
$$\mu_2-\mu_1=\sum_{j\not \in \{i,m-i+1\}} k_j\beta_j+(k_i-1)\beta_i+(k_{m-i+1}+1)\beta_{m-i+1}.$$
By comparing this expression with the one given in Lemma \ref{l:ammu1mu2}, it follows that $k_i-1=-(k_{m-i+1}+1)$ and $k_j=-k_{m-j+1}$ for all $j\not \in \{i,m-i+1\}$. Therefore $k_j=0$ for all $j$ and thus $\mu_2=\nu$, as required.
\end{proof}

\begin{lem}\label{l:centre}
Let $P_X=Q_XL_X$ be the parabolic subgroup of $X$ with $\Delta(L_X')=\{\beta_1,\dots,\beta_{m-1}\}$ and let $P=QL$ be the corresponding parabolic subgroup of $G$, constructed using the $Q_X$-levels of $W$ as in Lemma \ref{l:flag}. Assume 
$$\sum_{j=1}^mjc_j\neq \sum_{j=1}^m (m+1-j)c_j.$$ 
Then $V/[V,Q]$ is an irreducible $KL_X'$-module. 
\end{lem}

\begin{proof}
Recall that $V/[V,Q]$ is an irreducible module for $L$ and $L'$ (see Lemma \ref{l:vq}(i)). 
Further, as noted in Section \ref{ss:x2}, 
either $V/[V,Q]$ is an irreducible $KL_X'$-module, or 
$$V/[V,Q]=V/[V,Q_X]=V_1/[V_1,Q_X]\oplus V_2/[V_2, Q_X],$$
so we need to rule out the latter possibility. 
Let $Z=Z(L_X)^0$ and observe that
$$Z= \left\{h_{\beta_1}(c)h_{\beta_2}(c^2) \cdots h_{\beta_m}(c^m): c \in K^*\right\}.$$
By Lemma \ref{l:flag}(ii) we have $L=C_G(Z)$, so $Z \leqs Z(L)$.  Since $V/[V,Q]$ is an irreducible $KL$-module, Schur's lemma implies that $Z(L)$, and thus $Z$, acts as scalars on $V/[V,Q]$.  Now, if $V/[V,Q]$ is a reducible $KL_X'$-module then $\mu_1$ and $\mu_2$ both occur with non-zero multiplicity in $V/[V,Q]$, so $\mu_1|_{Z}=\mu_2|_{Z}$, that is 
$$\sum_{j=1}^mjc_j = \sum_{j=1}^m (m+1-j)c_j,$$ 
which is a contradiction.
\end{proof}

We close this section with a couple of miscellaneous results, which will be required in the later analysis.

\begin{lem}\label{l:am22} 
Suppose $m\geq 2$ and $p \neq 2$. Then 
$$\dim( 
V_{X}(2\delta_1+\delta_m))=\left\{\begin{array}{ll}
(m+1)(m^2+3m-2)/2 & \mbox{{\rm  if  $m+2\equiv 0 \imod{p}$}}\\
(m+1)(m^2+3m)/2 & \mbox{otherwise.}
\end{array}\right.$$
\end{lem}

\begin{proof}
Let $\nu=2\delta_1+\delta_m$ and observe that $V_X(\nu)$ has exactly three subdominant weights, namely $\nu_1=\nu$, $\nu_2=\nu-\beta_1$ and $\nu_3=\nu-\sum_{j=1}^m\beta_j$, with respective multiplicities 1, 1 and $m-\e$, where $\e=1$ if $p$ divides $m+2$, otherwise $\e=0$ (see Lemma \ref{l:s816}). Let $\mathcal{W} \cong S_{m+1}$ denote the Weyl group of $X$ and write 
$\mu^{\mathcal{W}}$ for the $\mathcal{W}$-orbit of a weight $\mu$. Then
$$ \dim V_X(\nu) = |\nu_1^{\mathcal{W}}|+|\nu_2^{\mathcal{W}}|+(m-\e)|\nu_3^{\mathcal{W}}|.$$ 
Since $\nu_1=2\delta_1+\delta_m$ we have $|\nu_1^{\mathcal{W}}|=(m+1)!/(m-1)!=(m+1)m$ (here we are using \cite[1.10]{Seitz2}). Similarly, since $\nu_2=\delta_2+\delta_m$ and $\nu_3=\delta_1$ we compute
$$|\nu_2^{\mathcal{W}}|=\frac{(m+1)!}{2(m-2)!}=\frac{1}{2}m(m+1)(m-1),\; |\nu_3^{\mathcal{W}}|=\frac{(m+1)!}{m!}=m+1.$$
The result follows. 
\end{proof}

Let $U$ be the natural $KX$-module. Following \cite{LS}, the subgroups of $X={\rm SL}(U)$ in the geometric $\C_2$ collection are defined to be the stabilizers of direct sum decompositions of the form $U = U_1 \oplus \cdots \oplus U_s$, where $s \ge 2$ and the $U_i$ have the same dimension. In the next lemma, $J$ is a $\C_2$-subgroup stabilizing a decomposition $U=U_1 \oplus U_2$. 

\begin{lem}\label{l:c2}
Suppose $m \ge 3$ is odd and let $M=V_X(\nu)$ be a $p$-restricted irreducible  $KX$-module with highest weight $\nu\not \in \{\delta_1,\delta_m\}$. Let $J = A_{(m-1)/2}A_{(m-1)/2}T_1.2$ be a $\C_2$-subgroup of $X$. Then $M|_{J^0}$ has at least three composition factors.
\end{lem}

\begin{proof}
Set $\nu=\sum d_i \delta_i$ and let $P_X=Q_XL_X$ be the parabolic subgroup of $X$ with $\Delta(L_X') = \Delta(X)\setminus \{\b_{(m+1)/2}\}$. By replacing $J$ by a suitable $X$-conjugate we may assume that $L_X=J^0$.
Now 
$$\la \nu, \b_1+\cdots+\b_{m}\ra = \sum_{i=1}^{m}d_i$$
so Corollary \ref{c:sat} implies that $\nu-k(\b_1+\cdots+\b_{m}) \in \L(M)$ for all $0 \le k \le \sum_{i}d_i$. In particular, if $\sum_{i}d_i>1$ then $M$ has weights at $Q_X$-levels $0,1$ and $2$, so $M|_{L_X}$ has at least $3$ composition factors, as required. 

Finally suppose $\nu=\delta_i$ for some $i$. Since $\nu \not \in \{\delta_1,\delta_m\}$, it follows that $2 \le i \le m-1$ and we have weights $\nu$ at $Q_X$-level $0$, $\nu-(\b_i + \b_{i+1} + \cdots + \b_{(m+1)/2})$ at level $1$ if $i \le (m+1)/2$ (and similarly if $i>(m+1)/2$), and also the weight 
$\nu - \b_1 - 2(\b_2+ \cdots + \b_{m-1})-\b_{m}$
at level $2$. The result follows. 
\end{proof}

\section{Some results on wedges}

In the proof of Theorem \ref{t:am} we will require some specific results on the number of composition factors in the wedges of some irreducible $KX$-modules.

\begin{lem}\label{l:wedge2}
Suppose $m\geq 2$ and $M=V_X(\nu)$, where $\nu=a\delta_i+a\delta_{m-i+1}$ with $1\leq a <p$ and $1 \leq i \leq m/2$. 
Then $\Lambda^2(M)$ has at least three distinct $KX$-composition factors.
\end{lem} 
 
\begin{proof}
Let $N=\Lambda^2(M)$ and let $w_1,w_2,w_3$ be non-zero vectors in $M$ of respective weights $\nu$, $\nu-\beta_i$ and $\nu-\beta_{m-i+1}$.  Then $w_1\wedge w_2$ and $w_1\wedge w_3$ are maximal vectors of $N$ with corresponding weights $x=2\nu-\beta_i$ and $y=2\nu-\beta_{m-i+1}$, affording composition factors $N(x)$ and $N(y)$, respectively. To prove the lemma it suffices to show that
\begin{equation}\label{e:mvz}
m_{N(x)}(z)+m_{N(y)}(z)< m_N(z)
\end{equation}
for some weight $z$ of $N$, where $m_N(z) = \dim N_z$ denotes the multiplicity of the weight $z$ in $N$, etc.
 
First assume $a\geq 2$ and $i<m/2$. Set $z=2\nu-\beta_i-2\beta_{m-i+1}$ and note that  
$$N_z=(M_{\nu}\wedge M_{\nu-\beta_i-2\beta_{m-i+1}})\oplus (M_{\nu-\beta_i}\wedge M_{\nu-2\beta_{m-i+1}})\oplus (M_{\nu-\beta_i-\beta_{m-i+1}}\wedge M_{\nu-\beta_{m-i+1}}).$$
Since $i+1 <m-i+1$ we have $\dim N_z = m_N(z)=3$. Now $m_{N(x)}(z)=1$ since $z=x-2\beta_{m-i+1}$, and similarly, since $z=y-\beta_i-\beta_{m-i+1}$ and $i+1<m-i+1$, we have $m_{N(y)}(z)=1$. Hence \eqref{e:mvz} holds. 
 
Next assume $a=1$ and $i<m/2$. There are two cases to consider. First assume $p=2$ and set $z=2\nu-\beta_i-\beta_{m-i+1}$. Then 
$$N_z=(M_{\nu}\wedge M_{\nu-\beta_i-\beta_{m-i+1}}) \oplus(M_{\nu-\beta_i}\wedge M_{\nu-\beta_{m-i+1}})$$ 
and thus $m_N(z)=2$ since $i+1 <m-i+1$. If $i=1$ then  $x=\delta_2+2 \delta_m$ and $y=2\delta_1+ \delta_{m-1}$, so $m_{N(x)}(z)=m_{N(y)}(z)=0$ since $z=x-\beta_m=y-\beta_1$. Therefore \eqref{e:mvz} holds in this case. Similarly, if $i \geq 2$ then $x=\delta_{i-1}+\delta_{i+1}+2\delta_{m-i+1}$, hence $m_{N(x)}(z)=0$ since $z=x-\beta_{m-i+1}$, and similarly we get $m_{N(y)}(z)=0$, so once again \eqref{e:mvz} holds. 

Now suppose $p \neq 2$ (and continue to assume that $a=1$ and $i < m/2$). Consider the weight  $z=2\nu-\beta_i-\dots-\beta_{m-i+1}$ in $N$.  To compute 
$m_N(z)$ we determine the different pairs $\omega_1,\omega_2 \in \L(M)$ such that $\omega_1+\omega_2 = z$, using Lemma \ref{l:s816}. If $\mu = \nu\wedge \left(\nu-\beta_i-\dots-\beta_{m-i+1}\right)$ then $m_{N}(\mu) = m-2i+1$ if $p$ divides $m+3-2i$, otherwise $m_{N}(\mu) = m-2i+2$. The only other pairs of weights $\omega_1,\omega_2$ with $\omega_1+\omega_2=z$ are of the form 
$$\omega_1 \wedge \omega_2 = \left(\nu- \sum_{j=i}^{k}\beta_j\right) \wedge \left(\nu-\sum_{j=k+1}^{m-i+1}\beta_j\right)$$ 
with $i \leq k \leq m-i$, and each of these weights has multiplicity $1$. It follows that 
 $$ m_N(z)=\left\{ \begin{array}{ll}
 2m-4i+2 & \mbox{if $p$ divides $m+3-2i$} \\
  2m-4i+3 & \mbox{otherwise}.
  \end{array}
  \right.$$
In $N(x)$, $z$ occurs as $x-\beta_{i+1}-\dots -\beta_{m-i+1}$, and we note that  $x=\delta_2+2\delta_m$ if $i=1$, otherwise $x=\delta_{i-1}+\delta_{i+1}+2\delta_{m-i+1}$.
Hence Lemma \ref{l:s816} gives 
 $$ m_{N(x)}(z)=\left\{
 \begin{array}{ll}
 m-2i & \mbox{if $p$ divides $m+3-2i$} \\
 m-2i+1 & \mbox{otherwise.} 
 \end{array}
  \right.$$
As $x$ and $y$ are symmetric, we have $m_{N(x)}(z)=m_{N(y)}(z)$ and we deduce that \eqref{e:mvz} holds.
 
To complete the proof we may assume that $m$ is even and $i=m/2$. Set $z=2\nu-\beta_i-\beta_{i+1}$. Then
$$N_z = \left(M_{\nu- \beta_i}\wedge M_{\nu-\beta_{i+1}}\right) \oplus \left(M_{\nu}\wedge M_{\nu-\beta_i-\beta_{i+1}}\right)$$
and thus $m_N(z)=3-\e$, where 
$\e=1$ if $p$ divides $2a+1$, otherwise $\e=0$. Now $x=2\nu-\beta_i=\sum_j d_j\delta_j$, where $d_{i+1}=2a+1$, so $z=x-\beta_{i+1}$ and thus $m_{N(x)}(z)=1-\e$. Similarly, $m_{N(y)}(z)=1-\e$ and we conclude that \eqref{e:mvz} holds.
\end{proof}

\begin{lem}\label{l:wedge3}
Suppose $m\geq 5$ is odd and $M=V_X(\nu)$ where $\nu=\delta_{(m+1)/2}$. Then $\Lambda^3(M)$ has at least three distinct $KX$-composition factors.
\end{lem}

\begin{proof}
Let $N = \Lambda^3(M)$, $k=(m+1)/2$ and let $w_1,w_2,w_3,w_4$ be non-zero vectors of $M$ of respective weights $\nu$, $\nu-\beta_k$, $\nu-\beta_k-\beta_{k+1}$ and $\nu-\beta_{k-1}-\beta_{k}$. Then $w_1\wedge w_2\wedge w_3$ and $w_1\wedge w_2\wedge w_4$ are maximal vectors of $N$ of respective weights $x=3\nu-2\beta_k-\beta_{k+1}$ and $y=3\nu-\beta_{k-1}-2\beta_{k}$. Let $N(x)$ and $N(y)$ be the composition factors of $N$ with highest weights $x$ and $y$, respectively. As in the proof of Lemma \ref{l:wedge2}, it suffices to find a weight $z$ of $N$ such that the inequality in \eqref{e:mvz} holds.
With this aim in mind, set $z=3\nu-\beta_{k-2}-2\beta_{k-1}-3\beta_k-2\beta_{k+1}-\beta_{k+2}$, which is a weight of $N$ occurring in both $N(x)$ and $N(y)$. Without loss of generality, we may assume $m=5$. 

Now $x = 2\delta_2+\delta_5$ and $z = \delta_3$. By inspecting \cite{LubeckW} we deduce that $m_{N(x)}(z)\leq 7$, and thus $m_{N(y)}(z)\leq 7$ since $x$ and $y$ are symmetric. Therefore, to see that \eqref{e:mvz} holds it suffices to show that $m_N(z)>14$.  This is a straightforward exercise. Indeed one can exhibit fifteen different ways in which $z$ arises as a sum of three distinct weights of $M$. 
\end{proof}

\section{The case $H^0=A_2$}\label{ss:am1} 

In this section we prove Theorem \ref{t:am} in the case where $H^0=X=A_2$. Let $\mathcal{T}$ be the set of irreducible triples $(G,H,V)$ satisfying the conditions stated in Hypothesis \ref{h:our}, with $H^0=A_2$. We will eventually show that $\mathcal{T}$ is empty. 

Suppose $(G,H,V) \in \mathcal{T}$. Let $\delta$ denote the highest weight of the natural $KG$-module $W$, viewed as an irreducible $KX$-module. To simplify the notation we write 
$$\delta=a\delta_1+a\delta_2$$ 
for some positive integer $a<p$ (see condition S5 in Hypothesis \ref{h:our}). Write $H=X\la t \ra$ and recall that $\l = \sum_{i}a_i\l_i$ denotes the highest weight of the irreducible $KG$-module $V=V_G(\l)$. Given a $t$-stable Borel subgroup $B_X = U_XT_X$ of $X$, we construct the corresponding parabolic subgroup $P=QL$ of $G$ as explained in Section \ref{ss:parabs}. As before, let $W_i$ be the $i$-th $U_X$-level of $W$. Note that the lowest weight $-\delta$ has level $\ell=4a$. Set $\ell'=\ell/2=2a$. 
  
\begin{prop}\label{p:a2red} 
Suppose $a \ge 2$ and $(a,p) \neq (2,5),(3,7)$. Then $L'$ has a unique $A_1$ factor, which corresponds to ${\rm Isom}(W_1)'$.
\end{prop}

\begin{proof}
First observe that $\dim W_0=1$ and $\dim W_1=2$, so by Remark \ref{r:a1factor} it suffices to show that $\dim W_i \ge 3$ for all $2 \le i \le \ell/2-1$, and $\dim W_{\ell/2} \ge 5$. For $a \geq 4$ we do this by exhibiting sufficiently many distinct weights at each $U_X$-level (by Lemma \ref{l:pr}, we can work in the Weyl module $W_X(\delta)$). First assume $a=4$, so $\ell'=8$. By applying Lemma \ref{l:t1} and Corollary \ref{c:sat}, it is easy to check that there are at least three distinct weights at level $i$ for $2\leq i\leq 7$, and at least five at level $8$. For example, if $2\leq i <8$ is even then 
$$\delta-\frac{1}{2}i\beta_1-\frac{1}{2}i\beta_2, \; \delta-(i/2+1)\beta_1-(i/2-1)\beta_2, \; \delta-(i/2-1)\beta_1-(i/2+1)\beta_2$$
are three distinct weights in $W_i$.

Now assume $a>4$. Let $\mu$ be the subdominant weight $\mu=\delta-\beta_1-\beta_2=(a-1)\delta_1+(a-1)\delta_2$. Using induction on $a$ and Lemma \ref{l:indwt} we deduce that there are at least three distinct weights at level $i$ for all $4\leq i<\ell/2$, and at least five at level $\ell/2=2a$. 
In addition, we also have $\dim W_i\geq 3$ for $i = 2,3$. Indeed, the weights at level 2 are $\delta-2\beta_1,\delta-\beta_1-\beta_2, \delta-2\beta_2$, while $\delta-3\beta_1, \delta-2\beta_1-\beta_2,\delta-3\beta_2$ are three weights at level 3. 

Next suppose $a=3$, so that $\ell/2=6$. It is straightforward to check that there are at least three distinct weights at levels 2, 3, 4 and 5, while at level $6$ the weights are 
$\delta-4\beta_1-2\beta_2, \delta-3\beta_1-3\beta_2,\delta-2\beta_1-4\beta_2$. Since $p\neq  7$, the weight $\delta-3\beta_1-3\beta_2$ occurs with multiplicity at least 3 (see \cite{LubeckW}), hence $\dim W_6\geq 5$.

Finally suppose $a=2$, so $\ell/2=4$. At level 2 we have weights 
$\delta-2\beta_1,\delta-\beta_1-\beta_2$ and $\delta-2\beta_2$, so $\dim W_2\geq 3$. The weights at level $3$ are $\delta-k\beta_1-(3-k)\beta_2$ with $k=1,2$, each of which is conjugate to $\delta-\beta_1-\beta_2$. Since $p\neq 5$, Lemma \ref{l:s816} implies that the latter weight has multiplicity $2$ and thus $\dim W_3 >3$. Finally the weights at level $4$ are $\delta-k\beta_1-(4-k)\beta_2$ with $k\in\{1,2,3\}$. Since $p\neq 5$ the weight $\delta-2\beta_1-2\beta_2$ has multiplicity at least 3 (see \cite{LubeckW}) so $\dim W_4 \geq 5$. The result follows. 
\end{proof}

We now establish Theorem \ref{t:am} for $X=A_2$ by a sequence of reductions.

\begin{lem}\label{p:a2redbis}
Suppose $(G,H,V) \in \mathcal{T}$. Then one of the following holds:
$$\mbox{{\rm (i)} $a=1$,\;\; {\rm (ii)} $(a,p)=(2,5)$,\;\;  {\rm (iii)} $(a,p)=(3,7)$ and $a_2 \neq 1$.}$$
\end{lem}

\begin{proof}
Seeking a contradiction, let us assume $a \ge 2$, $(a,p)\neq (2,5)$, and also assume $a_2=1$ if $(a,p)=(3,7)$. Then $\dim W \ge 18$ (see \cite{Lubeck}) and thus $n \ge 9$. Let $P=QL$ be the usual parabolic subgroup of $G$ constructed from a  $t$-stable Borel subgroup $B_X=U_XT_X$ of $X$.  Since $\ell/2 \geq 4$ we can choose an ordering of the $T_X$-weights in the $U_X$-levels $0$, $1$ and $2$ to obtain the root restrictions listed in Table \ref{t:r1} (here we are appealing to Remark \ref{r:ord}). Note that if $(a,p) \neq (3,7)$ then Proposition \ref{p:a2red} implies that  ${\rm Isom}(W_1)'$ is the unique $A_1$ factor in $L'$, so $a_2=1$ by Lemma \ref{l:main}. Therefore, in all cases, $\lambda-\alpha_2 \in \Lambda(V)$ and $(\lambda-\alpha_2)|_{X}=\lambda|_{X}-\beta_2+\beta_1$; so by Lemma \ref{l:ammu2} we deduce that $\mu_2=\mu_1-\beta_2+\beta_1$.

\renewcommand{\arraystretch}{1.2}
\begin{table}
$$\begin{array}{llll} \hline
\mbox{$U_X$-level} & \mbox{$T_X$-weight} & \mbox{$T$-weight} & \mbox{Root restriction} \\ \hline
0 &\delta& \lambda_1& \\
1 & \delta-\beta_1& \lambda_1-\alpha_1& \alpha_1|_{X}=\beta_1 \\
 & \delta-\beta_2& \lambda_1-\alpha_1-\alpha_2& \alpha_2|_{X}=\beta_2-\beta_1 \\
2 & \delta-2\beta_2& \lambda_1-\alpha_1-\alpha_2-\alpha_3& \alpha_3|_{X}=\beta_2 \\
 & \delta-2\beta_1& \lambda_1-\alpha_1-\alpha_2-\alpha_3-\alpha_4& \alpha_4|_{X}=2\beta_1-2\beta_2 \\
& \delta-\beta_1-\beta_2& \lambda_1-\alpha_1-\alpha_2-\alpha_3-\alpha_4-\alpha_5& \alpha_5|_{X}=\beta_2-\beta_1 \\ \hline
\end{array}$$
\caption{}
\label{t:r1}
\end{table}
\renewcommand{\arraystretch}{1}

Suppose $2a \neq p-1$. By Lemma \ref{l:s816}, $\delta-\beta_1-\beta_2$ has multiplicity 2, so we have an additional weight restriction at level $2$, namely $(\lambda_1-\sum_{i=1}^6\alpha_i)|_{X}=\delta-\beta_1-\beta_2$  and so $\alpha_6|_{X}=0$.
Since $a_2=1$, the weights $\lambda-\alpha_1-\alpha_2$, $\lambda-\alpha_2-\alpha_3-\alpha_4-\alpha_5$ and $\lambda-\alpha_2-\alpha_3-\alpha_4-\alpha_5-\alpha_6$ belong to $\Lambda(V)$ and they all restrict to the same $T_X$-weight $\nu=\lambda|_{X}-\beta_2$. Note that  $\nu$  occurs in $V_1$ as $\mu_1-\beta_2$, which has multiplicity at most 1, and it occurs in $V_2$ as $\mu_2-\beta_1$, again  with multiplicity at most 1. Since $m_{V}(\nu) \ge 3$, this contradicts the fact that $V = V_1\oplus V_2$.

To complete the proof we may assume that $2a=p-1$ and $a \geq 3$ (recall that we are assuming $(a,p) \neq (2,5)$). By Lemma \ref{l:s816} the weight $\delta-\beta_1-\beta_2$  has multiplicity 1. 
Since $a \ge 3$, we can order the $T_X$-weights at level $3$ to obtain the root restrictions recorded in Table \ref{t:r2} (see Remark \ref{r:ord}). Consequently, we deduce that the weights $\lambda-\alpha_1-2\alpha_2-\alpha_3$ and $\lambda-\sum_{i=2}^9\alpha_i$ both restrict  to the same $T_X$-weight $\nu=
\lambda|_X-3\beta_2+\beta_1$. Clearly, $\nu$ does not belong to $\Lambda(V_1)$. Furthermore, since $\nu=\mu_2-2\beta_2$ we see that it has multiplicity at most 1 in $V_2$ (see Lemma \ref{l:t1}). Once again, this is a contradiction since $V = V_1\oplus V_2$. 
\end{proof}

\renewcommand{\arraystretch}{1.2}
\begin{table}
$$\begin{array}{llll} \hline
\mbox{$U_X$-level} & \mbox{$T_X$-weight} & \mbox{$T$-weight} & \mbox{Root restriction} \\ \hline
3 & \delta-2\beta_1-\beta_2& \lambda_1-\sum_{i=1}^6\alpha_i& \alpha_6|_{X}=\beta_1 \\
 & \delta-\beta_1-2\beta_2& \lambda_1-\sum_{i=1}^7\alpha_i& \alpha_7|_{X}=\beta_2-\beta_1 \\
 & \delta-3\beta_1& \lambda_1-\sum_{i=1}^8\alpha_i& \alpha_8|_{X}=2\beta_1-2\beta_2 \\
& \delta-3\beta_2& \lambda_1-\sum_{i=1}^9\alpha_i& \alpha_9|_{X}=3\beta_2-3\beta_1 \\ \hline
\end{array}$$
\caption{}
\label{t:r2}
\end{table}
\renewcommand{\arraystretch}{1}

\begin{lem}\label{l:a21}
Suppose $(G,H,V) \in \mathcal{T}$. Then $a \neq 1$.
\end{lem}

\begin{proof}
Suppose $a=1$, so $\ell=4$.  First assume $p\neq3$, in which case $G=D_4$ (see \cite[Table 2]{Brundan}).  Let $P=QL$ be the parabolic subgroup of $G$ constructed from a $t$-stable Borel subgroup $B_X=U_XT_X$ of $X$. The $U_X$-levels $W_0$, $W_1$ and $W_2$ have respective dimensions $1$, $2$ and $2$ (note that the weight $\delta-\beta_1-\beta_2$ has multiplicity 2 in $W$ by Lemma \ref{l:s816}). Therefore Remark \ref{r:a1factor} and Lemma 
\ref{l:main} imply that $a_2=1$. In addition, by appealing to Remark \ref{r:ord} and the fact that there is a unique $T_X$-weight at level 2, we can order the $T$-weights in each level to give the root restrictions listed in Table \ref{t:r3}.

\renewcommand{\arraystretch}{1.2}
\begin{table}
$$\begin{array}{llll} \hline
\mbox{$U_X$-level} & \mbox{$T_X$-weight} & \mbox{$T$-weight} & \mbox{Root restriction} \\ \hline
0 & \delta& \lambda_1& \\
1 & \delta-\beta_1& \lambda_1-\alpha_1& \alpha_1|_{X}=\beta_1 \\
 & \delta-\beta_2& \lambda_1-\alpha_1-\alpha_2& \alpha_2|_{X}=\beta_2-\beta_1 \\
2 & \delta-\beta_1-\beta_2& \lambda_1-\alpha_1-\alpha_2-\alpha_3& \alpha_3|_{X}=\beta_1 \\
 & & \lambda_1-\alpha_1-\alpha_2-\alpha_4&  \alpha_4|_{X}=\beta_1 \\ 
\hline 
\end{array}$$
\caption{}
\label{t:r3}
\end{table}
\renewcommand{\arraystretch}{1}

Since $a_2=1$ it follows that $\lambda-\alpha_2$,  $\lambda-\alpha_1-\alpha_2$, $\lambda-\alpha_2-\alpha_3$ and $\lambda-\alpha_2-\alpha_4$ are all weights in $\Lambda(V)$. In addition, since $(\lambda-\alpha_2)|_{X}=\lambda|_{X}-\beta_2+\beta_1$, Lemma \ref{l:ammu2} implies that $\mu_2=\mu_1-\beta_2+\beta_1$. Now 
$\lambda-\alpha_1-\alpha_2$, $\lambda-\alpha_2-\alpha_3$ and $\lambda-\alpha_2-\alpha_4$ all restrict to the same  $T_X$-weight $\nu=\lambda|_{X}-\beta_2$. However, $\nu$ occurs in $V_1$ as $\mu_1-\beta_2$, which has multiplicity at most $1$, and it occurs in $V_2$ as $\mu_2-\beta_1$ also with multiplicity at most $1$. Since $m_{V}(\nu) \ge 3$ and $V = V_1\oplus V_2$, this is a contradiction.

Now assume $p=3$ so $G = B_3$. As before, let $P=QL$ be the parabolic subgroup of $G$ constructed from a $t$-stable Borel subgroup $B_X=U_XT_X$ of $X$. Here the weight $\delta-\beta_1-\beta_2$ has multiplicity 1 in $W$ (see Lemma \ref{l:s816}), and as before we obtain the root restrictions $\alpha_1|_{X}=\beta_1$, $\alpha_2|_{X}=\beta_2-\beta_1$ and $\alpha_3|_{X}=\beta_1$.
Since $\dim W_0=\dim W_2=1$ and $\dim W_1=2$, Lemma \ref{l:main} yields  $a_2=1$. Therefore $\lambda-\alpha_2 \in \Lambda(V)$ and so Lemma \ref{l:ammu2} implies that  $\mu_2=\mu_1-\beta_2+\beta_1$. 

If $a_1a_3 \neq 0$ then $\lambda-\alpha_1, \lambda-\alpha_3 \in \Lambda(V)$ both restrict to the same $T_X$-weight $\nu=\lambda|_{X}-\beta_1$. However, $\nu$ does not occur as a weight in $V_2$, while $\nu=\mu_1-\beta_1$ occurs in $V_1$ with multiplicity at most 1. This is a contradiction, so $a_1a_3=0$ and thus  
the possibilities for $\l$ are as follows:
$$\lambda_2, \lambda_2+\lambda_3,\lambda_1+\lambda_2,2\lambda_1+\lambda_2, \lambda_2+2\lambda_3,$$
giving $\dim V =  21, 104, 63, 309, 189$, respectively (see \cite[Table A.23]{Lubeck}). 

Now $V=V_1\oplus V_2$ and $\dim V_1=\dim V_2$, so $\dim V_1=52$ is the only possibility. If $p=0$ then it is easy to check that there are no $52$-dimensional irreducible $KX$-modules, so we may assume $p\neq 0$. By Steinberg's tensor product theorem we have
$$V_1 \cong S_1^{(q_1)} \otimes S_2^{(q_2)} \otimes \cdots \otimes S_k^{(q_k)}$$
for some $k \ge 1$, where each $S_i$ is a nontrivial $p$-restricted irreducible $KX$-module, the $q_i$ are certain powers of $p$, and $S_i^{(q_i)}$ is the twist of $S_i$ by the corresponding standard Frobenius morphism of $X$.
In particular, $\dim V_1= \prod_{i=1}^k \dim S_i$. However, by inspecting \cite[Table A.6]{Lubeck}, we see that we cannot obtain 52 as the product of dimensions of such $KX$-modules. This final contradiction completes the proof of the lemma.
\end{proof}

\begin{lem}\label{l:a22} 
Suppose $(G,H,V) \in \mathcal{T}$. Then $(a,p) \neq (2,5)$. 
\end{lem} 

\begin{proof} 
Suppose $(a,p)=(2,5)$, so $\ell=8$, $\dim W = 19$ (see \cite[Table A.6]{Lubeck}) and $G=B_9$. Let $P=QL$ be the parabolic subgroup of $G$ constructed from a $t$-stable Borel subgroup $B_X=U_XT_X$  of $X$. At the $U_X$-levels $0,1,2,3$ and $4$ we calculate that there are exactly $1, 2, 3, 2$ and $3$ distinct $T_X$-weights, respectively (see Table \ref{t:r5}), and since $\dim W=19$, it follows that each of these weights has multiplicity 1. Hence $\dim W_0=1$, $\dim W_1=\dim W_3=2$ and $\dim W_2=\dim W_4=3$, so 
$\Delta(L')=\{\alpha_2, \alpha_4,\alpha_5, \alpha_7,
\alpha_9\}$. In particular, Lemma \ref{l:main} implies that there exists $i\in\{2,7,9\}$
such that $a_i=1$ and $a_j=0$ for all $j\in\{2,4,5,7,9\}$, $j \neq i$. 

In the usual way, by appealing to Remark \ref{r:ord}, we may order the $T$-weights of $W$ (at levels less than 4) to give the root restrictions listed in Table \ref{t:r5}. At level 4, the zero $T$-weight $\lambda_1-\sum_{i=1}^{9}\alpha_i$ must restrict to the zero $T_X$-weight $\delta-2\beta_1-2\beta_2$. Moreover, since the weight space for $\lambda_1-\sum_{i
=1}^8\alpha_i$ is a singular 1-space, and since ${\rm Isom}(W_4)'$ acts transitively on the singular 1-spaces in $W_4$, we can assume  that the non-zero weight $\lambda_1-\sum_{i=1}^8\alpha_i$ restricts as a $T_X$-weight to $\delta-3\beta_1-\beta_2$.

\renewcommand{\arraystretch}{1.2}
\begin{table}
$$\begin{array}{llll} \hline
\mbox{$U_X$-level} & \mbox{$T_X$-weight} & \mbox{$T$-weight} & \mbox{Root restriction} \\ \hline
0 & \delta& \lambda_1& \\
1 & \delta-\beta_1& \lambda_1-\alpha_1& \alpha_1|_{X}=\beta_1 \\
 & \delta-\beta_2& \lambda_1-\alpha_1-\alpha_2& \alpha_2|_{X}=\beta_2-\beta_1 \\
2 & \delta-2\beta_2& \lambda_1-\alpha_1-\alpha_2-\alpha_3& \alpha_3|_{X}=\beta_2 \\
& \delta-2\beta_1& \lambda_1-\alpha_1-\alpha_2-\alpha_3-\alpha_4& \alpha_4|_{X}=2\beta_1-2\beta_2 \\
&  \delta-\beta_1-\beta_2 & \lambda_1-\alpha_1-\alpha_2-\alpha_3-\alpha_4-\alpha_5& \alpha_5|_{X}=\beta_2-\beta_1\\
3 & \delta-\beta_1-2\beta_2& \lambda_1-\sum_{i=1}^6\alpha_i& \alpha_6|_{X}=\beta_2 \\
& \delta-2\beta_1-\beta_2& \lambda_1-\sum_{i=1}^7\alpha_i& \alpha_7|_{X}=\beta_1-\beta_2 \\
4 & \delta-3\beta_1-\beta_2& \lambda_1-\sum_{i=1}^{8}\alpha_i& \alpha_{8}|_{X}=\beta_1 \\
& \delta-2\beta_1-2\beta_2& \lambda_1-\sum_{i=1}^{9}\alpha_i& \alpha_{9}|_{X}=\beta_2-\beta_1 \\
&  \delta-\beta_1-3\beta_2 & \lambda_1-\sum_{i=1}^{8}\alpha_i-2\alpha_9& \\
\hline 
\end{array}$$
\caption{}
\label{t:r5}
\end{table}
\renewcommand{\arraystretch}{1}

First assume $a_2=1$ and $a_j=0$ for all $j\in\{4,5,7,9\}$.  Now $\lambda-\alpha_2 \in \Lambda(V)$ and $(\lambda-\alpha_2)|_{X}=\lambda|_{X}-\beta_2+\beta_1$, so Lemma \ref{l:ammu2} yields $\mu_2=\mu_1-\beta_2+\beta_1$. Let $k \in \{1,3,6,8\}$ and suppose $a_k \neq 0$. Set $\nu_1 = \lambda-\alpha_1$, $\lambda-\alpha_3-\alpha_4-\alpha_5$, $\lambda-\alpha_6-\alpha_7$ or $\lambda-\alpha_8$ when $k = 1,3,6$ or $8$, respectively, and note that $\nu_1 \in \L(V)$. The weights $\lambda-\alpha_2-\alpha_3-\alpha_4$ and $\nu_1$ both restrict to the $T_X$-weight $\nu=\lambda|_{X}-\beta_1$. However, $\nu=\mu_1-\beta_1$ has multiplicity at most 1 in $V_1$ and it does not occur in $V_2$ since $\nu=\mu_2-2\beta_1+\beta_2$. This is a contradiction, so $a_k = 0$ and thus $\lambda=\lambda_2$.
Therefore \cite[Table A.29]{Lubeck} gives $\dim V=171$, which is a contradiction since $V$ is even-dimensional by Lemma \ref{l:vxsum}.

Next suppose that $a_7=1$ and $a_j=0$ for all $j\in \{2,4,5,9\}$. Then $\lambda-\alpha_7 \in \L(V)$ and $(\lambda-\alpha_7)|_{X}=\lambda|_{X}-\beta_1+\beta_2$, so Lemma \ref{l:ammu2} yields $\mu_2=\mu_1-\beta_1+\beta_2$.  Now $\lambda-\alpha_7-\alpha_8$  and $\lambda-\alpha_4-\alpha_5-\alpha_6-\alpha_7$ are $T$-weights of $V$ that both restrict to the same $T_X$-weight $\nu=\lambda|_{X}-2\beta_1+\beta_2$. However, $\nu$ occurs in $V_2$ as $\mu_2-\beta_1$ with multiplicity at most 1, and it does not occur in $V_1$ since $\nu=\mu_1-2\beta_1+\beta_2$. Again, this is a contradiction since $V= V_1\oplus V_2$. 

Finally, assume $a_9=1$ and $a_j=0$ for all $j \in\{2,4,5,7\}$. Here $\lambda-\alpha_9 \in \Lambda(V)$ and $(\lambda-\alpha_9)|_{X}=\lambda|_{X}-\beta_2+\beta_1$, so $\mu_2=\mu_1-\beta_2+\beta_1$ by Lemma \ref{l:ammu2}.  Let $k\in \{1,3,6,8\}$ and assume $a_k \neq 0$. Set $\nu_1 = \lambda-\alpha_1$, $\lambda-\alpha_3-\alpha_4-\alpha_5$, $\lambda-\alpha_6-\alpha_7$ or $\lambda-\alpha_8$ when $k = 1,3,6$ or $8$, respectively, and note that $\nu_1 \in \L(V)$. The weights $\lambda-\alpha_7-\alpha_8-\alpha_9$ and $\nu_1$  both restrict to the same $T_X$-weight $\nu=\lambda|_{X}-\beta_1$. However, $\nu=\mu_1-\beta_1$ occurs with multiplicity at most 1 in $V_1$, and it does not occur in $V_2$ since $\nu=\mu_2-2\beta_1+\beta_2$. This contradiction implies that $a_k = 0$ for all $k\in \{1,3,6,8\}$, so $\l = \lambda_9$ and $\dim V=512$. In particular, $\dim V_1=256$. As in the proof of Lemma \ref{l:a21}, we have $\dim V_1 = \prod_{i=1}^k \dim S_i$ for some $k \ge 1$, where each $S_i$ is a $p$-restricted irreducible $KX$-module. Since the dimension of each $S_i$ must be a power of 2, \cite[Table A.6]{Lubeck} implies that the only possibility is $\dim S_i=8$, but this gives a contradiction since 256 is not a power of 8.
\end{proof}

\begin{prop}\label{p:a2m}  
Theorem \ref{t:am} holds when $X = A_2$.
\end{prop}

\begin{proof} 
In view of the above results, it remains to eliminate the case $(a,p) = (3,7)$. Let us assume $(a,p) = (3,7)$ and $(G,H,V) \in \mathcal{T}$. Here $\ell=12$ and $\dim W=37$ (see 
\cite[Table A.6]{Lubeck}), so $G = B_{18}$. Let $P=QL$ be the parabolic subgroup of $G$ constructed from a $t$-stable Borel subgroup $B_X=U_XT_X$  of $X$. At the $U_X$-levels $0, 1, 2, 3, 4, 5$ and $6$ we calculate that there are exactly $1, 2, 3, 4, 3, 4$ and $3$ distinct $T_X$-weights, respectively (see Table \ref{t:r6}). Furthermore, since $\dim W=37$, each of these weights has multiplicity 1, so $\dim W_0=1$, $\dim W_1=2$,  $\dim W_2=\dim W_4=\dim W_6=3$ and $\dim W_3=\dim W_5=4$.
Now $a_2 \neq 1$ by Proposition \ref{p:a2redbis}, so  Lemma \ref{l:main} implies that 
$$\mbox{$a_{18}=1$  and $a_i=0$ for all $i \in \{2,4,5,7,8,9,11,12,14,15,16\}$.}$$

In view of Remark \ref{r:ord} we may order the $T$-weights in $W$ (at levels less than 6) to give the root restrictions listed in Table \ref{t:r6}. At level 6, the zero $T$-weight $\lambda_1-\sum_{i=1}^{18}\alpha_i$ must restrict to the zero $T_X$-weight $\delta-3\beta_1-3\beta_2$. Moreover, as the weight space for $\lambda_1-\sum_{i
=1}^{17}\alpha_i$ is a singular 1-space, and ${\rm Isom}(W_6)'$ is transitive on singular $1$-spaces in $W_6$, we can assume  that $\lambda_1-\sum_{i=1}^{17}\alpha_i$ restricts as a $T_X$-weight to $\delta-2\beta_1-4\beta_2$.

\renewcommand{\arraystretch}{1.2}
\begin{table}
$$\begin{array}{llll} \hline
\mbox{$U_X$-level} & \mbox{$T_X$-weight} & \mbox{$T$-weight} & \mbox{Root restriction} \\ \hline
0 & \delta& \lambda_1& \\
1 & \delta-\beta_1& \lambda_1-\alpha_1& \alpha_1|_{X}=\beta_1 \\
 & \delta-\beta_2& \lambda_1-\alpha_1-\alpha_2& \alpha_2|_{X}=\beta_2-\beta_1 \\
2 & \delta-\beta_1-\beta_2& \lambda_1-\alpha_1-\alpha_2-\alpha_3& \alpha_3|_{X}=\beta_1 \\
& \delta-2\beta_1& \lambda_1-\alpha_1-\alpha_2-\alpha_3-\alpha_4& \alpha_4|_{X}=\beta_1-\beta_2 \\
&  \delta-2\beta_2 & \lambda_1-\alpha_1-\alpha_2-\alpha_3-\alpha_4-\alpha_5& \alpha_5|_{X}=2\beta_2-2\beta_1\\
3 & \delta-\beta_1-2\beta_2& \lambda_1-\sum_{i=1}^6\alpha_i& \alpha_6|_{X}=\beta_1 \\
& \delta-2\beta_1-\beta_2& \lambda_1-\sum_{i=1}^7\alpha_i& \alpha_7|_{X}=\beta_1-\beta_2 \\
&  \delta-3\beta_1 & \lambda_1-\sum_{i=1}^8\alpha_i& \alpha_8|_{X}=\beta_1-\beta_2\\
& \delta-3\beta_2 & \lambda_1-\sum_{i=1}^9\alpha_i& \alpha_9|_{X}=3\beta_2-3\beta_1\\ 
4 & \delta-\beta_1-3\beta_2& \lambda_1-\sum_{i=1}^{10}\alpha_i& \alpha_{10}|_{X}=\beta_1 \\
& \delta-3\beta_1-\beta_2& \lambda_1-\sum_{i=1}^{11}\alpha_i& \alpha_{11}|_{X}=2\beta_1-2\beta_2 \\
&  \delta-2\beta_1-2\beta_2 & \lambda_1-\sum_{i=1}^{12}\alpha_i& \alpha_{12}|_{X}=\beta_2-\beta_1\\
5 & \delta-3\beta_1-2\beta_2& \lambda_1-\sum_{i=1}^{13}\alpha_i& \alpha_{13}|_{X}=\beta_1 \\
& \delta-2\beta_1-3\beta_2& \lambda_1-\sum_{i=1}^{14}\alpha_i& \alpha_{14}|_{X}=\beta_2-\beta_1 \\
&  \delta-4\beta_1-\beta_2 & \lambda_1-\sum_{i=1}^{15}\alpha_i& \alpha_{15}|_{X}=2\beta_1-2\beta_2\\
&  \delta-\beta_1-4\beta_2 & \lambda_1-\sum_{i=1}^{16}\alpha_i& \alpha_{16}|_{X}=3\beta_2-3\beta_1\\
6 & \delta-2\beta_1-4\beta_2& \lambda_1-\sum_{i=1}^{17}\alpha_i& \alpha_{17}|_{X}=\beta_1 \\
& \delta-3\beta_1-3\beta_2& \lambda_1-\sum_{i=1}^{18}\alpha_i& \alpha_{18}|_{X}=\beta_1-\beta_2 \\
&  \delta-4\beta_1-2\beta_2 & \lambda_1-\sum_{i=1}^{17}\alpha_i-2\alpha_{18}&\\
\hline 
\end{array}$$
\caption{}
\label{t:r6}
\end{table}
\renewcommand{\arraystretch}{1}

Since $a_{18}=1$ we have $\lambda-\alpha_{18} \in \Lambda(V)$ and this restricts to the $T_X$-weight $\lambda|_{X}-\beta_1+\beta_2$,  so Lemma \ref{l:ammu2} implies that $\mu_2=\mu_1-\beta_1+\beta_2$. 
Now the weight $\lambda-\alpha_{17}-2\alpha_{18}$ restricts as a $T_X$-weight to   
$\nu=\mu_1-3\beta_1+2\beta_2=\mu_2-2\beta_1+\beta_2$, but this is not under 
$\mu_1$ nor $\mu_2$, which contradicts the fact that $V = V_1\oplus V_2$.
\end{proof}

\section{The case $H^0=A_3$}\label{ss:am2}

Next we consider the triples $(G,H,V)$ with $H^0 = X = A_3$, so 
$$\delta=a\delta_1+b\delta_2+a\delta_3$$ 
for some non-negative integers $a$ and $b$. Note that if $\delta=\delta_2$ then 
$G=D_3\cong A_3$, which is incompatible with Hypothesis \ref{h:our}, so we may assume $\delta \neq \delta_2$.
Let $\mathcal{T}$ be the set of irreducible triples $(G,H,V)$ satisfying Hypothesis \ref{h:our}, where $H^0=A_3$ and $V=V_G(\l)$. We will eventually show that the only example in $\mathcal{T}$ is the following case:
$$\mbox{$G=D_{10}$, $\delta = 2\delta_2$, $\lambda=\lambda_9$ or $\lambda_{10}$, and 
$\lambda|_{X}=3\delta_1+\delta_2+\delta_3$ or $\delta_1+\delta_2+3\delta_3$}$$
with $p\neq 2,3,5,7$. Furthermore, in this case we will show that $H = A_3.2$ is not contained in $G$ (see Lemma \ref{l:a3d10}). We begin by recording a couple of preliminary lemmas.

\begin{lem}\label{l:a3subwt}
Suppose $\delta \not\in \{\delta_2,2\delta_2, 3\delta_2, \delta_1+\delta_3, \delta_1+\delta_2+\delta_3\}$.
If $b$ is even then set $\nu = 2\delta_1+2\delta_3$, otherwise set $\nu = 2\delta_1+\delta_2+2\delta_3$. Then $\nu$ is subdominant to $\delta$.
\end{lem}

\begin{proof}
First observe that 
\begin{align*}
\delta-(2\delta_1+2\delta_3) & =\frac12\left((2a+b-4)\beta_1+(2a+2b-4)\beta_2+(2a+b-4)\beta_3\right) \\
\delta-(2\delta_1+\delta_2+2\delta_3) & =\frac12\left((2a+b-5)\beta_1+(2a+2b-6)\beta_2+(2a+b-5)\beta_3\right)
\end{align*}
(see \cite[Table 1]{Hu1}).
Consequently, if $b$ is even then $\delta-(2\delta_1+2\delta_3)$ is an integral linear combination of the $\b_i$, and since $\delta\not \in \{2\delta_2,\delta_1+\delta_3\}$ we deduce that the coefficients are non-negative. Similarly, if $b$ is odd one can check that
$\delta-(2\delta_1+\delta_2+2\delta_3)$ is a linear combination of the $\b_i$ with non-negative integer coefficients.
\end{proof}

Let $B_X=U_XT_X$ be a  $t$-stable Borel subgroup of $X$ and let $W_i$ be the $i$-th $U_X$-level of $W$. As before, let $\ell$ be the level of the lowest weight $-\delta$. Note that $\ell=6a+4b$ and $\ell'=\ell/2=3a+2b$. Our main result on $U_X$-levels is Proposition \ref{p:a3levels} below. In the proof of this result we refer to the \emph{height} of a linear combination of simple roots $\nu = \sum_{i=1}^3d_i\beta_i$, denoted ${\rm ht}(\nu)$, where each $d_i$ is a non-negative integer. This is defined by 
${\rm ht}(\nu) = \sum_{i=1}^3d_i$. In addition, if $\nu$ is a dominant $T_X$-weight then it is convenient to write $\ell_{\nu}$ for the $U_X$-level of the lowest weight $-\nu$ in the module $V_{X}(\nu)$. Furthermore, if $\ell_{\nu}$ is even we set $\ell_{\nu}' = \ell_{\nu}/2$.

\begin{prop}\label{p:a3levels}
Suppose $\delta \not\in \{\delta_2, 2\delta_2,\delta_1+\delta_3\}$. Then $\dim W_i\geq 3$ 
for $2\leq i<\ell/2$ and $\dim W_{\ell/2}\geq 5$.
\end{prop}

We first handle some special cases.

\begin{lem}\label{l:a3base}
Proposition \ref{p:a3levels} holds if 
\begin{equation}\label{e:del}
\delta \in \{3\delta_2, 2\delta_1+2\delta_3, \delta_1+\delta_2+\delta_3, 2\delta_1+\delta_2+2\delta_3\}.
\end{equation}
Moreover, if $\delta\in \{3\delta_2, 2\delta_1+2\delta_3, 2\delta_1+\delta_2+2\delta_3\}$ then there are at least three distinct weights in $W_i$ for all $2 \leq i< \ell/2$, and at least five in $W_{\ell/2}$.
\end{lem}

\begin{proof}
We use Lemma \ref{l:pr}, which implies that every weight in the Weyl module $W_X(\delta)$ is also a weight of $W$. In this way one can check directly that the assertion holds for $\delta\in\{3\delta_2,2\delta_1+2\delta_3,2\delta_1+\delta_2+2\delta_3\}$. Similarly, if $\delta=\delta_1+\delta_2+\delta_3$ then $\ell=10$ and again we find that there are at least three distinct weights in $W_2,W_3$ and $W_4$. However, in $W_5$ there are only four distinct weights, but by inspecting \cite{LubeckW} we observe that at least one of these weights has multiplicity 2.
\end{proof}

\begin{proof}[Proof of Proposition \ref{p:a3levels}] 
In view of Lemma \ref{l:a3base}, we may assume that $\delta$ is not one of the weights in \eqref{e:del}.
By Lemma \ref{l:a3subwt}, there exists $\mu\in \{2\delta_1+2\delta_3,2\delta_1+\delta_2+2\delta_3\}$ such that $\mu$ is subdominant to $\delta$. More precisely, if $b$ is even then $\mu=2\delta_1+2\delta_3$, otherwise $\mu=2\delta_1+\delta_2+2\delta_3$. We proceed by induction on the height of $\delta - \mu$. 
If ${\rm ht}(\delta-\mu)=0$ then $\delta=\mu$ is one of the cases handled in Lemma \ref{l:a3base}, so we may assume that ${\rm ht}(\delta-\mu)>0$.

Set $\nu=\delta-\beta_2$ if $a=0$, otherwise $\nu=\delta-\beta_1-\beta_2-\beta_3$. Write $\nu=\sum_{i=1}^3d_i\delta_i$ and note that $\nu$ is subdominant to $\delta$. 
First assume $a=0$ and $b \ge 4$. Here $d_2=b-2$ and $\mu$ is subdominant to $\nu$ (see Lemma \ref{l:a3subwt}). Moreover, ${\rm ht}(\nu-\mu)<{\rm ht}(\delta-\mu)$ so by induction (and Lemma \ref{l:a3base}) it follows that $V_{X}(\nu)$ has at least three distinct weights from levels $2$ to $\ell_\nu'-1$, and at least five at level $\ell_\nu'$. Now ${\rm ht}(\delta-\nu)=1$ so Lemma \ref{l:indwt} implies that $W_i$ has at least three distinct weights for all $3 \le i <\ell/2$, and at least five at level $\ell/2$. It remains to exhibit three distinct weights at level 2, which is  straightforward.

Finally, suppose $a>0$. Here $d_2=b$ and once again Lemma \ref{l:a3subwt} implies that $\mu$ is subdominant to $\nu$ with ${\rm ht}(\nu-\mu)<{\rm ht}(\delta-\mu)$. Now ${\rm ht}(\delta-\nu)=3$ so by applying Lemma \ref{l:indwt} and induction we deduce that $W_i$ has sufficiently many distinct weights for all $i \ge 5$.  Finally, it is easy to find three distinct weights at levels $2$, $3$ and $4$. 
\end{proof}

\begin{lem}\label{l:a3red} 
Suppose $(G,H,V) \in \mathcal{T}$. Then $\delta=\delta_1+\delta_3$ or $\delta=2\delta_2$.
\end{lem}

\begin{proof} 
First we prove that either $\delta=\delta_1+\delta_3$ or $b \geq 1$. Seeking a contradiction, let us assume that $\delta = a\delta_1 +a\delta_3$ with $a \ge 2$. By Proposition \ref{p:a3levels} and Lemma \ref{l:main} it follows that $a_2=1$. Recall that $V|_X=V_1\oplus V_2$, where $V_1$ and $V_2$ are non-isomorphic irreducible $KX$-modules with respective $T_X$-highest weights $\mu_1$ and $\mu_2$, which are interchanged by an involutory graph automorphism $t$ of $X$. In particular, if $\mu_1=c_1\delta_1+c_2\delta_2+c_3\delta_3$ then $\mu_2=c_3\delta_1+c_2\delta_2+c_1\delta_3$, so $c_1\neq c_3$ since $V_1$ and $V_2$ are non-isomorphic $KX$-modules (by Proposition \ref{p:niso}).

Let $P_X=Q_XL_X$ be 
the maximal parabolic subgroup of $X$ with $\Delta(L_X')=\{\beta_1, \beta_2\}$, and let $P=QL$ be the  parabolic subgroup of $G$ constructed from $P_X$ (see Section \ref{ss:parabs}).  Since $c_1\neq c_3$, Lemma \ref{l:centre} implies that $V/[V,Q]$ is an irreducible $KL_X'$-module, so without loss of generality we may assume that $V/[V,Q] = V_1/[V_1,Q_X]$ as $KL_X'$-modules. Let $W_i$ denote the $i$-th $Q_X$-level of $W$. Write  $L'=L_1\cdots L_r$ (see Remark \ref{r:ldash}), where each $L_i$ is simple with natural module $Y_i=W_{i-1}$ (for example, if $a=2$ then the $Q_X$-level of the lowest weight $-\delta$ is $4$, and $\dim W_i \ge 6$ for $i=0,1,2$, so $r=3$). Also write $V/[V,Q]=M_1 \otimes \cdots \otimes M_r$, where each $M_i$ is a  $p$-restricted irreducible $KL_i$-module. Of course, since $V/[V,Q]$ is irreducible as a $KL_X'$-module, the $M_i$ are also irreducible $KL_X'$-modules. By construction we have $L_X'\leqs L'$ (see Lemma \ref{l:flag}). Let $\pi_i$ be the projection map from $L_X'$ to $L_i$, so either $\pi_i(L_X')$ is trivial or $\pi_i(L_X')=A_2$. 

Observe that $Y_1|_{L_X'}$ is irreducible with highest weight $a\delta_1|_{L_X'}$ (see Lemma \ref{l:vq}(ii)), so $\dim Y_1 \geq 6$ (hence $L_1 = A_s$ with $s \geq 5$) and $\pi_1(L_X')=A_2$. In particular, we  may view $L_X'$ as a proper subgroup of $L_1$.

Consider the triple $(L_X',L_1, M_1)$. Here $M_1$ is a $p$-restricted irreducible $KL_1$-module and $M_1|_{L_X'}$ is also irreducible (in particular, $M_1$ is tensor indecomposable as a $KL_1$-module -- see Proposition \ref{p:s16}).
Since $a_2=1$ we also note that $M_1 \neq Y_1, Y_1^*$.
Therefore, by the main theorem of \cite{Seitz2}, this triple must be one of the cases appearing in \cite[Table 1]{Seitz2}. Now $L_X' = A_2$ and $L_1 = A_s$ (with $s \ge 5$), so by inspecting \cite[Table 1]{Seitz2}  we deduce that $s=5$ and $Y_1|_{L_X'}$ has highest weight $2\delta_1|_{L_X'}$ (this is the case labelled ${\rm I}_{7}$ with $n=2$). Therefore $a=2$ is the only possibility and we have reduced to the case $\delta=2\delta_1+2\delta_3$.

We now return to a $t$-stable Borel subgroup $B_X=U_XT_X$ of $X$ and set $P=QL$ to be the parabolic subgroup of $G$ constructed from the $U_X$-levels of $W$. Since $\delta=2\delta_1+2\delta_3$ we have $\ell'=\ell/2=6$. By Remark \ref{r:ord} we can order the first ten $T$-weights of $W$ so that we obtain the root restrictions listed in Table \ref{t:r7}.

\renewcommand{\arraystretch}{1.2} 
\begin{table}
$$\begin{array}{llll} \hline
\mbox{$U_X$-level} & \mbox{$T_X$-weight} & \mbox{$T$-weight} & \mbox{Root restriction} \\ \hline
0 & \delta& \lambda_1& \\
1 & \delta-\beta_1& \lambda_1-\alpha_1& \alpha_1|_{X}=\beta_1 \\
 & \delta-\beta_3& \lambda_1-\alpha_1-\alpha_2& \alpha_2|_{X}=\beta_3-\beta_1 \\
2 & \delta-\beta_2-\beta_3& \lambda_1-\alpha_1-\alpha_2-\alpha_3& \alpha_3|_{X}=\beta_2 \\
& \delta-2\beta_3& \lambda_1-\alpha_1-\alpha_2-\alpha_3-\alpha_4& \alpha_4|_{X}=\beta_3-\beta_2 \\
&  \delta-\beta_1-\beta_2 & \lambda_1-\alpha_1-\alpha_2-\alpha_3-\alpha_4-\alpha_5& \alpha_5|_{X}=\beta_1+\beta_2-2\beta_3\\
& \delta-2\beta_1& \lambda-\sum_{i=1}^6 \alpha_i
&
 \alpha_6|_{X}=\beta_1-\beta_2\\
 & \delta-\beta_1-\beta_3 & \lambda-\sum_{i=1}^7 \alpha_i & \alpha_7|_{X}=\beta_3-\beta_1\\
3 & \delta-2\beta_1-\beta_3 & \lambda_1-\sum_{i=1}^8\alpha_i& \alpha_8|_{X}=\beta_1 \\
& \delta-\beta_1-2\beta_3 & \lambda_1-\sum_{i=1}^9\alpha_i& \alpha_9|_{X}=\beta_3-\beta_1 \\
\hline 
\end{array}$$
\caption{}
\label{t:r7}
\end{table}
\renewcommand{\arraystretch}{1}

Since $a_2=1$ we have $\lambda-\alpha_2 \in \Lambda(V)$, which restricts as a $T_X$-weight to $\lambda|_{X}-\beta_3+\beta_1$, so Lemma \ref{l:ammu2} yields 
$\mu_2=\mu_1-\beta_3+\beta_1$. Next we observe that the following weights of $V$
$$\lambda-\alpha_1-2(\alpha_2+\alpha_3+\alpha_4)-\alpha_5-\alpha_6, \; 
\lambda -\sum_{i=1}^9\alpha_i, \; \lambda -\alpha_1-2\alpha_2-\sum_{i=3}^8\alpha_i
$$ 
all restrict to the same $T_X$-weight $\nu=\lambda|_{X}-\beta_1-2\beta_3$. 
Now $\nu$ occurs in both $V_1$ and $V_2$ with multiplicity at most 1 since $\nu = \mu_1-\beta_1-2\beta_3 = \mu_2-2\beta_1-\beta_3$. But we have just observed that $m_{V}(\nu) \ge 3$, which contradicts the fact that $V =V_1\oplus V_2$. Therefore $\delta \neq a\delta_1 +a\delta_3$ with $a \ge 2$.

To complete the proof of the lemma we may assume that $\delta=a\delta_1+b\delta_2+a\delta_3$ where $b\geq 1$, and $b>2$ if $a=0$ (recall that $\delta \neq \delta_2$ since $H^0 \neq G$).  As before, Proposition \ref{p:a3levels} and Lemma \ref{l:main} imply that $a_2=1$, and we may consider the parabolic subgroup of $G$ constructed from the parabolic subgroup $P_X=Q_XL_X$ of $X$ with $\Delta(L_X')=\{\beta_1,\beta_2\}$. By repeating the earlier argument in the first part of the proof, we reach a contradiction. We leave the details to the reader.
\end{proof}

\begin{lem}\label{l:a3red22} 
Suppose $(G,H,V) \in \mathcal{T}$. Then $\delta \neq \delta_1+\delta_3$.
\end{lem}

\begin{proof} 
Seeking a contradiction, let us assume $\delta = \delta_1+\delta_3$, so $\dim W=15-\delta_{2,p}$ and $G$ is orthogonal (see \cite[Table 2]{Brundan}). If $B_X=U_XT_X$ is a $t$-stable Borel subgroup of $X$ then it is easy to verify that $6$ is the $U_X$-level of the lowest weight $-\delta$, and the dimensions of the $U_X$-levels $0,1,2$ and $3$ are $1,2,3$ and $3-\delta_{2,p}$, respectively (see Lemma \ref{l:am1m} to follow). In particular, Lemma \ref{l:main} implies that $a_2=1$ when $p=2$, and $a_2+a_7 = 1$ when $p \neq 2$.

Consider the parabolic subgroup $P_X=Q_XL_X$ of $X$ with $\Delta(L_X')=\{\beta_1,\beta_2\}$ and construct the parabolic subgroup $P=QL$ of $G$ in the usual way. Let $W_i$ denote the $i$-th $Q_X$-level of $W$ and note that $\ell=2$ is the $Q_X$-level of the lowest weight $-\delta$. If $p=2$ then $W_0$, $W_1$ and $W_2$ are irreducible $KL_X'$-modules with respective highest weights $\delta_1|_{L_X'}$, $(\delta_1+\delta_2)|_{L_X'}$ and $\delta_2|_{L_X'}$, and respective dimensions 3, 8 and 3. Similarly, if $p\neq 2$ then $W_0$ and $W_2$ are both irreducible $3$-dimensional $KL_X'$-modules with respective highest weights 
$\delta_1|_{L_X'}$ and $\delta_2|_{L_X'}$, while $W_1$ is a $9$-dimensional reducible $KL_X'$-module with a composition factor of highest weight $(\delta_1+\delta_2)|_{L_X'}$, and one (if $p \neq 3$) or two (if $p=3$) trivial composition factors. Therefore $L'=L_1L_2$, where each $L_i$ is simple with natural module $Y_i=W_{i-1}$. More precisely, $L_1=A_2$ has a root system with base $\{\alpha_1,\alpha_2\}$, and $L_2=D_4$ ($p=2$) or $B_4$ ($p \neq 2$) has root system with base   
$\{\alpha_4,\alpha_5,\alpha_6,\alpha_7\}$. Let $\pi_2$ be the projection map from $L_X'$ to $L_2$. Since $Y_2|_{L_X'}$ is nontrivial (it has a composition factor of highest weight $(\delta_1+\delta_2)|_{L_X'}$) it follows that $\pi_2(L_X')=A_2$ and we may view $L_X'$ as a proper subgroup of $L_2$. 

Write $V/[V,Q]=M_1\otimes M_2$ 
where each $M_i$ is a $p$-restricted irreducible $KL_i$-module.  
Using the fact that $c_1\neq c_3$ (recall that $V_1$ and $V_2$ are non-isomorphic $KX$-modules), Lemma \ref{l:centre} implies that  
 $V/[V,Q]$ is an irreducible $KL_X'$-module, and without loss of generality we may  assume that $V/[V,Q]=V_1/[V_1,Q_X]$ as $KL_X'$-modules.  In particular $M_1$ and $M_2$ are irreducible $KL_X'$-modules. 
 
We claim that $M_2$ is trivial. Seeking a contradiction, suppose $M_2$ is nontrivial and first assume $p=2$. If we consider the triple $(L_X',L_2,M_2)$, then the main theorem of \cite{Seitz2} implies that $M_2 = Y_2$. In particular,  
$M_2|_{L_X'}$ has highest weight $(\delta_1+\delta_2)|_{L_X'}$, which is $p$-restricted. We also note that $M_1|_{L_X'}$ is $p$-restricted. Indeed, given the embedding of $L_X'$ in $L_1$ (via the action on $W_0$), it follows that $M_1|_{L_X'}$ 
has highest weight $(a_1\delta_1+\delta_2)|_{L_X'}$. Therefore we have $V/[V,Q]=M_1\otimes M_2$, where each $M_i$ is a nontrivial $p$-restricted irreducible $KL_X'$-module. However, Proposition \ref{p:s16} now implies that $V/[V,Q]$ is a reducible $KL_X'$-module, which is a contradiction.
Similarly, if $p\neq 2$ then we focus on
the triple $(L_X', L_2, M_2)$. Here $M_2$ is a $p$-restricted irreducible $KL_2$-module and $M_2|_{L_X'}$ is irreducible, but $Y_2=W_1$ is reducible as a $KL_X'$-module. In particular, we note that $M_2 \neq Y_2,Y_2^*$, but this configuration contradicts Corollary \ref{c:g51}. We conclude that $M_2$ is trivial.
Therefore $a_4=a_5=a_6=a_7=0$. In particular, since $a_7=0$ it follows that $a_2=1$ in all characteristics, so we have reduced to the case
$$\lambda=a_1\lambda_1+\lambda_2+a_3\lambda_3.$$ 
In addition, since $V/[V,Q] = V_1/[V_1,Q_X]$ as $KL_X'$-modules, we have $\mu_1 = \lambda|_{X}=a_1\delta_1+\delta_2+c_3\delta_3$ for some $c_3 \geq 0$.

Let us now consider a $t$-stable Borel subgroup $B_X=U_XT_X$ of $X$. Let $W_i$ denote the $i$-th $U_X$-level of $W$ and let $P=QL$ be the parabolic subgroup of $G$ constructed from $B_X$ in the usual way. Note that $\dim W_0=1$, $\dim W_1=2$, $\dim W_2=3$ and by applying Lemma \ref{l:s816} we see that $\dim W_3 = 3-\delta_{2,p}$. 
Therefore $L'=L_1L_2$ ($p=2$) or $L_1L_2L_3$ ($p \neq 2$), where each $L_i$ is simple with natural module $Y_i=W_i$. 
As stated in Remark \ref{r:ord}, we may order the first six  $T$-weights in $W$ to give the root restrictions in Table \ref{t:r8}. Since $a_2 =1$ we have $\lambda-\alpha_2 \in \Lambda(V)$, which restricts as a $T_X$-weight to $\lambda|_{X}-\beta_3+\beta_1$, so Lemma \ref{l:ammu2} yields 
$\mu_2=\mu_1-\beta_3+\beta_1$. 

\renewcommand{\arraystretch}{1.2}
\begin{table}
$$\begin{array}{llll} \hline
\mbox{$U_X$-level} & \mbox{$T_X$-weight} & \mbox{$T$-weight} & \mbox{Root restriction} \\ \hline
0 & \delta& \lambda_1& \\
1 & \delta-\beta_1& \lambda_1-\alpha_1& \alpha_1|_{X}=\beta_1 \\
 & \delta-\beta_3& \lambda_1-\alpha_1-\alpha_2& \alpha_2|_{X}=\beta_3-\beta_1 \\
2 & \delta-\beta_1-\beta_3& \lambda_1-\alpha_1-\alpha_2-\alpha_3& \alpha_3|_{X}=\beta_1 \\
& \delta-\beta_2-\beta_3& \lambda_1-\alpha_1-\alpha_2-\alpha_3-\alpha_4& \alpha_4|_{X}=\beta_2-\beta_1 \\
& \delta-\beta_1-\beta_2& \lambda_1-\alpha_1-\alpha_2-\alpha_3-\alpha_4-\alpha_5& \alpha_5|_{X}=\beta_1-\beta_3 \\
\hline 
\end{array}$$
\caption{}
\label{t:r8}
\end{table}
\renewcommand{\arraystretch}{1}

Suppose  $a_3\neq 0$.  Then $\lambda-\alpha_3-\alpha_4$ and $\lambda-\sum_{i=2}^5\alpha_i$  belong to $\Lambda(V)$ and restrict to the same $T_X$-weight $\nu=\lambda |_{X}-\beta_2$. Now $\nu=\mu_1-\beta_2=\mu_2-\beta_1-\beta_2+\beta_3$, so  $\nu$ occurs in $V_1$ with multiplicity at most 1, but it does not occur in $V_2$. This is a contradiction, so $a_3=0$.

Next assume $a_1 \neq 0$. The weights $\lambda-\alpha_1-\alpha_2, \lambda-\alpha_2-\alpha_3 \in \L(V)$ both restrict to the $T_X$-weight $\nu=\lambda|_{X}-\beta_3$. Now $\nu = \mu_1-\beta_3 = \mu_2-\beta_1$, so $m_{V_i}(\nu) \le 1$ for $i=1,2$. By Lemma \ref{l:s816}, $\lambda-\alpha_1-\alpha_2$ has multiplicity 2 in $V$, unless $a_1=p-2$, in which case its multiplicity is 1. By comparing multiplicities, it follows that $a_1=p-2$ so $\mu_1=(p-2)\delta_1+\delta_2+c_3\delta_3$ and $\mu_2=c_3\delta_1+\delta_2+(p-2)\delta_3$. Since $\mu_2=\mu_1-\beta_3+\beta_1$ it follows that $c_3=p$ and thus $\nu$ is not a weight under $\mu_1$. Since $V=V_1\oplus V_2$ it follows that $m_{V}(\nu) \le 1$, which is a contradiction. We conclude that $a_1=0$.

We have now reduced to the case $\lambda=\lambda_2 = 2\lambda_1-\alpha_1$, so 
$\mu_1=2\delta-\beta_1=\delta_2+2\delta_3$.  
If $p\neq 2$ then \cite[Table A.27]{Lubeck} gives $\dim V=105$, which is absurd since  $\dim V$ has to be  even by Lemma \ref{l:vxsum}. Finally, if $p=2$ then $\dim V = 90$ (see \cite[Table A.44]{Lubeck}) and $V_X(\mu_1)=V_X(\delta_2) \otimes V_X(\delta_3)^{(2)}$. Therefore 
$\dim V_1=6\cdot 4=24 < \frac{1}{2}\dim V$, a contradiction.
\end{proof}

\begin{prop}\label{p:a3m}  
Theorem \ref{t:am} holds when $X = A_3$.
\end{prop}

\begin{proof} 
By Lemmas \ref{l:a3red} and \ref{l:a3red22}, it remains to deal with the case $\delta=2\delta_2$ (note that $p\neq 2$ since $\delta$ is $p$-restricted). By \cite[Table A.7]{Lubeck}, if $p= 3$  then $\dim W=19$ and so $G=B_9$, otherwise $\dim W=20$ and  $G=D_{10}$ (since $\delta(h_{\b}(-1))=1$ for all $\b \in \Phi^+(X)$, this follows from Lemma \ref{l:st}).

We first consider a $t$-stable Borel subgroup $B_X=U_XT_X$ of $X$.  As usual, let $W_i$ denote the $i$-th $U_X$-level of $W$ and note that $\ell=8$ is the $U_X$-level of the lowest weight $-\delta$. We calculate that $\dim W_0=1$, $\dim W_1=1$, $\dim W_2=3$, $\dim W_3=3$ and $\dim W_4 = 4-\delta_{3,p}$, and by applying Lemma \ref{l:main} we deduce that $a_3=a_4=a_6=a_7=0$. In addition, if $p=3$ then $a_9=1$, otherwise $a_9+a_{10}=1$. 

Next consider the parabolic subgroup $P_X=Q_XL_X$ of $X$ with $\Delta(L_X')=\{\beta_1,\beta_2\}$ and construct the parabolic subgroup $P=QL$ of $G$ in the usual way.  Let $W_i$ denote the $i$-th $Q_X$-level of $W$ and note that the $Q_X$-level of $-\delta$ is $\ell=2$. Now $W_0$ and $W_1$ are irreducible $KL_X'$-modules with respective highest weights $2\delta_2|_{L_X'}$ and $(\delta_1+\delta_2)|_{L_X'}$, and respective dimensions $6$ and $8-\delta_{3,p}$.  Therefore $L'=L_1L_2$ where each $L_i$ is simple with natural module $Y_i=W_{i-1}$. 
Write 
$V/[V,Q]=M_1\otimes M_2$ 
where each $M_i$ is a $p$-restricted irreducible $KL_i$-module. 

Using the fact that $c_1\neq c_3$ (see \eqref{e:ci10}, and recall that $V_1$ and $V_2$ are non-isomorphic $KX$-modules), Lemma \ref{l:centre} implies that  
$V/[V,Q]$ is an irreducible $KL_X'$-module. In particular, we note that both $M_1$ and $M_2$ are irreducible $KL_X'$-modules. Without loss of generality we may assume that $V/[V,Q]=V_1/[V_1,Q_X]$ as $KL_X'$-modules.  

Suppose $p=3$.   Let $\pi_2$ be the projection from $L_X'$ to $L_2$. As $\pi_2$ is the representation afforded by $(\delta_1+\delta_2)|_{L_X'}$, we have $\pi_2(L_X')=A_2$ so we may view $L_X'$ as a proper subgroup of $L_2$. Consider the triple $(L_X', L_2, M_2)$. Here $M_2$ is a nontrivial $p$-restricted irreducible $KL_2$-module (nontrivial since $a_9=1$) and $M_2|_{L_X'}$ is irreducible.  Also $M_2\neq Y_2,Y_2^*$ (since $a_9=1$). By applying the main theorem of \cite{Seitz2} we deduce that  there are no compatible configurations.

For the remainder we may assume $p \neq 2,3$, so $G=D_{10}$, $L_1=A_5$ and $L_2=D_4$. As before, let us consider the triple $(L_X', L_2, M_2)$. By inspection we see that there are no configurations of this type in  \cite[Table 1]{Seitz2}. However, the main theorem of \cite{Seitz2}  allows for additional configurations arising from graph automorphisms, and since $a_9+a_{10}=1$ we deduce that $M_2$ is one of the spin modules for $L_2$. In particular, since we know that $Y_2|_{L_X'}$ is irreducible with $p$-restricted highest weight $(\delta_1+\delta_2)|_{L_X'}$, and $M_2 = Y_2^{\tau}$ as $KL_2$-modules (where $\tau$ is a suitable triality graph automorphism of $L_2$), it follows that $M_2|_{L_X'}$ 
is $p$-restricted. 

Next we claim that $M_1$ is trivial, so $\lambda=\lambda_9$ or $\lambda_{10}$. Seeking a contradiction, suppose $M_1$ is nontrivial. Let $\pi_1$ be the projection from $L_X'$ to $L_1$. 
Since $Y_1|_{L_X'}$ has highest weight $2\delta_2|_{L_X'}$, it follows that $\pi_1(L_X')=A_2$ and so we may view $L_X'$ as a proper subgroup of $L_1$.  Consider the triple $(L_X', L_1, M_1)$. Here $M_1$ is a $p$-restricted irreducible $KL_1$-module and $M_1|_{L_X'}$ is irreducible.  If $M_1 \neq Y_1, Y_1^*$ then Seitz's main theorem in \cite{Seitz2} implies that $M_1$ has highest weight $\lambda_2|_{L_1}$, 
and $M_1|_{L_X'}$ has highest weight $(2\delta_1+\delta_2)|_{L_X'}$ (in \cite[Table 1]{Seitz2}, the only possibility is the case labelled ${\rm I}_{7}$, with $n=2$). On the other hand, if $M_1=Y_1$ or $Y_1^*$ then $M_1|_{L_X'}$ has highest weight $2\delta_2|_{L_X'}$ or $2\delta_1|_{L_X'}$, respectively. We conclude that in all cases $M_1|_{L_X'}$ is $p$-restricted. Therefore $V/[V,Q]=M_1\otimes M_2$, as a tensor product of two $p$-restricted irreducible $KL_X'$-modules, is reducible by Proposition \ref{p:s16}. This is a contradiction, so $M_1$ is trivial and thus $\lambda=\lambda_9$ or $\lambda_{10}$, so $\dim V=512$. 

To complete the proof it remains to show that $V|_{H}$ is irreducible if and only if $p \neq 5,7$. To do this, we take $B_X=U_XT_X$ a $t$-stable Borel subgroup of $X$ as above. By appealing to Remark \ref{r:ord} we may order the $T$-weights in the $U_X$-levels $0,1,2$ and $3$ of $W$ to obtain the root restrictions listed in Table \ref{t:r9}.

\renewcommand{\arraystretch}{1.2}
\begin{table}
$$\begin{array}{llll} \hline
\mbox{$U_X$-level} & \mbox{$T_X$-weight} & \mbox{$T$-weight} & \mbox{Root restriction} \\ \hline
0 & \delta& \lambda_1& \\
1 & \delta-\beta_2& \lambda_1-\alpha_1& \alpha_1|_{X}=\beta_2 \\
2 & \delta-2\beta_2 & \lambda_1-\alpha_1-\alpha_2& \alpha_2|_{X}=\beta_2 \\
& \delta-\beta_1-\beta_2 & \lambda_1-\alpha_1-\alpha_2-\alpha_3& \alpha_3|_{X}=\beta_1-\beta_2 \\
 & \delta-\beta_2-\beta_3 & \lambda_1-\alpha_1-\alpha_2-\alpha_3-\alpha_4 & \alpha_4|_{X}=\beta_3-\beta_1 \\
3 & \delta-2\beta_2-\beta_3& \lambda_1-\sum_{i=1}^5\alpha_i& \alpha_5|_{X}=\beta_2 \\
& \delta-\beta_1-\beta_2-\beta_3& \lambda_1- \sum_{i=1}^6 \alpha_i & \alpha_6|_{X}=\beta_1-\beta_2 \\
& \delta-\beta_1-2\beta_2& \lambda_1-\sum_{i=1}^7 \alpha_i& \alpha_7|_{X}=\beta_2-\beta_3 \\
\hline 
\end{array}$$
\caption{}
\label{t:r9}
\end{table}
\renewcommand{\arraystretch}{1}

We now examine the possibilities for the restriction of $T$-weights to $T_X$-weights at $U_X$-level $4$ in $W$. First note that  the weights at level 4 are $\delta-2\beta_1-2\beta_2$, $\delta-2\beta_2-2\beta_3=-(\delta-2\beta_1-2\beta_2)$ and the zero weight $\delta-\beta_1-2\beta_2-\beta_3$, which have respective multiplicities $1,1$ and $2$. 
Let 
$$\nu_1= \lambda_1-\sum_{i=1}^8 \alpha_i, \  \nu_2=\lambda_1-\sum_{i=1}^9\alpha_i, \ \nu_3=\lambda_1-\sum_{i=1}^8\alpha_i-\alpha_{10},\ \nu_4=\lambda_1-\sum_{i=1}^{10} \alpha_i$$
and note that  $\nu_1= -\nu_4$ and $\nu_2=-\nu_3$. Let 
$\mathcal{A} = \{\delta-2\beta_1-2\beta_2,\delta-2\beta_2-2\beta_3\}$ and $\mu = \delta-\beta_1-2\beta_2-\beta_3$. There are two possibilities: either 
$\{\nu_1|_{X},\nu_4|_{X}\} = \mathcal{A}$ and $\nu_2|_{X} = \nu_3|_{X} = \mu$, or $\{\nu_2|_{X},\nu_3|_{X}\} = \mathcal{A}$ and $\nu_1|_{X}=\nu_4|_{X} = \mu$. 
In the first case we get
$$\alpha_8 |_{X}=\beta_1, \ \alpha_9|_{X}=\alpha_{10}|_{X}=\beta_3-\beta_1; \mbox{ or } \alpha_8|_{X}=2\beta_3-\beta_1, \ \alpha_9|_{X}=\alpha_{10}|_{X}=\beta_1-\beta_3,$$
while the second case yields the restrictions
$$\alpha_8 |_{X}=\beta_3, \ \alpha_9|_{X}=\beta_1-\beta_3, \ \alpha_{10}|_{X}=\beta_3-\beta_1; \mbox{ or }$$
$$\alpha_8|_{X}=\beta_3, \ \alpha_9|_{X}=\beta_3-\beta_1,\ \alpha_{10}|_{X}=\beta_1-\beta_3.$$ 

Suppose that  $\lambda=\lambda_9$ (the case $\l = \l_{10}$ is entirely similar), so
$$\lambda= \frac12(\alpha_1+2\alpha_2+3\alpha_3+4\alpha_4+5\alpha_5+6\alpha_6+7\alpha_7+8\alpha_8+5\alpha_9+4\alpha_{10}).$$
From the above restrictions we get $(8\alpha_8+5\alpha_9+4\alpha_{10})|_{X}= 9\beta_3-\beta_1$ or $7\beta_3+\beta_1$, and thus  
$\lambda|_{X}= 2\beta_1+3\beta_2+3\beta_3$ or $3\beta_1+3\beta_2+2\beta_3$, so  
$\lambda|_{X}=\delta_1+\delta_2+3\delta_3$ or $3\delta_1+\delta_2+\delta_3$. 
Therefore  $V_1$ has highest weight   $\delta_1+\delta_2+3\delta_3$ and $V_2$ has highest weight
$3\delta_1+\delta_2+\delta_3$ (or vice versa). If $p = 5$ or $7$ then $\dim V_1$ is respectively 173 or 211 (see \cite[Table A.7]{Lubeck}), which is a contradiction since $\dim V = 512$. However, if $p>7$ (or $p=0$) then \cite[Table A.7]{Lubeck} states that $\dim V_1=256$, which establishes the desired irreducibility of  $V$ as a $KH$-module. This is the case recorded in part (ii) of Theorem \ref{t:new}. 
\end{proof}

Finally, in the one example which arises here we show that the disconnected subgroup $H=A_3.2$ is not contained in the simple group $G=D_{10}$.

\begin{lem}\label{l:a3d10}
Let $(G,H,V)$ be a triple satisfying Hypothesis \ref{h:our}, with $G=D_{10}$, $H=A_3.2$ and $\l = \l_9$ (or $\l_{10}$). Then $H$ is not a subgroup of $G$.  
\end{lem}

\begin{proof}
First observe that $p \neq 2,3,5,7$ and $H = A_3.2 \cong D_3.2 = {\rm GO}_{6}$. Let $U$ denote the natural $6$-dimensional module for ${\rm GO}_{6}$ and let $S^2(U)$ be the symmetric-square of $U$. As a module for ${\rm GO}_{6}$,  
$S^2(U)$ has exactly two composition factors, namely a trivial module, and an irreducible $20$-dimensional module which we can identify with $W$. Let $t \in {\rm GO}_{6}$ be a reflection (i.e. a diagonal matrix $[-I_1,I_5]$). To prove the lemma, it suffices to show that a matrix representing the action of $t$ on $W$ has determinant $-1$. An easy calculation with the symmetric-square reveals that $t$ acts on $S^2(U)$ as $[-I_5,I_{16}]$ (with respect to a suitable basis), so we just need to check that $t$ acts as $1$ on the $1$-dimensional submodule. One way to see this is to view $S^2(U)$ as the $K$-space of quadratic forms on $U$ and observe that ${\rm GO}_{6}$ fixes a unique quadratic form on $U$ (up to scalars). Therefore, $t$ acts on $W$ as $[-I_{5},I_{15}]$ (up to conjugacy) and the result follows.
\end{proof}

\section{The case $H^0=A_m$, $m \ge 4$}\label{ss:am3}  

We will henceforth assume that $H^0=X=A_m$ with 
$m \ge 4$. We begin by recording some preliminary results which will be useful later.

\subsection{Borel analysis}\label{sss:1}

Let $B_X=U_XT_X$ be a $t$-stable Borel subgroup of $X$ and let $P=QL$ be the corresponding parabolic subgroup of $G$ stabilizing the flag
\begin{equation}\label{e:flag}
W > [W,U_X] > [W,U_X^2] > \cdots > 0.
\end{equation}
Recall that the $U_X$-level of a $T_X$-weight $\mu = \delta-\sum_{i}d_i\b_i$ of $W$ is the integer $\sum_{i}d_i$, while the $i$-th $U_X$-level of $W$ is the subspace
$$W_i = \bigoplus_{(d_1, \ldots, d_m) \in \mathcal{S}_i}W_{\delta - \sum_{i}d_i\b_i}$$
where 
$$\mathcal{S}_{i} = \{(d_1, \ldots, d_m) \in \Z^m \mid \sum_{j}d_j=i, d_j \ge 0\}.$$
Let $\ell$ be the $U_X$-level of the lowest weight $-\delta$, so $\ell$ is minimal such that $[W,U_X^{\ell+1}]=0$. As before, we set $\ell'=\lfloor \ell/2\rfloor$, and we will sometimes write $\ell_\delta$ for $\ell$, and $\ell_\delta'$ for $\ell'$. Also recall that $(W_j)^* \cong W_{\ell-j}$ as $KL_X'$-modules, for all $0 \le j \le \ell'$.

First we consider the special case $\delta=\delta_1+\delta_m$. Note that the next lemma is valid for all $m\geq 3$. 

\begin{lem}\label{l:am1m}
Suppose $\delta=\delta_1+\delta_m$. Then $\ell=2m$, $\dim W_j=j+1$ for all $0 \leq j < \ell/2$, and $\dim W_{\ell/2}=m-\e$, where $\e=1$ if $p$ divides $m+1$, otherwise $\e=0$. In particular, $\dim W= (m+1)^2-1-\e$.  
\end{lem}

\begin{proof} 
Since weight multiplicities do not depend on the isogeny type of $X$, we may assume without loss of generality that $X$ is simply connected.
Then $W=\mathcal{L}(X)/Z(\mathcal{L}(X))$, where $\mathcal{L}(X)$ is the Lie algebra of $X$, and $\delta=\beta_1+\dots+\beta_m$, so $\ell=2m$. For $0 \le j < \ell/2$, the weights in $W_j$ correspond to the roots in $\Phi(X)$ of height $m-j$, and they each have multiplicity one. Since there are $j+1$ such roots, we deduce that $\dim W_j=j+1$. Finally, $W_{\ell/2}$ is the zero-weight space and thus Lemma \ref{l:s816} implies that $\dim W_{\ell/2}=m-\e$.
\end{proof}

We now handle the general case.

\begin{prop}\label{p:amlevels}
Suppose $m\geq 4$ and $\delta\neq\delta_1+\delta_m$. 
\begin{itemize}\addtolength{\itemsep}{0.3\baselineskip}
\item[{\rm (i)}]  We have $\dim W_j \geq 3$ for all $3\leq j \leq \ell/2$, and $\dim W_{\ell/2}\geq 5$ if $\ell$ is even. 

\item[{\rm (ii)}] If $\delta\neq \delta_{(m+1)/2}$ then $\dim W_2\geq 3$, otherwise $\dim W_2=2$. 

\item[{\rm (iii)}] If $\delta \not \in \{ a\delta_j+a\delta_{m-j+1},a\delta_{(m+1)/2}\}$ where $a\geq 1$ and $1\leq j \le m/2$, then $\dim W_1\geq 3$.

\item[{\rm (iv)}] If $\delta=a\delta_j+a\delta_{m-j+1}$ or $\delta= a\delta_{(m+1)/2}$, where $a\geq 1$ and $1\leq j \le m/2$, then $ \dim W_1 = 2$ or $1$, respectively.
\end{itemize}
\end{prop}

We will prove Proposition \ref{p:amlevels} in a sequence of separate lemmas. 

\begin{lem}\label{l:subdom}
Suppose $\delta=\sum_{i}b_i\delta_i\neq \delta_1+\delta_m$. If $m$ is even, or if $m$ is odd and $b_{(m+1)/2}$ is even, then set $\nu = \delta_2+\delta_{m-1}$, otherwise set $\nu = \delta_{(m+1)/2}$. Then $\nu$ is subdominant  to $\delta$.
\end{lem}

\begin{proof}
Set $(\e,k)=(0,m/2)$ if $m$ is even, otherwise $(\e,k)=(1,(m+1)/2)$.  
As $\delta$ is a symmetric weight we have $b_i=b_{m-i+1}$ for all $1\leq i\leq k-\epsilon$, and so 
$$\delta=\sum_{i=1}^{k}b_i(\delta_i+\delta_{m-i+1}) \mbox{ or } \sum_{i=1}^{k-1}b_i(\delta_i+\delta_{m-i+1})+b_k\delta_k$$ 
if $m$ is even or odd, respectively. By \cite[Table 1]{Hu1} we have 
\begin{equation}\label{e:one}
\delta_i+\delta_{m-i+1}=\sum_{j=1}^{i-1}j(\beta_j+\beta_{m-j+1})+\sum_{j=i}^{m-i+1}i\beta_j
\end{equation}
for all $1\leq i\leq k-\epsilon$, and similarly if $m$ is odd then
\begin{equation}\label{e:two}
\delta_{k}= \frac12\left(\sum_{j=1}^{k-1}j(\beta_j+\beta_{m-j+1})+k\beta_k\right).
\end{equation}

Suppose $m$ is even, so 
$$\delta-(\delta_2+\delta_{m-1}) = \sum_{i=1}^{k}b_i(\delta_i+\delta_{m-i+1})-(\delta_2+\delta_{m-1}).$$
If $b_2\neq 0$ then \eqref{e:one} implies that $\delta-(\delta_2+\delta_{m-1})$ is a non-negative integer linear combination of the $\beta_i$, so $\delta_2+\delta_{m-1}$ is subdominant  to $\delta$. If $b_2=0$ then either $b_1\geq 2$,
or $b_i\geq 1$ for some $3\leq i \leq m/2$ (recall that $\delta \neq \delta_1+\delta_m$). Therefore \eqref{e:one} gives the desired result.

Next suppose $m$ is odd and $b_k$ is even, so 
$$\delta-(\delta_2+\delta_{m-1}) = \sum_{i=1}^{k-1}b_i(\delta_i+\delta_{m-i+1})+b_k\delta_k-(\delta_2+\delta_{m-1}).$$
If $b_2\neq 0$ then \eqref{e:one} and \eqref{e:two} imply that $\delta-(\delta_2+\delta_{m-1})$ is a non-negative integer linear combination of the $\beta_i$ (since $b_k$ is even), so $\delta_2+\delta_{m-1}$ is subdominant  to $\delta$. Similarly, if $b_2=0$ then
either $b_1\geq 2$,
or $b_i\geq 1$ for some $3\leq i \leq k-1$, or $b_k>0$. Again, since $b_k$ is even, the result follows from \eqref{e:one} and \eqref{e:two}.

Finally, let us assume $m$ is odd and $b_k$ is odd, so 
$$\delta-\delta_k =   \sum_{i=1}^{k-1}b_i(\delta_i+\delta_{m-i+1})+(b_k-1)\delta_k.$$
Since $b_k$ is odd, \eqref{e:two} implies that $(b_k-1)\delta_k$ is a non-negative integer linear combination of the $\beta_i$, and so $\delta_k$ is subdominant to $\delta$. 
\end{proof}

Our proof of Proposition \ref{p:amlevels} consists of exhibiting a certain number of weights at the appropriate levels. In order to do this, by Lemma \ref{l:pr}, it suffices to work in the Weyl module with highest weight $\delta=\sum_{i=1}^m b_i\delta_i$.  We first handle some special cases.

\begin{lem}\label{l:amredbc1}
Proposition \ref{p:amlevels} holds if $\delta=\delta_2+\delta_{m-1}$. 
\end{lem}

\begin{proof}
Here $\ell=4m-4$ (see \eqref{e:one}) and $\ell'=2m-2$. Set $\mu=\delta-\sum_{i=2}^{m-1} \beta_i = \delta_1+\delta_m$. By Lemmas \ref{l:am1m} and \ref{l:indwt}, we need to verify the result for $1\leq j \leq m-1$ and $j=2m-2$. This is a straightforward calculation. Indeed, at level 1 there are two weights, each of multiplicity one, namely $\delta-\beta_2$ and $\delta-\beta_{m-1}$, and we can exhibit the following three distinct weights at level $j$ for all $2 \leq j \leq m-1$: 
$$\delta-\beta_2-\dots-\beta_{j+1}, \; \delta-\beta_1-\dots-\beta_j, \; \nu$$
where 
$$\nu = \left\{\begin{array}{ll}
\delta-\beta_1-2\beta_2-\beta_3-\dots-\beta_{j-1} & \mbox{if $j \ge 4$} \\
\delta-\b_{1}-\b_2-\b_{m-1} & \mbox{if $j=3$ and $m \ge 5$} \\
\delta-\b_2-2\b_{3} & \mbox{if $j=3$ and $m=4$} \\
\delta - \b_{m-1}-\b_m & \mbox{if $j=2$}
\end{array}\right.$$
Finally, the following five weights are at level $2m-2$: 
$$\delta-2\sum_{i=1}^{m-1}\beta_i, \; \delta-2\sum_{i=2}^m \beta_i, \; \delta-\beta_1-2\sum_{i=2}^{m-1}\beta_i-\beta_m$$ 
$$\delta-\beta_2-2\sum_{i=3}^{m-2}\beta_i-3\beta_{m-1}-2\beta_m, \;\delta-2\beta_1-3\beta_2-2\sum_{i=3}^{m-2}\b_i-\beta_{m-1}.$$
\end{proof}

\begin{lem}\label{l:amredbc2}
Proposition \ref{p:amlevels} holds if $m$ is odd and $\delta=\delta_{(m+1)/2}$. 
\end{lem}

\begin{proof}
Set $k=(m+1)/2$ and note that $\ell=\frac{1}{4}(m+1)^2$ (see \eqref{e:two}). It is clear that $\dim W_1=1$ and $\dim W_2=2$. Indeed there is only one weight at level 1, namely $\delta-\beta_k$, and only two weights at level 2: $\delta-\beta_{k-1}-\beta_{k}$ and $\delta-\beta_k-\beta_{k+1}$ (all of these weights have multiplicity 1). For the remainder we may assume $j \ge 3$.

The result for $m\in \{5,7,9,11,13\}$ can be checked directly. For example, if $m=5$ then $\ell=9$, $\ell'=4$ and it is easy to find three distinct weights at levels 3 and 4. For instance, at level 3 we have the weights $\delta-\beta_1-\beta_2-\beta_3$, $\delta-\beta_2-\beta_3-\beta_4$, 
$\delta-\beta_3-\beta_4-\beta_5$, and at level 4 we have $\delta-\beta_2-2\beta_3-\beta_4$, $\delta-\sum_{i=1}^{4}\beta_i$, $\delta-\sum_{i=2}^5\beta_i$. The other cases are similar.
 
Now assume $m \geq 15$. First we will exhibit  three distinct weights at level $j$ for all 
$3\leq j \leq \ell'$. To do this we argue by induction on $m$. Let $P_X=Q_XL_X$ be the parabolic subgroup of $X$ with $L_X'=A_{m-2}$ and
$\Delta(L_X')=\{\beta_2,\dots,\beta_{m-1}\}$. By induction on $m$, using the full range of levels for $A_{m-2}$, we deduce that there are at least three distinct weights at each level $j$ with $3\leq j\leq \frac{1}{4}(m-1)^2-3$. Since $m\geq 15$, we have $\frac{1}{4}(m-1)^2-3>\frac{1}{8}(m+1)^2$ and the result follows. 

To complete the proof it suffices to show that if $\ell=\frac{1}{4}(m+1)^2$ is even (that is, if $m\equiv 3 \imod{4}$) then there are at least five distinct weights at level $\ell'=\frac{1}{8}(m+1)^2$. 
By the argument in the previous paragraph, there are at least three distinct weights at level $\ell'$ of the form $\delta-\sum_{i=2}^{m-1}d_i\beta_i$. It is therefore sufficient to produce two distinct weights $\nu_1,\nu_2$ at level $\ell'$ of the form $\delta-\sum_{i=1}^md_i\beta_i$ with $d_1$ or $d_m$  non-zero.  

First assume $m\equiv 3 \imod{8}$, so $k \equiv 2 \imod{4}$. Recall that $t$ denotes the graph automorphism of $X$ and let
\begin{equation}\label{e:nu12}
\nu_1=\delta-\sum_{i=1}^m d_i\beta_i,\;\; \nu_2=t(\nu_1)= \delta-\sum_{i=1}^{m}d_{m-i+1}\beta_i
\end{equation}
where $d_i=1$ if $1\leq i\leq (k+2)/4$ or $m-(k+2)/4\leq i \leq m$, $d_i=i-(k-2)/4$ if $(k+6)/4\leq i \leq (3k-10)/4$, $d_i=k/2-1$ if $(3k-6)/4\leq i \leq k-1$ or $k+1\leq i \leq (5k+2)/4$, $d_i=m-i-(k-2)/4$ if $(5k+6)/4\leq  i\leq m-(k+6)/4$, and finally $d_k=k/2$.  
Now $\nu_1$ and $\nu_2$ are distinct weights of $W$ (with $d_1=d_m=1$) and one can check that the $U_X$-level of both weights is $\frac{1}{8}(m+1)^2$. 

Finally suppose $m\equiv 7 \imod{8}$, so $k\equiv 0 \imod{4}$.
First assume $m=15$, so $k=8$, $\ell=64$ and $\ell'=32$. Let 
$$\nu_1=\delta-\sum_{i=1}^4\beta_i-2\beta_5-3\beta_6-3\beta_7-4\beta_8-3\sum_{i=9}^{12}\beta_i-2\beta_{13}-\beta_{14}-\beta_{15}$$ 
and 
$$\nu_2=t(\nu_1)=\delta-\beta_1-\beta_2-2\beta_3-3\sum_{i=4}^7\beta_i-4\beta_8-3\beta_{9}-3\beta_{10}-2\beta_{11}-\sum_{i=12}^{15}\beta_i.$$ 
Then $\nu_1$ and $\nu_2$ are distinct weights of $W$ at level 32, and $d_1=d_{15}=1$, so the desired result holds when $m=15$. Finally, if $m \geq 23$ then define $\nu_1$ and $\nu_2$ as in \eqref{e:nu12}, where 
$d_i=1$ if $1\leq i\leq 4$ or $ i \in\{m-1,m\}$, $d_i=2$ for $i\in \{5,m-2\}$, $d_i=3$ if $6 \leq i \leq k/4+4$ or $m-k/4-3\leq m-3$, $d_i=i-k/4-1$ if $k/4+5\leq i \leq 3k/4-1$, $d_i=k/2-1$ if $3k/4\leq i \leq k-1$ or $k+1\leq i \leq 5k/4$, $d_i=m-i-k/4$ if $5k/4+1\leq  i\leq m-k/4-4$, and finally $d_k=k/2$. It is easy to check that 
$\nu_1$ and $\nu_2$ have the required properties.
\end{proof}

\begin{proof}[Proof of Proposition \ref{p:amlevels}]
As before, if $m$ is odd we set $k=(m+1)/2$. First we consider levels $1$ and $2$. Suppose that $\delta$ is not of the form $a\delta_k$ or $a\delta_i+a\delta_{m-i+1}$, where $a$ and $i$ are positive integers with $1\leq i \leq m/2$. Then at least three of the $b_i$ are non-zero, say $b_{i_1}, b_{i_2}, b_{i_3}$ and so there are at least three weights at level 1, namely $\delta-\beta_{i_j}$ for $1\leq j\leq 3$, and also at least three distinct weights at level 2. Now assume that $\delta$ has the form $a\delta_k$ or $a\delta_i+a\delta_{m-i+1}$, where $a$ and $i$  are positive integers with $1 \leq i \leq m/2$. (Note that $a>1$ if $i=1$ since we are assuming $\delta \neq \delta_1+\delta_m$.) Clearly, there are respectively one or two distinct weights at level 1, namely $\delta-\beta_k$, respectively $\delta-\beta_i$, $\delta-\beta_{m-i+1}$.  Also if $\delta \neq \delta_k$ then there are at least three weights at level 2. For example, we have  $\delta-\beta_1-\beta_2$, $\delta-\beta_1-\beta_m$, $\delta-\beta_{m-1}-\beta_m$ (if $i=1$), 
$\delta-\beta_{i-1}-\beta_i$, $\delta-\beta_i-\beta_{m-i+1}$, $\delta-\beta_{m-i+1}-\beta_{m-i+2}$ (if $i>1$), or 
$\delta-\beta_{k-1}-\beta_k$, $\delta-\beta_{k}-\beta_{k+1}$, $\delta-2\beta_k$ (if $\delta = a\delta_k$ with $a>1$). Finally if 
$\delta=\delta_k$ then there are only two weights at level 2, namely $\delta-\beta_{k-1}-\beta_k$ and $\delta-\beta_k-\beta_{k+1}$.

We now exhibit the appropriate number of weights at levels 3 to $\ell_\delta'$. 
By Lemma \ref{l:subdom}, some weight  
$\mu \in \{\delta_2+\delta_{m-1},  \delta_k\}$ is  subdominant  to $\delta$. We proceed by induction on the height of $\delta-\mu$, which we denote by ${\rm ht}(\delta-\mu)$ as before.
If ${\rm ht}(\delta-\mu)=0$ then we are in the situation handled by Lemmas \ref{l:amredbc1} and \ref{l:amredbc2}, so we may assume that ${\rm ht}(\delta-\mu)>0$. 
For the remainder of the proof, if $\nu$ is a dominant weight for $X$ then we will refer to the weights of the irreducible $KX$-module with $T_X$-highest weight $\nu$ as the \emph{weights of $\nu$}.  

First assume $m$ is odd and $b_k \geq 2$.  We set $\mu=\delta_2+\delta_{m-1}$ if $b_k$ is even, otherwise $\mu=\delta_k$, so that $\mu$ is  subdominant  to $\delta$ by Lemma \ref{l:subdom}. Let $\nu=\delta-\beta_k$ and write $\nu=\sum_id_i\delta_i$. Then $\nu$ is subdominant  to $\delta$, and Lemma \ref{l:subdom} implies that $\mu$ is  subdominant  to $\nu$ (since $d_k=b_k-2$). In addition, ${\rm ht}(\nu-\mu)<{\rm ht}(\delta -\mu)$, $\ell_\nu=\ell_\delta-2$ and $\ell_\nu'=\ell_\delta'-1$, so by induction we see that $\nu$ has at least three distinct weights at levels 3 to $\ell_\nu'-1$, and at least five at level $\ell_\nu'$. Therefore Lemma \ref{l:indwt} implies that $W$ has at least three distinct weights from levels 4 to $\ell_\nu'=\ell_\delta'-1$, and at least five at level $\ell_\nu'+1=\ell_\delta'$. It remains to exhibit three distinct weights at level 3, which is easy: $\delta-2\beta_k-\beta_{k+1}$, $\delta-\beta_{k-1}-2\beta_k$, $\delta-\beta_{k-1}-\beta_k-\beta_{k+1}$.
 
To complete the proof we may assume that $b_i=b_{m-i+1}>0$ for some $1\leq i\leq m/2$, and also $b_k \leq 1$ if $m$ is odd. Let $1\leq r\leq m/2$ be maximal such that $b_r=b_{m-r+1}>0$. Set $\mu=\delta_2+\delta_{m-1}$ if $m$ is even, or if $m$ is odd and $b_k=0$, otherwise we set $\mu=\delta_k$. By Lemma \ref{l:subdom}, 
$\mu$ is  subdominant  to $\delta$.   Let $\nu=\delta-\sum_{i=r}^{m-r+1}\beta_i$ and write $\nu=\sum_id_i\delta_i$. Then $\nu$ is subdominant  to  $\delta$, and if $m$ is odd we have $d_k=b_k$. Hence by Lemma \ref{l:subdom}, $\mu$ is  subdominant  to $\nu$ with ${\rm ht}(\nu-\mu)<{\rm ht}(\delta -\mu)$, $\ell_\nu=\ell_\delta-2(m-2r+2)$ and $\ell_\nu'=\ell_\delta'-(m-2r+2)$. By induction, $\nu$ has at least  three distinct weights  from levels 3 to $\ell_\nu'-1$, and at least five at level $\ell_\nu'$. Since ${\rm ht}(\delta-\nu)=m-2r+2$, it follows that $W$ has at least three distinct weights at levels $m-2r+5$ to $\ell_\delta'-1$, and at least five at level $\ell_\delta'$. It remains to exhibit three distinct weights in $W_i$, where $3\leq i \leq m-2r+4$. 
 
First assume $r=1$ and $b_1=1$. Here $m$ is odd and $b_k=1$ since $\delta \neq \delta_1+\delta_m$. If $i$ is even and $4 \leq i < m$ then we have the following three distinct weights at level $i$:
$$\delta-\sum_{j=1}^{i}\beta_j,\; \delta-\beta_k - \sum_{j=(m-i+3)/2}^{(m+i-1)/2}\beta_j, \; \delta-\sum_{j=m+1-i}^m\beta_j,$$ 
and at level $m+1$ we have
$$\delta-\beta_k - \sum_{j=1}^{m}\beta_j, \; \delta- \beta_{k-1}- \beta_k - \sum_{j=1}^{m-1}\beta_j, \; \delta- \beta_k - \beta_{k+1} - \sum_{j=1}^{m-1}\beta_j.$$ 
Similarly, at level $3$ we have 
$\delta-\beta_1-\beta_2-\beta_3$, $\delta-\beta_{m-2}-\beta_{m-1}-\beta_m$ and 
$\delta-\beta_{k-1}-\beta_k-\beta_{k+1}$. If $i$ is odd and $5 \leq i <m$ then we have the following three weights at level $i$:
$$\delta-\sum_{j=1}^{i}\beta_j, \; \delta-\beta_k - \sum_{j=(m-i+2)/2}^{(m+i-2)/2}\beta_j, \; \delta-\sum_{j=m+1-i}^m\beta_j.$$
Finally, at levels $m$ and $m+2$ we have distinct weights
$$\delta-\sum_{j=1}^m\beta_j, \; 
\delta-\beta_k - \sum_{j=2}^{m}\beta_j,  \; \delta-\beta_k - \sum_{j=1}^{m-1}\beta_j$$
and 
$$\delta-\beta_k-\beta_{k+1} - \sum_{j=1}^{m}\beta_j, \; 
\delta-\beta_{k-1}-\beta_{k} - \sum_{j=1}^{m}\beta_j, \; 
\delta-\sum_{j=1}^{k-2}\beta_j-2\sum_{j=k-1}^{k+1}\beta_{j}-\sum_{j=k+2}^{m-1}\beta_j$$
respectively.

Next assume $r=1$ and $b_1>1$. If $3 \leq i \leq m$ then we have the following weights at level $i$: 
$$\delta-\sum_{j=1}^{i}\beta_j, \;  
\delta-\beta_1 - \sum_{j=1}^{i-1} \beta_j, \; \delta-\beta_m - \sum_{j=m-i+2}^{m}\beta_j.$$ 
Similarly, at level $m+1$ we have weights 
$$\delta-\beta_1 - \sum_{j=1}^m\beta_j, \; \delta-\beta_m - \sum_{j=1}^{m}\beta_j, \; \delta-\beta_1-\beta_2 - \sum_{j=1}^{m-1} \beta_j,$$ 
and at level $m+2$ we have 
$$\delta-\beta_1-\beta_2 - \sum_{j=1}^m \beta_j, \; \delta-\beta_{m-1}-\beta_m - \sum_{j=1}^{m}\beta_j$$
and also $\delta-2\beta_1-2\beta_2-2\beta_3-\sum_{j=4}^{m-1} \beta_j$ (if $m >4$), and $\delta-2\beta_1-2\beta_2-2\beta_3$ (if $m=4$).
 
Finally, suppose $r\geq 2$. If $r=m/2$ then at level $3$ we have 
$$\delta-\beta_{r-1}-\beta_r-\beta_{r+1}, \; \delta-\beta_{r}-\beta_{r+1}-\beta_{r+2},\; \delta-\beta_r-2\beta_{r+1},$$
while at level 4 we have weights 
$$\delta-\beta_{r-1}-\beta_r-\beta_{r+1}-\beta_{r+2}, \; \delta-\beta_{r-1}-2\beta_r-\beta_{r+1}, \; \delta-\beta_r-2\beta_{r+1}-\beta_{r+2}.$$ 
Now assume $2 \le r<m/2$. If $4 \leq i \leq \max\{4,m-2r+2\}$ then at level $i$ we have weights
$$\delta-\sum_{j=r}^{r+i-1}\beta_j, \; \delta-\sum_{j=r-1}^{r+i-2} \beta_j, \; \delta-\b_r - \sum_{j=r-1}^{r+i-3}\beta_j.$$  
At level $3$ we have 
$$\delta-\beta_{r-1}-\beta_r-\beta_{r+1}, \; \delta-\beta_r-\beta_{r+1}-\beta_{r+2}, \; \delta-\beta_{r-1}-\beta_{r}-\beta_{m-r+1},$$
while at level $m-2r+3$ we have
$$\delta-\sum_{j=r-1}^{m-r+1}\beta_j, \; \delta-\sum_{j=r}^{m-r+2}\beta_j, \; \delta- \b_r - \sum_{j=r-1}^{m-r}\beta_j.$$ 
Finally at level $m-2r+4$ we have the weights 
$$\delta-\sum_{j=r-1}^{m-r+2}\beta_j, \; \delta-\b_r - \sum_{j=r-1}^{m-r+1}\beta_j,$$
together with $\delta-\sum_{j=r-1}^{m-r}\beta_j - \b_r - \b_{r+1}$ (if $r\neq (m-1)/2$) and $\delta-\beta_{r-1}-\beta_{r}-2\beta_{r+1}-\beta_{r+2}$ (if $r=(m-1)/2$). This completes the proof of Proposition \ref{p:amlevels}.
\end{proof}

Let $\mathcal{T}$ be the set of triples $(G,H,V)$ satisfying Hypothesis \ref{h:our}, with $H^0=A_m$ and $m \ge 4$.

\begin{cor}\label{c:red}
If $(G,H,V) \in \mathcal{T}$ then either
$\delta = a\delta_{i}+a\delta_{m+1-i}$ for some $a,i \ge 1$, or $m$ is odd and $\delta = 
\delta_{(m+1)/2}$.
\end{cor}

\begin{proof}
In view of Remark \ref{r:a1factor} and Lemma \ref{l:main}, this follows immediately from Proposition \ref{p:amlevels}.
\end{proof}

Therefore, to complete the proof of Theorem \ref{t:am} we may assume that $m \ge 4$ and either $\delta = a\delta_{i}+a\delta_{m+1-i}$ or $\delta_{(m+1)/2}$. 

\begin{lem}\label{l:embed}
Suppose $p \neq 2$ and $\delta = a\delta_{i}+a\delta_{m+1-i}$ or $\delta_{(m+1)/2}$. 
Then $G$ is a symplectic group if and only if $\delta = \delta_{(m+1)/2}$ and $m \equiv 1 \imod{4}$. 
\end{lem}

\begin{proof}
This follows from Steinberg's criterion (see Lemma \ref{l:st}). Let $h = \prod{h_{\b}(-1)}$, the product over the positive roots in $\Phi(X)$. First assume $m$ is odd and $\delta = \delta_{(m+1)/2}$. It is straightforward to check that there are exactly $\frac{1}{4}(m+1)^2$ roots in $\Phi^+(X)$ of the form $\sum_{j}d_j\b_j$ with $d_{(m+1)/2}=1$, whence $\delta(h) = (-1)^{\frac{1}{4}(m+1)^2}$ and thus $G$ is symplectic if and only if $m \equiv 1 \imod{4}$. Similarly, if $\delta = a\delta_{i}+a\delta_{m+1-i}$ then there are exactly $i(m-2i+1)$ positive roots $\sum_{j}d_j\b_j$ with $d_{i}=1$ and $d_{m+1-i}=0$, and of course the same number with $d_{i}=0$ and $d_{m+1-i}=1$. Therefore $\delta(h) = (-1)^{2ai(m+1-i)}=1$. The result follows. 
\end{proof}

\subsection{Parabolic analysis}

We will need information on the $Q_X$-levels of $W$,
where $P_X=Q_XL_X$ is either the $t$-stable parabolic subgroup of $X$ with $\Delta(L_X') = \{\b_2,\ldots, \b_{m-1}\}$, or the asymmetric parabolic with $\Delta(L_X') = \{\b_1, \ldots, \b_{m-1}\}$. This is recorded in Lemmas \ref{l:levels2} and \ref{l:levels3} below, for the weights $\delta = a\delta_{i}+a\delta_{m+1-i}$ and $\delta = \delta_{(m+1)/2}$.

\begin{lem}\label{l:levels2}
Suppose $m \ge 4$ and let $P_X=Q_XL_X$ be the $t$-stable parabolic subgroup of $X$ with $\Delta(L_X') = \{\b_2,\ldots, \b_{m-1}\}$. Let $W_j$ be the $j$-th $Q_X$-level of $W$ and let $\ell \ge 0$ be minimal such that $[W,Q_X^{\ell+1}]=0$.
\begin{itemize}\addtolength{\itemsep}{0.3\baselineskip}
\item[{\rm (a)}] If $\delta = a\delta_{i}+a\delta_{m+1-i}$ then $\ell = 4a$ and the following hold:
\begin{itemize}\addtolength{\itemsep}{0.3\baselineskip}
\item[{\rm (i)}] If $i \ge 2$ then $\dim W_0 \ge 7$, $\dim W_{\ell/2} \ge 5$ and $\dim W_j \ge 2$ for all $1 \le j <\ell/2$;
\item[{\rm (ii)}] If $i=1$ then $\dim W_0 = 1$, $\dim W_1=2(m-1)$, $\dim W_{\ell/2} \ge 5$ and $\dim W_j \ge 2$ for all $2 \le j <\ell/2$;
\item[{\rm (iii)}] $W_j|_{L_X'}$ is reducible for all $1 \le j \le \ell/2$;
\item[{\rm (iv)}] $W_j|_{L_X'}$ is nontrivial for all $j \ge \gamma$, where $\gamma = 1$ if $i=1$, otherwise $\gamma=0$. 
\end{itemize}
\item[{\rm (b)}] Suppose $m$ is odd and $\delta = \delta_{(m+1)/2}$. Then $\ell = 2$ and 
$$\dim W_0 = \binom{m-1}{(m-1)/2}, \;\; \dim W_1 = 2\binom{m-1}{(m+1)/2}.$$ 
Further, both $W_0|_{L_X'}$ and $W_1|_{L_X'}$ are nontrivial, and $W_1|_{L_X'}$ is reducible with precisely two composition factors, afforded by the highest weights $\delta_{(m+3)/2}|_{L_X'}$ and $\delta_{(m-1)/2}|_{L_X'}$.
\end{itemize}
\end{lem}

\begin{proof}
Recall that if $\mu = \delta - \sum_{i}{d_i\b_i}$ is a weight of $W$ then the $Q_X$-level of $\mu$ is $d_1+d_m$, and its $Q_X$-shape is $(d_1,d_m)$. First consider (a), so $\delta = a\delta_{i}+a\delta_{m+1-i}$. From \eqref{e:one}
we calculate that
$$-\delta = \delta - \frac{2a}{m+1}\left((m-i+1)\b_1+i\b_m+i\b_1+(m+1-i)\b_m\right) - \sum_{j=2}^{m-1}d_j\b_j$$
for some $d_j$, whence $\ell=4a$ as claimed.

Suppose $1 \le j \le \ell/2$. We claim that each $W_j$ contains at least two weights with distinct $Q_X$-shapes, which implies that $\dim W_j \ge 2$ and  $W_j|_{L_{X}'}$ is reducible (see Lemma \ref{l:cl}). This is a straightforward calculation. By Lemma \ref{l:pr}, it is sufficient to work in the Weyl module $W_X(\delta)$.
First observe that the set of $T_{X}$-weights of $W$ is invariant under the induced action of the graph automorphism $t$, so the claim is clear when $j$ is odd. If $j$ is even and $i \ge 2$ then  
$$\delta - \frac{1}{2}j(\b_i +\b_{i+1} + \cdots + \b_{m-i}) - j(\b_{m+1-i} + \b_{m+2-i} + \cdots + \b_{m}) \in \L(W)$$ 
has $Q_X$-shape $(0,j)$, and the result follows. Next suppose $j$ is even, $i=1$ and $a \ge 2$. If $j<2a$ then $\mu = \delta -\frac{1}{2}(j-2)\b_1-\frac{1}{2}(j+2)\b_m \in \L(W)$, so let us assume $j=2a$. If $a$ is even then 
$$\delta-\frac{1}{2}a(\b_1+\cdots+\b_{m-1})-\frac{3}{2}a\b_m \in \L(W)$$
is a weight with $Q_X$-shape $(a/2,3a/2)$, while if $a \ge 3$ is odd then
$$\delta - \frac{1}{2}(a+1)(\b_1+\cdots+\b_{m-1}) - \frac{1}{2}(3a-1)\b_m \in \L(W)$$
is a weight at $Q_X$-level $2a$ with $Q_X$-shape $((a+1)/2, (3a-1)/2)$. Finally, suppose $i=a=1$. Here $\ell=4$ and $W_1|_{L_{X}'}$ is clearly reducible since $\delta - \b_1, \delta-\b_m \in \L(W)$. Similarly, $W_2|_{L_{X}'}$ has composition factors with $T_{L_{X}'}$-highest weights $0$ and $(\delta_2+\delta_{m-1})|_{L_X'}$. This justifies the claim. 

In addition, the above argument also shows that each $W_j|_{L_X'}$ is nontrivial when $j$ is even, and similar reasoning applies when $j$ is odd. For instance, if $i=1$ and $j$ is an even integer in the range $2 \le j \le 2a-2$ then we have observed that 
$$\mu = \delta -\frac{1}{2}(j-2)\b_1-\frac{1}{2}(j+2)\b_m = (a-j+2)\delta_1+\frac{1}{2}(j-2)\delta_2+\frac{1}{2}(j+2)\delta_{m-1}+(a-j+2)\delta_m$$
is a weight in $W_j$ with
$\mu|_{L_X'} = (\frac{1}{2}(j-2)\delta_2+\frac{1}{2}(j+2)\delta_{m-1})|_{L_X'} \neq 0$. The other cases are very similar. 

The remaining assertions in (a) are easily verified. Indeed, it is straightforward to exhibit at least five distinct weights at level $W_{\ell/2}$, whence $\dim W_{\ell/2} \ge 5$. Clearly, if $i=1$ then $\dim W_0=1$ and we find that every weight in $W_1$ is under $\delta-\b_1$ or $\delta-\b_m$, so $W_1|_{L_X'}$ has exactly two composition factors, with corresponding $T_{L_X'}$-highest weights $\delta_2|_{L_X'}$ and $\delta_{m-1}|_{L_X'}$. In other words, $W_1|_{L_{X}'} \cong U \oplus U^*$, where $U$ is the natural $KL_{X}'$-module, whence $\dim W_1=2(m-1)$ as claimed. Finally, if $i \ge 2$ then $W_0|_{L_X'}$ is irreducible with highest weight $(a\delta_i+a\delta_{m+1-i})|_{L_X'}$, so $W_0|_{L_X'}$ is nontrivial and \cite{Lubeck} implies that $\dim W_0 \ge 7$ (with equality if and only if $(m,a,i,p)=(4,1,2,3)$).

\par

Finally, let us consider (b), so $m \ge 5$ is odd and $\delta = \delta_{(m+1)/2}$. Note that $V_{X}(\delta_j) \cong \L^j(W')$ for all $j$, where $W'$ denotes the natural $KX$-module.
Now $W_0 \cong W/[W,Q_X]$ is an irreducible $KL_X'$-module with $T_{L_X'}$-highest weight $\delta_{(m+1)/2}|_{L_X'}$, so $\dim W_0 = \dim \L^{(m-1)/2}(U)$ where $U$ is the natural $KL_{X}'$-module. Next observe that  
$$\delta - \b_1 - \b_2 - \cdots - \b_{(m+1)/2} \mbox{ and } \delta - \b_{(m+1)/2} - \b_{(m+3)/2} - \cdots - \b_m$$
afford highest weights of composition factors for $W_1|_{L_X'}$. Indeed, $W_1$ has precisely two $KL_{X}'$-composition factors, with highest weights $\delta_{(m+3)/2}|_{L_X'}$ and $\delta_{(m-1)/2}|_{L_X'}$, respectively, and thus
$$\dim W_1 = \dim \L^{(m+1)/2}(U)+\dim \L^{(m-3)/2}(U) = 2\binom{m-1}{(m+1)/2}$$
as required. Finally, to see that $\ell=2$ it suffices to show that
\begin{eqnarray*}
2\dim \L^{(m-1)/2}(U) + \dim \L^{(m+1)/2}(U) +\dim \L^{(m-3)/2}(U) & = & \dim \L^{(m+1)/2}(W') \\
& = & \dim W.
\end{eqnarray*}
This is a straightforward calculation.
\end{proof}

\begin{lem}\label{l:levels3}
Suppose $m \ge 4$ and let $P_X=Q_XL_X$ be the parabolic subgroup of $X$ with $\Delta(L_X') = \{\b_1,\ldots, \b_{m-1}\}$. Let $W_j$ be the $j$-th $Q_X$-level of $W$ and let $\ell \ge 0$ be minimal such that $[W,Q_X^{\ell+1}]=0$.
\begin{itemize}\addtolength{\itemsep}{0.3\baselineskip}
\item[{\rm (a)}] If $\delta = a\delta_{1}+a\delta_{m}$ then $\ell = 2a$ and the following hold:
\begin{itemize}\addtolength{\itemsep}{0.3\baselineskip}
\item[{\rm (i)}] $\dim W_0  = \binom{m-1+a}{a} \ge m$, $\dim W_{\ell/2} \ge 5$ and $\dim W_j \ge 2$ for all $1 \le j <\ell/2$;
\item[{\rm (ii)}] $W_j|_{L_X'}$ is nontrivial for all $j \ge 0$;
\item[{\rm (iii)}] If $a=1$ then $\dim W_1=m^2-\e$, where $\e=1$ if $p$ divides $m+1$, otherwise $\e=0$.
\end{itemize}
\item[{\rm (b)}] Suppose $m$ is odd and $\delta = \delta_{(m+1)/2}$. Then $\ell = 1$ and $W_0|_{L_X'}$ has highest weight $\delta_{(m+1)/2}|_{L_X'}$, so $\dim W_0 = \binom{m-1}{(m-1)/2}$.
\end{itemize}
\end{lem}

\begin{proof}
This is similar to the proof of Lemma \ref{l:levels2}. First consider (a), so $\delta = a\delta_1+a\delta_m$. To establish the lower bound on $\dim W_j$ we exhibit suitably many distinct weights at level $j$, working in the Weyl module $W_X(\delta)$. In this way, we also deduce that each $W_j|_{L_X'}$ is nontrivial; we ask the reader to check the details. In addition, Lemma \ref{l:vq}(ii) indicates that $W_0 \cong W/[W,Q_X]$ has highest weight $a\delta_1|_{L_X'}$, so \cite[1.14]{Seitz2} implies that $\dim W_0 = \binom{m-1+a}{a}$ as claimed. 

Now assume $a=1$. Then $W_1|_{L_X'}$ has a composition factor $U_1$ of highest weight $(\delta-\beta_m)|_{L_X'}=(\delta_1+\delta_{m-1})|_{L_X'}$, and any other composition factor of $W_1|_{L_X'}$ has highest weight $(\delta-\sum_{j=1}^m\beta_j)|_{L_X'}=0$. By Lemma \ref{l:am1m} we have $\dim W=(m+1)^2-1-\e$ (where $\e$ is defined as in part (iii) of (a)), whereas $\dim U_1=m^2-2$ if $p$ divides $m$, otherwise $\dim U_1=m^2-1$. Using Lemma \ref{l:s816} we can compute the multiplicity of the weight $\delta-\sum_{j=1}^m \beta_j$ in both $W$ and $W_1$, which yields the number of trivial composition factors in $W_1$.  Indeed, we find that there is a single trivial factor in $W_1$ when $m(m+1)\not\equiv 0 \imod{p}$, there are no factors when $p$ divides $m+1$, and two trivial factors in $W_1$ when $p$ divides $m$. We conclude that $\dim W_1=m^2-\e$.

To prove (b), first observe that Lemma \ref{l:vq}(ii) implies that $W_0 \cong W/[W,Q_X]$ is an irreducible $KL_X'$-module with highest weight $\delta_{(m+1)/2}|_{L_X'}$, so 
$$\dim W_0 = \dim \L^{(m-1)/2}(U) = \binom{m}{(m-1)/2},$$
where $U$ is the natural $KL_X'$-module.
Finally, since
$$\dim W = \dim \L^{(m+1)/2}(W') = \binom{m+1}{(m+1)/2} = 2\dim W_0,$$
where $W'$ is the natural $KX$-module, we conclude that $\ell=1$. 
\end{proof}

\subsection{The case $\delta = a\delta_{i}+a\delta_{m+1-i}$}\label{sss:2}

\begin{lem}\label{l:aim1}
Suppose $(G,H,V) \in \mathcal{T}$ and $\delta = a\delta_{i}+a\delta_{m+1-i}$. Then $i=1$.
\end{lem}

\begin{proof}
Seeking a contradiction, let us assume $i \ge 2$. By considering the dimensions of the $U_X$-levels of $W$ (corresponding to a $t$-stable Borel subgroup $B_X = U_XT_X$ of $X$) given in Proposition \ref{p:amlevels}, and applying 
Lemma \ref{l:main}, we deduce that $a_2=1$. In addition, since the $U_X$-level of the lowest weight $-\delta$ is even (the level is $2ai(m+1-i)$), Lemma \ref{l:main} also implies that $a_n=0$.  

Let $P_{X}=Q_XL_X$ be the $t$-stable parabolic subgroup of $X$ with $L_X'=A_{m-2}$ and 
$\Delta(L_X') = \{\b_2, \ldots, \b_{m-1}\}$. Let $W_j$ denote the $j$-th $Q_X$-level of $W$ and let $P=QL$ be the corresponding parabolic subgroup of $G$ (constructed as in Lemma \ref{l:flag}). In view of Lemma \ref{l:levels2}(a) we deduce that $L'=L_1 \cdots L_r$, where $r=2a+1$ and each $L_j$ is simple with natural module $Y_j = W_{j-1}$. Since $L_X' \leqs L'$, there is a natural projection map $\pi_j:L_X' \to L_j$ for each $j$.
Moreover, since $Y_j|_{L_{X}'} = W_{j-1}|_{L_X'}$ is nontrivial (see Lemma \ref{l:levels2}(a)) it follows that $\pi_j(L_X')$ is nontrivial, so we may view $L_X'$ as a subgroup of $L_j$ for all $j$. Now $Y_1=W_0 \cong W/[W,Q_X]$ so Lemma \ref{l:vq}(ii) implies that $Y_1|_{L_X'}$ is irreducible (with highest weight $(a\delta_{i}+a\delta_{m+1-i})|_{L_X'}$), while $Y_j|_{L_X'}$ is reducible for all $j \ge 2$ (see Lemma \ref{l:levels2}(a)(iii)). Also note that $\dim Y_1 \ge 7$ (minimal if $(m,a,i,p) = (4,1,2,3)$).

By Lemma \ref{l:vq}, $V/[V,Q]$ is a $p$-restricted irreducible $KL'$-module (with highest weight $\lambda|_{T \cap L'}$), so we may write
$$V/[V,Q]=M_1 \otimes \cdots \otimes M_r$$ 
where each $M_j$ is a $p$-restricted irreducible $KL_j$-module. Moreover, since $P_X$ is $t$-stable, Lemma \ref{l:tstable} implies that 
\begin{equation}\label{e:vvq}
V/[V,Q] = V/[V,Q_X] = V_1/[V_1,Q_X] \oplus V_2/[V_2,Q_X]
\end{equation}
as $KL_X'$-modules. By Lemma \ref{l:vq}(ii), $V_1/[V_1,Q_X]$ and $V_2/[V_2,Q_X]$ are irreducible $KL_X'$-modules with $T_{L_X'}$-highest weights $c_2\delta_2 + \cdots + c_{m-1}\delta_{m-1}$ and 
$c_{m-1}\delta_2 + \cdots + c_{2}\delta_{m-1}$, respectively (recall that $V_1$ has $T_X$-highest weight $\mu_1 = \sum_{i}c_i\delta_i$).

Since $Y_1|_{L_X'}$ is irreducible with symmetric highest weight $\delta' = (a\delta_{i}+a\delta_{m+1-i})|_{L_X'}$, it follows that $L_X'$ fixes a non-degenerate form on $Y_1$ (since $\delta'$ is fixed under the graph automorphism of $L_X'$ induced by $t$) and thus  
$$L_X' < Cl(Y_1) < L_1,$$
where $Cl(Y_1)$ is a simple symplectic or orthogonal group with natural module $Y_1$. In addition, since $a_2=1$, we note that  $M_1$ is nontrivial and $M_1 \neq Y_1, Y_1^*$. We consider the restriction of $M_1$ to $Cl(Y_1)$. 

Suppose $M_1|_{Cl(Y_1)}$ is irreducible. Then the configuration $(Cl(Y_1), L_1,M_1)$ must be one of the cases in \cite[Table 1]{Seitz2}. Since $a_2=1$ and $\dim Y_1 \ge 7$, close inspection of this table reveals that there are only three possibilities (each with $p \neq 2$), labelled ${\rm I}_{1}'$, ${\rm I}_{2}$ and ${\rm I}_{4}$:
\begin{itemize}\addtolength{\itemsep}{0.3\baselineskip}
\item[(i)] ${\rm I}_{1}'$ with $(a,b,k) = (p-2,1,1)$ or $(1,p-2,2)$;
\item[(ii)] ${\rm I}_{2}$ with $k=2$;
\item[(iii)] ${\rm I}_{4}$ with $k=2$.
\end{itemize}
(Note that $\dim Y_1 \neq 10$, so case ${\rm I}_{8}$ in \cite[Table 1]{Seitz2} does not arise here.)
We can immediately eliminate case (i) since the configuration in \cite[Table 1]{Seitz2} requires $Cl(Y_1) = {\rm Sp}(Y_1)$, but Lemma \ref{l:embed} implies that $L_X'<{\rm SO}(Y_1)$ since $p \neq 2$.
In (ii) and (iii) we have $M_1 = \L^2(Y_1)$ as a $KL_1$-module. However, Lemma \ref{l:wedge2} implies that $M_1|_{L_X'}$ has at least three composition factors, which contradicts \eqref{e:vvq}. 

Therefore $M_1|_{Cl(Y_1)}$ is reducible, so $M_1|_{L_X'}$ is also reducible. More precisely, in view of \eqref{e:vvq}, $L_X'$ has exactly two composition factors on $M_1$, so $M_1|_{Cl(Y_1)}$ also has exactly two composition factors, say $U_1$ and $U_2$. By applying Steinberg's tensor product theorem, we may write 
$$U_1 \cong S_1^{(q_1)} \otimes S_2^{(q_2)} \otimes \cdots \otimes S_k^{(q_k)}$$
for some $k \ge 1$, where $k=q_1=1$ if $p=0$, otherwise each $S_j$ is a nontrivial $p$-restricted irreducible $KCl(Y_1)$-module, the $q_j$ are certain powers of $p$, and $S_j^{(q_j)}$ is the twist of $S_j$ by a suitable standard Frobenius morphism of $Cl(Y_1)$. Note that each $S_j$ is an irreducible $KL_X'$-module since $U_1$ is also a $KL_X'$-composition factor of $M_1|_{L_X'}$. Since $a_2=1$, there exists some $j$ such that the highest weight of $S_j$ as a $KCl(Y_1)$-module has a non-zero coefficient on the second fundamental dominant weight (see Lemma \ref{l:nat}). Therefore $S_j \neq Y_1, Y_1^*$ and thus the configuration $(L_X',Cl(Y_1),S_j)$ must occur in \cite[Table 1]{Seitz2}. (Note that $S_j$ may be 
tensor decomposable as a $KCl(Y_1)$-module in some special cases (by Proposition \ref{p:s16}, this can only happen if $p=2$ and $Cl(Y_1)$ is of type $B$ or $C$). In this situation, $S_j = S_{j,1} \otimes S_{j,2}$ and at least one of the configurations $(L_X',Cl(Y_1),S_{j,k})$ must occur in \cite[Table 1]{Seitz2}.)
However, using the known action of $L_X'$ on $Y_1$, it is easy to check that no examples of this form arise. This final contradiction completes the proof of the lemma.
\end{proof}

\begin{lem}\label{l:a1m1}
Suppose $(G,H,V) \in \mathcal{T}$ and $\delta = a\delta_1+a\delta_m$ with $a \ge 2$. Then $a=2$ and $(a_1,a_2)=(0,1)$.
\end{lem}

\begin{proof}
Here $p \neq 2$ since $\delta$ is $p$-restricted, so $G$ is an orthogonal group (see Lemma \ref{l:embed}). As in the previous case, by considering a $t$-stable Borel subgroup $B_X=U_XT_X$ of $X$, and by applying Lemma \ref{l:main} and Proposition \ref{p:amlevels}, we deduce that $a_2=1$ and $a_i = 0$ for all $i$ such that $\a_i \in \Delta(L')\setminus\{\a_2\}$, where $P=QL$ is the usual parabolic subgroup of $G$ constructed from the $U_X$-levels of $W$. 

Let $P_{X}=Q_XL_X$ be the $t$-stable parabolic subgroup of $X$ with $\Delta(L_X') = \{\b_2, \ldots, \b_{m-1}\}$, so $L_X'=A_{m-2}$. Let $W_j$ denote the $j$-th $Q_X$-level of $W$ and let $P=QL$ be the corresponding parabolic subgroup of $G$. Write $L'=L_1 \cdots L_r$ where each $L_j$ is simple with natural module $Y_{j} = W_j$ (here we are using part (a) of Lemma \ref{l:levels2}). In particular, $r = \ell/2 = 2a$ where  $\ell=4a$ is the $Q_X$-level of the lowest weight $-\delta$. Now $L_X' \leqs L'$ and Lemma \ref{l:levels2}(a)(iv) states that each $Y_j|_{L_X'}$ is nontrivial, so 
we may view $L_X'$ as a subgroup of $L_j$. 
In addition, by Lemma \ref{l:levels2}(a)(iii), we note that $Y_j|_{L_X'}$ is reducible for all $j$. 

Since $V/[V,Q]$ is an irreducible $KL'$-module we have 
$V/[V,Q]=M_1 \otimes \cdots \otimes M_r$
where each $M_j$ is a $p$-restricted irreducible $KL_j$-module. Also note that \eqref{e:vvq} holds since $P_X$ is $t$-stable.
Suppose 
$M_1|_{L_X'}$ is irreducible. Now $M_1$ is nontrivial since $a_2=1$, and we note that $M_1 \neq Y_1, Y_1^*$ since we have already observed that $Y_1|_{L_X'}$ is reducible. Therefore we have a configuration $(L_X',L_1,M_1)$, where the restriction of $M_1$ to $L_X'$ is irreducible, but $L_X'$ acts reducibly on the natural $KL_1$-module. This contradicts Corollary \ref{c:g51}, so 
$M_1|_{L_X'}$ must be reducible. More precisely, in view of \eqref{e:vvq}, 
$M_1|_{L_X'}$ has exactly two composition factors, and $M_j|_{L_X'}$ is irreducible for all $j>1$. A further application of Corollary \ref{c:g51} reveals that $M_j|_{L_X'}$ is trivial for all $j>1$, whence 
\begin{equation}\label{e:vvq2}
V/[V,Q] = M_1 = V/[V,Q_X] = V_1/[V_1,Q_X] \oplus V_2/[V_2,Q_X]
\end{equation}
as $KL_X'$-modules. 

By Lemma \ref{l:levels2}, $\dim W_1=2(m-1)$. Also every weight in $Y_1 = W_1$ occurs under $\delta - \b_1$ or $\delta - \b_m$, so $Y_1|_{L_{X}'}$ has exactly two composition factors, with corresponding $T_{L_X'}$-highest weights $\delta_2|_{L_X'}$ and $\delta_{m-1}|_{L_X'}$ (in other words, $Y_1|_{L_{X}'} = U \oplus U^*$, where $U$ is the natural $KL_{X}'$-module). 

First assume $V_1/[V_1,Q_X] \cong V_2/[V_2,Q_X]$ as $KL_{X}'$-modules, which implies that 
\begin{equation}\label{e:ci}
\mbox{$c_{j}=c_{m+1-j}$ for all $j \ge 2$.}
\end{equation}
(Consequently, $c_1 \neq c_m$ since $V_1$ and $V_2$ are non-isomorphic $KX$-modules by Proposition \ref{p:niso}.) Let us switch to the parabolic subgroup $P_X=Q_XL_X$ of $X$ with $L_X'=A_{m-1}$ and $\Delta(L_X') = \{\b_1, \ldots, \b_{m-1}\}$. As usual, let $W_j$ be the $j$-th $Q_X$-level of $W$ and let $P=QL$ be the corresponding parabolic subgroup of $G$. By Lemma \ref{l:levels3}(a) we have $L'=L_1 \cdots L_r$ and each $L_j$ is simple with natural module $Y_j = W_{j-1}$. Furthermore, each projection map $\pi_j: L_X' \to L_j$ is nontrivial, so we may view $L_X' \leqs L_j$ for all $j$. Since $V/[V,Q]$ is an irreducible $KL'$-module (see Lemma \ref{l:vq}(i)) we have
$V/[V,Q]=M_1 \otimes \cdots \otimes M_r$
where each $M_j$ is a $p$-restricted irreducible $KL_j$-module. 

Since \eqref{e:ci} holds and $c_1 \neq c_m$, it follows that
\begin{equation}\label{e:summ}
\sum_{j=1}^mjc_j\neq \sum_{j=1}^m (m+1-j)c_j.
\end{equation}
Therefore Lemma \ref{l:centre} implies that $V/[V,Q]$ is an irreducible $KL_X'$-module, so each $M_j$ is also irreducible as a $KL_X'$-module. Now $Y_1 = W_0 \cong  W/[W,Q_X]$ is also an irreducible $KL_X'$-module with $T_{L_{X}'}$-highest weight $a\delta_1|_{L_X'}$ (by Lemma \ref{l:vq}(ii)), and we observe that $M_1 \neq Y_1, Y_1^*$ since $a_2=1$. Therefore we have a configuration $(L_X',L_1,M_1)$ which must appear in \cite[Table 1]{Seitz2}, and by inspecting this table we deduce that $a=2$ and $M_1 = \L^2(Y_1)$ is the only possibility (this is the case labelled ${\rm I}_{7}$ in  \cite[Table 1]{Seitz2}). In particular, we have $a=2$ and $(a_1,a_2)=(0,1)$ as required.

Now assume $V_1/[V_1,Q_X] \not\cong V_2/[V_2,Q_X]$ as $KL_{X}'$-modules, where $P_X= Q_XL_X$ is the $t$-stable parabolic subgroup of $X$ with $\Delta(L_X') = \{\b_2, \ldots, \b_{m-1}\}$. Define the $L_j, Y_j$ and $M_j$ as before (corresponding here to the parabolic subgroup $P=QL$ constructed from $P_X$). Let $U$ denote the natural $KL_X'$-module and observe that $L_X'<J^0<L_1$, where $J = A_{m-2}A_{m-2}T_1.2$ is the stabilizer in $L_1$ of the direct sum decomposition $Y_1 = U \oplus U^*$. Now $L_X'\la t \ra$ acts irreducibly on $M_1$ (since \eqref{e:vvq2} holds and $V_1/[V_1,Q_X] \not\cong V_2/[V_2,Q_X]$), so $J=J^0.2$ must also act irreducibly on $M_1$ (since $t$ interchanges $U$ and $U^*$ we have $L_X'\la t \ra < J$). Now $J$ is a maximal $\C_2$-subgroup of $L_1 = A_{2m-3}$ and thus the configuration $(J,L_1,M_1)$ must be one of the cases recorded in the main theorem of \cite{BGT} on geometric subgroups (note that $M_1$ is nontrivial since $a_2=1$). However, by inspecting \cite{BGT} we see that $M_1 = Y_1$ or $Y_1^*$ are the only possibilities, and the latter is ruled out since $a_2=1$. Therefore $M_1=Y_1$. We have previously observed that $Y_1|_{L_X'}$ has exactly two composition factors, with corresponding $T_{L_{X}'}$-highest weights $\delta_2|_{L_X'}$ and $\delta_{m-1}|_{L_X'}$, and we recall that \eqref{e:vvq2} holds. Since $M_1=Y_1$, by relabelling if necessary, we may assume that $V_1/[V_1,Q_X]$ has $T_{L_{X}'}$-highest weight $\delta_2|_{L_X'}$, while $V_2/[V_2,Q_X]$ has highest weight $\delta_{m-1}|_{L_X'}$. In particular, $V_1$ and $V_2$ have $T_X$-highest weights $c_1\delta_1+\delta_2+c_m\delta_m$ and $c_m\delta_1+\delta_{m-1}+c_1\delta_m$, respectively.

As before, we now switch to the maximal parabolic subgroup $P_X=Q_XL_X$ with $\Delta(L_X') = \{\b_1, \ldots, \b_{m-1}\}$. Let $W_j$ be the $j$-th $Q_X$-level of $W$ and write $L' = L_1\cdots L_r$, where each $L_j$ is simple with natural module $Y_j = W_{j-1}$. As usual we may view $L_X' \leqs L_j$ for all $j$, and write $V/[V,Q]=M_1 \otimes \cdots \otimes M_r$, where each $M_j$ is a $p$-restricted irreducible $KL_j$-module.
Using the expressions for the $T_X$-highest weights of $V_1$ and $V_2$ obtained in the previous paragraph, together with the bound $m \ge 4$, we deduce that \eqref{e:summ} holds and thus Lemma \ref{l:centre} implies that 
$V/[V,Q]$ is an irreducible $KL_X'$-module.
Therefore $M_1$ is an irreducible $KL_X'$-module and we have a configuration $(L_X',L_1,M_1)$, where $Y_1|_{L_X'}$ has highest weight $a\delta_1|_{L_X'}$. Once again, by inspecting \cite[Table 1]{Seitz2} we deduce that $a=2$ and $M_1 = \L^2(Y_1)$ is the only possibility.
\end{proof}

\begin{lem}\label{l:a1m3}
If $\delta = 2\delta_1+2\delta_m$ then $(G,H,V) \not\in \mathcal{T}$. 
\end{lem}

\begin{proof}
First observe that $p \neq 2$ since $\delta$ is $p$-restricted. Seeking a contradiction, let us assume
$(G,H,V) \in \mathcal{T}$. By Lemma \ref{l:a1m1}, we have $(a_1,a_2)=(0,1)$. Let $P_X=Q_XL_X$ be the $t$-stable parabolic subgroup of $X$ with $\Delta(L_X')=\{\beta_1,\beta_{m}\}$, let $W_j$ denote the $j$-th $Q_X$-level of $W$, and construct the parabolic subgroup $P=QL$ of $G$ in the usual way. Write $L'=L_1 \cdots L_r$ and $V/[V,Q]=M_1\otimes \cdots \otimes M_r$, where each $L_j$ is simple with natural module $Y_j=W_{j-1}$, and each $M_j$ is a $p$-restricted irreducible $KL_j$-module (see Lemma \ref{l:levels3}(a)). Note that $Y_1 =W_0 \cong W/[W,Q_X]$ is an irreducible $KL_X'$-module with highest weight $(2\delta_1+2\delta_{m})|_{L_X'}$, so we may view $L_X'=A_1A_1$ as a subgroup of $L_1=A_8$. We can now apply Lemma \ref{l:a1a1am} with respect to the triple 
$(L_X',L_1,M_1)$, and we deduce that $M_1|_{L_X'}$ has more than two composition factors (recall that $a_2=1$). This is a contradiction (indeed, \eqref{e:vvq} holds by the $t$-stability of $P_X$).
\end{proof}

\subsection{The case $\delta = \delta_{1}+\delta_{m}$}\label{sss:4}

Here $W$ is the nontrivial composition factor of the Lie algebra $\mathcal{L}(X)$, so the non-zero $T_X$-weights in $W$ are in bijection with the roots of $X$. 
In addition, 
$\dim W = m^2+2m-\e$, where $\e=1$ if $p$ divides $m+1$, otherwise $\e=0$, and \cite[Table 2]{Brundan} indicates that $G$ is symplectic if and only if $p=2$ and $m \equiv 1 \imod{4}$.

\begin{lem}\label{l:aaa1}
Suppose $(G,H,V) \in \mathcal{T}$ and $\delta = \delta_1+\delta_m$. Then $a_2=1$.
\end{lem}

\begin{proof}
Let $B_X=U_XT_X$ be a $t$-stable Borel subgroup of $X$ and let $W_j$ denote the $j$-th $U_X$-level of $W$. Let $\ell=2m$ denote the $U_X$-level of the lowest weight $-\delta$. By Lemma \ref{l:am1m} we have $\dim W_{j} = j+1$ for all $j<m$, while $\dim W_{m}=m-\e$, where $\e=1$ if $p$ divides $m+1$, otherwise $\e=0$. 
Therefore Lemma \ref{l:main} implies that $a_2=1$ if $m \ge 6$. We need to treat the cases $m=4,5$ separately.

Suppose $m=4$, so $\ell=8$ and Lemma \ref{l:am1m} implies that $\dim W=24-\delta_{5,p}$. Further, as noted above, $G$ is orthogonal.
Let $P=QL$ be the corresponding parabolic subgroup of $G$ constructed in the usual way. Then $L'=L_1L_2L_3L_4$ where each $L_i$ with $1\leq i \leq 3$ is simple with natural module $Y_i=W_i$. 
By applying Remark \ref{r:ord} we may order the $T$-weights in $W$ at levels $0$ to $3$ so that we  obtain the root restrictions listed in Table \ref{t:r11} (note that there is a single $T_X$-weight at level $4$, namely the zero weight).

\renewcommand{\arraystretch}{1.2}
\begin{table}
$$\begin{array}{llll} \hline
\mbox{$U_X$-level} & \mbox{$T_X$-weight} & \mbox{$T$-weight} & \mbox{Root restriction} \\ \hline
0 & \delta& \lambda_1& \\
1 & \delta-\beta_1& \lambda_1-\alpha_1& \alpha_1|_{X}=\beta_1 \\
 & \delta-\beta_4& \lambda_1-\alpha_1-\alpha_2& \alpha_2|_{X}=\beta_4-\beta_1 \\
2 & \delta-\beta_3-\beta_4& \lambda_1-\sum_{i=1}^3\alpha_i & \alpha_3|_{X}=\beta_3 \\
& \delta-\beta_1-\beta_2& \lambda_1-\sum_{i=1}^4\alpha_i & \alpha_4|_{X}=\beta_1+\beta_2-\beta_3-\beta_4 \\
& \delta-\beta_1-\beta_4& \lambda_1-\sum_{i=1}^5\alpha_i & \alpha_5|_{X}=\beta_4-\beta_2 \\
3 & \delta-\beta_1-\beta_2-\beta_4& \lambda_1-\sum_{i=1}^6\alpha_i& \alpha_6|_{X}=\beta_2 \\
&  \delta-\beta_1-\beta_3-\beta_4& \lambda_1-\sum_{i=1}^7\alpha_i& \alpha_7|_{X}=\beta_3-\beta_2 \\
&  \delta-\beta_1-\beta_2-\beta_3& \lambda_1-\sum_{i=1}^8\alpha_i& \alpha_8|_{X}=\beta_2-\beta_4 \\
&  \delta-\beta_2-\beta_3-\beta_4& \lambda_1-\sum_{i=1}^9\alpha_i& \alpha_9|_{X}=\beta_4-\beta_1 \\
4 & \delta-\beta_1-\beta_2-\beta_3-\beta_4 & \lambda_1-\sum_{i=1}^{10} \alpha_i & \alpha_{10} |_{X}=\beta_1 \\
& & \lambda_1-\sum_{i=1}^{11}\alpha_i & \alpha_{11}|_{X}=0\\
& & \lambda_1-\sum_{i=1}^{10}\alpha_i-\alpha_{12} & \alpha_{12}|_{X}=0 
\mbox{ ($p\neq 5$ only)} \\ \hline
\end{array}$$
\caption{}
\label{t:r11}
\end{table}
\renewcommand{\arraystretch}{1}

By Lemma \ref{l:main} we deduce that $\{a_2,a_{11}\}=\{0,1\}$ if $p=5$, otherwise 
$\{a_2,a_{11}+a_{12}\}=\{0,1\}$. Recall that $V|_X=V_1\oplus V_2$, where $V_1$ and $V_2$ have highest weights $\mu_1=\l|_X$ and $\mu_2$, respectively. Suppose that $a_{11}\neq 0$. Then $a_{11}=1$ and the weight $\lambda-\alpha_{11} \in \L(V)$ restricts to $\mu_1$ as a $T_X$-weight. Using Lemma \ref{l:ammu1mu2}, one checks that $\mu_1$ is under $\mu_2$ only if $\mu_1=\mu_2$. As this is not the case (indeed, $V_1$ and $V_2$ are non-isomorphic by Proposition \ref{p:niso}), it follows that the multiplicity of $\mu_1$ in $V_1$ is at least $2$, which is a contradiction. Therefore $a_{11}=0$. If $p \neq 5$ then a similar argument yields $a_{12}=0$. Hence $a_2=1$, as required. 

To complete the proof of the lemma we may assume that $m=5$, so $\ell=10$. By 
Lemma \ref{l:am1m} we have $\dim W_i=i+1$ for all $0\leq i \leq 4$, and $\dim W_5=4$ 
if $p=2$ or $3$, otherwise $\dim W_5=5$. In particular, if $p\neq 2,3$ then Lemma \ref{l:main} implies that $a_2=1$. Similarly, if $p=2$ then $\dim W=34$
and $G=C_{17}$ (see \cite[Table 2]{Brundan}), and once again we deduce that $a_2=1$.

Finally suppose that $p=3$. Here $\dim W=34$ and Lemma \ref{l:embed} implies that $G=D_{17}$. Let $P=QL$ be the parabolic subgroup of $G$ constructed from the Borel subgroup $B_X=U_XT_X$ of $X$, so $L'=L_1L_2L_3L_4L_5$, where each $L_i$ is simple with natural module $Y_i=W_i$. 
We now argue as in the case $m=4$, using root restrictions, to show that $a_2=1$.  We leave the details to the reader.
\end{proof}

\begin{lem}\label{l:aaa2}
Suppose $(G,H,V) \in \mathcal{T}$ and $\delta = \delta_1+\delta_m$. Then $\l = a_1\l_1+\l_2+a_{2m-1}\l_{2m-1}$.
\end{lem}

\begin{proof}
Let $P_X=Q_XL_X$ be the parabolic subgroup of $X$ with $\Delta(L_X') = 
\{\b_2, \ldots, \b_{m-1}\}$. Let $W_i$ denote the $i$-th $Q_X$-level of $W$ and note that $\ell=4$ is the $Q_X$-level of the lowest weight $-\delta$. Then $\dim W_0=1$ and as in the proof of Lemma \ref{l:levels2} we find that $W_{1}|_{L_{X}'}$ has exactly two composition factors, with highest weights $\delta_{2}|_{L_X'}$ and $\delta_{m-1}|_{L_X'}$, respectively. (In other words, $W_{1}|_{L_{X}'} = U \oplus U^*$ where $U$ is the natural $KL_X'$-module.) Let $P=QL$ be the corresponding parabolic subgroup of $G$ constructed from the $Q_X$-levels of $W$. Write $L'=L_1L_2$, where each $L_i$ is simple with natural module $Y_i=W_i$ (note that $L_1=A_{2m-3}$ and $L_2$ is a classical group). By Lemma \ref{l:levels2}, $L_X'$ acts nontrivially on each $Y_i$, so the projection maps $\pi_i:L_X' \to L_i$ are nontrivial and we may view $L_X'$ as a subgroup of $L_i$. Finally, note that $V/[V,Q]$ is an irreducible $KL'$-module (Lemma \ref{l:vq}(i)), so $V/[V,Q]=M_1 \otimes M_2$ and each $M_i$ is a $p$-restricted irreducible $KL_i$-module.  

Suppose $M_1|_{L_{X}'}$ is irreducible. Then $M_1 \neq Y_1,Y_1^*$ (since $Y_1|_{L_X'}$ is reducible), and we note that $M_1$ is nontrivial since $a_2=1$. Therefore we have a configuration $(L_X',L_1,M_1)$, where $M_1|_{L_{X}'}$ is irreducible but $L_X'$ is reducible on the natural $KL_1$-module. Such a situation is ruled out by Corollary \ref{c:g51}, so $M_1|_{L_{X}'}$ must be reducible (with exactly two composition factors) and hence $M_2|_{L_{X}'}$ is irreducible (here we are using the fact that $V/[V,Q]$ has exactly two $KL_X'$-composition factors, which follows from \eqref{e:vvq}). 
In addition, we observe that $Y_2|_{L_X'}$ has at least two composition factors, with highest weights $0$ and $(\delta_2+\delta_{m-1})|_{L_X'}$. Therefore $M_2 \neq Y_2, Y_2^*$ and thus Corollary \ref{c:g51} implies that $M_2$ is trivial, so $a_{i} = 0$ for all $i \ge 2m$ (since $\Delta(L_2) = \{\a_{2m}, \a_{2m+1}, \ldots, \a_{n}\}$). 

To complete the proof it remains to show that $M_1=Y_1$. As in the proof of Lemma \ref{l:a1m1}, we observe that $L_X'<J^0<L_1$, where $J$ is a maximal $\C_2$-subgroup of $L_1$ stabilizing the decomposition $Y_1 = U \oplus U^*$ of the natural $KL_1$-module. Now $M_1|_{L_X'}$ has exactly two composition factors, so $M_1|_{J^0}$ has at most two. However, if we assume $M_1 \neq Y_1$ (note that  $M_1 \neq Y_1^*$ since $a_2=1$) then Lemma \ref{l:c2} reveals that $M_1|_{J^0}$ has at least three composition factors. This is a contradiction, so $M_1=Y_1$ as required.
\end{proof}

\begin{lem}\label{l:aaa3}
Suppose $(G,H,V) \in \mathcal{T}$ and $\delta = \delta_1+\delta_m$. Then $\l = \l_2$.
\end{lem}

\begin{proof}
By Lemma \ref{l:aaa2} we may assume that $\l = a_1\l_1+\l_2+a_{2m-1}\l_{2m-1}$.
First we will show that $a_{2m-1}=0$. 

Let $P_X=Q_XL_X$ be the parabolic subgroup of $X$ with $\Delta(L_X')=\{\beta_1,\dots,\beta_{m-1}\}$, and construct the corresponding parabolic subgroup $P=QL$ of $G$ in the usual way. Let $W_i$ be the 
$i$-th $Q_X$-level of $W$, and note that $\ell=2$ is the $Q_X$-level of the lowest weight $-\delta$. By Lemma \ref{l:levels3}(a) we have $\dim W_0=m$ and $\dim W_1=m^2-\e$, where $\e=1$ if $p$ divides $m+1$, otherwise $\e=0$. Write $L'=L_1L_2$ where each $L_i$ is simple with natural module $Y_i=W_{i-1}$. Note that  $L_1=A_{m-1}$ and $\Delta(L_1)=\{\alpha_1,\dots,\alpha_{m-1}\}$. Also write
$V/[V,Q]=M_1\otimes M_2$ where each $M_i$ is a $p$-restricted irreducible $KL_i$-module. Since $\lambda_1|_{X}=\delta_1$ and $\pi_1(L_X')=A_{m-1}$, we may assume that the $T$-weights in $W$ at level 0 restrict as $T_X$-weights so that 
$(\lambda_1-\sum_{k=1}^j \alpha_k)|_{X} = \delta - \sum_{k=1}^{j} \beta_k$ for all 
$1\leq j \leq m-1$. This implies that $\alpha_j|_{X}=\beta_j$ for all $1\leq j \leq m-1$.

Since $M_1$ is an irreducible $KL_1$-module of highest weight $(a_1\lambda_1+\lambda_2)|_{L_1}$, it follows from the above root restrictions that $M_1|_{L_X'}$ is irreducible of highest weight $(a_1\delta_1+\delta_2)|_{L_X'}$. In particular, $\mu_1=a_1\delta_1+\delta_2+c_m\delta_m$ and $\mu_2=c_m\delta_1+\delta_{m-1}+a_1\delta_m$, where as usual $\mu_i$ is the highest weight of the irreducible $KX$-module $V_i$ in the decomposition $V = V_1 \oplus V_2$ (we label the $V_i$ so that $\mu_1=\l|_{X}$).
Since  $m\geq 4$ we have  $a_1+2+mc_m\neq ma_1+(m-1)+c_m$ and thus Lemma \ref{l:centre} implies that $V/[V,Q]$ is an irreducible $KL_X'$-module (and without loss of generality, we may assume $V/[V,Q] \cong V_1/[V_1,Q_X]$ as $KL_X'$-modules). In particular, $M_1$ and $M_2$ are both irreducible $KL_X'$-modules.  

Next we observe that $Y_2|_{L_X'}$ is nontrivial as it affords a composition factor of highest weight $(\delta_1+\delta_{m-1})|_{L_X'}$. Therefore the projection map $\pi_2: L_X' \to L_2$ is nontrivial, so we may view $L_X'$ as a (proper) subgroup of $L_2$. Consider the triple $(L_X',L_2,M_2)$. Here $M_2$ is a $p$-restricted irreducible $KL_2$-module and $M_2|_{L_X'}$ is irreducible. If $Y_2|_{L_X'}$ is reducible then Corollary \ref{c:g51} implies that $M_2$ is trivial and thus $a_{2m-1}=0$. Now assume $Y_2|_{L_X'}$ is irreducible, so $Y_2|_{L_X'}$ has highest weight $(\delta_1+\delta_{m-1})|_{L_X'}$. Now $M_2 \neq Y_2, Y_2^*$ (since $\lambda=a_1\lambda_1+\lambda_2+a_{2m-1}\lambda_{2m-1}$), so the configuration $(L_X',L_2,M_2)$ must be one of the cases listed in \cite[Table 1]{Seitz2}. However, the only possibility is the case labelled ${\rm S}_{7}$, with $m=4$, $p=2$ and $L_2=D_7$, but this implies that $\dim Y_2=14$, which contradicts Lemma \ref{l:levels3}(a)(iii). 
 
We have now reduced to the case $\l = a_1\l_1+\l_2$ and so it remains to show that $a_1=0$. Seeking a contradiction, assume $a_1 \neq 0$. Here we generalize part of the argument given in the proof of \cite[Lemma 3.5]{Ford2}. Let $P_X=Q_XL_X$ be the maximal parabolic subgroup of $X$ with $\Delta(L_X')=\{\beta_1,\beta_3, \ldots, \beta_{m}\}$ and construct the corresponding parabolic subgroup $P=QL$ of $G$ in the usual way. Let $W_i$ denote the $i$-th $Q_X$-level of $W$ and note that $\ell=2$ is the level of the lowest weight $-\delta$. Then $L'=L_1L_2$, where each $L_i$ is simple with natural module $Y_i=W_{i-1}$. Let $\pi_i: L_X' \to L_i$  be the $i$-th projection map
and write 
$V/[V,Q]=M_1\otimes M_2$, where each $M_i$ is a $p$-restricted irreducible $KL_i$-module.  

Suppose $V/[V,Q]$ is an irreducible $KL_X'$-module, so each $M_i$ is also irreducible as a $KL_X'$-module. Now $W_0 \cong W/[W,Q_X]$ is an irreducible $KL_X'$-module with highest weight $\delta|_{L_X'}$, so $Y_1|_{L_X'}=V_{L_X'}((\delta_1+\delta_m)|_{L_X'})$ and thus $\pi_1(L_X') =A_1A_{m-2}$. In particular, we may view $L_X'$ as a proper subgroup of $L_1$. 
Consider the triple $(L_X',L_1,M_1)$. Here $M_1$ is a $p$-restricted irreducible $KL_1$-module and $M_1\neq Y_1, Y_1^*$ since $a_2=1$. Therefore \cite[Theorem 1]{Seitz2} implies that this triple $(L_X',L_1,M_1)$ must appear in \cite[Table 1]{Seitz2}, but by inspection we see that there are no compatible examples. This is a contradiction, whence $V/[V,Q]$ is a reducible $KL_X'$-module. More precisely, as $KL_X'$-modules we have 
$$V/[V,Q]=V_1/[V_1,Q_X] \oplus V_2/[V_2, Q_X],$$ 
so the weights $\mu_1$ and $\mu_2$ both occur with non-zero multiplicity in $V/[V,Q]$.  

Let $Z=Z(L_X)^0$ and note that $Z\leqs Z(L)$ since $L=C_G(Z)$ (see Lemma \ref{l:flag}). Now $V/[V,Q]$ is an irreducible $KL$-module, so $Z(L)$ acts as scalars on $V/[V,Q]$ (by Schur's lemma) and thus 
$\mu_1|_Z=\mu_2|_Z$. Since 
$$Z=\left\{h_{\beta_1}(c^{m-1})h_{\beta_2}(c^{2(m-1)})h_{\beta_3}(c^{2(m-2)}) \dots h_{\beta_{m-1}}(c^4) h_{\beta_m}(c^2) \mid c\in K^*\right\}$$ 
this implies that 
$$(m-1)a_1+2(m-1)+2c_m=(m-1)c_m+4+2a_1.$$ 
In particular, $c_m=a_1+2$ since $m \ge 4$.

Now let $P_X=Q_XL_X$ be the parabolic subgroup of $X$ with $\Delta(L_X')=\{\beta_1,\beta_m\}$.
As usual, let $W_i$ be the $i$-th  $Q_X$-level of $W$ and let $P=QL$ be the corresponding parabolic subgroup of $G$. We have $L'=L_1\cdots L_r$ and $V/[V,Q]=M_1\otimes \cdots \otimes M_r$, where each $L_i$ is simple with natural module $Y_i = W_{i-1}$, and each $M_i$ is a $p$-restricted irreducible $KL_i$-module.  The weights at level 0 are $\delta$, $\delta-\beta_1$, $\delta-\beta_m$ and $\delta-\beta_1-\beta_m$, so $\dim W_0=4$ and thus $L_1= A_3$.  Since $\l = a_1\l_1+\l_2$ it follows that  $M_i$ is a trivial $KL_i$-module for all $i\geq 2$ and so $V/[V,Q]= M_1$ as a $KL'$-module.  In addition, we note that $V/[V,Q]|_{L_1}$ has highest weight $(a_1\lambda_1+\lambda_2)|_{L_1}$. Since $P_X$ is $t$-stable, Lemma \ref{l:tstable} implies that 
$$V/[V,Q]=V_1/[V_1,Q_X]\oplus V_2/[V_2,Q_X]$$ 
as $KL_X'$-modules, and $\dim V/[V,Q]=2\dim V_1/[V_1,Q_X]$. We also note that 
$$V_1/[V_1,Q_X]= V_{L_X'}((a_1\delta_1+(a_1+2)\delta_m)|_{L_X'}).$$

If $a_1 \neq p-2$ then Lemma \ref{l:a1a3_2} immediately implies that $V/[V,Q]$ has more than two $KL_X'$-composition factors, which is a contradiction. (Here we are applying Lemma \ref{l:a1a3_2} with $G=L_1$, $Y=L_X'$, $V=V/[V,Q]$ and $W=Y_1$. Note that $Y_1|_{L_X'}$ is irreducible with highest weight $(\delta_1+\delta_m)|_{L_X'}$, so we may view $L_X'=A_1A_1$ as a subgroup of $L_1 = A_3$.)

Finally, let us consider the case $a_1=p-2$. Note that $p \neq 2$ since we are assuming $a_1 \neq 0$. 
Now $V_1/[V_1,Q_X]=V_{L_X'}(((p-2)\delta_1+p\delta_m)|_{L_X'})$, so 
$\dim V_1/[V_1,Q_X]=2(p-1)$ and thus $\dim V/[V,Q]=4(p-1)$. However, if $p=3,5$ or $7$ then \cite[Table A.7]{Lubeck} implies that $\dim V/[V,Q]>4(p-1)$, which is a contradiction. For example, if $p=3$ then $a_1=1$ and $\dim V/[V,Q] = \dim V_{A_3}(\l_1+\l_2) = 16$. For $p>7$, each weight $\lambda-k\alpha_1$ with $0 \leq k <  (p-3)/2$ is dominant and it has precisely $4!/2=12$ distinct $\mathcal{W}(A_3)$-conjugates (where $\mathcal{W}(A_3)$ is the Weyl group of $A_3$), whence 
$$\dim V/[V,Q]\geq 12 \cdot (p-3)/2=6(p-3) > 4(p-1),$$
which once again is a contradiction. 
\end{proof}

\begin{lem}\label{l:aaa4}
If $\delta = \delta_1+\delta_m$ then $(G,H,V) \not\in \mathcal{T}$.
\end{lem}

\begin{proof}
By Lemma \ref{l:aaa3} we may assume that $\l=\l_2$. By Lemma \ref{l:st}, if $p\neq2$ then $G={\rm SO}(W)$ and thus $V=\L^2(W)$. Otherwise, $V$ is the unique nontrivial composition factor of $\L^2(W)$.
Seeking a contradiction, suppose $(G,H,V) \in \mathcal{T}$.
Recall that $V|_{X} = V_1 \oplus V_2$, where $V_1$ and $V_2$ have $T_X$-highest weights 
$$\mu_1=\l|_{X} = 2\delta-\b_1=\delta_2+2\delta_m,\;\; \mu_2 = 2\delta-\b_m=2\delta_1+\delta_{m-1},$$ 
respectively. We now proceed as in the proof of Lemma \ref{l:wedge2}; it suffices to find a weight $T_X$-weight $\nu$ such that $m_{V}(\nu)>m_{V_1}(\nu)+m_{V_2}(\nu)$.

Consider the weight
$$\nu=\l-\b_2-\cdots-\b_m = 2\delta-\b_1-\cdots - \b_m \in \L(V).$$
First assume $p \neq 2$. We claim that $m_{V}(\nu)=2m-1-\e$, where $\e=1$ if $p$ divides $m+1$, otherwise $\e=0$. To see this, we determine the different pairs $\omega_1,\omega_2 \in \L(W)$ such that $\omega_1+\omega_2 = \nu$. If $\mu = \delta \wedge (\delta - \b_1 - \cdots - \b_m)$ then Lemma \ref{l:s816} implies that $m_{V}(\mu) = m-\e$. The only other suitable pairs of weights $\omega_1,\omega_2$ are of the form 
$$\omega_1 \wedge \omega_2 = \left(\delta- \sum_{j=1}^{i}\beta_j \right) \wedge \left(\delta-\sum_{j=i+1}^{m}\beta_j\right)$$ 
with $1 \leq i \leq m-1$, and each of these weights has multiplicity $1$. This establishes the claim. Now, a further application of Lemma \ref{l:s816}
implies that $\nu$ has multiplicity $m-1-\e$ in $V_1$, and similarly $\nu=\mu_2-\b_1-\cdots - \b_{m-1}$ has the same multiplicity in $V_2$, whence $m_{V}(\nu)>m_{V_1}(\nu)+m_{V_2}(\nu)$, which is a contradiction. 

Finally, if $p=2$ then $V_{1}=V_{11}^{(2)} \otimes V_{12}$ as $KX$-modules, where $V_{11}$ and $V_{12}$ have $T_{X}$-highest weights $\delta_{m}$ and $\delta_{2}$, respectively. Therefore 
$$m_{V}(\nu) \ge 2m-2 >m_{V_{1}}(\nu)+m_{V_{2}}(\nu) = 2m_{V_{1}}(\nu)=2$$
and this final contradiction rules out the case $\l=\l_2$.
\end{proof}

\subsection{The case $\delta = \delta_{(m+1)/2}$}

By Corollary \ref{c:red}, in order to complete the proof of Theorem \ref{t:am} we may assume 
$\delta = \delta_{(m+1)/2}$, where $m \ge 5$ is odd. We start by dealing with the case $m=5$.

\begin{lem}\label{l:aa1}
Suppose $X=A_5$ and $\delta=\delta_3$. Then $(G,H,V) \in \mathcal{T}$ if and only if $G=C_{10}$, $p \neq 2,3$, $\l=\l_3$ and $\l|_{X} = \delta_1+2\delta_4$ or $2\delta_2+\delta_5$.
\end{lem}

\begin{proof}
Here $\dim W=20$ and $G =D_{10}$ if $p=2$, otherwise $G = C_{10}$ (see \cite[Table 2]{Brundan}). Write $V|_X=V_1\oplus V_2$, where $V_1$ and $V_2$ are irreducible $KX$-modules interchanged by $t$, with highest weights $\mu_1=\sum_{i=1}^5c_i\delta_i$ and $\mu_2=\sum_{i=1}^5c_{6-i}\delta_i$, respectively. Without loss of generality, we assume that $\lambda|_{X}=\mu_1$.

Let $B_X=U_XT_X$ be a $t$-stable Borel subgroup of $X$, let $W_i$ denote the $i$-th $U_X$-level of $W$ and let $\ell$ be the level of the lowest weight $-\delta$. Then $\ell=9$ and we calculate that $\dim W_0=\dim W_1=1$, $\dim W_2=2$ and $\dim W_3=\dim W_4=3$, so 
Lemma \ref{l:main} implies that $a_3=1$. In addition, Lemma \ref{l:main} also implies that $a_6=0$. By appealing to Remark \ref{r:ord} we may order the $T$-weights in $W$ at levels $0$ to $3$ to give the root restrictions presented in Table \ref{t:r12}.

\renewcommand{\arraystretch}{1.2}
\begin{table}
$$\begin{array}{llll} \hline
\mbox{$U_X$-level} & \mbox{$T_X$-weight} & \mbox{$T$-weight} & \mbox{Root restriction} \\ \hline
0 & \delta & \lambda_1 & \\
1& \delta-\beta_3 &  \lambda_1-\alpha_1 &\alpha_1|_{X}=\beta_3\\
2 & \delta-\beta_2-\beta_3 &  \lambda_1-\alpha_1-\alpha_2 & \alpha_2|_{X}=\beta_2\\
& \delta-\beta_3-\beta_4 &   \lambda_1-\alpha_1-\alpha_2-\alpha_3& \alpha_3|_{X}=\beta_4-\beta_2\\
3 & \delta-\beta_2-\beta_3-\beta_4 &   \lambda_1-\alpha_1-\alpha_2-\alpha_3-\alpha_4& \alpha_4|_{X}=\beta_2\\
& \delta-\beta_1-\beta_2-\beta_3 &   \lambda_1-\sum_{i=1}^5\alpha_i & \alpha_5|_{X}=\beta_1-\beta_4\\
 & \delta-\beta_3-\beta_4-\beta_5 &   \lambda_1-\sum_{i=1}^6\alpha_i & \alpha_6|_{X}=\beta_4+\beta_5-\beta_1-\beta_2\\ \hline
\end{array}$$
\caption{}
\label{t:r12}
\end{table}
\renewcommand{\arraystretch}{1}

Next let $P_X=Q_XL_X$ be the $t$-stable parabolic subgroup of $X$ with $\Delta(L_X')=\{\b_2,\b_3,\b_4\}$,  let $P=QL$ be the corresponding parabolic subgroup of $G$ and let $W_i$ be the $i$-th $Q_X$-level of $W$. Then $\ell=2$ is the $Q_X$-level of the lowest weight, and we have $\dim W_0=6$ and $\dim W_1=8$. 
Write $L=L_1L_2$, where $L_1=A_5$, and $L_2=D_4$ or $C_4$ if $p=2$ or $p\neq 2$, respectively. Let $Y_{i}$ be the natural module for $L_i$ and note that $Y_{i}=W_{i-1}$. Also write 
$V/[V,Q]=M_1\otimes M_2$ where each $M_i$ is a $p$-restricted irreducible $KL_{i}$-module. As usual, let $\pi_i$ be the projection from $L_X'$ to $L_{i}$. Now $Y_1|_{L_1'}$ is irreducible of highest weight $\delta_3|_{L_X'}$, while $Y_2|_{L_X'}$ has two composition factors of highest weights $\delta_4|_{L_X'}$ and $\delta_2|_{L_X'}$. Therefore $\pi_1(L_X')=\pi_2(L_X')=A_3$ so we may view $L_X'$ as a proper subgroup of both $L_{1}$ and $L_{2}$. 
 
By Lemma \ref{l:tstable} we have
$$V/[V,Q]=V/[V,Q_X]=V_1/[V_1,Q_X]\oplus V_2/[V_2,Q_X]$$
as $KL_X'$-modules, where $V_1/[V_1,Q_X]$ and $V_2/[V_2,Q_X]$ are irreducible $KL_X'$-modules. 
Suppose $M_1|_{L_X'}$ is irreducible and consider the triple $(L_X',L_1,M_1)$. Then $L_X'$ is a proper subgroup of $L_1$,   $M_1$ is a nontrivial  $p$-restricted irreducible $KL_1$-module, and $M_1 \neq Y_1, Y_1^*$ since $a_3=1$. Then \cite[Theorem 1]{Seitz2} implies that this configuration is one of the cases listed in \cite[Table 1]{Seitz2}, but we see that there are no compatible examples. This is a contradiction, so $M_1|_{L_X'}$ is 
reducible,  and thus $M_2|_{L_X'}$ is irreducible. 

Next suppose $M_2$ is nontrivial and consider the triple $(L_X',L_{2},M_2)$. Recall that $Y_2|_{L_X'}$ has two composition factors, so $M_2\neq Y_2, Y_2^*$ since $M_2|_{L_X'}$ is irreducible. This situation is ruled out by Corollary \ref{c:g51}, so 
$M_2$ is trivial and thus 
\begin{equation}\label{e:ee}
V/[V,Q_1]=M_1= V_1/[V_1,{Q_X}]\oplus V_2/[V_2,Q_X]
\end{equation}
as $KL_X'$-modules.
 
We claim that $V_1/[V_1,Q_X]$ and $V_2/[V_2,Q_X]$ are non-isomorphic $KL_X'$-modules. Seeking a contradiction, let us assume they are isomorphic. As $KL_X'$-modules, $V_1/[V_1,Q_X]$ and $V_2/[V_2,Q_X]$ have respective highest weights 
$(c_2\delta_2+c_3\delta_3+c_4\delta_4)|_{L_X'}$ and $(c_4\delta_2+c_3\delta_3+c_2\delta_4)|_{L_X'}$, so if these two modules are isomorphic then $c_2=c_{4}$. In particular, since $V_1$ and $V_2$ are non-isomorphic $KX$-modules (see Proposition \ref{p:niso}), it follows that $c_1\neq c_5$ and thus
\begin{equation}\label{e:jcj}
\sum_{j=1}^{5}jc_j \neq \sum_{j=1}^{5}(6-j)c_j.
\end{equation}

At this stage, we switch to the asymmetric parabolic subgroup $P_X=Q_XL_X$ of $X$ with $\Delta(L_X')=\{\b_1, \b_2,\b_3, \b_{4}\}$, and we construct the corresponding parabolic subgroup $P=QL$ of $G$ in the usual way. Let $W_i$ be the $i$-th $Q_X$-level of $W$ and note that $\ell=1$ and $\dim W_0=10$. Therefore $L'=L_{1} = A_9$ is a simple group with natural module $Y_1=W_0$, and $V/[V,Q]=M_1$ is a $p$-restricted irreducible $KL_{1}$-module. Now $Y_1|_{L_X'}=W_0|_{L_X'}$ is irreducible with highest weight $\delta_3|_{L_X'}$, so we may view $L_X'=A_4$ as a proper subgroup of $L_{1}$.
Consider the triple $(L_X',L_{1},M_1)$ and note that $M_1$ is nontrivial and $M_1\neq Y_1,Y_1^*$ since $a_3=1$. In view of \eqref{e:jcj}, Lemma \ref{l:centre} implies that $V/[V,Q]=M_1$ is irreducible as a $KL_X'$-module, hence \cite[Theorem 2]{Seitz2} implies that this triple must be one of the cases in \cite[Table 1]{Seitz2}, but since $a_3=1$ it is easy to check that there are no such examples. This final contradiction proves that $V_1/[V_1,Q_X]$ and $V_2/[V_2,Q_X]$ are non-isomorphic $KL_X'$-modules (with respect to the parabolic subgroup $P_X=Q_XL_X$ with $\Delta(L_X') = \{\b_2,\b_3,\b_4\}$).

To complete the proof of the lemma we now return to the $t$-stable parabolic subgroup $P_X={Q_X}{L_X}$ with $\Delta(L_X')=\{\b_2,\b_3,\b_4\}$. Consider the triple $(L_X', L_{1}, M_1)$. Here $M_1$ is a $p$-restricted irreducible $KL_{1}$-module, and $M_1|_{L_X'}$ is the sum of two irreducible non-isomorphic modules interchanged by an involutory graph automorphism of  $L_X'$ (see \eqref{e:ee}). We also note that $Y_1|_{L_X'}$ is irreducible with highest weight $\delta_3|_{L_X'}$. Since $L_X'=A_3\cong D_3$ and $D_3<A_5$ is a geometric subgroup in the $\C_6$ collection, the main theorem of \cite{BGT} implies that $M_1=\Lambda^3(Y_1)$ and $p\neq 2$. Therefore $G=C_{10}$,  $\lambda=\lambda_3$ and \cite[Table A.39]{Lubeck} states that $\dim V=1100$ if $p=3$, otherwise $\dim V = 1120$. It now remains to determine $\mu_1$ and $\mu_2$. Since 
$\lambda=\lambda_3=3\lambda_1-2\alpha_1-\alpha_2$ and $\delta=\delta_3$, the root restrictions given in Table \ref{t:r12} imply that 
$$\lambda|_{X}=3\delta_3-\beta_2-2\beta_3=\delta_1+2\delta_4,$$ 
so $\mu_1=\delta_1+2\delta_4$ and $\mu_2=2\delta_2+\delta_5$. Now \cite[Table A.9]{Lubeck} gives $\dim V_1=440$ when $p=3$, otherwise $\dim V_1=560$. Since $1100>2\cdot 440$ we can rule out the case $p=3$. However, if $p \neq 2,3$ then $\dim V_1+\dim V_2=\dim V$ and it follows that $(G,H,V) \in \mathcal{T}$. This case is recorded in Tables \ref{tab:main2} and \ref{tab:main}.
\end{proof}

\begin{rmk}\label{r:ford}
Let $(G,H,V)$ be the irreducible triple arising in Lemma \ref{l:aa1}. Notice that the 
$KX$-composition factors of $V$ have $p$-restricted highest weights, but this example is missing from Ford's main theorem in \cite{Ford1}. The explanation for this mistake can be found by carefully
reading the argument in the penultimate paragraph of page 3888 in \cite{Ford2}.
Here Ford is considering the same set-up as we are, with the same $t$-stable parabolic subgroup $P_X$. He observes that $V/[V,Q]$ is the sum of two irreducible $KL_X'$-modules, interchanged by $t$, and then concludes that there are no such configurations, by induction. But as we have seen above, the 
embedding $L_X' = A_3\cong D_3<L_1= A_5$ is one of the geometric 
configurations giving rise to an irreducible triple.
\end{rmk}

\begin{lem}\label{l:aa2}
If $\delta=\delta_{(m+1)/2}$ and $m \ge 7$ then $(G,H,V) \not\in \mathcal{T}$.
\end{lem}

\begin{proof}
Seeking a contradiction, suppose $(G,H,V) \in \mathcal{T}$. By applying Proposition \ref{p:amlevels} and Lemma \ref{l:main}, we have $a_3=1$. Let $P_X=Q_XL_X$ be the parabolic subgroup of $X$ with $\Delta(L_X')=\{\b_2, \ldots, \b_{m-1}\}$ and let $P=QL$ be the corresponding parabolic subgroup of $G$. Let $W_i$ be the $i$-th $Q_X$-level of $W$ and note that $\ell=2$ is the level of the lowest weight $-\delta$. By Lemma \ref{l:levels2}(b) we have $L'=L_1L_2$, where each $L_i$ is simple with natural module $Y_i=W_{i-1}$ (more precisely, $L_1=A_k$ with $k=\binom{m-1}{(m-1)/2}-1$, and $L_2$ is a classical group). Write $V/[V,Q]=M_1\otimes M_2$, where each $M_i$ is a $p$-restricted irreducible $KL_i$-module. Now $Y_1|_{L_X'}$ is irreducible with highest weight $\delta_{(m+1)/2}|_{L_X'}$, and $Y_2|_{L_X'}$ has two composition factors of highest weights  $\delta_{(m-1)/2}|_{L_X'}$ and  $\delta_{(m+3)/2}|_{L_X'}$. (In particular, note that $\dim Y_1=\binom{m-1}{(m-1)/2}$ is even.) Therefore $L_X'$ acts nontrivially on both $Y_1$ and $Y_2$, so we may view $L_X'$ as a subgroup of both $L_1$ and $L_2$. By arguing as in the proof of Lemma \ref{l:aa1} we deduce that $M_1|_{L_X'}$ has exactly two composition factors, while $M_2|_{L_X'}$ is irreducible. We also deduce that $M_2$ is trivial and  $V_1/[V_1,Q_X]\not \cong V_2/[V_2,Q_X]$ as $KL_X'$-modules, so \eqref{e:ee} holds and $M_1|_{L_X'.2}$ is irreducible.

Since $Y_1|_{L_X'}$ is irreducible with symmetric highest weight $\delta_{(m+1)/2}|_{L_X'}$ (which is fixed by the graph automorphism of $L_X'$ induced by $t$), it follows that $L_X' $ fixes a non-degenerate form on $Y_1$. As $m\ge 7$, we have $L_X'<Cl(Y_1)<L_1$ where $Cl(Y_1)$ is a simple orthogonal or symplectic group. There are two cases to consider according to whether or not $M_1|_{Cl(Y_1)}$ is irreducible. (The rest of the proof is similar to the final argument in the proof of Lemma \ref{l:aim1}.)

First assume $M_1|_{Cl(Y_1)}$ is reducible. Since $M_1|_{L_X'}$ has exactly two composition factors, it follows that  $M_1|_{Cl(Y_1)}$ also has two composition factors, say $U_1$ and $U_2$. Consider $U_1$ and note that $U_1|_{L_X'}$ is irreducible. By Steinberg's tensor product theorem we may write  
$$U_1\cong S_1^{(q_1)}\otimes S_2^{(q_2)}\otimes \dots \otimes S_k^{(q_k)}$$
for some $k \ge 1$, where $k=q_1=1$ if $p=0$, otherwise each $S_j$ is a nontrivial $p$-restricted irreducible $KCl(Y_1)$-module, and the $q_j$ are certain powers of $p$. Since $a_3=1$, without loss of generality (relabelling if necessary) we may assume that there exists a $j$ such that the highest weight of $S_j$ as a $KCl(Y_1)$-module has non-zero coefficient on the third fundamental dominant weight (see Lemma \ref{l:nat}). Hence $S_j\neq Y_j, Y_j^*$ and the main theorem of \cite{Seitz2} implies that the triple $(L_X',Cl(Y_1),S_j)$ must appear in \cite[Table 1]{Seitz2}. (Note that each $S_j|_{L_X'}$ is irreducible since $U_1|_{L_X'}$ is irreducible. Also, if $S_j|_{Cl(Y_1)}$ is tensor decomposable then we can argue as in the proof of Lemma \ref{l:aim1} to reduce to a tensor indecomposable configuration of the same type.) However, there are no compatible examples, so this is a contradiction. 

Finally,  let us assume that $M_1|_{Cl(Y_1)}$ is irreducible. Then by 
\cite[Theorem 1]{Seitz2} the triple $(Cl(Y_1),L_1,M_1)$ must appear in \cite[Table 1]{Seitz2}. Since $a_3=1$, $m \ge 7$ and $Cl(Y_1)$ is of type $C$ or $D$ (recall that $\dim Y_1$ is even), we find that two possible configurations may arise (both with $p\neq 2$):
\begin{itemize}\addtolength{\itemsep}{0.3\baselineskip}
\item[(i)] ${\rm I}_1'$ with $(a,b,k)=(1,p-2,3)$ or $(p-2,1,2)$;
\item[(ii)] ${\rm I}_4$ with $k=3$. 
\end{itemize}
(Note that $\dim Y_1 \ge 20$ so $L_1 \neq A_7$ and the case labelled ${\rm I}_5$ (with $n=4$) in \cite[Table 1]{Seitz2} does not arise.)  In (ii), $M_1$ is the $KL_1$-module $\Lambda^3(Y_1)$, so Lemma \ref{l:wedge3} implies that $M_1|_{L_X'}$ has at least three distinct composition factors, a contradiction. Now consider (i). Here $Cl(Y_1)={\rm Sp}(Y_1)$ and thus $m\equiv 3 \imod{4}$ by Lemma \ref{l:embed}. Furthermore, by inspecting \cite[Table 1]{Seitz2} we observe that $M_1|_{Cl(Y_1)}$ is $p$-restricted. Also, recall that $M_1|_{L_X'.2}$ is irreducible since $V_1/[V_1,Q_X]\not \cong V_2/[V_2,Q_X]$ and \eqref{e:ee} holds. Therefore $(L_X'.2,Cl(Y_1),M_1)$ must be a triple satisfying Hypothesis \ref{h:our}. However, there are no compatible examples. Indeed, for $m=7$ this follows from Lemma \ref{l:aa1} (the only possibility is $\l = \l_3$, but this case is not compatible with the labelling given by ${\rm I}_{1}'$ in \cite[Table 1]{Seitz2}), while induction on $m$ gives the result for $m>7$.
\end{proof}

\vs

This completes the proof of Theorem \ref{t:am}.

\chapter{The case $H^0=D_m$, $m \ge 5$}\label{s:dm5}

Let us now turn to the case $H^0=X=D_m$. In this section we will assume $m \ge 5$, so 
$H=X\langle t \rangle$, where $t$ is an involutory graph automorphism of $X$. (The analysis of the special case $X=D_4$ is postponed to Section \ref{s:d4}.) We continue with the same notation: $T_X$ is a $t$-stable maximal torus of $X$, $\{\delta_1, \ldots, \delta_m\}$ are the fundamental dominant weights of $X$ corresponding to a fixed base $\Delta(X)=\{\b_1, \ldots, \b_m\}$ of the root system $\Phi(X)$, and $\delta = \sum_{i}b_i\delta_i$ is the highest weight of $W$ as a $KX$-module. Since $t$ acts on $W$, it follows that $\delta$ is fixed under the induced action of $t$ on the set of weights of $X$, so $b_{m-1} = b_{m}$. As before, since $\delta$ is $t$-invariant it follows that $X$ fixes a non-degenerate bilinear form on $W$ (see \cite[1.8]{Seitz2}), so $G = {\rm Sp}(W)$ or ${\rm SO}(W)$. Since $X$ is simply laced, the set of $T_X$-weights of $W$ is the same as the set of $T_X$-weights of the Weyl module $W_X(\delta)$ (see Lemma \ref{l:pr}). Also recall that $\delta$ is $p$-restricted as a $T_X$-weight. 

Let $V=V_G(\lambda)$ be a rational $p$-restricted irreducible tensor indecomposable $KG$-module with highest weight $\l=\sum_{i}a_i\l_i$, where $\{\l_1, \ldots, \l_n\}$ are the fundamental dominant weights of $G$ with respect to a base $\{\a_1, \ldots, \a_n\}$ of the root system $\Phi(G)$. As in Hypothesis \ref{h:our}, let us 
assume $V|_{H}$ is irreducible and $V|_{X}$ is reducible,
so 
$$V|_{X}=V_1 \oplus V_2,$$
where the $V_i$ are non-isomorphic irreducible $KX$-modules interchanged by $t$ (see Proposition \ref{p:niso}). Let $\mu_i$ denote the highest weight of $V_i$, say 
$$\mu_1 = \sum_{i=1}^mc_i\delta_i,\;\; \mu_2 = \sum_{i=1}^{m-2}c_i\delta_i+c_{m}\delta_{m-1}+c_{m-1}\delta_m.$$ 
Note that $c_{m-1} \neq c_m$ since the $V_i$ are non-isomorphic. Without loss of generality, we may assume that $\l|_{X} = \mu_1$. 

\section{The main result}

\begin{thm}\label{T:DM5}
There are no triples $(G,H,V)$ satisfying Hypothesis \ref{h:our} with $H^0=D_m$ and $m \ge 5$. 
\end{thm}

\section{Preliminaries}\label{s:dm5_prel}

Let $P_X=Q_XL_X$ be a parabolic subgroup of $X$ and let $P=QL$ be the parabolic subgroup of $G$ which stabilizes the flag 
$$W > [W,Q_X] > [W,Q_X^2] > \cdots > 0.$$
As before, let $W_i$ denote the $i$-th $Q_X$-level of $W$ and let $\ell$ be the $Q_X$-level of the lowest weight $-\delta$ of $W$. In addition, we set $\ell'=\lfloor \ell/2\rfloor$. We note that the next result also holds when $m=4$.

\begin{lem}\label{l:asym}
Let $P_X = Q_XL_X$ be the parabolic subgroup of $X$ with $L_X'=A_{m-1}$ and $\Delta(L_X') = \{\b_1, \ldots, \b_{m-1}\}$, and let $P=QL$ be the corresponding parabolic subgroup of $G$. Then $V/[V,Q]$ is an irreducible $KL_X'$-module.
\end{lem}

\begin{proof}
First observe that 
$$V/[V,Q_X] = V_1/[V_1,Q_X] \oplus V_2/[V_2,Q_X]$$
as $KL_X'$-modules, where the two $KL_X'$-modules on the right are irreducible and non-isomorphic (recall that $c_{m-1} \neq c_m$). Since $[V,Q_X] \leqs [V,Q]$ it follows that $V/[V,Q]$ is a quotient of $V/[V,Q_X]$.

Let $Z=Z(L_X)^0$ and note that $Z \leqs Z(L)$ by Lemma \ref{l:flag}(ii). It is straightforward to show that 
$$Z = \{h_{\b_1}(c^2)h_{\b_2}(c^4) \cdots h_{\b_{m-2}}(c^{2(m-2)})h_{\b_{m-1}}(c^{m-2})h_{\b_m}(c^m) \mid c \in K^*\}.$$
Since $V/[V,Q]$ is an irreducible $KL$-module (see Lemma \ref{l:vq}(i)), Schur's lemma implies that $Z$ acts by scalars on $V/[V,Q]$. Seeking a contradiction, let us assume $V/[V,Q]$ is reducible as a $KL_X'$-module, so
$$V/[V,Q] \cong V_1/[V_1,Q_X] \oplus V_2/[V_2,Q_X]$$
as $KL_X'$-modules.  
Then $\mu_1$ and $\mu_2$ each occur with non-zero multiplicity in $V/[V,Q]$, and we have 
$\mu_1|_Z=\mu_2|_Z$, that is 
$$2\sum_{i=1}^{m-2}ic_i + (m-2)c_{m-1} + mc_m= 2\sum_{i=1}^{m-2}ic_i + mc_{m-1} + (m-2)c_m.$$ 
This equation implies that $c_{m-1}= c_m$, which is a contradiction.
\end{proof}

\begin{lem}\label{l:embeddm}
Suppose  $\delta=\delta_{m-1}+\delta_m$ and $p\neq 2$. Then $G$ is an orthogonal group.
\end{lem}

\begin{proof}
This follows from Lemma \ref{l:st}. Set $h = \prod{h_{\b}(-1)}$, the product over the positive roots of $X$, and observe that if $\b = \sum_{i}d_i\b_i \in \Phi^+(X)$ then $\delta(h_{\b}(-1)) = (-1)^{d_{m-1}+d_{m}}$. By inspecting the list of roots in $\Phi^+(X)$ (see \cite[Plate IV]{Bou}, for example), we deduce that $\delta(h_{\b}(-1))=-1$ if and only if 
$$\b \in \left\{\sum_{k=i}^{m-1}\b_k,\; \b_m+\sum_{k=i}^{m-2}\b_k \mid 1 \le i \le m-1\right\}.$$
Therefore $\delta(h)=(-1)^{2(m-1)} =1$, so the fixed form on $W$ is symmetric by Lemma \ref{l:st} and thus $G$ is orthogonal.  
\end{proof}

The next result holds for all $m \ge 3$.

\begin{lem}\label{l:compfdm}
Let $N$ be the $KX$-module $\L^2(W)$, where $W=V_X(\delta)$ with $\delta=\delta_{m-1}+\delta_{m}$.
Then $N$ has at least three distinct $KX$-composition factors.
\end{lem}
 
\begin{proof}
This is similar to the proof of Lemma \ref{l:wedge2}. Let $w_1,w_2,w_3$ be non-zero vectors in $W$ of respective weights $\delta$, $\delta-\beta_{m-1}$ and $\delta-\beta_{m}$.  Then $w_1\wedge w_2$, $w_1\wedge w_3$ are maximal vectors of $N$ of respective weights $x=2\delta-\beta_{m-1} = \delta_{m-2}+2\delta_m$ and $y=2\delta-\beta_{m} = \delta_{m-2}+2\delta_{m-1}$. Let $N(x)$ and $N(y)$ be the composition factors of $N$ afforded by $x$ and $y$. As in the proof of Lemma \ref{l:wedge2}, to prove the lemma it suffices to show that 
\begin{equation}\label{e:mvz2}
m_{N(x)}(z)+m_{N(y)}(z)< m_N(z)
\end{equation}
for some weight $z$ of $N$.
 
First assume $p=2$. Let $z=2\delta-\beta_{m-1}-\beta_{m}$ and note that  
$$N_z=(W_{\delta}\wedge W_{\delta-\beta_{m-1}-\beta_{m}}) \oplus(W_{\delta-\beta_{m-1}}\wedge W_{\delta-\beta_{m}}),$$
so $m_N(z)=2$. Since $z=x-\beta_m=y-\beta_{m-1}$ it follows that $m_{N(x)}(z)=m_{N(y)}(z)=0$ and thus \eqref{e:mvz2} holds. 
  
Now assume $p \neq 2$. Consider the weight $z=2\delta-\beta_{m-2}-\beta_{m-1}-\beta_{m}$ in $N$. There are exactly three distinct ways to write $z=\omega_1 \wedge \omega_2$ with $\omega_i \in \L(W)$, namely 
$$\nu_1 = \delta\wedge \left(\delta-\beta_{m-2}-\beta_{m-1}-\beta_{m}\right),\; \nu_2 = \left(\delta- \beta_{m-1} \right) \wedge \left(\delta-\beta_{m-2}-\beta_m\right),$$
$$\nu_3 = \left(\delta-\beta_{m}\right) \wedge \left(\delta-\beta_{m-2}-\beta_{m-1}\right),$$
and using Lemma \ref{l:s816} we calculate that $m_N(\nu_1) = 3$ and $m_N(\nu_2) =  m_N(\nu_3) = 1$, whence $m_N(z) = 5$. Now $z = x-\beta_{m-2}-\beta_{m}$, so Lemma \ref{l:s816} yields $m_{N(x)}(z) = 2$ since $x=\delta_{m-2}+2\delta_m$. Also 
$m_{N(y)}(z) = 2$, by symmetry, so \eqref{e:mvz2} holds.
\end{proof}

\begin{lem}\label{l:new1}
We have 
$$\mu_2 - \mu_1 = \frac{1}{2}(c_{m}-c_{m-1})(\b_{m-1} - \b_m)$$
In particular, $\mu_2$ is not under $\mu_1$.
\end{lem}

\begin{proof}
From the information in \cite[Table 1]{Hu1} we calculate that
$$\mu_2 - \mu_1 = (c_{m}-c_{m-1})(\delta_{m-1} - \delta_m) = \frac{1}{2}(c_{m}-c_{m-1})(\b_{m-1} - \b_m)$$
and the result follows.
\end{proof}

\section{The Borel reduction}

 Here we determine some lower bounds on the dimensions of the $U_X$-levels of $W$, where $U_X$ is the unipotent radical of a $t$-stable Borel subgroup $B_X=U_XT_X$ of $X$. 
Throughout this section, let $W_i$ be the $i$-th $U_X$-level of $W$, and let $W_X(\delta)$ denote the corresponding Weyl module of $X$.
 
\begin{prop}\label{p:dm5_sub}
Suppose $\delta \not\in \{\delta_1, \delta_2, 2\delta_1\}$ and set $b = \sum_{i \in I}b_i$, where $I$ is the set of odd integers $i$ in the range $1 \le i \le m-1$. If $b$ is odd set $\mu=\delta_3$, if $b$ is even and $m \ge 6$ then set $\mu=\delta_4$, and if $b$ is even and $m=5$ then set $\mu = \delta_4+\delta_5$. Then $\mu$ is subdominant to $\delta$.
\end{prop}

\begin{proof}
First observe that 
\begin{equation}\label{e:di}
\delta_i = \sum_{j=1}^{i-1}j\b_j + i\sum_{j=i}^{m-2}\b_j + \frac{1}{2}i(\b_{m-1}+\b_m)
\end{equation}
if $i \le m-2$, and 
$$\delta_{m-1}+\delta_{m} = \sum_{j=1}^{m-2}j\b_j + \frac{1}{2}(m-1)(\b_{m-1}+\b_m)$$
(see \cite[Table 1]{Hu1}, for example). Therefore, $\delta-\delta_3=\sum_{i}d_i\beta_i$ where 
$$d_1=-1+\sum_{j=1}^{m-1}b_j,\;\; d_2=-2+b_1+2\sum_{j=2}^{m-1}b_j,\;\; d_{m-1}=d_{m}=-\frac{3}{2}+\frac{1}{2}\sum_{j=1}^{m-1}jb_j$$
and
$$d_i=-3+\sum_{j=1}^{i-1}jb_j+i\sum_{j=i}^{m-1}b_j$$
for all $3\leq i \leq m-2$. Similarly, if $m>5$ then $\delta-\delta_4=\sum_{i}e_i\beta_i$ where 
$$e_1=-1+\sum_{j=1}^{m-1}b_j, \;\; e_2=-2+b_1+2\sum_{j=2}^{m-1}b_j,\;\; e_3=-3+b_1+2b_2+3b_3+3\sum_{j=4}^{m-1}b_j,$$
$$e_{m-1}=e_{m}=-2+\frac{1}{2}\sum_{j=1}^{m-1}jb_j$$
and
$$ e_i=-4+\sum_{j=1}^{i-1}jb_j+i\sum_{j=i}^{m-1}b_j$$
for all $4\leq i \leq m-2$. Finally, if $m=5$ then 
\begin{eqnarray*}
\delta-(\delta_4+\delta_5)  & = & (b_1+b_2+b_3+b_4-1)\beta_1+(b_1+2(b_2+b_3+b_4-1))\beta_2 \\
  & & +(b_1+2b_2+3(b_3+b_4-1))\beta_3 \\
  & & +\frac{1}{2}(b_1+2b_2+3b_3+4b_4-4)(\beta_4+\beta_5).
\end{eqnarray*}
Now, if $b$ is odd it is easy to see that each $d_i$ is a non-negative integer (note that $\delta \neq  \delta_1$). Similarly, if $b$ is even and $m>5$ then each $e_i$ is a non-negative integer, and the same is true for the coefficients in the above expression for $\delta-(\delta_4+\delta_5)$ when $m=5$. The result follows. 
\end{proof}

\begin{rmk}\label{r:lwddi}
Using \eqref{e:di}, we deduce that if $\delta = \delta_i$ with $i \le m-2$ then the $U_X$-level of the lowest weight $-\delta$ of $W$ is $\ell = i(2m-i-1)$.
\end{rmk}

\begin{lem}\label{l:delta2}
Let $\delta=\delta_2$. Then $\ell = 4m-6$ and $W_i$ contains at least $3$ distinct weights for all $4 \le i < \ell/2$. Moreover, $W_{\ell/2}$ is the zero-weight space and
$$\dim W_{\ell/2} = \left\{\begin{array}{ll} m & \mbox{if $p \neq 2$} \\
m-1 & \mbox{if $p=2$ and $m$ is odd} \\
m-2 & \mbox{if $p=2$ and $m$ is even.}
\end{array}\right.$$
\end{lem}

\begin{proof}
By  Remark \ref{r:lwddi}, $\ell=4m-6$ as claimed. 

Suppose $4 \le i < \ell/2$. By Lemma \ref{l:pr}, it suffices to show that the $i$-th $U_X$-level of the Weyl module $W_X(\delta)$ contains at least $3$ distinct weights. First observe that $\L(W_X(\delta))=\L(\mathcal{L}(X))$ where $\mathcal{L}(X)$ is the Lie algebra  of $X$. In particular, there is a bijection between the non-zero weights in $W_X(\delta)$ and the set of roots $\Phi(X)$. More precisely, the weights at level $i$ in $W_X(\delta)$ correspond to the roots in $\Phi(X)$ of height $\ell/2-i$. As explained in \cite[p.83]{Hu2}, the number of roots in $\Phi(X)$ of a given (positive) height is closely related to the degrees of the invariant polynomials for the Weyl group of $X$. Indeed, if $\partial = \{d_i \mid 1 \le i \le m\}$ is the set of such degrees then $\rho = (d_1 -1, \ldots, d_m-1)$ is a partition of $|\Phi^+(X)|$, and the number of roots in $\Phi^+(X)$ of height $i$, which we denote by $k_i$,  is equal to the number of parts in the partition $\rho$ which are greater than or equal to $i$. By \cite[Table 3.1]{Hu2}, the set of degrees for $X = D_m$ is
$$\partial = \{2,4,6, \ldots, 2m-2, m\}$$
and the desired result quickly follows. 

For example, if $m=5$ then $\ell = 14$, $|\Phi^+(X)| = 20$ and $\partial = \{2,4,6,8,5\}$, so $\rho = (1,3,5,7,4)$ and thus
$$\begin{array}{llllllll} \hline
i & 1 & 2 & 3 & 4 & 5 & 6 & 7 \\ \hline
k_i & 5 & 4 & 4 & 3 & 2 & 1 & 1 \\ \hline
\end{array}$$
Therefore $W_4$ contains $k_{\ell/2-4} = k_3 = 4$ distinct weights, and so on. Finally, it is clear that $W_{\ell/2}$ coincides with the zero-weight space of $W$, so
$\dim W_{\ell/2} = \dim W - |\Phi(X)|$ and the result follows since \cite[Theorem 5.1]{Lubeck} states that
$\dim W = \dim X$ if $p \neq 2$, otherwise $\dim W = \dim X - {\rm gcd}(2,m)$.
\end{proof}

\begin{rmk}
According to Lemma \ref{l:delta2}, if $\delta = \delta_2$ then $\dim W_{\ell/2} \ge 5$ unless $(m,p) = (5,2)$ or $(6,2)$. Note that if $(m,p) = (6,2)$ then $G={\rm Sp}(W)$ (see \cite[Table 2]{Brundan}, for example).
\end{rmk}

\begin{lem}\label{l:delta21}
Let $\delta=2\delta_1$ (so $p>2$). Then $\ell = 4m-4$, $\dim W_{\ell/2} \ge 5$ and $W_i$ contains at least $3$ distinct weights for all $4 \le i < \ell/2$. 
\end{lem}

\begin{proof}
Here $W=V_{X}(\delta)$ is the unique nontrivial composition factor of the symmetric-square $S^2(U)$, where $U$ is the natural $KX$-module.  The cases $m=5,6$  can be checked by direct calculation, so let us assume $m \ge 7$. By \cite[Theorem 5.1]{Lubeck}, $\dim W = \dim S^2(U) - \e$, where $\e = 2$ if $p$ divides $m$, otherwise $\e=1$. Now $\delta - \b_1 = \delta_2$ and so by applying Lemmas \ref{l:indwt} and \ref{l:delta2} we deduce that $W_i$ has at least $3$ distinct weights for all $5 \le i < \ell/2$, while $\delta - 2(\b_1+\b_2)$, $\delta - 2\b_1-\b_2-\b_3$ and $\delta - \b_1 - \b_2 - \b_3 - \b_4$ are distinct weights in $W_4$. Finally, we note that the zero-weight has multiplicity $m-\e$ in $W_{\ell/2}$, whence $\dim W_{\ell/2}\geq 5$ as required.
\end{proof}

\begin{lem}\label{l:delta3}
Let $\delta=\delta_3$. Then $\ell = 6m-12$ and $W_i$ contains at least $3$ distinct weights for all $4 \le i < \ell/2$. Further, there are at least $5$ distinct weights in $W_{\ell/2}$.
\end{lem}

\begin{proof}
By direct calculation 
it is easy to check that the result holds when $m=5$ or $6$, so we will assume $m \ge 7$. Let $P_X=Q_XL_X$ be the parabolic subgroup of $X$ with $L_X'=D_{m-1}$ and $\Delta(L_X') = \{\b_2, \ldots, \b_{m}\}$. By considering the weights in the $KL_X'$-module $W_{L_X'}(\delta_3)$ (where $\{\delta_2, \ldots, \delta_{m}\}$ are the fundamental dominant weights of $L_X'$), and by applying Lemma \ref{l:delta2}, we deduce that $W_i$ has at least $3$ weights for all $4 \le i \le 2(m-1)-4 = 2m-6$. Therefore, to complete the proof of the lemma we may assume $2m-5 \le i \le 3m-6$.

To deal with the remaining $U_X$-levels we use Lemma \ref{l:pr} and the fact that $\L(W_X(\delta)) = \L(\L^3(U))$, where $U$ is the natural $KX$-module. (Indeed, $\L^3(U)$ is an irreducible $KA_{2m-1}$-module, and the restriction to a subgroup $D_m<A_{2m-1}$ remains irreducible when $p \neq 2$ -- see Case ${\rm I}_{4}$ in \cite[Table 1]{Seitz2}.  If $p=2$ then $\L^3(U)$ has two composition factors, namely $V_X(\delta_3)$ and $V_X(\delta_1)$; the same is true for $W_X(\delta)$, so once again we deduce that $\L(W_X(\delta)) = \L(\L^3(U))$.) For example, suppose $i = 3m-6 = \ell/2$. Then $W_{i}$ contains the following five distinct weights $\eta_1 \wedge \eta_2 \wedge \eta_3$:
\renewcommand{\arraystretch}{1.2}
$$\begin{array}{llll} \hline
 & \eta_1 & \eta_2 & \eta_3 \\ \hline
\xi_1 & \delta_1 - \b_1 - \cdots - \b_{m-1} & \delta_1 - \b_1 - \cdots - \b_{m-2} - \b_m & \delta_1 - \b_1 - \cdots - \b_{m-4} \\
\xi_2 &  \delta_1 - \b_1 - \cdots - \b_m &  \delta_1 - \b_1 - \cdots - \b_{m-2} & \delta_1 - \b_1 - \cdots - \b_{m-4} \\
\xi_3 &  \delta_1 - \b_1 - \cdots - \b_m & \delta_1 - \b_1 - \cdots - \b_{m-1} & \delta_1 - \b_1 - \cdots - \b_{m-5} \\
\xi_4 &   \delta_1 - \gamma_1 &  \delta_1 - \b_1 &  \delta_1 - \b_1 - \cdots - \b_{m-5} \\
\xi_5 &  \delta_1 - \gamma_2 &  \delta_1 - \b_1 &  \delta_1 - \b_1 - \cdots - \b_{m-4} \\ \hline
\end{array}$$
\renewcommand{\arraystretch}{1}
where
$$\gamma_1 = 2(\b_1 + \cdots + \b_{m-2})+\b_{m-1}+\b_m,\;\; \gamma_2 = \b_1 + 2(\b_2 + \cdots + \b_{m-2})+\b_{m-1}+\b_m.$$

It is easy to suitably modify the above weights to see that there are at least $3$ weights in $W_i$ for all $2m-5 \le i \le 3m-5$. For example, if $i = 2m-5$ then by modifying $\xi_1$, $\xi_2$ and $\xi_3$ above we obtain the following weights $\eta_1' \wedge \eta_2' \wedge \eta_3'$ in $W_{i}$:
\renewcommand{\arraystretch}{1.2}
$$\begin{array}{lll} \hline
  \eta_1' & \eta_2' & \eta_3' \\ \hline
 \delta_1  & \delta_1 - \b_1 - \cdots - \b_{m-2} - \b_m & \delta_1 - \b_1 - \cdots - \b_{m-4} \\
  \delta_1 - \b_1  &  \delta_1 - \b_1 - \cdots - \b_{m-2} & \delta_1 - \b_1 - \cdots - \b_{m-4} \\
  \delta_1 - \b_1 & \delta_1 - \b_1 - \cdots - \b_{m-1} & \delta_1 - \b_1 - \cdots - \b_{m-5} \\ \hline
\end{array}$$
\renewcommand{\arraystretch}{1}
The remaining cases are very similar, and we leave the reader to check the details. 
\end{proof}

\begin{lem}\label{l:delta4}
Suppose $m \ge 6$ and $\delta=\delta_4$. Then $\ell = 8m-20$ and $W_i$ contains at least $3$ distinct weights for all $4 \le i < \ell/2$. Further, there are at least $5$ distinct weights in $W_{\ell/2}$.
\end{lem}

\begin{proof}
The cases $m=6,7$ can be checked by direct calculation, so let us assume $m \ge 8$. 
By considering the $D_{m-1}$ Levi subgroup of $X$, and applying Lemma \ref{l:delta3}, we deduce that $W_i$ has sufficiently many weights for all $4 \le i \le 3(m-1)-7 = 3m-10$, so we may assume $3m-9 \le i \le \ell/2$. 

Next observe that $\delta - \b = \delta_2$, where 
$$\b = \b_3 + 2(\b_4 +\cdots + \b_{m-2})+\b_{m-1}+\b_m$$
is a root of height $2m-7$. By Lemmas \ref{l:indwt} and \ref{l:delta2}, it follows that $W_i$ has enough weights when $(2m-7)+4 \le i < \ell/2$, so it remains to deal with the middle level $W_{\ell/2}$. By arguing as in the proof of Lemma \ref{l:delta3}, it is easy to see that $\L(W_X(\delta)) = \L(\L^4(U))$, where $U$ is the natural $KX$-module. As a consequence, it is straightforward to find $5$ distinct weights $\xi_1, \ldots, \xi_5$ (of the form $\eta_1 \wedge \eta_2 \wedge \eta_3 \wedge \eta_4$) at level $\ell/2$. Indeed, the following weights
\renewcommand{\arraystretch}{1.2}
$$\begin{array}{lllll} \hline
 & \eta_1 & \eta_2 & \eta_3 & \eta_4 \\ \hline
\xi_1 & \delta_1 - \sum_{i=1}^{m-1}\b_i & \delta_1 - \b_1 - \cdots  - \b_{m-2} - \b_m & \delta_1 - \sum_{i=1}^{m-2}\b_i & \eta \\

\xi_2 & \delta_1 - \gamma_1 & \delta_1 - \b_1 - \cdots  - \b_{m-3} & \delta_1 - \b_1 & \eta \\

\xi_3 & \delta_1 - \gamma_2 & \delta_1 - \b_1 - \cdots  - \b_{m-2} & \delta_1 - \b_1 & \eta  \\

\xi_4 & \delta_1 - \gamma_3 & \delta_1 - \b_1 - \cdots  - \b_{m-1} & \delta_1 - \b_1 & \eta  \\

\xi_5 & \delta_1 - \gamma_3 & \delta_1 - \b_1 - \cdots  - \b_{m-2} - \b_m  & \delta_1 - \b_1 & \eta \\ \hline
\end{array}$$
\renewcommand{\arraystretch}{1}
are in $W_{\ell/2}$, where $\eta =  \delta_1 - \b_1 - \cdots  - \b_{m-6}$, $\gamma_1$ and $\gamma_2$ are defined as in the proof of Lemma \ref{l:delta3}, and
$\gamma_3 = \b_1+\b_2 + 2(\b_3 + \cdots + \b_{m-2}) + \b_{m-1}+\b_m$.
\end{proof}

\begin{lem}\label{l:delta45}
Suppose $m=5$ and $\delta=\delta_4+\delta_5$. Then $\ell=20$, $\dim W_1=2$ and $W_i$ contains at least $3$ distinct weights for all $2 \leq i<10$. Moreover, there are at least $5$ distinct weights in $W_{10}$.
\end{lem}

\begin{proof}
This is a straightforward calculation.
\end{proof}

\begin{lem}\label{l:deltam}
Let $\delta=\delta_i$, where $3 \le i \le m-2$. Then $\ell = i(2m -i-1)$ and $W_j$ contains at least $3$ distinct weights for all $4 \le j < \ell/2$. Further, there are at least $5$ distinct weights in $W_{\ell/2}$.
\end{lem}

\begin{proof}
We proceed by induction on $i$. The cases $i=3$ and $4$ were dealt with in Lemmas \ref{l:delta3} and \ref{l:delta4}, respectively, so let us assume $i>4$. In particular, $m \ge 7$. Now $\delta - \b = \delta_{i-2}$, where
$$\beta = \b_{i-1}+2(\b_i + \cdots + \b_{m-2})+\b_{m-1}+\b_m$$
is a root of height $2(m-i)+1$. By induction (and Lemma \ref{l:indwt}), it follows that $W_j$ has at least $3$ distinct weights for all $j \ge 2(m-i)+5$ (and also $W_{\ell/2}$ has at least $5$ weights), so we may assume $4 \le  j \le 2(m-i)+4$. 

Let $P_X = Q_XL_X$ be the parabolic subgroup of $X$ with $\Delta(L_X') = \{\b_{i-2}, \ldots, \b_m\}$, so $L_X' = D_{m-i+3}$. By considering the weights in the Weyl module $W_{L_X'}(\delta_i)$ (where we label the fundamental dominant weights for $L_X'$ by $\{\delta_{i-2}, \delta_{i-1}, \ldots, \delta_m\}$), and applying Lemma \ref{l:delta3}, we deduce that $W_j$ has at least $3$ distinct weights for all $4 \le j \le 3(m-i+3)-6 = 3(m-i)+3$. Since $3(m-i)+3 \ge 2(m-i)+4$, the result follows.
\end{proof}

\begin{prop}\label{p:ind}
Let $\Omega = \{\delta_1, \delta_2, 2\delta_1\}$ and suppose $\delta = \sum_{i}b_i\delta_i \not\in \Omega$. Then $W_j$ contains at least $3$ distinct weights for all $4 \le j < \ell/2$. Moreover, if $\ell$ is even then there are at least $5$ distinct weights in $W_{\ell/2}$.
\end{prop}

\begin{proof}
Define $b$ and $\mu$ as in the statement of Proposition \ref{p:dm5_sub}, so $\mu$ is subdominant to $\delta$. We proceed by induction on the height ${\rm ht}(\delta-\mu)$ of $\delta-\mu$ (recall that ${\rm ht}(\delta-\mu)= \sum_{i}d_i$ where $\delta - \mu = \sum_{i}d_i\b_i$). If ${\rm ht}(\delta-\mu)=0$ then $\delta=\mu$ and the result follows from Lemmas 
\ref{l:delta3}, \ref{l:delta4} and \ref{l:delta45}. Therefore, we may assume that ${\rm ht}(\delta-\mu)>0$. We partition the analysis into a number of separate cases. We say that a dominant weight $\sum_{i}d_i\delta_i$ is \emph{symmetric} if $d_{m-1}=d_m$.

\vs

\noindent \emph{Case 1.} $\delta = \delta_i$ for some $3 \le i \le m-2$. 

\vs

Here the result follows from Lemma \ref{l:deltam}.

\vs

\noindent \emph{Case 2.} $b_i \ge 2$ for some $1 \le i \le m-2$. 

\vs

Let $\nu = \delta - \b_i$ and recall that $\delta \neq 2\delta_1$. Then $\nu \not\in \Omega$ is symmetric and  subdominant to $\delta$. Moreover, Proposition \ref{p:dm5_sub} implies that $\mu$ is subdominant to $\nu$ and ${\rm ht}(\nu-\mu)<{\rm ht}(\delta-\mu)$, so by induction and Lemma \ref{l:indwt} we see that the desired result holds for all $j \ge 5$. Finally, it is straightforward to check that $W_4$ has at least $3$ weights. For example, if $i=2$ then $\delta - 2(\b_1+\b_2)$, $\delta - 2(\b_2+\b_3)$ and $\delta - \b_1 - 2\b_2-\b_3$ are distinct weights in $W_4$.

\vs

\noindent \emph{Case 3.} $b_{m-1} = b_m \ge 2$. 

\vs

Let $\nu = \delta - \b_{m-2} - \b_{m-1} - \b_m$. Then $\nu \not\in \Omega$ is symmetric and subdominant to $\delta$. Moreover, $\mu$ is subdominant to $\nu$ and ${\rm ht}(\nu-\mu)<{\rm ht}(\delta-\mu)$, so by induction it remains to check that $W_4$, $W_5$ and $W_6$ each contain at least $3$ distinct weights. This is entirely straightforward:
$$\begin{array}{ll} 
W_4: & \delta - 2(\b_{m-1} +\b_m), \; \delta - \b_{m-2} - 2\b_{m-1} - \b_m, \; \delta - \b_{m-2}-\b_{m-1}-2\b_m \\
W_5: & \delta - \b_{m-2} - 2(\b_{m-1} +\b_m), \; \delta - 2(\b_{m-2} +\b_{m-1}) - \b_m, \\
& \delta - 2(\b_{m-2}+\b_{m})-\b_{m-1} \\
W_6: & \delta - 2(\b_{m-2}+\b_{m-1} +\b_m), \; \delta - \b_{m-3} - \b_{m-2} - 2(\b_{m-1}+\b_m), \\
& \delta - \b_{m-4}- \b_{m-3} - \b_{m-2}- \b_{m-1}-2\b_m \\
\end{array}$$

\vs

\noindent \emph{Case 4.} $b_{i}b_k \neq 0$ for some $1 \le i<k \le m-2$.

\vs

Choose $i,k$ so that $k-i$ is minimal. By Case 2, we can now assume that $b_i = b_k = 1$. Let $\nu = \delta - \b_i - \b_{i+1} - \cdots - \b_k$. This is a symmetric subdominant weight to $\delta$, and $\la \mu,\b_{k+1}\ra  \neq 0$ so $\mu \not\in \Omega$. Also observe that $\mu$ is subdominant to $\nu$ and ${\rm ht}(\nu-\mu)<{\rm ht}(\delta-\mu)$, so by induction and  Lemma \ref{l:indwt} we deduce that $W_j$ has sufficiently many weights for all $j \ge 4+(k-i+1) = 5+k-i$. For the remainder we may assume 
$4 \le j \le 4+k-i$. 

Consider the $A_{k-i+1}$ Levi subgroup of $X$ with base $\{\beta_i,\dots,\beta_k\}$. For clarity we will assume $i=1$ (a similar argument applies equally well if $i>1$), so we need to check the levels $W_4,\ldots,W_{3+k}$. Note that $b_j=0$ for all $2\leq j \leq k-1$, by the  minimality of $k-i$, so  $\delta|_{A_k}=(\delta_1+\delta_k)|_{A_k}$. If $k\geq 5$ then the proof of Lemma \ref{l:am1m} implies that there are at least 3 distinct weights in $W_j$ for all $4 \leq j \leq 3+k$, $j \neq k$, and it is easy to see that there are also at least 3 distinct weights in $W_k$; for example 
$$\delta-\sum_{j=1}^k\beta_j, \;\; \delta-\sum_{j=2}^{k+1}\beta_j, \;\;
\delta-\sum_{j=3}^{k+2}\beta_j.$$
It remains to show that there are at least 3 distinct weights in $W_4, \ldots, W_{3+k}$ when $k \in\{2,3,4\}$. This is an easy check (note that $m \ge 6$ if $k=4$).  

Notice that we have now reduced to the case $\delta = b_i\delta_i + \delta_{m-1} + \delta_m$, with $1 \le i \le m-2$ and $b_i \le 1$.

\vs

\noindent \emph{Case 5.} $\delta=b_i\delta_i+\delta_{m-1}+\delta_m$ where $1\leq i\leq m-2$ and $b_i\leq 1$. 

\vs

Let $\nu = \delta - \b_{m-2} - \b_{m-1} - \b_m$. By Lemma \ref{l:delta45} we can assume $m>5$ if $\delta=\delta_{m-1}+\delta_m$. Then $\nu \not\in \Omega$ is symmetric and subdominant to $\delta$. Further, $\mu$ is subdominant to $\nu$ and ${\rm ht}(\nu-\mu)<{\rm ht}(\delta-\mu)$, so in the usual way we deduce that $W_j$ has the desired number of distinct weights for all $j \ge 7$. The remaining levels $W_4, W_5$ and $W_6$ are straightforward to check directly. This completes the proof of the proposition.
\end{proof}

\begin{lem}\label{l:123}
Suppose $m \ge 5$ and $\delta$ is one of the following:
$$\delta_1+\delta_2, \; 2\delta_1, \; \delta_{m-1}+\delta_{m}, \; \delta_i \; (2 \le i \le m-2).$$
Then the dimensions of the $U_X$-levels $W_1$, $W_2$ and $W_3$ are as follows:
\renewcommand{\arraystretch}{1.2}
$$\begin{array}{ccccccc} \hline
 & \delta_2 & \delta_i, \, 3 \le i \le m-3 & \delta_{m-2} & 2\delta_1 & \delta_{m-1}+\delta_{m} & \delta_1+\delta_2 \\ \hline
\dim W_1 & 1 & 1 & 1 & 1 & 2 & 2 \\
\dim W_2  & 2 & 2 & 3 & 2 & 3 & 3-\delta_{3,p} \\ 
\dim W_3  & 2+\delta_{5,m} & 3+\delta_{i,m-3} & 4 & 2 & 5 - \delta_{2,p} & \geq 3\\ \hline 
\end{array}$$
\renewcommand{\arraystretch}{1}
\end{lem}

\begin{proof}
This is an easy calculation.
\end{proof}

The next result summarises much of the above discussion on $U_X$-levels.

\begin{prop}\label{p:main}
Let $\delta = \sum_ib_i\delta_i$ be a nontrivial weight such that $b_{m-1}=b_m$ and $\delta \neq \delta_1$. Let $W_i$ denote the $i$-th $U_X$-level of $W$ and let $\ell$ be the $U_X$-level of the lowest weight $-\delta$. Then exactly one of the following holds:
\begin{itemize}\addtolength{\itemsep}{0.3\baselineskip}
\item[{\rm (I)}] $\dim W_i \neq 2$ for all $i$. Moreover, if $\ell$ is even then $\dim W_{\ell/2} \ge 5$;
\item[{\rm (II)}] Either $\delta=b_{m-1}(\delta_{m-1}+\delta_m)$ with $b_{m-1}\geq 2$, or $\delta=b_i\delta_i+b_k\delta_k$ for some $1\leq i<k\leq m-2$, where $b_ib_k \neq 0$ and $\delta \neq \delta_1+\delta_2$ if $p=3$. Here  
$\dim W_1 = 2$ and $\dim W_j \ge 3$ for all $2 \le j < \ell/2$. In addition, if $\ell$ is even then $\dim W_{\ell/2} \ge 5$;
\item[{\rm (III)}] $\delta=b_1\delta_1$ with $b_1\geq 3$. Here $\dim W_1 = 1$, $\dim W_2 = 2$ and $\dim W_i \ge 3$ for all $3 \le i < \ell/2$. Also, if $\ell$ is even then $\dim W_{\ell/2} \ge 5$;
\item[{\rm (IV)}] $\delta \in \{\delta_i \; (2 \le i \le m-3), \; 2\delta_1, \; \delta_{m-1}+\delta_m, \; \delta_1+\delta_2 \; (p=3)\}$.
\end{itemize}
\end{prop}

\begin{proof}
We may assume $\delta$ is not one of the weights listed in Case (IV). (Note that we allow $\delta = \delta_1+\delta_2$ when $p\neq 3$.) Then in view of Proposition \ref{p:ind}, we need only consider $W_1$, $W_2$ and $W_3$.  In particular, note that Case (I) applies if 
$\delta = \delta_{m-2}$ (see Lemma \ref{l:123}). Now $\dim W_1$ is equal to the number of non-zero coefficients $b_i$ in the expression for $\delta$. In particular, if $\dim W_1 = 2$ then we may assume $\delta = b_{m-1}(\delta_{m-1}+\delta_m)$ with $b_{m-1} \ge 2$, or 
$\delta = b_i\delta_i + b_k \delta_k$ for some $1 \le i < k \le m-2$ with $b_ib_k \neq 0$. In both cases it is easy to check that  $\dim W_2\geq 3$  (here we require $p\neq 3$ if $\delta=\delta_1+\delta_2$) and $\dim W_3 \ge 3$.

Next suppose $\dim W_2 = 2$, so at most two of the $b_i$ are non-zero. By the above analysis of $W_1$ it follows that $\delta = b_i\delta_i$ for some $b_i \ge 2$ and $1 \le i \le m-2$ (note that $(b_i,i) \neq (2,1)$ since $\delta = 2\delta_1$ is one of the weights in Case (IV)). However, if $i=m-2$ then we calculate that 
$\dim W_1 = 1$ and $\dim W_j \ge 4$ for $j=2,3$, so this case does not arise. Similarly, if $1<i < m-2$ then $\dim W_2 \geq 3$ and so we can also discard these cases. Finally, if $i=1$ then $\dim W_1=1$, $\dim W_2=2$ and $\dim W_3\geq 3$ (see Case (III)). 
 
Finally, suppose $\dim W_3 = 2$. By the above analysis we may assume that $\delta$ has at least three non-zero coefficients. Then it is straightforward to verify that $\dim W_3 \ge 3$ and so this possibility  does not arise.
\end{proof}

\section{Proof of Theorem \ref{T:DM5}}

We are now ready to begin the proof of Theorem \ref{T:DM5}. 
By combining Proposition \ref{p:main} and Lemma \ref{l:main}, we may assume that we are in one of the cases labelled (II) -- (IV) in the statement of Proposition \ref{p:main}. We consider each of these cases in turn. As before, let $\mathcal{T}$ be the set of triples $(G,H,V)$ with $H^0=X = D_m$ and $m \ge 5$, satisfying Hypothesis \ref{h:our}.

\begin{rmk}\label{r:p}
Let $P_X = Q_XL_X$ be the parabolic subgroup of $X$ with $L_X'=A_{m-1}$ and $\Delta(L_X') = \{\b_1, \ldots, \b_{m-1}\}$ and let $P=QL$ be the corresponding parabolic subgroup of $G$, constructed in the usual manner from the $Q_X$-levels of $W$. Let $W_i$ denote the $i$-th $Q_X$-level of $W$. Write $L' = L_1\cdots L_r$, where each $L_i$ is simple, and let $Y_i$ be the natural $KL_i$-module. (It is easy to check that $W_0$ and $W_1$ both contain at least two distinct weights, so $Y_1=W_0$ and $Y_2=W_1$.) Since $V/[V,Q]$ is a nontrivial irreducible $KL'$-module we have
$$V/[V,Q] = M_1 \otimes \cdots \otimes M_r,$$
where each $M_i$ is a $p$-restricted irreducible $KL_i$-module. By Lemma \ref{l:asym}, $V/[V,Q]$ is an irreducible $KL_X'$-module. 
\end{rmk}

\begin{lem}\label{l:case2}
Suppose $(G,H,V) \in \mathcal{T}$ and $\delta = b_i\delta_i+b_k\delta_k$ with $1\leq i<k\leq m$ and $b_ib_k \neq 0$. Then either $\delta=\delta_{m-1}+\delta_{m}$, or $\delta=\delta_1+\delta_2$ and $p=3$. 
\end{lem}

\begin{proof} 
Seeking a contradiction, suppose that $\delta \neq \delta_{m-1}+\delta_{m}$ nor $\delta_1+\delta_2$ (if $p=3$). Then $\delta$ satisfies the conditions in Case (II) of Proposition \ref{p:main}, so Lemma \ref{l:main} implies that $a_2=1$ (where $V=V_G(\l)$ and $\l=\sum_{i=1}^{n}a_i\l_i$, as usual). Let $P_X = Q_XL_X$ be the parabolic subgroup of $X$ in Remark \ref{r:p} and define the $W_i$, $L_i$, $Y_i$ and $M_i$ as in that remark. Note that $Y_1 = W_0 \cong W/[W,Q_X]$ is an irreducible $KL_X'$-module with highest weight $(b_1\delta_1 + \cdots + b_{m-1}\delta_{m-1})|_{L_X'}$. Since $\delta \neq 0$ and $b_{m-1}=b_m$ it follows that $Y_1$ is a nontrivial $KL_X'$-module and thus $\dim Y_1 \ge m$ since $L_X'=A_{m-1}$. In addition, we note that $\pi_1(L_X') = A_{m-1}$, where $\pi_1:L_X'\to L_1$ is the projection map, so we may view $L_X'$ as a subgroup of $L_1$.

First assume $\dim Y_1 > m$. Then $Y_1|_{L_X'}$ is not the natural $KL_X'$-module (nor its dual) and we have a configuration $(L_X',L_1,M_1)$, where $M_1|_{L_X'}$ is irreducible (since $V/[V,Q]$ is an irreducible $KL_X'$-module). Since $M_1$ is nontrivial, and $M_1 \neq Y_1$ nor $Y_1^*$ (recall that $a_2=1$), the main theorem of \cite{Seitz2} implies that this must be one of the cases listed in \cite[Table 1]{Seitz2}. However, the only possibilities which arise are the cases labelled ${\rm I}_{6}$ and ${\rm I}_{7}$, but neither are compatible with the highest weight of $Y_1|_{L_X'}$. 
This is a contradiction.

Finally, suppose $\dim Y_1 = m$. Let $U$ be the natural $KL_X'$-module. Since $Y_1$ is an irreducible $KL_X'$-module it follows that $Y_1|_{L_X'}=U$ or $U^*$. If $Y_1|_{L_X'}=U$ then $b_1 = 1$ and $b_i = 0$ for all $1 < i \le m-1$, whence $b_m = 0$ by symmetry and thus $\delta = \delta_1$, which is a contradiction. Similarly, if $Y_1|_{L_X'}=U^*$ then $\delta = \delta_{m-1}+\delta_m$ is the only possibility, which contradicts our initial hypothesis. 
\end{proof}

\begin{lem}\label{l:case3}
Suppose $(G,H,V) \in \mathcal{T}$ and $\delta = b_1\delta_1$. Then $b_1 \le 2$.
\end{lem}

\begin{proof}
Seeking a contradiction, assume $b_1 \ge 3$. In view of Case (III) in Proposition \ref{p:main}, Lemma 
\ref{l:main} implies that $a_3=1$. Let $P_X = Q_XL_X$ be the parabolic subgroup of $X$ in Remark \ref{r:p} and define the $W_i$, $L_i$, $Y_i$ and $M_i$ as before. As in the previous lemma, $Y_1|_{L_X'}$ is irreducible with highest weight $(b_1\delta_1 + \cdots + b_{m-1}\delta_{m-1})|_{L_X'}$, and thus $\dim Y_1 > m$ (since $b_1 \ge 3$). We now finish the proof as we did in Lemma \ref{l:case2}.
\end{proof}

To complete the proof of Theorem \ref{T:DM5}, it remains to consider the specific weights  listed in Case (IV) of Proposition \ref{p:main}.

\begin{lem}\label{l:case41}
Suppose $\delta = \delta_i$ with $2 \le i \le m-3$. Then $(G,H,V) \not\in \mathcal{T}$. 
\end{lem}

\begin{proof}
Seeking a contradiction, let us assume $(G,H,V) \in \mathcal{T}$. Let $U_j$ denote the $j$-th $U_X$-level of $W$, where $B_X = U_XT_X$ is a $t$-stable Borel subgroup of $X$ containing the $t$-stable maximal torus $T_X$. As in Lemma \ref{l:deltam}, $\ell = i(2m-i-1)$ is the $U_X$-level of the lowest weight $-\delta$. We claim that either $a_3=1$, or $m \geq 6$, $\delta=\delta_2$ and $1 \in \{a_3,a_5\}$. 

For $i \ge 3$ the claim follows by combining Lemma \ref{l:main} with the bounds on $\dim  U_j$ given in  Proposition \ref{p:ind} and Lemma \ref{l:123}. 
Now suppose $\delta=\delta_2$, so $\ell=4m-6$.  First assume $m=5$. Here $\dim U_1 = 1$, $\dim U_2 = 2$, $\dim U_j \ge 3$ for all $3 \le j \le 6$, and $\dim U_7 = 5-\delta_{2,p}$ (see Lemma \ref{l:delta2}), so Lemma \ref{l:main} yields $a_3=1$ when $p \neq 2$. Now assume $(m,p)=(5,2)$, so $G=D_{22}$ (see \cite[Table 2]{Brundan}). Suppose $a_3 \neq 1$. Then Lemma \ref{l:main} implies that $1 \in \{a_{21},a_{22}\}$, and without loss of generality we may assume $a_{21}=1$. By Lemma \ref{l:delta2}, the zero-weight is the only weight in $U_7$, so in view of Remark \ref{r:ord} we deduce that
$\alpha_{21}|_{X}=\alpha_{22}|_{X}=0$.  Therefore, $\lambda$ and $\lambda-\alpha_{21}$ are distinct weights in $V$ which restrict to the same $T_X$-weight $\nu=\lambda|_{X}=\mu_1$. Now $\nu$ occurs in $V_1$ with multiplicity 1, but it does not occur in $V_2$ since $\mu_1$ is not under $\mu_2$ (see Lemma \ref{l:new1}). This contradicts the fact that $V=V_1\oplus V_2$, hence $a_3=1$ when $(m,p)=(5,2)$. 

To complete the proof of the claim, let us assume $\delta=\delta_2$ and $m \ge 6$. Here $\dim U_1 = 1$, $\dim U_2 = \dim U_3 = 2$ and $\dim U_j \ge 3$ for all $4 \le j <\ell/2$, while $\dim U_{\ell/2} \ge 5$ unless $(m,p) = (6,2)$, in which case $\dim U_{\ell/2} =4$. Therefore, if $(m,p) \neq (6,2)$ then Lemma \ref{l:main} implies that $1 \in \{a_3, a_5\}$, as required. Finally, if $(m,p)=(6,2)$ then $G=C_{32}$ (see \cite[Table 2]{Brundan}) and the same conclusion holds. This justifies the claim. 

Next let $P_X=Q_XL_X$ be the parabolic subgroup of $X$ in Remark \ref{r:p}, and define $W_j$, $L_j$, $Y_j$ and $M_j$ in the usual way. Note that  $Y_1 = W_0 \cong W/[W,Q_X]$ is an irreducible $KL_X'$-module with highest weight $\delta_i|_{L_X'}$, so $\dim Y_1 = \dim \L^i(U)$ where $U$ is a natural $KL_X'$-module. In particular, we may view $L_X'$ as a subgroup of $L_1$. Also note that each $M_j|_{L_X'}$ is irreducible (since $V/[V,Q]$ is an irreducible $KL_X'$-module).

Suppose $i \ge 3$ or $(i,m)=(2,5)$. By the previous claim we have $a_3=1$, so $M_1$ is nontrivial and $M_1 \neq Y_1$ nor $Y_1^*$. Consequently, the configuration $(L_X',L_1,M_1)$ must be one of the cases listed in \cite[Table 1]{Seitz2}, but one checks that there are no compatible examples in this table. This is a contradiction. Finally, suppose $\delta=\delta_2$ and $m\geq 6$, so $a_3 = 1$ or $a_5 = 1$. Consider the triple $(L_X',L_1,M_1)$. Here \cite[Table 1]{Seitz2} indicates that $M_1 = \L^2(Y_1)$ is the only possibility, in which case $a_i = 0$ for all $3 \le i \le k$, where $k = \dim Y_1 -1 = m(m-1)/2 -1$. This is a contradiction since $k>5$.
\end{proof}

\begin{lem}\label{l:case42}
If $\delta = 2\delta_1$ then $(G,H,V) \not\in \mathcal{T}$.
\end{lem}

\begin{proof}
Suppose $\delta = 2\delta_1$ and $(G,H,V) \in \mathcal{T}$. By combining Lemmas \ref{l:main}, \ref{l:delta21} and \ref{l:123} we deduce that $a_3=1$ or $a_5=1$. Let $P_X = Q_XL_X$ be the parabolic subgroup of $X$ in Remark \ref{r:p} and define the $L_i,Y_i$ and $M_i$ as before. Now $Y_1|_{L_X'}$ is irreducible with highest weight $2\delta_1|_{L_X'}$, so if $U$ denotes a natural $KL_X'$-module then $\dim Y_1 = \dim S^2(U) = m(m+1)/2$. However, by considering the configuration $(L_X',L_1,M_1)$ and inspecting \cite[Table 1]{Seitz2} we see that $M_1 = \L^2(Y_1)$ is the only possibility, whence $a_i = 0$ for all $3 \le i \le k$ where $k = m(m+1)/2-1>5$. This is a contradiction.
\end{proof}

\begin{lem}\label{l:case 4last}
If $(\delta,p) = (\delta_1+\delta_2,3)$ then $(G,H,V) \not\in \mathcal{T}$.
\end{lem}

\begin{proof}
Suppose otherwise. By applying Lemmas \ref{l:main}, \ref{l:123} and Proposition \ref{p:ind} we deduce that $a_2=1$ or $a_4=1$. 
Let $P_X=Q_XL_X$ be the usual parabolic subgroup of $X$, and define the $L_i, Y_i$ and $M_i$ as before. Note that $Y_1 = W/[W,Q_X]$ is an irreducible $KL_X'$-module with highest weight $(\delta_1+\delta_2)|_{L_X'}$, so $\dim Y_1\geq m+1$. In particular, since $1 \in \{a_2, a_4\}$, $M_1$ is nontrivial and $M_1\neq Y_1$ nor $Y_1^*$.  Consequently, the configuration $(L_X',L_1,M_1)$ must be one of the cases listed in \cite[Table 1]{Seitz2}, but we find that there are no compatible examples.
\end{proof}

To complete the proof of Theorem \ref{T:DM5}, it remains to eliminate the case $\delta = \delta_{m-1}+\delta_m$.

\begin{lem}\label{l:case43}
If $\delta = \delta_{m-1}+\delta_m$ then $(G,H,V) \not\in \mathcal{T}$.
\end{lem}

\begin{proof}
Suppose $\delta = \delta_{m-1}+\delta_m$ and $(G,H,V) \in \mathcal{T}$. By applying Lemma \ref{l:main} (together with Proposition \ref{p:ind} and Lemma \ref{l:123}) we deduce that $a_2=1$, $a_4 = a_5 = 0$ and $a_n = 0$ (the latter condition follows from the fact that 
the $U_X$-level of the lowest weight $-\delta$ is $m(m-1)$, and is therefore even). Here the usual parabolic subgroup of $X$ (see Remark \ref{r:p}) is not much use since $Y_1|_{L_X'}$ is simply a natural module for $L_X'$, so we cannot effectively apply \cite[Theorem 1]{Seitz2}. A different approach is required.

Let $P_X = Q_XL_X$ be the $t$-stable parabolic subgroup of $X$ with $L_X'=D_{m-1}$ and $\Delta(L_X') = \{\b_2, \ldots, \b_{m}\}$, and let $P=QL$ be the corresponding parabolic subgroup of $G$. Here $\ell=2$, so $L' = L_1L_2$ (with each $L_i$ simple) and $V/[V,Q] = M_1 \otimes M_2$, where each $M_i$ is a $p$-restricted irreducible $KL_i$-module. Let $W_i$ denote the $i$-th $Q_X$-level of $W$ and let $Y_i$ be the natural module for $L_i$, so $Y_i = W_{i-1}$. By Lemma \ref{l:tstable} we have 
\begin{equation}\label{e:v1q}
V/[V,Q] = V/[V,Q_X] = V_1/[V_1,Q_X] \oplus V_2/[V_2,Q_X]
\end{equation}
as $KL_X'$-modules. Here $V_1/[V_1,Q_X]$ and $V_2/[V_2,Q_X]$ are irreducible $KL_X'$-modules with respective highest weights 
$$(\sum_{i=2}^{m}c_i\delta_{i})|_{L_X'},\;\; (\sum_{i=2}^{m-2}c_i\delta_{i}+c_{m}\delta_{m-1}+c_{m-1}\delta_m)|_{L_X'}.$$  
Moreover, since $c_{m-1}\neq c_m$, it follows that  $V_1/[V_1,Q_X]$ and $V_2/[V_2,Q_X]$ are non-isomorphic $KL_X'$-modules.

Next observe that $Y_1 = W_0 \cong W/[W,Q_X]$ is an irreducible $KL_X'$-module, with highest weight $(\delta_{m-1}+\delta_m)|_{L_X'}$, so we may view $L_X'$ as a subgroup of $L_1$. In fact, since $Y_1 = V_{L_X'}(\delta_{m-1}+\delta_m)$ and $\delta_{m-1}+\delta_m$ is symmetric, it follows that $L_X' \leqs Cl(Y_1)<L_1$, where $Cl(Y_1)$ is a simple classical group of symplectic or orthogonal type. (By Lemma \ref{l:embeddm}, if $p \neq 2$ then $Cl(Y_1)$ is orthogonal.)
We consider the restriction of $M_1$ to $Cl(Y_1)$.

First assume $M_1|_{Cl(Y_1)}$ is irreducible, so we have a configuration $(Cl(Y_1), L_1, M_1)$. Since $a_2=1$, by inspecting \cite[Table 1]{Seitz2} we reduce to one of the following situations (each with $p \neq 2$):
\begin{itemize}\addtolength{\itemsep}{0.3\baselineskip}
\item[(i)] ${\rm I}_{1}'$, with $(k,a,b) = (1,p-2,1)$ or $(2,1,p-2)$;
\item[(ii)] ${\rm I}_{2}$ or ${\rm I}_{4}$ (both with $k=2$).
\end{itemize}
In case (i) we have $Cl(Y_1) = {\rm Sp}(Y_1)$, which contradicts Lemma \ref{l:embeddm}.  In (ii), $M_1 = \L^2(Y_1)$ and thus Lemma \ref{l:compfdm} implies that $M_1|_{L_X'}$ has at least three composition factors. This contradicts \eqref{e:v1q}. 

Therefore $M_1|_{Cl(Y_1)}$ is reducible, so $M_1|_{L_X'}$ is also reducible. More precisely, in view of \eqref{e:v1q}, it follows that $M_1|_{Cl(Y_1)}$ has exactly two composition factors, say $U_1$ and $U_2$, and $L_X'$ acts irreducibly on each of them. In the usual way, via Steinberg's tensor product theorem, we may write
$$U_1 \cong S_1^{(q_1)} \otimes \cdots \otimes S_k^{(q_k)}$$
for some $k \ge 1$, where $k=q_1=1$ if $p=0$, otherwise each $S_i$ is a nontrivial $p$-restricted irreducible  
$KCl(Y_1)$-module. Since $a_2=1$, there exists an $i$ such that $S_i \not\cong Y_1$ nor $Y_1^*$, so the configuration $(L_X',Cl(Y_1),S_i)$ must be in \cite[Table 1]{Seitz2}. (If $S_i|_{Cl(Y_1)}$ is tensor decomposable then we can repeat the argument in the proof of Lemma \ref{l:aim1}.) However, since $Y_1|_{L_X'}$ has highest weight $(\delta_{m-1}+\delta_m)|_{L_X'}$, it is easy to see that there are no compatible examples.
\end{proof}

This completes the proof of Theorem \ref{T:DM5}.

\chapter{The case $H^0=E_6$}\label{s:e6}

In this section we will assume $H^0 = X=E_6$, so $H=X\langle t \rangle = X.2$ with $t$ an involutory graph automorphism of $X$. Since $t$ acts on $W$, the $T_{X}$-highest weight $\delta$ of $W$ must be fixed by $t$, so 
$$\delta = a\delta_1+b\delta_2+c\delta_3+d\delta_4+c\delta_5+a\delta_6$$
for some non-negative integers $a,b,c$ and $d$, where we label the $\delta_i$ as indicated in the labelled Dynkin diagram of type $E_6$ given below (see \cite{Bou}). In particular, $G = {\rm Sp}(W)$ or ${\rm SO}(W)$. 

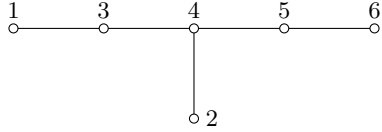
\begin{center}
\begin{tikzpicture}[scale = 1.2]
\tikzstyle{every node}=[font=\small]
\draw (0,0) circle (0.05);
\draw (0.05,0) -- (0.95,0);
\draw (1,0) circle (0.05);
\draw (1.05,0) -- (1.95,0);
\draw (2,-0.05) -- (2,-0.95);
\draw (2,-1) circle (0.05);
\draw (2.05,0) -- (2.95,0);
\draw (2,0) circle (0.05);
\draw (3.05,0) -- (3.95,0);
\draw (3,0) circle (0.05);
\draw (4,0) circle (0.05);
\node at (0,0.2){$1$};
\node at (1,0.2){$3$};
\node at (2,0.2){$4$};
\node at (3,0.2){$5$};
\node at (4,0.2){$6$};
\node at (2.2,-1){$2$};
 \end{tikzpicture}
\end{center}

As before, since $X$ is simply laced, the set of $T_X$-weights of $W$ is the same as the set of $T_X$-weights of the Weyl module $W_X(\delta)$ (see Lemma \ref{l:pr}).

We continue to adopt the same set-up as before. Let $V=V_G(\lambda)$ be a rational $p$-restricted irreducible tensor indecomposable $KG$-module with highest weight $\l=\sum_{i}a_i\l_i$, where $\{\l_1, \ldots, \l_n\}$ are the fundamental dominant weights of $G$ with respect to a base $\{\a_1, \ldots, \a_n\}$ of the root system $\Phi(G)$. Our goal is to classify the triples $(G,H,V)$, where $V|_{H}$ is irreducible and $V|_{X}$ is reducible. Recall that if $(G,H,V)$ is such a triple then $V|_{X}=V_1 \oplus V_2$, where the $V_i$ are non-isomorphic irreducible $KX$-modules interchanged by $t$ (see Proposition \ref{p:niso}). Let $\mu_i$ denote the highest weight of $V_i$, so
$$\mu_1=\sum_{i=1}^6c_i\delta_i, \;\;  \mu_2=c_6\delta_1+c_2\delta_2+c_5\delta_3+c_4\delta_4+c_3\delta_5+c_1\delta_6.$$
Without loss of generality, we may assume that $\lambda|_X=\mu_1$. 
Let $\mathcal{T}$ be the set of triples $(G,H,V)$ satisfying Hypothesis \ref{h:our} with $H^0=E_6$. 

\section{The main result}

\begin{thm}\label{T:E6}
If $H^0=E_6$ then there are no triples $(G,H,V)$ satisfying Hypothesis \ref{h:our}.
\end{thm}

\section{Preliminaries}

\begin{lem}\label{l:e6mu1mu2} 
We have 
$$\mu_1-\mu_2=\frac{1}{3}(2c_1+c_3-c_5-2c_6)(\beta_1-\beta_6)+\frac{1}{3}(c_1+2c_3-2c_5-c_6)(\beta_3-\beta_5).$$
\end{lem}

\begin{proof}
This is an easy check, using \cite[Table 1]{Hu1} to express $\mu_1$ and $\mu_2$ as a linear combination of the simple roots $\beta_1,\dots,\beta_6$.
\end{proof}

\begin{cor}\label{c:e6mu2}
Suppose $(G,H,V) \in \mathcal{T}$ and let $\mu \in \Lambda(V)$.
\begin{itemize}\addtolength{\itemsep}{0.3\baselineskip}
\item[{\rm (i)}] If $\mu|_{X}=\lambda|_{X}-\beta_1+\beta_6$  then $\mu_2=\mu_1-\beta_1+\beta_6$.
\item[{\rm (ii)}] If $\mu |_{X}=\lambda|_{X}-\beta_3+\beta_5$  then $\mu_2=\mu_1-\beta_3+\beta_5$.
\end{itemize}
\end{cor}

\begin{proof}
Consider (i). Let $\nu=\mu|_{X}$. Now $\nu=\mu_1-\beta_1+\beta_6$, so $\nu$ is not under $\mu_1$ and thus $\nu \not \in \Lambda(V_1)$. Therefore $\nu \in \Lambda(V_2)$ and  
$\nu=\mu_2-\sum_{j=1}^6k_j\beta_j$ for some non-negative integers $k_j$. Hence 
$$\mu_2-\mu_1=(k_1-1)\beta_1+k_2\beta_2+k_3\beta_3+k_4\beta_4+k_5\beta_5+(k_6+1)\beta_6$$ 
and thus Lemma \ref{l:e6mu1mu2} implies that 
$k_1-1=-(k_6+1)$, $k_2=k_4=0$ and $k_3=-k_5$. Therefore $k_j=0$ for all $1\leq j\leq 6$, and part (i) follows. The proof of (ii) is entirely similar.  
\end{proof}

Let $P_X=Q_XL_X$ be a parabolic subgroup of $X$, let $W_i$ be the $i$-th $Q_X$-level of $W$ and let $P=QL$ be the corresponding parabolic subgroup of $G$ (see Section \ref{ss:parabs}). As before, let $\ell$ denote the $Q_X$-level of the lowest weight $-\delta$ of $W$. To avoid any confusion, we will sometimes write $\ell_{\delta}$ in place of $\ell$ if we wish to specify the highest weight of the relevant $KX$-module. The next result provides some useful information on the dimensions of the $W_i$ in the special case where $P_X$ is a Borel subgroup of $X$.

\begin{prop}\label{p:levels_e6}
Let $B_X=U_XT_X$ be a $t$-stable Borel subgroup of $X$ and let $W_i$ denote the $i$-th $U_X$-level of $W$. Then $\ell = 2(16a+11b+30c+21d)$, $\dim W_i \geq 3$ for all $4\leq i <\ell/2$, and $\dim W_{\ell/2} \geq 5$. Furthermore, exactly one of the following holds:
\begin{itemize}\addtolength{\itemsep}{0.3\baselineskip}
\item[{\rm (i)}] $\dim W_i \ge 3$ for all $1 \le i \leq 3$. 
\item[{\rm (ii)}] $\delta = \delta_{2}$ and $\dim W_1=\dim W_2=1$, $\dim W_3=2$.
\item[{\rm (iii)}] $\delta = b\delta_{2}$ with $b\geq 2$ and $\dim W_1=1$,  $\dim W_2=2$, $\dim W_3\geq 3$. 
\item[{\rm (iv)}] $\delta=d\delta_4$ with $d\geq 1$ and $\dim W_1=1$, $\dim W_2\geq 3$, $\dim W_3\geq 3$.
\item[{\rm (v)}] $\delta \in \{a(\delta_{1}+\delta_{6}), c(\delta_{3}+\delta_{5}), b\delta_{2}+d\delta_{4}\}$ where  $a,b,c,d \ge 1$ and $\dim W_{1}=2$, $\dim W_2\geq 3$, $\dim W_3\geq 3$.
\end{itemize}
\end{prop}

In order to prove the proposition, we require a preliminary lemma.

\begin{lem}\label{l:subdome6}
If $\delta \neq \delta_2$ then $\delta_1+\delta_6$ is subdominant to $\delta$. 
\end{lem} 

\begin{proof}
Let $\nu=\delta-(\delta_1+\delta_6)$. By expressing $\nu$ as a linear combination of the $\beta_i$ (see \cite[Table 1]{Hu1}), we get
\begin{align*}
\nu= & \; (2a+b+3c+2d-2)\beta_1+(2a+2b+4c+3d-2)\beta_2 
\\
& +(3a+2b+6c+4d-3)\beta_3+(4a+3b+8c+6d-4)\beta_4 \\
& +(3a+2b+6c+4d-3)\beta_5+(2a+b+3c+2d-2)\beta_6.  
\end{align*}
Therefore, since $\delta \neq \delta_2$, it follows that $\nu$ is a non-negative integral linear combination of the $\beta_i$. The result follows.
\end{proof}

\begin{proof}[Proof of Proposition \ref{p:levels_e6}]
Using \cite[Table 1]{Hu1} it is easy to check that $\ell=2(16a+11b+30c+21d)$ as claimed.
Suppose $\delta=\delta_2$. Here the corresponding weight lattice for the Weyl module $W_X(\delta)$ is the same as the weight lattice for the Lie algebra $\mathcal{L}(X)$, and we can argue as in the proof of Lemma \ref{l:delta2} to show that there are at least three distinct weights in $W_i$ for $4\leq i<\ell/2=11$. Note that the weights in $W_1$, $W_2$ and $W_3$ are respectively $\delta-\beta_2$; $\delta-\beta_2-\beta_4$; and $\delta-\beta_2-\beta_3-\beta_4$, $\beta_2-\beta_4-\beta_5$, so $\dim W_1=\dim W_2=1$ and $\dim W_3=2$ as claimed. Finally, $W_{11}$ coincides with the zero-weight space, and by inspecting \cite{LubeckW} we deduce that $\dim W_{11} = 6-\delta_{3,p}$.   

For $\delta \neq \delta_2$ our aim is to exhibit sufficiently many distinct weights in the appropriate $U_X$-level of $W$ (in view of Lemma \ref{l:pr}, it suffices to exhibit weights in the Weyl module $W_{X}(\delta)$). Let $\mu=\delta_1+\delta_6$. By Lemma \ref{l:subdome6}, $\mu$ is subdominant to $\delta$, and we proceed by induction on the height ${\rm ht}(\delta-\mu)$.
 
First assume ${\rm ht}(\delta-\mu)=0$, so $\delta=\mu=\delta_1+\delta_6$. Here $\ell=32$ and $\dim W_1=2$ since the only weights in $W_1$ are $\delta-\beta_1$ and $\delta-\beta_6$, each with multiplicity 1. It is straightforward to check that there are at least three distinct weights in $W_i$ for all $2\leq i \leq 15$, and at least five in $W_{16}$. We leave the details to the reader.

Next suppose $b \ge 2$. Set 
$$\nu=\delta-\beta_2=a(\delta_1+\delta_6)+(b-2)\delta_2+c(\delta_3+\delta_5)+(d+1)\delta_4,$$ 
so $\nu$ is subdominant to $\delta$, and Lemma \ref{l:subdome6} implies that $\mu$ is subdominant to $\nu$. Also note that ${\rm ht}(\nu-\mu)<{\rm ht}(\delta-\mu)$ and ${\rm ht}(\delta-\nu) = 1$.
By induction, $V_{X}(\nu)$ has at least three distinct weights in the $U_X$-levels from $2$ to $\ell_{\nu}/2-1$, and at least five at level $\ell_{\nu}/2$. Therefore Lemma \ref{l:indwt} implies that there are at least three distinct weights in each $W_i$ with $3 \le i \le \ell_{\nu}/2=\ell/2-1$, and at least five in $W_{\ell/2}$.  Now, if $\delta=b\delta_2$ then clearly $\dim W_1=1$ and $\dim W_2=2$. Similarly, if $\delta=b\delta_2+d\delta_4$ and $d \neq 0$ then $\dim W_1=2$ and $\dim W_2 \geq 3$. For any other $\delta$ (with $b \ge 2$) it is clear that $\dim W_i \geq 3$ for $i=1,2$. 

Now assume $b=1$ and $\delta \neq \delta_2$. There are several cases to consider.
Suppose $d \neq 0$. Set  
$$\nu=\delta-\beta_2-\beta_4=a(\delta_1+\delta_6)+(c+1)(\delta_3+\delta_5)+(d-1)\delta_4.$$ 
As before, $\nu$ is subdominant to $\delta$, $\mu$ is subdominant to $\nu$ and ${\rm ht}(\nu-\mu)<{\rm ht}(\delta-\mu)$. Since ${\rm ht}(\delta-\nu) = 2$, by induction (and Lemma \ref{l:indwt}), it follows that $W_i$ has sufficiently many weights for all $i \ge 4$. The remaining $U_X$-levels are easy to check directly; we have $\dim W_1 \ge 2$ (with equality if and only if $a=c=0$), $\dim W_2\geq 3$ and $\dim W_3\geq 3$.  
A very similar argument applies if $a \neq 0$ or $c \neq 0$ (with $b=1$). Indeed, if $c \neq 0$ we define
$$\nu=\delta-\beta_3-\beta_4-\beta_5=(a+1)(\delta_1+\delta_6)+2\delta_2+(c-1)(\delta_3+\delta_5),$$
while we set
$$\nu=\delta-\beta_1-\beta_3-\beta_4-\beta_5-\beta_6=(a-1)(\delta_1+\delta_6)+2\delta_2$$
if $a \neq 0$. In both cases (using Lemma \ref{l:indwt}, together with induction and direct calculation) it is easy to check that $\dim W_i \ge 3$ for all $i \ge 1$, and $\dim W_{\ell/2} \ge 5$. We leave the details to the reader.
For the remainder we may assume that $b=0$. 

Next suppose $d \neq 0$. We first assume $d=1$. Let 
$$\nu=\delta-\beta_2-\beta_3-2\beta_4-\beta_5=(a+1)(\delta_1+\delta_6)+c(\delta_3+\delta_5).$$ 
Then $\nu$ is subdominant to $\delta$, $\mu$ is subdominant to $\nu$ (see Lemma \ref{l:subdome6}), ${\rm ht}(\nu-\mu)<{\rm ht}(\delta-\mu)$ and ${\rm ht}(\delta-\nu) = 5$. By induction, $V_{X}(\nu)$ has at least three distinct weights in the $U_X$-levels from $2$ to $\ell_{\nu}/2-1$, and at least five at level $\ell_{\nu}/2$. Therefore, by applying Lemma \ref{l:indwt}, it remains to check the levels $W_1, \ldots, W_6$. If $\delta = \delta_4$ then a straightforward calculation yields $\dim W_1=1$ and $\dim W_i\geq 3$ for all $2\leq i \leq 6$.
Similarly, if $a \neq 0$ or $c \neq 0$ then it is easy to check that $\dim W_i\geq 3$ for all $1\leq i \leq 6$. 

Now assume $d \ge 2$.  Here we set 
$$\nu=\delta-\beta_4=a(\delta_1+\delta_6)+\delta_2+(c+1)(\delta_3+\delta_5)+(d-2)\delta_4,$$ 
so $\nu$ is subdominant to $\delta$, and $\mu$ is subdominant to $\nu$. Since 
${\rm ht}(\nu-\mu)<{\rm ht}(\delta-\mu)$ and ${\rm ht}(\delta-\nu) = 1$, the usual induction argument implies that $W_i$ has at least three distinct weights for all $i \ge 3$, and there are at least five in $W_{\ell/2}$. Clearly, if $a=c=0$ then 
$\dim W_1=1$, otherwise $\dim W_1\geq 3$. In addition, it is easy to exhibit three 
distinct weights in $W_2$ (for example $\delta-\beta_2-\beta_4$, $\delta-\beta_4-\beta_5$ and $\delta-\beta_3-\beta_4$), hence $\dim W_2\geq 3$. For the remainder we may assume that $b=d=0$. 

Suppose $c \neq 0$. Let 
$$\nu=\delta-\beta_3-\beta_4-\beta_5=(a+1)(\delta_1+\delta_6)+\delta_2+(c-1)(\delta_3+\delta_5)$$
and note that $\nu$ is subdominant to $\delta$, and $\mu$ is subdominant to $\nu$.  Since ${\rm ht}(\nu-\mu)<{\rm ht}(\delta-\mu)$ and ${\rm ht}(\delta-\nu) = 3$, the usual argument implies that there are at least three distinct weights in $W_i$ for all $i \ge 5$, and at least five in $W_{\ell/2}$. 
Clearly, $\dim W_1 = 4-2\delta_{a,0}$. By exhibiting suitable weights, we also calculate that $\dim W_i\geq 3$ for $i=2,3,4$. For instance, in $W_2$ we have $\delta-\b_3-\b_5$, $\delta - \b_3 - \b_4$ and $\delta - \b_4 - \b_5$.

To complete the proof we may assume that $\delta=a\delta_1+a\delta_6$ with $a \ge 2$. Let 
$$\nu=\delta-\beta_1-\beta_3-\beta_4-\beta_5-\beta_6=(a-1)(\delta_1+\delta_6)+\delta_2$$
and observe that $\nu$ is subdominant to $\delta$, and $\mu$ is subdominant to $\nu$. In the usual way, by induction, we deduce that there are at least three distinct weights in $W_i$ for all $i \ge 7$, and at least five in $W_{\ell/2}$. Now $\dim W_1=2$, and it is easy to exhibit three distinct weights in $W_i$ for all $2\leq i\leq 6$.
This completes the proof of Proposition \ref{p:levels_e6}.
\end{proof}

\begin{lem}\label{l:centree6a1a4}
Let $P_X=Q_XL_X$ be the  parabolic subgroup of $X$ with $L_X'=A_1A_4$ and $\Delta(L_X')=\{\beta_1,\beta_2,\beta_4,\beta_5,\beta_6\}$, and let $P=QL$ be the corresponding parabolic subgroup of $G$. If $c_1+2c_3\neq 2c_5+c_6$ then 
$V/[V,Q]$ is an irreducible $KL_X'$-module.
\end{lem}

\begin{proof}
This is similar to the proof of Lemma \ref{l:asym}. Recall that $V/[V,Q]$ is an irreducible module for $L$ and $L'$ (see Lemma \ref{l:vq}(i)), and since $[V,Q_X]\leqs [V,Q]$ it follows that $V/[V,Q]$ is a quotient of 
$$V/[V,Q_X]=V_1/[V_1,Q_X]\oplus V_2/[V_2, Q_X],$$ 
where each summand is an irreducible $KL_X'$-module. Seeking a contradiction, suppose that $V/[V,Q]$ is a reducible $KL_X'$-module, so 
$$V/[V,Q]=V/[V,Q_X]=V_1/[V_1,Q_X] \oplus V_2/[V_2, Q_X]$$
as $KL_X'$-modules, and thus $\mu_1$ and $\mu_2$ both occur with non-zero multiplicity in $V/[V,Q]$.

Let $Z=Z(L_X)^0$.  As $L=C_G(Z)$ (see Lemma \ref{l:flag}(ii)), we have $Z \leqs Z(L)$ and thus Schur's lemma implies that $Z$ acts as scalars on $V/[V,Q]$. In particular, $\mu_1|_{Z}=\mu_2|_{Z}$ so
$$5c_1+6c_2+10c_3+12c_4+8c_5+4c_6 = 4c_1+6c_2+8c_3+12c_4+10c_5+5c_6$$
since  
$$Z= \left\{h_{\beta_1}(c^{5})h_{\beta_2}(c^{6})h_{\beta_3}(c^{10})h_{\beta_4}(c^{12})h_{\beta_5}(c^{8})h_{\beta_6}(c^{4}) \mid c \in K^*\right\}.$$
This equality implies that 
$c_1+2c_3=2c_5+c_6$, which is a contradiction.
\end{proof}

\begin{lem}\label{l:centree6d5}
Let $P_X=Q_XL_X$ be the  parabolic subgroup of $X$ with $L_X'=D_5$ and $\Delta(L_X') = \{\beta_1,\beta_2,\beta_3,\beta_4,\beta_5\}$, and let $P=QL$ be the corresponding parabolic subgroup of $G$. If 
$2c_1+c_3\neq c_5+2c_6$ then $V/[V,Q]$ is an irreducible $KL_X'$-module. 
\end{lem}

\begin{proof}
This is entirely similar to the proof of Lemma \ref{l:centree6a1a4}, noting that 
$$Z= Z(L_X)^0 = \left\{h_{\beta_1}(c^{2})h_{\beta_2}(c^{3})h_{\beta_3}(c^{4})h_{\beta_4}(c^{6})h_{\beta_5}(c^{5})h_{\beta_6}(c^{4}) \mid c \in K^*\right\}.$$
We leave the reader to check the details.
\end{proof}

\section{Proof of Theorem \ref{T:E6}}

By applying Lemma \ref{l:main} and Proposition \ref{p:levels_e6}, we may assume  
$$\delta \in \{a(\delta_1+\delta_6), b\delta_2, c(\delta_3+\delta_5), b\delta_2+d\delta_4 \mid a,b,c,d \ge 1\}.$$ 
Therefore, to prove Theorem \ref{T:E6} it remains to eliminate each of these cases. Recall that $\mathcal{T}$ is the set of triples $(G,H,V)$ satisfying Hypothesis \ref{h:our} with $H^0=E_6$.

\begin{lem}\label{l:e6delta2}
If $\delta = \delta_2$ then $(G,H,V) \not\in \mathcal{T}$.
\end{lem}

\begin{proof} 
Here $\dim W=78-\delta_{3,p}$ and $G={\rm SO}(W)$ (see \cite[Table 2]{Brundan}, for example). Seeking a contradiction, let us assume $\delta = \delta_2$ and $(G,H,V) \in \mathcal{T}$. In view of Lemma \ref{l:main} and Proposition \ref{p:levels_e6}(ii), we deduce that $a_4=1$. 

Let $P_X=Q_XL_X$ be the $t$-stable parabolic subgroup of $X$ with $L_X'=A_5$ and $\Delta(L_X') = \{\b_1,\b_3,\b_4,\b_5,\b_6\}$. Let $W_i$ be the $i$-th $Q_X$-level of $W$ and let $\ell$ denote the level of the lowest weight $-\delta$. Then $\ell=4$ and we calculate that $\dim W_0=1$, $\dim W_1=20$ and $\dim W_2=36-\delta_{3,p}$ (note that every weight of $W$ occurs with multiplicity 1, except for the zero weight, which has multiplicity $6-\delta_{3,p}$). Let $P=QL$ be the corresponding parabolic subgroup of $G$, and write $L'=L_1L_2$ and 
$V/[V,Q]=M_1 \otimes M_2$, where each $L_i$ is simple and each $M_i$ is a $p$-restricted irreducible $KL_i$-module. (More precisely, we have $L_1 = A_{19}$, while $L_2 = D_{18}$ if $p \neq 3$, otherwise $L_2 = B_{17}$.) Let $Y_i$ be the natural $KL_i$-module and note that $Y_i = W_i$.  
Since $P_X$ is $t$-stable, Lemma \ref{l:tstable} implies that 
$$V/[V,Q] = V_1/[V_1,Q_X] \oplus V_2/[V_2,Q_X]$$
as $KL_X'$-modules, so $V/[V,Q]$ has exactly two $KL_X'$-composition factors.

Let $\pi_i$ denote the projection map from $L_X'$ to $L_i$. Now every weight in $W_1$ is under $\delta - \beta_2 = -\delta_2+\delta_4$, so $Y_1|_{L_{X}'}$ is irreducible with highest weight $\delta_4|_{L_X'}$. Therefore $\pi_1(L_X')=A_5$ and we may view $L_X'$ as a subgroup of $L_1$. In addition, it follows that $Y_1|_{L_X'} = \L^3(U)$, where $U$ is the natural $6$-dimensional module for $L_X'$, and thus $L_X'<Cl(Y_1)<L_1$ where
$Cl(Y_1) = C_{10}$ if $p \neq 2$, otherwise $Cl(Y_1) = D_{10}$ (see Lemma \ref{l:st} and \cite[Table 2]{Brundan}). Furthermore, since $a_4=1$ it follows that $M_1$ is nontrivial, and $M_1 \neq Y_1$ nor $Y_1^*$. We now consider the restriction of $M_1$ to $Cl(Y_1)$ (the rest of the argument is similar to the proof of Lemma \ref{l:case43}).
 
First assume $M_1|_{Cl(Y_1)}$ is irreducible. Here the configuration $(Cl(Y_1),L_1,M_1)$ must be one of the cases in \cite[Table 1]{Seitz2}; using the fact that $a_4=1$ it is easy to see that the case labelled ${\rm I}_{1}'$ is the only possibility, with $p \neq 2$ and either $(k,a,b) = (3,1,p-2)$ or $(2,p-2,1)$ (in the notation of \cite[Table 1]{Seitz2}). In particular, $M_1|_{Cl(Y_1)}$ is $p$-restricted with highest weight $a\lambda_k+b\lambda_{k+1}$, so the triple $(L_X',Cl(Y_1),M_1) = (A_5, C_{10},V_{C_{10}}(a\lambda_k+b\lambda_{k+1}))$ must have arisen in the analysis of the case $X=A_5$ in Section \ref{s:am}. However, Theorem \ref{t:am} states that there are no compatible triples $(G,H,V)$ satisfying Hypothesis \ref{h:our} with $H^0=A_5$.

Therefore $M_1|_{Cl(Y_1)}$ is reducible. More precisely, $M_1|_{Cl(Y_1)}$ has exactly two composition factors, say $U_1$ and $U_2$, and $L_X'$ acts irreducibly on each of them. By Steinberg's tensor product theorem, we may write
$$U_1 \cong S_1^{(q_1)} \otimes \cdots \otimes S_k^{(q_k)}$$
for some $k \ge 1$, where $k=q_1=1$ if $p=0$, otherwise each $S_i$ is a nontrivial $p$-restricted irreducible  $KCl(Y_1)$-module. Since $a_4=1$, we may assume that there exists some $i$ such that the highest weight of $S_i$ as a $KCl(Y_1)$-module has a non-zero coefficient on the third fundamental dominant weight (see Lemma \ref{l:nat}). Then $S_i \neq Y_i,Y_i^*$ and so the main theorem of \cite{Seitz2} implies that the configuration $(L_X',Cl(Y_1),S_i)$ is one of the cases in \cite[Table 1]{Seitz2} (recall that  $Cl(Y_1) = D_{10}$ if $p=2$, so each $S_i$ is tensor indecomposable as a $KCl(Y_1)$-module). However, it is easy to check that there are no compatible examples.
\end{proof}

\begin{lem}\label{l:e6bdelta2bnot1}
Suppose $\delta = b\delta_2$ with $b\geq 2$. Then $(G,H,V) \not\in \mathcal{T}$.
\end{lem}   

\begin{proof}
Suppose otherwise. Let $B_X=U_XT_X$ be a $t$-stable Borel subgroup of $X$ and let $P=QL$ be the corresponding parabolic subgroup of $G$ 
constructed in the usual way.  By appealing to Remark \ref{r:ord}, we may order the $T$-weights in the $U_X$-levels $0$, $1$ and $2$ so that we obtain the root restrictions given in Table \ref{t:r13}.

\renewcommand{\arraystretch}{1.2}
\begin{table}
$$\begin{array}{llll} \hline
\mbox{$U_X$-level} & \mbox{$T_X$-weight} & \mbox{$T$-weight} & \mbox{Root restriction} \\ \hline
0 & \delta & \lambda_1& \\
1& \delta-\beta_2& \lambda_1-\alpha_1 & \alpha_1|_{X}=\beta_2\\
2 & \delta-\beta_2-\beta_4 & \lambda_1-\alpha_1-\alpha_2 & \alpha_2|_{X}=\beta_4\\
& \delta-2\beta_2 & \lambda_1-\alpha_1-\alpha_2-\alpha_3 & \alpha_3|_{X}=\beta_2-\beta_4 \\ \hline
\end{array}$$
\caption{}
\label{t:r13}
\end{table}
\renewcommand{\arraystretch}{1}

By Lemma \ref{l:main} and Proposition \ref{p:levels_e6}, we have $a_3=1$ and 
thus $\lambda-\alpha_3 \in \Lambda(V)$. Now $\lambda-\alpha_3$ restricts as a $T_X$-weight to $\nu=\mu_1-\beta_2+\beta_4$. Clearly $\nu$ does not occur in $V_1$,  and Lemma \ref{l:e6mu1mu2}  implies that $\nu$ does not occur in $V_2$. This contradicts the fact that $V=V_1\oplus V_2$.
\end{proof}

\begin{lem}\label{l:e6bdelta2ddelta4}
Suppose $\delta = b\delta_2+d\delta_4$ with $bd \neq 0$. Then $(G,H,V) \not\in \mathcal{T}$.
\end{lem}

\begin{proof}
This is very similar to the proof of the previous lemma. Let $B_X=U_XT_X$ be a  $t$-stable Borel subgroup of $X$ and let $P=QL$ be the corresponding parabolic subgroup of $G$. By suitably ordering the $T$-weights in the $U_X$-levels $0$ and $1$ (see Remark \ref{r:ord})  we may assume that $\alpha_1|_{X}=\beta_4$ and $\alpha_2|_{X}=\beta_2-\beta_4$. Now 
Lemma \ref{l:main} and Proposition \ref{p:levels_e6} give $a_2=1$, so 
$\lambda-\alpha_2 \in \Lambda(V)$. This weight restricts to  
$\nu=\mu_1-\beta_2+\beta_4$, which does not occur in $V_1$ nor $V_2$ (see Lemma \ref{l:e6mu1mu2}). This is a contradiction. 
\end{proof}

\begin{lem}\label{l:e6cdelta3cdelta5}
If $\delta = c(\delta_3+\delta_5)$ then $(G,H,V) \not\in \mathcal{T}$.
\end{lem}

\begin{proof}
Seeking a contradiction, suppose $(G,H,V) \in \mathcal{T}$. Let $B_X=U_XT_X$ be a $t$-stable Borel subgroup of $X$ and let $P=QL$ be the corresponding parabolic subgroup of $G$. We may order the $T$-weights in the $U_X$-levels $0$ and $1$ so that we obtain the root restrictions $\alpha_1|_{X}=\beta_5$ and $\alpha_2|_{X}=\beta_3-\beta_5$ (see Remark \ref{r:ord}). Now Lemma \ref{l:main} and Proposition \ref{p:levels_e6} imply that $a_2=1$, so $\lambda-\alpha_2 \in \Lambda(V)$ and this weight restricts to $\nu=\mu_1-\beta_3+\beta_5$. Therefore Corollary \ref{c:e6mu2} yields $\mu_2=\mu_1-\beta_3+\beta_5$ and thus Lemma \ref{l:e6mu1mu2} implies that 
\begin{equation}\label{e:cii}
2c_1+c_3=c_5+2c_6,\;\; c_1+2c_3=2c_5+c_6+3.
\end{equation}
 
Let $P_X=Q_XL_X$ be the parabolic subgroup of $X$ with $L_X'=A_1A_4$ and $\Delta(L_X')=\{\beta_1,\beta_2,\beta_4,\beta_5,\beta_6\}$. Let $P=QL$ be the parabolic subgroup of $G$ constructed from $P_X$ in the usual way.  As usual, write $L'=L_1\cdots L_r$  and $V/[V,Q]=M_1\otimes \dots \otimes M_r$, where each $L_i$ is simple with natural module $Y_i$, and each $M_i$ is a $p$-restricted irreducible $KL_i'$-module. In addition, let $\pi_i$ be the projection map from $L_X'$ to $L_i$. Now $Y_1=W_0$ and $Y_1|_{L_X'}$ is irreducible of highest weight $c\delta_5|_{L_X'}$. In particular, we note that $\dim Y_1 \ge 10$ (equality if $c=1$) and $\pi_1(L_X')\neq 1$. More precisely, since the $A_1$-factor of $L_X'$ acts trivially on $Y_1$, it follows that $\pi_1(L_X')=A_4$, so if $J_1$ denotes the $A_4$-factor of $L_X'$ then we may view $J_1$ as a subgroup of $L_1$.

Next observe that $c_1+2c_3\neq 2c_5+c_6$ (see \eqref{e:cii}), so $V/[V,Q]$ is an irreducible $KL_X'$-module by Lemma \ref{l:centree6a1a4}. In particular, each $M_i$ is an irreducible $KL_X'$-module, hence  $M_1$ is an irreducible $KJ_1$-module. Consider the triple $(J_1,L_1,M_1)$. Here
 $J_1=A_4$  is a proper subgroup of $L_1$ (since $\dim Y_1\geq 10$), $M_1$ is a nontrivial $p$-restricted irreducible $KL_1$-module and $M_1 \neq Y_1$ nor $Y_1^*$ (since $a_2=1$). It follows that this  configuration must appear in \cite[Table 1]{Seitz2}. Since $a_2=1$ and $Y_1|_{J_1}$ has highest weight $c\delta_5|_{J_1}$,  it is easy to see that the only possibility is the case labelled ${\rm I}_6$ in \cite[Table 1]{Seitz2}. In particular, we have now reduced to the case $p\neq 2$, $c=1$, $L_1=A_{9}$ and $\lambda|_{L_1}=\lambda_2|_{L_1}$, so $a_i=0$ for all $1 \le i \le 9$, $i \neq 2$. 
 
Now let us consider the $t$-stable parabolic subgroup $P_X=Q_XL_X$ of $X$ with $L_X'=A_2A_1A_2$ and $\Delta(L_X')=\{\beta_1, \beta_2,\beta_3, \beta_5,\beta_6\}$. Let $J_1,J_2,J_3$ denote the factors of $L_X'$ whose root systems have bases $\{\beta_1,\beta_3\}$, $\{\beta_2\}$ and $\{\beta_5,\beta_6\}$, respectively.  Let $P=QL$ be the corresponding parabolic subgroup of $G$, and define the $W_i$, $L_i$, $Y_i$, $M_i$ and $\pi_i$ in the usual way. Note that $Y_1=W_0$ is an irreducible $KL_X'$-module with $Y_1|_{L_X'}=V_{J_1}(\delta_3)\otimes V_{J_3}(\delta_5)$, so $\dim Y_1=9$. We have already observed that $a_i=0$ for all $1 \le i \le 9$, $i \neq 2$ (and $a_2=1$), whence $M_1|_{L_1}=\Lambda^2(Y_1)$ and $\dim M_1=36$. It follows that $M_1|_{L_X'}$ has two composition factors: 
$V_{J_1}(\delta_1)\otimes V_{J_3}(2\delta_5)$ and $V_{J_1}(2\delta_3)\otimes V_{J_3}(\delta_6)$.  
Since $V/[V,Q]|_{L_X'}$ has exactly two composition factors (see Lemma \ref{l:tstable}), we deduce that 
$M_2|_{L_X'}$ is irreducible. 
Next observe that $W_1|_{L_X'}$ has at least two composition factors, namely
$$V_{J_1}(\delta_1)\otimes V_{J_2}(\delta_2)\otimes V_{J_3}(2\delta_5), \;\; V_{J_1}(2\delta_3)\otimes V_{J_2}(\delta_2)\otimes V_{J_3}(\delta_6),$$  
each of dimension $36$, with respective highest weights $\delta-\beta_3-\beta_4$ and $\delta-\beta_4-\beta_5$ (so $\dim Y_2 \ge 72$). Therefore $Y_2|_{L_X'}$ is reducible and $\pi_2(L_X')=A_2A_1A_2$, so we may view $L_X'$ as a subgroup of $L_2$.  Finally, if we consider the triple $(L_X',L_2,M_2)$ then \cite[Theorem 5.1]{Seitz2} implies that $M_2$ is trivial, whence $a_i=0$ for all $1 \le i \le 80$, $i \neq 2$. 
 
To complete the proof of the lemma, let us turn to the $t$-stable parabolic subgroup $P_X=Q_XL_X$ of $X$ with $L_X'=A_3$ and $\Delta(L_X')=\{\beta_3,\beta_4,\beta_5\}$. Let $P=QL$ be the corresponding parabolic subgroup of $G$ and define the $W_i,L_i,Y_i,M_i$ and $\pi_i$ as before. Now $Y_1=W_0$ so  $Y_1|_{L_X'}$ is irreducible of highest weight $(\delta_3+\delta_5)|_{L_X'}$. In particular, Lemma \ref{l:am1m} implies that $\dim Y_1=15$ (since $p \neq 2$) and thus $L_1=A_{14}$. Since we have $a_i=0$ for all $1 \le i \le 14$, $i \neq 2$ (and $a_2=1$), it follows that $M_1=\Lambda^2(Y_1)$.
By Lemma \ref{l:tstable} we have
$$V/[V,Q]=V/[V,Q_X]=V_1/[V_1,Q_X]\oplus V_2/[V_2,Q_X]$$
as $KL_X'$-modules, so $M_1|_{L_X'}$ has at most two composition factors. However,  Lemma \ref{l:wedge2} states that $M_1|_{L_X'}$ has at least three composition factors, which is a contradiction.
\end{proof}

\begin{lem}\label{l:e6adelta1adelta6}
If $\delta = a(\delta_1+\delta_6)$ then $(G,H,V) \not\in \mathcal{T}$.
\end{lem}

\begin{proof}
Suppose there is an example $(G,H,V) \in \mathcal{T}$ with $\delta = a(\delta_1+\delta_6)$. Let $B_X=U_XT_X$ be a $t$-stable Borel subgroup of $X$. In the usual way, by considering $U_X$-levels and applying Lemma \ref{l:main}, we deduce that $a_2=1$. Also, by fixing a suitable ordering of the $T$-weights in the $U_X$-levels $0$ and $1$ of $W$, we may also assume that $\alpha_1|_{X}=\beta_6$ and $\alpha_2|_{X}=\beta_1-\beta_6$. Now the $T_X$-restriction of the weight $\lambda-\alpha_2 \in \Lambda(V)$ is given by $\nu=\mu_1-\beta_1+\beta_6$, so Corollary \ref{c:e6mu2} implies that $\mu_2=\mu_1-\beta_1+\beta_6$ and thus Lemma \ref{l:e6mu1mu2} yields
\begin{equation}\label{e:ci2}
2c_1+c_3=c_5+2c_6+3,\;\; c_1+2c_3=2c_5+c_6.
\end{equation}

Let $P_X=Q_XL_X$ be the parabolic subgroup of $X$ with $L_X'=D_5$ and $\Delta(L_X')=\{\beta_1,\beta_2,\beta_3,\beta_4,\beta_5\}$. Let $W_i$ denote the $i$-th $Q_X$-level of $W$. Let $P=QL$ be the corresponding parabolic subgroup of $G$, and define the $L_i,Y_i,M_i$ and $\pi_i$ in the usual way. Then  $Y_1=W_0$ and $Y_1|_{L_X'}$ is irreducible of highest weight $a\delta_1|_{L_X'}$, so $\pi_1(L_X')=D_5$ and we may view $L_X'$ as a subgroup of $L_1$. Also note that $\dim Y_1 \ge 10$ (with equality if $a=1$).

Since $2c_1+c_3\neq c_5+2c_6$ (see \eqref{e:ci2}), Lemma \ref{l:centree6d5} implies that $V/[V,Q]$ is an irreducible $KL_X'$-module, so $M_1|_{L_X'}$ is irreducible. Consider the triple $(L_X',L_1,M_1)$. Here $L_X'=D_5$  is a proper subgroup of $L_1$, $M_1$ is a nontrivial $p$-restricted irreducible $KL_1$-module and $M_1 \neq Y_1$ nor $Y_1^*$ (since $a_2=1$). Therefore, this configuration must be one of the cases in \cite[Table 1]{Seitz2}. Since $a_2=1$ and $M_1|_{L_X'}$ has highest weight $a\delta_1|_{L_X'}$, it is easy to see that the only possibility is the case labelled ${\rm I}_4$ with $k=2$. In particular, $p\neq 2$, $a=1$, $L_1=A_9$ and $\lambda|_{L_1}=\lambda_2|_{L_1}$, 
so we have $a_i=0$ for all $1 \le i \le 9$, $i \neq 2$. 
 
Next consider the $t$-stable parabolic subgroup $P_X=Q_XL_X$ of $X$ such that $L_X'=A_2A_1A_2$ and $\Delta(L_X')=\{\beta_1, \beta_2,\beta_3, \beta_5,\beta_6\}$. Here $\ell=8$ is the $Q_X$-level of the lowest weight $-\delta$. As in the proof of the previous lemma, let $J_1, J_2, J_3$ denote the factors of $L_X'$ whose root systems have bases $\{\beta_1,\beta_3\}$, $\{\beta_2\}$ and $\{\beta_5,\beta_6\}$, respectively. Let $P=QL$ be the corresponding parabolic subgroup of $G$ and define the $W_i,L_i,Y_i,M_i$ and $\pi_i$ as before. Here $Y_1=W_0$ is an irreducible $KL_X'$-module with $Y_1|_{L_X'}=V_{J_1}(\delta_1)\otimes V_{J_3}(\delta_6)$, so $\dim Y_1=9$. We have already observed that $a_i=0$ for all $1 \le i \le 9$, $i \neq 2$ (and also $a_2=1$), so $M_1|_{L_1}=\Lambda^2(Y_1)$ and thus $\dim M_1=36$. It follows that $M_1|_{L_X'}$ has two composition factors, namely $V_{J_1}(\delta_3)\otimes V_{J_3}(2\delta_6)$ and $V_{J_1}(2\delta_1)\otimes V_{J_3}(\delta_5)$, so $M_2|_{L_X'}$ is irreducible (since $V/[V,Q]$ has exactly two $KL_X'$-composition factors). 

Next observe that $W_1|_{L_X'}$ has at least two composition factors. Indeed,  
$V_{J_2}(\delta_2)\otimes V_{J_3}(\delta_5+\delta_6)$ and $V_{J_1}(\delta_1+\delta_3)\otimes V_{J_2}(\delta_2)$
are composition factors of $W_1|_{L_X'}$, each of dimension $16-2\delta_{3,p}$, with highest weights $\delta-\beta_1-\beta_3-\beta_4$ and $\delta-\beta_4-\beta_5-\beta_6$, respectively. Since $Y_2=W_1$ it follows that $Y_2|_{L_X'}$ is reducible and $\pi_2(L_X')=A_2A_1A_2$, so we may view $L_X'$ as a subgroup of $L_2$.   
By considering the triple $(L_X',L_2,M_2)$, \cite[Theorem 5.1]{Seitz2} implies that $M_2$ is trivial, whence $a_i=0$ for all $1 \le i \le 36$, $i \neq 2$.  
 
Finally, let us consider the $t$-stable parabolic subgroup $P_X=Q_XL_X$ of $X$ with $L_X'=A_5$ and $\Delta(L_X')=\{\beta_1,\beta_3,\beta_4,\beta_5,\beta_6\}$. Define $P=QL$, $L_i$, $Y_i$, $M_i$ and $\pi_i$ in the usual way. Here $Y_1=W_0$ and $Y_1|_{L_X'}$ is irreducible of highest weight $(\delta_1+\delta_6)|_{L_X'}$, so Lemma \ref{l:am1m} implies that $\dim Y_1=35-\delta_{3,p}$  and thus $L_1=A_{34-\delta_{3,p}}$. Since $a_i=0$ for all $1 \le i \le 36$, $i \neq 2$ (and $a_2=1$), we deduce that $M_1=\Lambda^2(Y_1)$. 
By Lemma \ref{l:tstable}, $V/[V,Q]$ has exactly two $KL_X'$-composition factors, so $M_1|_{L_X'}$ must have at most two. However, this contradicts Lemma \ref{l:wedge2}. 
\end{proof}

\vs

This completes the proof of Theorem \ref{T:E6}. 

\chapter{The case $H^0 = D_4$}\label{s:d4}

To complete the proof of Theorem \ref{t:main}, it remains to deal with the case $X=D_4$ (recall that in Section \ref{s:dm5} we handled the case $X=D_m$ with $m \ge 5$). This is a special case because the Dynkin diagram of type $D_4$ admits an order $3$ rotational symmetry, which gives rise to so-called \emph{triality} graph  automorphisms of $D_4$. In particular, if $H$ is a disconnected almost simple group with $H^0=D_4$ then there are three possibilities for $H$, namely $D_4\la t \ra = D_4.2$, $D_4\la s \ra = D_4.3$ and $D_4\la t,s \ra = D_4.S_3$. We will deal with each of these cases in turn. Let $\mathcal{T}$ be the set of triples $(G,H,V)$ satisfying Hypothesis \ref{h:our} with $H^0=D_4$. 

\section{The main result}

\begin{thm}\label{t:d4}
If $H^0=D_4$ then there are no triples $(G,H,V)$ satisfying Hypothesis \ref{h:our}.
\end{thm}

\section{The Borel reduction}

Let $B_X=U_XT_X$ be a $t$-stable Borel subgroup of $X=H^0=D_4$ and let $\delta$ denote the highest weight of $W$ as a $KX$-module. We continue to adopt the same notation for roots and weights, so $\{\delta_1, \ldots, \delta_4\}$ are the fundamental dominant weights of $X$ corresponding to a fixed base $\Delta(X) = \{\b_1, \ldots, \b_4\}$ of the root system $\Phi(X)$. We label these weights and roots as indicated in the Dynkin diagram of type $D_4$ given below.

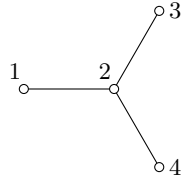
\begin{center}
\begin{tikzpicture}[scale = 1.2]
\tikzstyle{every node}=[font=\small]
\draw (0,0) circle (0.05);
\draw (-0.05,0) -- (-0.95,0);
\draw (-1,0) circle (0.05);
\draw (60:0.05) -- (60:0.95);
\draw (-60:0.05) -- (-60:0.95);
\draw (60:1) circle (0.05);
\draw (-60:1) circle (0.05);
\node at (-1.1,0.2){$1$};
\node at (-0.1,0.2){$2$};
\node at (52:1.1){$3$};
\node at (-52:1.1){$4$};
 \end{tikzpicture}
\end{center}

Suppose $H=X\la t\ra = X.2$, where we may assume without loss of generality that $t$ is the involutory graph automorphism of $X$ interchanging the simple roots $\b_3$ and $\b_4$. Now $\delta$ is $t$-stable since $t$ acts on $W$, so
$$\delta = a\delta_1+b\delta_2+c(\delta_3+\delta_4)$$
for some non-negative integers $a,b$ and $c$. In addition, we have $a=c$ if $H = X.3$ or $X.S_3$ since in this situation $\delta$ is also $s$-stable, where $s$ is the triality graph automorphism of $X$ cyclically permuting the simple roots $\b_1, \b_3$ and $\b_4$.

As before, let $W_i$ denote the $i$-th $U_X$-level of $W$ and let $\ell$ be the $U_X$-level of the lowest weight $-\delta$. It is easy to check that $\ell = 6a+10b+12c$ (see the proof of Proposition \ref{p:dm5_sub}). In particular, $\ell$ is always even. In this section we obtain some useful information on the dimensions of the $W_i$.

\begin{prop}\label{p:d4_sub}
Let $\delta = \sum_ib_i\delta_i$ be a  nontrivial $p$-restricted dominant weight for $X$ with $b_3=b_{4}$, and assume $\delta \not \in\{\delta_1, \delta_2,\delta_3+\delta_4 ,2\delta_1,\delta_1+\delta_2,3\delta_1\}$. Let $d = b_1+b_3$ and set $\mu = \delta_1+\delta_3+\delta_4$ if $d$ is even, and $\mu = \delta_2+\delta_3+\delta_4$ if $d$ is odd. Then $\mu$ is subdominant to $\delta$.
\end{prop}

\begin{proof}
First observe that 
\renewcommand\arraystretch{1.3}
\begin{equation}\label{e:d44}
\begin{array}{c}
\displaystyle \delta_1 = \b_1+\b_2+\frac{1}{2}(\b_3+\b_4),\; \delta_2 = \b_1+2\b_2+\b_3+\b_4, \\
\displaystyle \delta_3 = \frac{1}{2}(\b_1+2\b_2+2\b_3+\b_4),\; \delta_4 = \frac{1}{2}(\b_1+2\b_2+\b_3+2\b_4)
\end{array}
\end{equation}
\renewcommand\arraystretch{1}
so $\delta-(\delta_1+\delta_3+\delta_4)=\sum_{i}d_i\beta_i$ where 
$$d_1=b_1+b_2+b_3-2, \; d_2=b_1+2(b_2+b_3)-3,\; d_3=d_4 = \frac{1}{2}(b_1+2b_2+3b_3-4),$$ 
and $\delta-(\delta_2+\delta_3+\delta_4)=\sum_{i}e_i\beta_i$ with 
$$e_1=b_1+b_2+b_3-2,\; e_2=b_1+2(b_2+b_3)-4, \; e_3=e_{4}=\frac{1}{2}(b_1+2b_2+3b_3-5).$$
Now, if $d = b_1+b_3$ is even then each $d_i$ is an integer. Moreover, each $d_i$ is non-negative since $\delta \neq \delta_2, 2\delta_1$. Similarly, one can check that each $e_i$ is a non-negative integer when $d$ is odd. The result follows.
\end{proof}

\begin{lem}\label{l:d4delta2}
If $\delta=\delta_2$ then $\ell=10$ and 
$$\dim W_1=1,\; \dim W_2=3,\; \dim W_3=3,\; \dim W_4=4 \mbox{ and } \dim W_5=4-2\delta_{2,p}.$$
\end{lem}

\begin{proof}
This follows from Lemma \ref{l:delta2} (note that every weight in $W_i$ with $i < \ell/2$ occurs with multiplicity one).
\end{proof}

\begin{lem}\label{l:d4borelredsc}
Let $\delta\in\{\delta_3+\delta_4, 2\delta_1, \delta_1+\delta_2,3\delta_1\}$. Then $\dim W_{\ell/2}\geq 5$  and $W_i$ contains at least $3$ distinct weights for all $3\leq i< \ell/2$. Moreover, if 
$\delta \in \{\delta_1+\delta_2, \delta_3+\delta_4\}$ then $\dim W_1=2$ and $\dim W_2\geq 3$, while 
$\dim W_1=1$ and $\dim W_2=2$ if $\delta \in \{2\delta_1, 3\delta_1\}$.
\end{lem}

\begin{proof}
First assume $\delta=\delta_3+\delta_4$, so $\ell=12$. The weights in $W_1$ are $\delta-\beta_3$ and $\delta-\beta_4$, both with multiplicity 1, so $\dim W_1=2$. For $2 \le i \le 5$ it is easy to exhibit three distinct weights in $W_i$ (by Lemma \ref{l:pr}, we may work in the Weyl module $W_X(\delta)$). There are only $4$ distinct weights in $W_{6}$:
$\delta-\beta_1-2\beta_2-2\beta_3-\beta_4$, $\delta-\beta_1-2\beta_2-\beta_3-2\beta_4$, $\delta-2\beta_1-2\beta_2-\beta_3-\beta_4$ and $\delta-2\beta_2-2\beta_3-2\beta_4$. However, 
$\delta-\beta_1-2\beta_2-2\beta_3-\beta_4$ and $\delta-\beta_2-\beta_3-\beta_4$ are conjugate under the Weyl group of $X$ and the multiplicity of the latter weight is at least $2$ (see Lemma \ref{l:s816}), so $\dim W_6\geq 5$ as required.

The remaining cases are very similar. For example, if $\delta = 2\delta_1$ then $\ell=12$ and the weights in $W_1$ and $W_2$ are $\delta-\b_1$, and $\delta-2\b_1$, $\delta - \b_1-\b_2$, respectively, so $\dim W_1 = 1$ and $\dim W_2=2$ since each of these weights has multiplicity $1$. For $3 \le i \le 5$ it is easy to find three distinct weights in the corresponding Weyl module. In $W_6$ we have the zero weight $\delta - 2\b_1-2\b_2-\b_3-\b_4$, together with $\delta-2\beta_1-2\beta_2-2\beta_3$ and $\delta-2\beta_1-2\beta_2-2\beta_4$. By \cite{LubeckW}, the zero weight has multiplicity $3$ and thus $\dim W_6=5$. 
\end{proof}

\begin{lem}\label{l:d4borelredbcone}
Let $\delta=\delta_1+\delta_3+\delta_4$. Then $\ell=18$ and $W_i$ contains at least $7$ distinct weights for all $4\leq i \leq \ell/2$. Further $\dim W_1=3$,  $\dim W_2=6$ and $\dim W_3 \geq 7$.
\end{lem}

\begin{proof}
The weights in $W_1$ are $\delta-\beta_1$, $\delta-\beta_3$ and $\delta-\beta_4$, each occurring with multiplicity $1$, and so $\dim W_1=3$. Similarly, there are $6$ distinct weights in $W_2$, each with multiplicity $1$, hence $\dim W_2=6$.  In $W_3$ there are only $4$ distinct weights: $\delta-\beta_1-\beta_2-\beta_3$, $\delta-\beta_1-\beta_2-\beta_4$, $\delta-\beta_2-\beta_3-\beta_4$ and $\delta-\beta_1-\beta_3-\beta_4$. By Lemma \ref{l:s816}, the first $3$ weights in this list all have multiplicity at least $2$, so $\dim W_3\geq 7$.  Finally, for $4\leq i \leq 9$ it is easy to exhibit at least $7$ distinct weights in $W_i$.
\end{proof}
 
\begin{lem}\label{l:d4borelredbctwo}
Let $\delta=\delta_2+\delta_3+\delta_4$ or $2\delta_1+\delta_3+\delta_4$. Then $\ell =22$ or $24$, respectively, and $W_i$ contains at least $3$ distinct weights for all $1\leq i < \ell/2$. Further, there are at least $5$ distinct weights in $W_{\ell/2}$. 
\end{lem}

\begin{proof}
This is an easy check.  
\end{proof}

Our main result on the $U_X$-levels of $W$ is the following:

\begin{prop}\label{p:d4b}
Let $\delta = a\delta_1+b\delta_2+c(\delta_3+\delta_4)$ be the $T_X$-highest weight of $W$, where $\delta \neq \delta_1$. Then $\dim W_0 = 1$, $\dim W_1 = 4-\delta_{a,0} - \delta_{b,0} - 2\delta_{c,0}$ and the following hold:
\begin{itemize}\addtolength{\itemsep}{0.3\baselineskip}
\item[{\rm (i)}] If  $a=c$ then exactly one of  the following holds:

\vspace{1mm}

\begin{itemize}\addtolength{\itemsep}{0.3\baselineskip}
\item[{\rm (a)}] $\delta=\delta_2$, $\ell=10$, $\dim W_2=\dim W_3 =3$, $\dim W_4=4$ and $\dim W_5 = 4 - 2\delta_{2,p}$.
\item[{\rm (b)}] $\delta=2\delta_2$, $\dim W_2=4$, $\dim W_3=6$ and $\dim W_i \geq 7$ for $4\leq i \leq {\ell/2}$.
\item[{\rm (c)}] $\delta=b\delta_2$ with $b \ge 3$, $\dim W_2=4$ and $\dim W_i \geq 7$ for $3\leq i \leq {\ell/2}$.
\item[{\rm (d)}] $\delta=\delta_1+\delta_3+\delta_4$, $\dim W_2=6$ and $\dim W_i\geq 7$ for $3\leq i \leq {\ell/2}$.
\item[{\rm (e)}] $\delta=\delta_1+\delta_2+\delta_3+\delta_4$, $p=3$, $\dim W_2=6$ and $\dim W_i \geq 7$ for $3\leq i \leq {\ell/2}$. 
\item[{\rm (f)}] $\dim W_i \geq 7$ for $2\leq i\leq \ell/2$.
\end{itemize}
 \item[{\rm (ii)}] If $a\neq c$ then exactly one of the following holds:
 
 \vspace{1mm}
 
\begin{itemize}\addtolength{\itemsep}{0.3\baselineskip}
\item[{\rm (a)}]  $\delta=a\delta_1$ with $a \ge 2$, $\dim W_2=2$, $\dim W_i\geq3$ for $3\leq i <\ell/2$, and $\dim W_{\ell/2}\geq 5$.
\item[{\rm (b)}] $\dim W_i \geq 3$ for $2\leq i <\ell/2$ and $\dim W_{\ell/2} \geq 5$.
\end{itemize} 
\end{itemize}
\end{prop}

\begin{proof}
In view of Lemmas \ref{l:d4delta2}, \ref{l:d4borelredsc} and \ref{l:d4borelredbcone}, we may assume that $\delta \not \in\{\delta_2,\delta_3+\delta_4,2\delta_1, \delta_1+\delta_2, 3\delta_1,\delta_1+\delta_3+\delta_4\}$.  The stated dimensions of $W_0$ and $W_1$ are clear. Next we consider $W_2$ and $W_3$.

Suppose $a=c$. For now let us assume $a=0$, so $\delta=b\delta_2$ with $b \ge 2$. The weights in $W_2$ are $\delta-\beta_1-\beta_2$, $\delta-\beta_2-\beta_3$, $\delta-\beta_2-\beta_4$ and $\delta-2\beta_2$, each with multiplicity $1$, so $\dim W_2=4$. Similarly, the weights in $W_3$ are $\delta-\beta_1-\beta_2-\beta_3$, $\delta-\beta_1-\beta_2-\beta_4$, $\delta-\beta_2-\beta_3-\beta_4$, $\delta-\beta_1-2\beta_2$, $\delta-2\beta_2-\beta_3$, $\delta-2\beta_2-\beta_4$ and $\delta-3\beta_2$ (the latter weight only occurs if $b \ge 3$). Therefore $\dim W_3 \geq 7$ if $b \ge 3$, while $\dim W_3=6$ if $\delta = 2\delta_2$ since the first six weights in this list all have multiplicity $1$. 

Now assume $a=c>0$. Suppose $b=0$.  By assumption $a\geq 2$ (the case $\delta = \delta_1+\delta_3+\delta_4$ is handled in Lemma \ref{l:d4borelredbcone}) and it is easy to find at least $7$ distinct weights in $W_2$ and $W_3$. Now assume $b \neq 0$. If $(a,b) = (1,1)$ then $\delta=\delta_1+\delta_2+\delta_3+\delta_4$ and one can exhibit $7$ distinct weights in $W_3$, while the weights in $W_2$ are $\delta-\beta_1-\beta_2$, $\delta-\beta_2-\beta_3$, $\delta-\beta_2-\beta_4$, $\delta-\beta_1-\beta_3$, $\delta-\beta_1-\beta_4$ and $\delta_1-\beta_3-\beta_4$. The first three weights in this list have multiplicity $2-\delta_{3,p}$, while the latter three have multiplicity $1$. Therefore $\dim W_2 = 9-3\delta_{3,p}$. Finally, if $b \neq 0$ and $(a,b) \neq (1,1)$ then it is easy to find seven distinct weights in both $W_2$ and $W_3$.

Suppose $a\neq c$.  If $b \neq 0$ or $c \neq 0$ then it is easy to exhibit three distinct weights in $W_2$ and $W_3$, so we may assume that $\delta=a\delta_1$ with $a \ge 3$ (recall that we may assume $\delta \neq 2\delta_1$). Here $\dim W_2=2$ since the weights in $W_2$ are $\delta-\beta_1-\beta_2$ and $\delta-2\beta_1$, each with multiplicity 1, and we have $\dim W_3 \geq 3$ since $\delta-\beta_1-\beta_2-\beta_3$, $\delta-\beta_1-\beta_2-\beta_4$ and $\delta-2\beta_1-\beta_2$ are weights in $W_3$. 

To complete the proof of the proposition, it suffices to show that if $a=c$ then there are at least seven distinct weights in $W_i$ for all $4\leq i\leq \ell/2$, and if $a \neq c$ then there are at least three  distinct weights in $W_i$ for $4\leq i <\ell/2$ and at least five in $W_{\ell/2}$. In view of Lemma \ref{l:pr}, it suffices to work in the corresponding Weyl module $W_X(\delta)$. As in the statement of Proposition \ref{p:d4_sub}, set $d=a+c$ and define  $\mu=\delta_1+\delta_3+\delta_4$ if $d$ is even, otherwise $\mu = \delta_2+\delta_3+\delta_4$. By Proposition \ref{p:d4_sub}, $\mu$ is subdominant to $\delta$. We proceed by induction on the height ${\rm ht}(\delta-\mu)$ of $\delta-\mu$. If ${\rm ht}(\delta-\mu)=0$ then the result follows from Lemmas \ref{l:d4borelredbcone} and \ref{l:d4borelredbctwo}, so we may  assume that ${\rm ht}(\delta-\mu)>0$.

First assume that $b \le 1$, so $(a,c) \neq (0,0)$. Suppose $a \neq 0$ and $c=0$, so $a \ge 2$ since we have already handled the case $\delta = \delta_1+\delta_2$. Set 
$$\nu=\delta-\beta_1=(a-2)\delta_1+(b+1)\delta_2.$$ 
Then $\nu$ is subdominant to $\delta$, $\mu$ is subdominant to $\nu$ (by Proposition \ref{p:d4_sub}) and we have ${\rm ht}(\nu-\mu)<{\rm ht}(\delta-\mu)$. By induction and Lemma \ref{l:indwt} we see that the desired result holds for all $i\geq 5$, and it is easy to check that 
$W_4$ has at least $3$ distinct weights. Now assume $c \neq 0$.  If $(a,b,c) = (2,0,1)$ then Lemma \ref{l:d4borelredbctwo} applies, so let us assume otherwise. Set 
$$\nu=\delta-\beta_2-\beta_3-\beta_4=(a+1)\delta_1+b\delta_2+(c-1)(\delta_3+\delta_4).$$ 
As before, $\nu$ is subdominant to $\delta$,  $\mu$ is subdominant to $\nu$ and ${\rm ht}(\nu-\mu)<{\rm ht}(\delta-\mu)$, so by using induction and Lemma \ref{l:indwt} we reduce to the case $4 \le i \le 6$. It is straightforward to find sufficiently many weights in $W_4$, $W_5$ and $W_6$. For example, if $a=c$ then 
$\delta-\beta_1-\beta_2-\beta_3-\beta_4$, $\delta-\beta_1-2\beta_2-\beta_3$, $\delta-\beta_1-2\beta_2-\beta_4$, $\delta-2\beta_2-\beta_3-\beta_4$, $\delta-2\beta_1-\beta_2-\beta_3$, $\delta-2\beta_1-\beta_2-\beta_4$, $\delta-\beta_2-\beta_3-2\beta_4$
are 7 distinct weights in $W_4$. 

To complete the proof we may assume that $b \ge 2$. Set 
$$\nu=\delta-\beta_2=(a+1)\delta_1+(b-2)\delta_2+(c+1)(\delta_3+\delta_4)$$ 
and observe that $\nu$ is subdominant to $\delta$, $\mu$ is subdominant to $\nu$ and ${\rm ht}(\nu-\mu)<{\rm ht}(\delta-\mu)$. Arguing in the usual way (using induction and Lemma \ref{l:indwt}), we deduce that the desired result holds for all $i\geq 5$, and it is straightforward to show that $W_4$ has enough distinct weights. 
\end{proof}

\section{The case $H=D_4.2$}\label{ss:d42}

Let $(G,H,V)$ be an irreducible triple satisfying Hypothesis \ref{h:our}, where $H = D_4\la t \ra = D_4.2$ and $t$ is an involutory graph automorphism of $X=H^0=D_4$. As before, let 
$\l = \sum_{i}a_i\l_i$ be the ($p$-restricted) highest weight of the irreducible $KG$-module $V=V_G(\l)$. Since $V|_{H}$ is irreducible and $V|_{X}$ is reducible, we have 
\begin{equation}\label{e:v1v222}
V|_{X} = V_1 \oplus V_2,
\end{equation}
where the $V_i$ are non-isomorphic irreducible $KX$-modules interchanged by $t$. In particular, if $\mu_i$ denotes the $T_X$-highest weight of $V_i$ then we may assume 
$$\mu_1 = \l|_{X} = \sum_{i=1}^4c_i\delta_i,\;\; \mu_2 = c_1\delta_1+c_2\delta_2+c_4\delta_3+c_3\delta_4$$
and thus 
\begin{equation}\label{e:d41}
\mu_2-\mu_1 = \frac{1}{2}(c_4-c_3)(\b_3-\b_4).
\end{equation}
By applying Lemma \ref{l:main} and Proposition \ref{p:d4b}, we immediately deduce that $\delta$ is one of the following:
$$a\delta_1 \, (a \ge 2), \; \delta_2, \; a\delta_{1}+b\delta_2\, (ab \neq 0),\; c(\delta_3+\delta_4).$$ 

\begin{rmk}\label{r:d4rem}
Note that in the following proof of Theorem \ref{t:d4} when $H=D_4.2$, we do \emph{not} use the fact that $V_1$ and $V_2$ are non-isomorphic $KX$-modules. This intentional feature of the proof will be important in Section \ref{ss:d4s3}, when we deal with the case $H=D_4.S_3$.
\end{rmk}

\begin{lem}\label{l:d42_1}
If $\delta = a\delta_1$ $(a \ge 2)$ then $(G,H,V) \not\in \mathcal{T}$.
\end{lem}

\begin{proof}
Seeking a contradiction, let us assume $(G,H,V) \in \mathcal{T}$. In view of the bounds on the dimensions of the $U_X$-levels of $W$ given in part (a) of Proposition \ref{p:d4b}(ii), an application of Lemma \ref{l:main} yields $a_3=1$. Further, by appealing to Remark \ref{r:ord}, we may choose an ordering of the $T$-weights in the $U_X$-levels $0,1$ and $2$ to give $\a_3|_{X} = \b_1-\b_2$, whence $\l-\a_3 \in \L(V)$ restricts to the $T_X$-weight $\nu = \mu_1-\b_1+\b_2$. Clearly, $\nu$ does not occur in $V_1$. Suppose $\nu$ occurs in $V_2$. Then $\nu = \mu_2 - \sum_{i}k_i\b_i$ for some non-negative integers $k_i$, whence
$$\mu_2-\mu_1 = (k_1-1)\b_1+(k_2+1)\b_2+k_3\b_3+k_4\b_4,$$
which contradicts \eqref{e:d41}. Therefore $\nu$ does not occur in $V_1$ nor $V_2$, which is a contradiction since $V=V_1\oplus V_2$.
\end{proof}

\begin{lem}\label{l:d42_2}
If $\delta = \delta_2$ then $(G,H,V) \not\in \mathcal{T}$.
\end{lem}

\begin{proof}
Here $G = D_{14-\delta_{2,p}}$ (see \cite[Table 2]{Brundan}) and the dimensions of the 
$U_X$-levels are given in Lemma \ref{l:d4delta2}. Suppose there is an irreducible 
triple $(G,H,V) \in \mathcal{T}$. Let $P=QL$ be the parabolic subgroup of $G$ 
constructed in the usual way from the $U_X$-levels of $W$. If $p=2$ then
 Lemma \ref{l:d4delta2} implies that $L'$ does not have an $A_1$ factor 
(see Remark \ref{r:a1factor}), which contradicts Lemma \ref{l:main}. 
For the remainder we may assume $p \neq 2$, so $G=D_{14}$. Here Lemma \ref{l:main} implies 
that $1 \in \{a_{13},a_{14}\}$, and without loss of generality we will assume that $a_{13}=1$. Moreover,
$a_j=0$ for $j\in\{3,4,6,7,9,10,11,14\}$.

Let $W_i$ denote the $i$-th $U_X$-level of $W$ and note that $\ell=10$. By suitably ordering the weights in the 
levels $W_0, \ldots, W_4$ (see Remark \ref{r:ord}) we obtain the root restrictions recorded in Table \ref{t:d4root}. 
Moreover, since $\delta - \b_1-2\b_2-\b_3-\b_4 = 0$ is the only weight in $W_5$, we note that the 
weights  $\nu = 
\lambda-\sum_{i=1}^{12}\a_i$, $\nu-\alpha_{13}$, $\nu-\alpha_{14}$ and $\nu-\alpha_{13}-\alpha_{14}$ 
all restrict to the zero-weight. Hence, $\alpha_{12}|_X = \beta_1$ and $\a_{13}|_{X}=\a_{14}|_{X}=0$. 

\renewcommand{\arraystretch}{1.2}
\begin{table}
$$\begin{array}{llll} \hline
\mbox{$U_X$-level} & \mbox{$T_X$-weight} & \mbox{$T$-weight} & \mbox{Root restriction} \\ \hline
0 & \delta& \lambda_1& \\
1 & \delta-\beta_2 & \lambda_1-\alpha_1& \alpha_1|_{X}=\beta_2 \\
2 &  \delta-\beta_1-\beta_2 & \lambda_1-\alpha_1-\alpha_2 & \alpha_2|_{X}=\beta_1 \\
& \delta-\beta_2 - \beta_3 & \lambda_1-\alpha_1-\alpha_2 - \a_3 & \alpha_3|_{X}=\beta_3-\beta_1 \\
& \delta-\beta_2 - \beta_4 & \lambda_1-\sum_{i=1}^4\alpha_i & \alpha_4|_{X}=\beta_4-\b_3  \\
3 & \delta-\beta_1 - \b_2 - \b_3 & \lambda_1-\sum_{i=1}^5\alpha_i & \alpha_5|_{X}=\beta_1+\b_3-\beta_4 \\
 & \delta-\beta_1-\beta_2 - \b_4 & \lambda_1-\sum_{i=1}^6\alpha_i& \alpha_6|_{X}=\beta_4-\b_3 \\
& \delta-\beta_2-\beta_3 - \b_4 & \lambda_1-\sum_{i=1}^7\alpha_i& \alpha_7|_{X}=\beta_3-\beta_1 \\
4 & \delta-\beta_1-\beta_2- \b_3 - \b_4 & \lambda_1-\sum_{i=1}^{8}\alpha_i& \alpha_{8}|_{X}=\beta_1 \\
& \delta-\beta_1-2\beta_2 - \b_3 & \lambda_1-\sum_{i=1}^{9}\alpha_i& \alpha_{9}|_{X}=\beta_2-\beta_4 \\
&  \delta-\beta_1-2\beta_2 - \b_4  & \lambda_1-\sum_{i=1}^{10}\alpha_i & \alpha_{10}|_{X}=\beta_4-\beta_3 \\
&  \delta-2\beta_2-\beta_3 - \b_4  & \lambda_1-\sum_{i=1}^{11}\alpha_i & \alpha_{11}|_{X}=\beta_3-\beta_1 \\ \hline 
\end{array}$$
\caption{}
\label{t:d4root}
\end{table}
\renewcommand{\arraystretch}{1}

Let $P_X=Q_XL_X$ be the $t$-stable parabolic subgroup of $X$ with $\Delta(L_X')=\{\beta_2\}$, and let $P=QL$ be the 
corresponding parabolic subgroup of $G$. Then $L'=L_1 \cdots L_r$ and $V/[V,Q]=M_1 \otimes \cdots \otimes M_r$, 
where each $L_i$ is simple with natural module $Y_i=W_{i-1}$ (where $W_i$ denotes the $i$-th $Q_X$-level of $W$), and
 each $M_i$ is a $p$-restricted irreducible $KL_i$-module. Here $r=4$ and we have $L_1=A_1$, $L_2=A_2$, $L_3=A_5$ and 
$L_4=D_3$. We also observe that
$Y_3|_{L_X'} = U \oplus U \oplus U$, where $U$ is the natural $2$-dimensional module for $L_X'=A_1$.  
Furthermore, $Y_4|_{L_X'}$ is the direct sum of a $3$-dimensional trivial module and the irreducible $3$-dimensional 
$KL_X'$-module of highest weight $2\delta_2|_{L_X'}$.
Viewing $D_3={\rm SO}(Y_4)$ as an image of $A_3={\rm SL}(M)$, we see that $Y_4=\Lambda^2(M)$ as a $KA_3$-module. 
In particular, the action of $L_X'$ on $Y_4$ implies that some central extension of $L_X'$ acts on 
$M$ as the sum of two $2$-dimensional natural modules for $A_1$.
By applying Lemmas \ref{l:a1a3}(i) and \ref{l:a1a5} we deduce that $a_{8}=a_{12}=0$; 
if not, then $M_3|_{L_X'}$ or $M_4|_{L_X'}$ has at least three composition factors, which is a contradiction. 

Now $\l$ and $\l-\a_{13}$ are distinct weights in $V$ (recall that $a_{13} =1$) which restrict to the same 
$T_X$-weight $\mu_1$. Since $\mu_1$ occurs with multiplicity $1$ in $V_1$, the second occurrence of $\mu_1$ must
be as a weight in $V_2$. Then \eqref{e:d41} implies that $\mu_2 = \mu_1$. Note that $a_5= 0$, 
since otherwise $\lambda-\alpha_5\in\Lambda(V)$, but $(\lambda-\alpha_5)|_X = \mu_1-\beta_1-\beta_3+\beta_4$, which 
does not lie under $\mu_1=\mu_2$. So we have reduced to the case 
$\lambda = a_1\lambda_1+a_2\lambda_2+\lambda_{13}$.

We now turn to the $t$-stable parabolic subgroup $P_X=Q_XL_X$ of $X$ with 
$\Delta(L_X')=\{\beta_2,\beta_3,\beta_4\}$, and we define $P=QL$ to be the 
corresponding parabolic subgroup of $G$. Then $L'=L_1L_2$ and $V/[V,Q]=M_1 \otimes M_2$, 
where each $L_i$ is simple with natural module $Y_i=W_{i}$ (where $W_i$ denotes the $i$-th $Q_X$-level of $W$), and
 each $M_i$ is a $p$-restricted irreducible $KL_i$-module. We have $L_1=A_5$, $L_2=D_8$, and 
$Y_1|_{L_X'}$ is the irreducible $KL_X'$-module with highest weight $\delta_2|_{L_X'}$ and $Y_2$ is reducible as a $KL_X'$-module. If we consider the configuration $(L_X',L_2,M_2)$ then Corollary \ref{c:g51} implies that $M_2$ is a reducible $KL_X'$-module
and so $M_1$ must be irreducible. By \cite[Theorem 2]{Seitz2}, we see that the irreducible configuration 
$(L_X',L_1,M_1)$ occurs if and only if $M_1$ is trivial or $M_1 \in \{Y_1, Y_1^*, \L^2(Y_1),
\L^2(Y_1)^*\}$. Given that $\lambda = a_1\lambda_1+a_2\lambda_2+\lambda_{13}$,
we deduce that $\lambda = \lambda_1+\lambda_{13}, \lambda_2+\lambda_{13}$ or $\lambda_{13}$.

For the case $\lambda = \lambda_1+\lambda_{13}$, we return to the consideration of the parabolic subgroup $P_X$ 
with $\Delta(L_X') = \{\beta_2\}$. Here $M_1$ is the $2$-dimensional natural module for $L_X'$, while
$M_4|_{L_X'} = U_1 \oplus U_2$ is the sum of two $p$-restricted $2$-dimensional irreducible $KL_X'$-modules, and $M_j|_{L_X'}$ is trivial for $j=2,3$. By Proposition \ref{p:s16}, the $KL_X'$-modules $M_1 \otimes U_i$ are reducible, which implies that the $KL_X'$-module $V/[V,Q] = M_1\otimes M_4$ has at least four composition factors. This is a contradiction.

Next suppose $\lambda = \lambda_2+\lambda_{13}$.  Here we may use the above root restrictions to show that $\mu_1 = 2(\delta_1+\delta_3+\delta_4)+\delta_2$, so $\mu_1$ is $p$-restricted. 
The weight $\lambda-\sum_{i=2}^{13}\alpha_i$
 has multiplicity
at least 11 (see Lemma 2.4), as does its conjugate $\lambda-\sum_{i=2}^{14}\alpha_i$. 
So the weight $(\lambda-\sum_{i=2}^{13}\alpha_i)|_X = \mu_1-\beta_1-\beta_2-\beta_3-\beta_4$ has multiplicity at least 22 in $V$. But
 a direct application of the PBW-basis theorem (see \cite[Section 17.3]{Hu1})
shows that this weight has multiplicity at most $8$ in each
of the irreducible modules $V_1,V_2$, leading to a contradiction. 

Finally, let us assume that $\lambda = \lambda_{13}$. Now we may write $\lambda$ as a linear 
combination of the simple roots $\alpha_i$ and conclude that $\lambda|_{X} = \delta_1+\delta_2+\delta_3+\delta_4$. Using the method explained in \cite{Lubeck}, Frank L\"{u}beck (personal communication) has calculated that $\dim V_1 = 1841, 2451, 3797$ when $p=3,5,7$, respectively (see \cite[Table A.41]{Lubeck} for the case $p=3$), and $\dim V_1 = 4096$ in all other cases. In particular, if $p=3,5,7$ then $2\dim V_1<\dim V$, and in the remaining cases we have $\dim V = 2^{13} = 2\dim V_1$.

To see that the configuration $\lambda = \lambda_{13}$ does not give rise to an example $(G,H,V)$ in 
$\mathcal{T}$ when $p \neq 2,3,5,7$, we first note that the group $D_4.2$ does not lie in the simple group $D_{14}$. Indeed, since $W$ is the Lie algebra of $X$,
we may study the action of $D_4.2$ on $W$ via the conjugation action on the Lie algebra $\mathcal{L}(X)$. It is then a straightforward calculation to see that an element in ${\rm GO}_8\setminus {\rm SO}_8$ acts with determinant $-1$ on $W$. 
Hence we have $D_4.2$ as a subgroup of $D_{14}.2$. Moreover, since the weight $\lambda_{13}$ is not invariant under the graph automorphism of $D_{14}$, this representation does not extend to a representation of $D_{14}.2$. Hence, we do not have a configuration satisfying Hypothesis \ref{h:our}.
\end{proof}

\begin{rmk}\label{r:d14}
We wish to make note of the following unique situation arising in the proof of the previous lemma, even though this does not give rise to an example in the statement of our main theorem. The result follows from the proof of Lemma \ref{l:d42_2}:

\begin{quote}
\emph{Suppose $X=D_4$ and $\delta=\delta_2$ with $p\ne 2,3,5,7$, so $G = D_{14}$. Let $V$ be one of the two 
spin modules for $G$. Then $V|_{X}$ is the sum of two irreducible isomorphic $KX$-modules, each having highest weight $\delta_1+\delta_2+\delta_3+\delta_4$.}
\end{quote}
\end{rmk}

\begin{lem}\label{l:d42_3}
If $\delta = a\delta_1+b\delta_2$ $(ab \neq 0)$ then $(G,H,V) \not\in \mathcal{T}$.
\end{lem}

\begin{proof}
Suppose $(G,H,V) \in \mathcal{T}$. By applying Proposition \ref{p:d4b}(ii)(b) and 
Lemma \ref{l:main} we deduce that $a_2=1$. By suitably ordering the weights in the $U_X$-levels 
$W_0$ and $W_1$ (see Remark~\ref{r:ord}), we have that $\alpha_1|_X = \beta_1$ and $\alpha_2|_X = \beta_2-\beta_1$.
Since $a_2=1$, $\lambda-\alpha_2\in\Lambda(V)$ and so $\mu_1-\beta_2+\beta_1\in\Lambda(V)$.
Since this weight does not occur in $V_1$, it must occur in $V_2$ and so $\mu_1-\beta_2+\beta_1
=\mu_2-\sum_{i=1}^4 k_i\beta_i$ for non-negative integers $k_i$. But then
we have 
$$\mu_2-\mu_1 = (k_1+1)\b_1+(k_2-1)\b_2+k_3\b_3+k_4\b_4,$$ 
which contradicts \eqref{e:d41}.
\end{proof}

\begin{lem}\label{l:d42_4}
If $\delta = c(\delta_3+\delta_4)$ then $(G,H,V) \not\in \mathcal{T}$.
\end{lem}

\begin{proof}
Let us assume $(G,H,V) \in \mathcal{T}$. As in the proof of the previous lemma, 
we deduce that $a_2=1$. By suitably ordering the weights in the $U_X$-levels 
$W_0$ and $W_1$ (see Remark~\ref{r:ord}), we have that $\alpha_1|_X = \beta_3$ and $\alpha_2|_X = \beta_4-\beta_3$.
Since $a_2=1$, $\lambda-\alpha_2\in\Lambda(V)$ and so $\mu_1-\beta_4+\beta_3\in\Lambda(V)$. 
Since this weight does not occur in $V_1$, it must occur in $V_2$ and so there
exist non-negative integers $k_i$ such that $\mu_1-\beta_4+\beta_3
=\mu_2-\sum_{i=1}^4 k_i\beta_i$. Comparing with \eqref{e:d41}, we deduce that 
$\mu_2 = \mu_1-\beta_4
+\beta_3\ne\mu_1$. In particular, we have that $c_3\ne c_4$.

Now let $P_X = Q_XL_X$ be the parabolic subgroup of $X$ with $L_X'=A_3$ and 
$\Delta(L_X') = \{\b_1,\b_2,\b_3\}$, let $P=QL$ be the corresponding parabolic subgroup of 
$G$ and define the $L_i,M_i,Y_i$ as before. Note that $Y_1 = W/[W,Q_X]$ is an irreducible 
$KL_X'$-module with highest weight $c\delta_3|_{L_X'}$, so 
$$\dim Y_1 =  \frac{1}{6}(c+1)(c+2)(c+3)$$
(see \cite[1.14]{Seitz2}). By Lemma \ref{l:asym}, which applies since $c_3\ne c_4$, 
 $V/[V,Q]$ is an irreducible $KL_X'$-module. 
As before, the triple $(L_X',L_1,M_1)$ must be one of the cases in \cite[Table 1]{Seitz2}, and
 we see that the only possibility is the case labelled ${\rm I}_{7}$ with $n=3$, so $c=2$, 
$p \neq 2$ and $M_1 = \L^2(Y_1)$ as a $KL_1$-module. 

Let us now switch to the $t$-stable parabolic subgroup $P_X=Q_XL_X$ of $X$ with $\Delta(L_X')=\{\beta_1,\beta_3,\beta_4\}$, and 
let $P=QL$ be the corresponding parabolic subgroup of $G$.
Define the $W_i,L_i,M_i,Y_i$ and $\pi_i$ in the usual  way, and note that $\dim Y_1=\dim W_0= 9$ so $L_1=A_8$.
In addition, we note that $Y_1|_{L_X'}$ is irreducible of highest weight $(2\delta_3+2\delta_4)|_{L_X'}$, so $\pi_1(L_X')=A_1A_1$. 
Therefore, Lemma \ref{l:a1a1am} implies that $M_1|_{L_X'}$ has at least three composition factors, which is a contradiction 
since $V/[V,Q]$ has precisely two composition factors as a $KL_X'$-module.
\end{proof}

\vs

This completes the proof of Theorem \ref{t:d4} in the case $H=D_4.2$.

\section{The case $H=D_4.3$}\label{ss:d43}

Next we establish Theorem \ref{t:d4} in the case where $H=X\langle s\rangle = X.3$, with $s$ a triality graph automorphism of $X = D_4$. As usual, let $V=V_G(\l)$ be a rational $p$-restricted irreducible tensor indecomposable $KG$-module with highest weight $\lambda=\sum_i a_i \lambda_i$, and assume $V|_H$ is irreducible, but $V|_X$ is reducible, so  
$$V|_X=V_1\oplus V_2\oplus V_3,$$ 
where the $V_i$ are non-isomorphic 
irreducible $KX$-modules permuted by $s$ (see Proposition \ref{p:niso}). Also recall that $\delta$ denotes the highest weight of $W$ as a $KX$-module, which is $s$-stable since $s$ acts on $W$. Therefore $\delta$ is of the form
\begin{equation}\label{e:deltad4}
\delta = a(\delta_1+\delta_3+\delta_4)+b\delta_2.
\end{equation}
Also note that if $P_X = Q_XL_X$ is a parabolic subgroup of $X$ and $P=QL$ is the corresponding parabolic subgroup of $G$ (constructed using the $Q_X$-levels of $W$ in the usual way) then $V/[V,Q]$ is a quotient of
$$V/[V,Q_X] = V_1/[V_1,Q_X] \oplus V_2/[V_2,Q_X] \oplus V_3/[V_3,Q_X],$$
where each summand is an irreducible $KL_X'$-module. In particular, $V/[V,Q]$ has at most three composition factors as a $KL_X'$-module.

The next result provides an analogue of Lemma \ref{l:main}.

\begin{lem}\label{l:maind4_3}
Let $P=QL$ be the parabolic subgroup of $G$ constructed from an $s$-stable Borel subgroup of $X$ as in Lemma \ref{l:flag}. Write $L'=L_1\cdots L_r$ as a product of simple factors. 
Then exactly one of the following holds:
\begin{itemize}\addtolength{\itemsep}{0.3\baselineskip}
\item[{\rm (i)}]  There exists $1\leq i\leq r$ such that $L_i = A_1$, $\Delta(L_i)=\{\alpha\}$, $\langle \lambda, \alpha \rangle=2$ and $\langle \lambda,\beta\rangle=0$ for all $\beta \in \Delta(L')\backslash\{\alpha\}$.
\item[{\rm (ii)}] There exists $1 \leq i \leq r$ such that $L_i = A_2$, $\Delta(L_i)=\{\alpha,\beta\}$,
$\{\langle \lambda,\alpha\rangle,\langle \lambda,\beta\rangle\}=\{0,1\}$
and $\langle \lambda, \gamma\rangle=0$ for all $\gamma \in \Delta(L')\backslash\{\alpha,\beta\}$.
\end{itemize}
\end{lem}

\begin{proof}
This follows from \cite[Lemma 6.2]{Ford2}.
\end{proof}

\begin{rmk}\label{r:borel1}
It will be useful to interpret Lemma \ref{l:maind4_3} in terms of the $U_X$-levels of $W$, where $B_X=U_XT_X$ is an $s$-stable Borel subgroup of $X$. Indeed, if $W_i$ denotes the $i$-th $U_X$-level of $W$, and $\ell$ is the $U_X$-level of the lowest weight $-\delta$, then Lemma \ref{l:maind4_3} implies that one of the following holds:
\begin{itemize}\addtolength{\itemsep}{0.3\baselineskip}
\item[{\rm (i)}] There exists $1 \le i < \ell/2$ such that $\dim W_i = 2$ or $3$;
\item[{\rm (ii)}] $\ell$ is even, $G$ is symplectic and $\dim W_{\ell/2}=2$;
\item[{\rm (iii)}] $\ell$ is even, $G$ is orthogonal and $\dim W_{\ell/2}=3$ or $4$.
\end{itemize}
\end{rmk}

In particular, by appealing to Proposition \ref{p:d4b} and Lemma \ref{l:maind4_3}, we may assume that $\delta=a(\delta_1+\delta_3+\delta_4)$ or $\delta_2$.

\begin{lem}\label{l:d411}
If $\delta = a(\delta_1+\delta_3+\delta_4)$, $a \neq 0$, then $(G,H,V) \not\in \mathcal{T}$.
\end{lem}

\begin{proof}
Seeking a contradiction, suppose that $(G,H,V) \in \mathcal{T}$.  
Let $B_X=U_XT_X$  be an $s$-stable Borel subgroup of $X$ and let  $P=QL$ be the corresponding parabolic subgroup of $G$.  By applying Proposition \ref{p:d4b} and Lemma \ref{l:maind4_3} we deduce that $\{a_2,a_3\}=\{0,1\}$ and $a_i=0$ for all $5 \le i \le 9$.  

Next let $P_X=Q_XL_X$ be the parabolic subgroup of $X$ with $\Delta(L_X')=\{\beta_2\}$, and let $P=QL$ be the parabolic subgroup of $G$ constructed in the usual way. As usual, we write $L'=L_1\cdots L_r$ and $V/[V,Q]=M_1\otimes \cdots \otimes M_r$, where each $L_i$ is simple with natural module $Y_i$, and each $M_i$ is a $p$-restricted irreducible $KL_i$-module. Let $W_i$ denote the $i$-th $Q_X$-level of $W$ and note that $Y_1=W_1$ and $\dim W_1=6$, so  $L_1=A_5$. 
In addition, we have $Y_1|_{L_X'} = U \oplus U \oplus U$, where $U$ is the natural $2$-dimensional module for $L_X'=A_1$. 
Now, if $a_3=1$, or if $a_2=1$ and $a_4\neq 0$, then Lemma \ref{l:a1a5} implies that 
$M_1|_{L_X'}$ has at least four composition factors, contradicting the fact that $V/[V,Q]$ has at most three composition factors as a $KL_X'$-module. It follows that $a_2=1$ and $a_i=0$ for all $3\leq i \leq 9$.

To complete the proof, we switch to the $s$-stable parabolic subgroup $P_X=Q_XL_X$ of $X$ with $\Delta(L_X')=\{\beta_1,\beta_3,\beta_4\}$. Let $P=QL$ be the parabolic subgroup of $G$ constructed from $P_X$, and define the $L_i,M_i,W_i$ and $Y_i$ in the usual way. Note that $Y_1=W_0$ and $\dim W_0\geq 8$, so $L_1=A_m$ with $m \geq 7$. Now $Y_1|_{L_X'}$ is irreducible, with highest weight $(a\delta_1+a\delta_3+a\delta_4)|_{L_X'}$, so we may view $L_X'=A_1A_1A_1$ as a subgroup of $L_1$, and thus Lemma \ref{l:a1a1a1am} implies that $M_1|_{L_X'}$ has more than three composition factors. Once again, this is a contradiction.
\end{proof}

\begin{lem}\label{l:d412}
If $\delta = \delta_2$ then $(G,H,V) \not\in \mathcal{T}$. 
\end{lem}

\begin{proof}
Here $G=D_{14-\delta_{2,p}}$ (see \cite[Table 2]{Brundan}). Seeking a contradiction, let us assume $(G,H,V) \in \mathcal{T}$. First we deal with the case $p \neq 2$. In the usual way, by considering the dimensions of the $U_X$-levels of $W$ given in Proposition \ref{p:d4b} and by applying Lemma \ref{l:maind4_3}, we see that one of the following holds:
\begin{itemize}\addtolength{\itemsep}{0.3\baselineskip}
\item[(i)] $a_i=0$ for $i \in\{3,4,6,7,9,10,11\}$, $\{a_{13},a_{14}\}=\{0,2\}$.
\item[(ii)] $a_i=0$ for $i \in\{3,4,9,10,11,13,14\}$, $\{a_6,a_7\}=\{0,1\}$.
\item[(iii)] $a_i=0$ for $i \in \{6,7,9,10,11,13,14\}$, $\{a_3,a_4\}=\{0,1\}$.
\end{itemize}

The first step is to reduce to the cases $\{a_3,a_4\}=\{0,1\}$ or $\{a_6,a_7\}=\{0,1\}$. To do this, we argue as in the proof of Lemma \ref{l:d42_2}, working with the parabolic subgroup 
$P_X=Q_XL_X$ of $X$ with $\Delta(L_X')=\{\beta_2\}$. Let $P=QL$ be the corresponding parabolic subgroup of $G$. Then $L'=L_1 \cdots L_r$ and $V/[V,Q]=M_1 \otimes \cdots \otimes M_r$, where each $L_i$ is simple with natural module $Y_i=W_{i-1}$ (where $W_i$ denotes the $i$-th $Q_X$-level of $W$), and each $M_i$ is a $p$-restricted irreducible $KL_i$-module. Now $r=4$ and we have $L_1=A_1$, $L_2=A_2$, $L_3=A_5$ and $L_4=D_3$. We also observe that
$Y_3|_{L_X'} = U \oplus U \oplus U$, where $U$ is the natural $2$-dimensional module for $L_X'=A_1$.  Furthermore, $Y_4|_{L_X'}$ is the direct sum of a $3$-dimensional trivial module and the irreducible $3$-dimensional $KL_X'$-module of highest weight $2\delta_2|_{L_X'}$.
We may view $D_3={\rm SO}(Y_4)$ as an image of $A_3={\rm SL}(M)$, in which case  $Y_4=\Lambda^2(M)$ as a $KA_3$-module. In particular, the action of $L_X'$ on $Y_4$ implies that some central extension of $L_X'$ acts on $M$ as the sum of two $2$-dimensional natural modules for $A_1$.
By applying Lemma \ref{l:a1a3}(i) we deduce that $a_{12}=a_{13}=a_{14}=0$; if not, then $M_4|_{L_X'}$ has at least four composition factors, which is a contradiction. 
We have now reduced to the cases $\{a_3,a_4\}=\{0,1\}$ or $\{a_6,a_7\}=\{0,1\}$. 

Suppose that $\{a_6,a_7\}=\{0,1\}$.
Let $P_X=Q_XL_X$ be the parabolic subgroup of $X$ with $\Delta(L_X')=\{\beta_1,\beta_3,\beta_4\}$ and let $P=QL$ be the parabolic subgroup of $G$ constructed from $P_X$. Then $L'=L_1L_2$ and $V/[V,Q]=M_1\otimes M_2$, where $L_1= A_7$, $L_2=D_5$ and each $M_i$ is a $p$-restricted irreducible $KL_i$-module. Define the $W_i$ and $Y_i$ in the usual way and note that 
$Y_i=W_i$ for $i=1,2$. 
Now $Y_1|_{L_X'}$ is the tensor product of three $2$-dimensional natural $KA_1$-modules, so Lemma \ref{l:a1a1a1am} implies that $M_1|_{L_X'}$ has more than three composition factors. Once again, this is a contradiction. An entirely similar argument applies if $\{a_3,a_4\}=\{0,1\}$.

This completes the proof of the lemma when $p \neq 2$, so for the remainder we may assume $p=2$ (so $G=D_{13}$). By applying Proposition \ref{p:d4b} and Lemma \ref{l:maind4_3}, we reduce to one of the following two cases:
\begin{itemize}\addtolength{\itemsep}{0.3\baselineskip}
\item[(i)] $a_i=0$ for $i \in\{3,4,9,10,11\}$, $\{a_{6},a_{7}\}=\{0,1\}$.
\item[(ii)] $a_i=0$ for $i \in\{6,7,9,10,11\}$, $\{a_3,a_4\}=\{0,1\}$.
\end{itemize}
We now complete the proof as in the case $p \neq 2$, working with the $s$-stable parabolic subgroup  $P_X=Q_XL_X$ with $\Delta(L_X')=\{\beta_1,\beta_3,\beta_4\}$.
\end{proof}

\vs

This completes the proof of Theorem \ref{t:d4} in the case $H = D_4.3$.

\section{The case $H=D_4.S_3$}\label{ss:d4s3}

To complete the proof of Theorem \ref{t:d4} it remains to deal with the case $H=X\langle s,t\rangle = X.S_3$, where $s$ is a triality graph automorphism of $X=D_4$ cyclically permuting the simple roots $\{\b_1,\b_3, \b_4\}$, and $t$ is an involutory graph automorphism interchanging $\beta_3$ and $\beta_4$. Since we are assuming $s$ and $t$ both act on $W$, it follows that $\delta$ has the form given in \eqref{e:deltad4}. As usual, we take $V=V_G(\l)$ to be a rational  $p$-restricted irreducible tensor indecomposable $KG$-module with highest weight $\lambda=\sum_i a_i\lambda_i$, with the property that $V|_H$ is irreducible but $V|_X$ is reducible. 

Let $N$ be the normal subgroup $D_4.3$ of $H$ and consider the restriction of $V$ to $N$. 
If $V|_N$ is irreducible, then the conclusion to Theorem \ref{t:d4} follows from the previous analysis of the $D_4.3$ case in Section \ref{ss:d43}. Therefore, we may assume that $V|_N$ is reducible. Then by
Clifford theory we have $V|_N = V_1'\oplus V_2'$, where $V_1'$ and 
$V_2'$ are non-isomorphic irreducible $KN$-modules, interchanged by the 
involutory graph automorphism of $D_4$ (these modules are non-isomorphic by Proposition \ref{p:niso}).

Now, if $V_1'|_{X}$ and $V_2'|_{X}$ are irreducible then $V|_{X}=V_1'|_X \oplus V_2'|_X$ but we cannot assume $V_1'|_X,V_2'|_X$ are non-isomorphic as $KX$-modules. However, in our earlier analysis of the case $D_4.2$ in Section \ref{ss:d42}, we do \emph{not} assume that the modules $V_1$ and $V_2$ in \eqref{e:v1v222} are non-isomorphic (see Remark \ref{r:d4rem}), so we can conclude that no examples arise in this situation. In particular, for the remainder of this section we may assume that $V_1'|_{X}$ and $V_2'|_{X}$ are reducible, so by a further application of Clifford theory we deduce that  
\begin{equation}\label{e:vx2}
V|_X=V_1\oplus V_2\oplus V_3\oplus V_4\oplus V_5\oplus V_6,
\end{equation}
where the $V_i$ are irreducible $KX$-modules permuted by $\langle s,t\rangle \cong S_3$ (note that we do not assume the $V_i$ are pairwise non-isomorphic). In particular, if  $P=QL$ is a parabolic subgroup of $G$, constructed in the usual way from a parabolic subgroup $P_X=Q_XL_X$ of $X$, then the irreducible $KL'$-module $V/[V,Q]$ has at most six composition factors as a $KL_X'$-module.

Let $\mu_i$ denote the highest weight of $V_i$. We may assume that  
$$\mu_1=\lambda|_{X}=c_1\delta_1+c_2\delta_2+c_3\delta_3+c_4\delta_4,$$ 
Furthermore, we may also assume that  
$$t(\mu_1)=c_1\delta_1+c_2\delta_2+c_4\delta_3+c_3\delta_4, \;\;  s(\mu_1)=c_3\delta_1+c_2\delta_2+c_4\delta_3+c_1\delta_4,$$ 
so
$$s^2(\mu_1)=c_4\delta_1+c_2\delta_2+c_1\delta_3+c_3\delta_4, \; \; st(\mu_1)=c_4\delta_1+c_2\delta_2+c_3\delta_3+c_1\delta_4$$ 
and
$$ts(\mu_1)=c_3\delta_1+c_2\delta_2+c_1\delta_3+c_4\delta_4.$$ 
Set $\sigma_0=1$, $\sigma_1=t$, $\sigma_2=ts$, $\sigma_3=st$, $\sigma_4=s^2$ and $\sigma_5=s$. By relabelling the $V_i$, if necessary, we may assume that $\mu_i=\sigma_{i-1}(\mu_1)$ for all $1\leq i \leq 6$.

\begin{lem}\label{l:mu1d4s3}
The following  assertions hold:
\begin{align*}
\sigma_1(\mu_1)-\mu_1 & =\frac12(c_3-c_4)(\beta_4-\beta_3) \\
\sigma_2(\mu_1)-\mu_1 & =\frac12(c_3-c_1)(\beta_1-\beta_3) \\
\sigma_3(\mu_1)-\mu_1 & =\frac12(c_4-c_1)(\beta_1-\beta_4) \\
\sigma_4(\mu_1)-\mu_1 & =\frac12(c_4-c_1)\beta_1+\frac12(c_1-c_3)\beta_3+\frac12(c_3-c_4)\beta_4 \\
\sigma_5 (\mu_1)-\mu_1 & =  \frac12(c_3-c_1)\beta_1+\frac12(c_4-c_3)\beta_3+\frac12(c_1-c_4)\beta_4
\end{align*}
\end{lem}

\begin{proof}
This is an easy calculation, given the expression of each fundamental weight $\delta_i$ in terms of the simple roots $\beta_i$ (see \eqref{e:d44}).
\end{proof} 

Recall that a $T_X$-weight $\nu$ is said to be \emph{under} a weight $\mu$ if there exist non-negative integers $d_i$ such that $\nu = \mu - d_1\b_1 - d_2\b_2 - d_3\b_3 - d_4\b_4$.

\begin{prop}\label{p:notunderd4}
If $i\in \{1,3,4\}$ then neither $\mu_1-\beta_2+\beta_i$ nor $\mu_1+\beta_2-\beta_i$ is under $\mu_j$, for any $1 \le j \le 6$.
\end{prop}

\begin{proof}
The case $j=1$ is clear. Suppose $\mu_1-\b_2+\b_1$ is under $\mu_2$. Then
$$\mu_2 -\mu_1 = (d_1+1)\b_1+(d_2-1)\b_2+d_3\b_3+d_4\b_4$$
for some non-negative integers $d_i$. This contradicts Lemma \ref{l:mu1d4s3}, which states that $\mu_2-\mu_1 = \frac{1}{2}(c_3-c_4)(\b_4-\b_3)$. An entirely similar argument applies in each of the remaining 
cases.
\end{proof}

\begin{lem}\label{l:maind4s3}
Let $P=QL$ be the parabolic subgroup of $G$ constructed from a Borel subgroup $B_X=U_XT_X$ of $X$ that is stable under $s$ and $t$. Write $L'=L_1\cdots L_r$ as a product of simple factors, so $V/[V,Q]  =M_1 \otimes \cdots \otimes M_r$ with each $M_i$ a $p$-restricted irreducible $KL_i$-module. Let $Y_i$ denote the natural $KL_i$-module. Then exactly one of the following holds:
\begin{itemize}\addtolength{\itemsep}{0.3\baselineskip}
\item[{\rm (i)}]  There exists $1 \leq i \leq r$ such that $L_i= A_5$, $M_i = Y_i$ or $Y_i^*$, and $M_j$ is trivial for all $j \neq i$;
\item[{\rm (ii)}]  There exists $1 \leq i \leq r$ such that $L_i= A_3$, $M_i = \L^2(Y_i)$, and $M_j$ is trivial for all $j \neq i$;
\item[{\rm (iii)}]  There exists $1 \leq i \leq r$ such that $L_i= A_2$, $M_i = S^2(Y_i)$ or $S^2(Y_i)^*$, $p \neq 2$ and $M_j$ is trivial for all $j \neq i$;
\item[{\rm (iv)}]  There exists $1 \leq i \leq r$ such that $L_i=A_1$, $M_i = S^5(Y_i)$, $p \neq 2,3,5$ and $M_j$ is trivial for all $j \neq i$;
\item[{\rm (v)}]  There exist $1 \leq i, j \leq r$, $i \neq j$, such that $L_i= A_2$, $L_j = A_1$, $M_i = Y_i$ or $Y_i^*$, $M_j = Y_j$ and $M_k$ is trivial for all $k \neq i,j$;
\item[{\rm (vi)}]  There exist $1 \leq i, j \leq r$, $i \neq j$, such that $L_i= L_j = A_1$, $M_i = S^2(Y_i)$, $M_j = Y_j$, $p \neq 2$ and $M_k$ is trivial for all $k \neq i,j$;
\item[{\rm (vii)}]  $L_r = C_3$, $M_r = Y_r$, and $M_i$ is trivial for all $i<r$.
\end{itemize}
\end{lem}

\begin{proof}
This follows from \cite[Lemma 6.2]{Ford2}.
\end{proof}

\begin{lem}
If $\delta \neq \delta_2$ then $(G,H,V) \not\in \mathcal{T}$.
\end{lem}

\begin{proof}
Seeking a contradiction, suppose $\delta  =a(\delta_1+\delta_3+\delta_4)+b\delta_2 \neq \delta_2$ and $(G,H,V) \in \mathcal{T}$. We consider the following cases in turn:
$$\mbox{(i) $a=0$, $b \ge 2$; \; (ii)  $a \neq 0$, $b=0$; \; (iii) $ab \neq 0$.}$$

First consider (i). Let $B_X=U_XT_X$ be a Borel subgroup of $X$ that is stable under $s$ and $t$. By appealing to Remark \ref{r:ord}, we may order the $T$-weights of $W$ to obtain the root restrictions recorded in Table \ref{t:d43}. First assume $b>2$. Then Proposition \ref{p:d4b} and Lemma \ref{l:maind4s3} imply that $a_4=1$, so $\nu=\lambda-\alpha_4-\alpha_5$ is a weight of $V$. However, $\nu|_{X}=\lambda|_{X}-\beta_2+\beta_1=\mu_1-\beta_2+\beta_1$,  which by Proposition \ref{p:notunderd4} is not a weight of $V_i$ for any $1\leq i \leq 6$. This is a contradiction. 
 
\renewcommand{\arraystretch}{1.2}
\begin{table}
$$\begin{array}{llll} \hline
\mbox{$U_X$-level} & \mbox{$T_X$-weight} & \mbox{$T$-weight} & \mbox{Root restriction} \\ \hline
0 & \delta & \lambda_1 & \\
1 & \delta-\beta_2 & \lambda_1-\alpha_1 & \alpha_1|_{X}=\beta_2 \\
2 & \delta-\beta_2-\beta_3 & \lambda_1-\alpha_1-\alpha_2 & \alpha_2|_{X}=\beta_3 \\
& \delta-\beta_1-\beta_2 & \lambda_1-\alpha_1-\alpha_2-\alpha_3 & \alpha_3|_{X}=\beta_1-\beta_3 \\
& \delta-\beta_2-\beta_4 & \lambda_1-\alpha_1-\alpha_2-\alpha_3-\alpha_4 & \alpha_4|_{X}=\beta_4-\beta_1 \\
& \delta-2\beta_2 & \lambda_1-\sum_{i=1}^5\alpha_i & \alpha_5|_{X}=\beta_2-\beta_4 \\
3 & \delta-\beta_2-\beta_3-\beta_4 & \lambda_1-\sum_{i=1}^6\alpha_i & \alpha_6|_{X}=\beta_3+\beta_4-\beta_2 \\
& \delta-\beta_1-\beta_2-\beta_4 & \lambda_1-\sum_{i=1}^7\alpha_i & \alpha_7|_{X}=\beta_1-\beta_3\\
& \delta-\beta_1-\beta_2-\beta_3 & \lambda_1-\sum_{i=1}^8\alpha_i & \alpha_8|_{X}=\beta_3-\beta_4\\
& \delta-\beta_1-2\beta_2 & \lambda_1-\sum_{i=1}^9\alpha_i & \alpha_9|_{X}=\beta_2-\beta_3 \\
& \delta-2\beta_2-\beta_3 & \lambda_1-\sum_{i=1}^{10}\alpha_i & \alpha_{10}|_{X}=\beta_3-\beta_1 \\
& \delta-2\beta_2-\beta_4 & \lambda_1-\sum_{i=1}^{11}\alpha_i & \alpha_{11}|_{X}=\beta_4-\beta_3 \\ \hline
\end{array}$$
\caption{}
\label{t:d43}
\end{table}
\renewcommand{\arraystretch}{1}

Now assume $b=2$ (and $a=0$). Here Proposition \ref{p:d4b} and Lemma \ref{l:maind4s3} imply that either $a_4=1$ and $a_3=a_5=0$,  or $a_7=1$ and $a_i=0$ for $i\in\{8,9,10,11\}$, or $a_{11}=1$ and $a_i=0$ for $i \in \{7,8,9,10\}$. If $a_4=1$ then one can repeat the previous argument to obtain a contradiction, so let us assume $a_7=1$ or $a_{11}=1$. Now $\nu=\lambda-\sum_{i=7}^{11}\alpha_i$ is a weight of $V$, and in view of the root restrictions listed in Table \ref{t:d43} we deduce that $\nu|_{X}=\lambda|_{X}-\beta_2+\beta_3=\mu_1-\beta_2+\beta_3$. Once again, this is ruled out by Proposition \ref{p:notunderd4}. This completes the analysis of (i).

Next consider (ii), so $a\neq 0$ and $b= 0$. To begin with let us assume $a \ge 2$, in which case  Proposition \ref{p:d4b} and Lemma \ref{l:maind4s3} combine to give $\{a_2,a_3\}=\{0,2\}$. 
Let $P_X=Q_XL_X$ be the parabolic subgroup of $X$ with $\Delta(L_X')=\{\beta_2\}$, and let $P=QL$ be the parabolic subgroup of $G$ constructed from the $Q_X$-levels of $W$ in the usual way. Write $L'=L_1\cdots L_r$ and $V/[V,Q]=M_1\otimes \cdots \otimes M_r$, where each $L_i$ is simple with natural module $Y_i$, and each $M_i$ is a $p$-restricted irreducible $KL_i$-module. Note that  $Y_i$ coincides with the $i$-th $Q_X$-level of $W$. In particular, $L_1=A_5$ 
and $Y_1|_{L_X'} = U \oplus U \oplus U$ is the sum of three $2$-dimensional natural modules for $L_X' = A_1$. Since $\{a_2,a_3\}=\{0,2\}$,  Lemma \ref{l:a1a5}(iii) implies that $M_1|_{L_X'}$ has more than six composition factors, which contradicts the fact that $V/[V,Q]$ has at most six composition factors as a $KL_X'$-module. 

Now assume that $a=1$ and $b=0$.  Here Proposition \ref{p:d4b} and Lemma \ref{l:maind4s3} imply that $\{a_2,a_3\}=\{0,2\}$, or $\{a_5,a_9\}=\{0,1\}$ and $a_6=a_7=a_8=0$. If $\{a_2,a_3\}=\{0,2\}$ then we can proceed as in the previous paragraph, so let us assume $\{a_5,a_9\}=\{0,1\}$ and $a_6=a_7=a_8=0$. If $a_5=1$ then we can argue as before, taking the parabolic subgroups $P_X$ and $P$ defined in the previous paragraph to reach a contradiction via Lemma \ref{l:a1a5}(iii). Finally, let us assume $a_9=1$. Let $B_X=U_XT_X$ be a Borel subgroup of $X$ that is stable under $s$ and $t$. We can order the $T$-weights of $W$ in the first three $U_X$-levels to give the root restrictions recorded in Table \ref{t:d44}. Since $a_9=1$ it follows that $\nu=\lambda-\alpha_6-\alpha_7-\alpha_8-\alpha_9$ is a weight of $V$, and using the root restrictions listed in Table \ref{t:d44} we calculate that $\nu|_{X}=\lambda|_{X}-\beta_2+\beta_3=\mu_1-\beta_2+\beta_3$. However, Proposition \ref{p:notunderd4} implies that this is not a weight of $V_i$ for any $1\leq i \leq 6$, a contradiction. 

\renewcommand{\arraystretch}{1.2}
\begin{table}
$$\begin{array}{llll} \hline
\mbox{$U_X$-level} & \mbox{$T_X$-weight} & \mbox{$T$-weight} & \mbox{Root restriction} \\ \hline
0 & \delta & \lambda_1 & \\
1 & \delta-\beta_1 & \lambda_1-\alpha_1 & \alpha_1|_{X}=\beta_1 \\
& \delta-\beta_3 & \lambda_1-\alpha_1-\alpha_2 & \alpha_2|_{X}=\beta_3-\beta_1 \\
& \delta-\beta_4 & \lambda_1-\alpha_1-\alpha_2-\alpha_3 & \alpha_3|_{X}=\beta_4-\beta_3 \\
2 & \delta-\beta_1-\beta_4 & \lambda_1-\alpha_1-\alpha_2-\alpha_3-\alpha_4 & \alpha_4|_{X}=\beta_1 \\
& \delta-\beta_3-\beta_4 & \lambda_1-\sum_{i=1}^5\alpha_i & \alpha_5|_{X}=\beta_3-\beta_1 \\
& \delta-\beta_1-\beta_3 & \lambda_1-\sum_{i=1}^6\alpha_i & \alpha_6|_{X}=\beta_1-\beta_4 \\
& \delta-\beta_1-\beta_2 & \lambda_1-\sum_{i=1}^7\alpha_i & \alpha_7|_{X}=\beta_2-\beta_3 \\
& \delta-\beta_2-\beta_3 & \lambda_1-\sum_{i=1}^8\alpha_i & \alpha_8|_{X}=\beta_3-\beta_1 \\
& \delta-\beta_2-\beta_4 & \lambda_1-\sum_{i=1}^9\alpha_i & \alpha_9|_{X}=\beta_4-\beta_3 \\ \hline
\end{array}$$
\caption{}
\label{t:d44}
\end{table}
\renewcommand{\arraystretch}{1}

To complete the proof, let us deal with case (iii) above, so $ab \neq 0$. In the usual way, by considering $U_X$-levels and applying Proposition \ref{p:d4b} and Lemma \ref{l:maind4s3}, we deduce that either $a_3=1$ and $a_2=a_4=0$, or $\{a_6,a_{10}\}=\{0,1\}$ and $a_7=a_8=a_9=0$. Note that the latter case can only occur if  $(a,b,p)=(1,1,3)$. By suitably ordering the $T$-weights in the first three $U_X$-levels of $W$ we derive the root restrictions given in Table \ref{t:d45}. If $a_3=1$ then $\lambda-\alpha_3-\alpha_4$ is a weight of $V$ which restricts to $\lambda|_{X}-\beta_2+\beta_3=\mu_1-\beta_2+\beta_3$. In the usual way, this is ruled out by Proposition \ref{p:notunderd4}. Similarly, if $a_6=1$ (respectively, $a_{10}=1$) then $\lambda-\alpha_6-\alpha_7-\alpha_8$ (respectively, $\lambda-\alpha_8-\alpha_9-\alpha_{10}$) is a weight of $V$ and $\nu|_{X}=\lambda|_{X}-\beta_3+\beta_2=\mu_1-\beta_3+\beta_2$. Once again we reach a contradiction via Proposition \ref{p:notunderd4}.  
\end{proof}

\renewcommand{\arraystretch}{1.2}
\begin{table}
$$\begin{array}{llll} \hline
\mbox{$U_X$-level} & \mbox{$T_X$-weight} & \mbox{$T$-weight} & \mbox{Root restriction} \\ \hline
0 & \delta & \lambda_1 & \\
1 & \delta-\beta_1 & \lambda_1-\alpha_1 & \alpha_1|_{X}=\beta_1 \\
& \delta-\beta_3 & \lambda_1-\alpha_1-\alpha_2 & \alpha_2|_{X}=\beta_3-\beta_1 \\
& \delta-\beta_4 & \lambda_1-\alpha_1-\alpha_2-\alpha_3 & \alpha_3|_{X}=\beta_4-\beta_3 \\
& \delta-\beta_2 & \lambda_1-\alpha_1-\alpha_2-\alpha_3-\alpha_4 & \alpha_4|_{X}=\beta_2-\beta_4 \\
2 & \delta-\beta_1-\beta_2 & \lambda_1-\sum_{i=1}^5\alpha_i & \alpha_5|_{X}=\beta_1\\
& \delta-\beta_2-\beta_3 & \lambda_1-\sum_{i=1}^6\alpha_i & \alpha_6|_{X}=\beta_3-\beta_1\\
& \delta-\beta_2-\beta_4 & \lambda_1-\sum_{i=1}^7\alpha_i & \alpha_7|_{X}=\beta_4-\beta_3\\
& \delta-\beta_1-\beta_3 & \lambda_1-\sum_{i=1}^8\alpha_i & \alpha_8|_{X}=\beta_1+\beta_3-\beta_2-\beta_4\\
& \delta-\beta_1-\beta_4 & \lambda_1-\sum_{i=1}^9\alpha_i & \alpha_9|_{X}=\beta_4-\beta_3\\
& \delta-\beta_3-\beta_4 & \lambda_1-\sum_{i=1}^{10}\alpha_i & \alpha_{10}|_{X}=\beta_3-\beta_1 \\ \hline
\end{array}$$
\caption{}
\label{t:d45}
\end{table}
\renewcommand{\arraystretch}{1}

\begin{lem}\label{p:d4s3delta22}
If $\delta = \delta_2$ and $p=2$ then $(G,H,V) \not\in \mathcal{T}$. 
\end{lem}

\begin{proof}
Here $G=D_{13}$ (see \cite[Table 2]{Brundan}). Seeking a contradiction, let us assume $(G,H,V) \in \mathcal{T}$. The usual Borel analysis (using Proposition \ref{p:d4b} and Lemma \ref{l:maind4s3}) yields  $a_{10}=1$ and $a_9=a_{11}=0$. In addition, we may order the $T$-weights in the first five $U_X$-levels of $W$ to give the root restrictions recorded earlier in Table \ref{t:d4root}.
Since  $a_{10}=1$ it follows that $\lambda-\alpha_9-\alpha_{10}-\alpha_{11}$ is a weight of $V$. This restricts to give $\nu|_{X}=\lambda|_{X}-\beta_2+\beta_1=\mu_1-\beta_2+\beta_1$, which is ruled out by Proposition \ref{p:notunderd4}. 
\end{proof}

\begin{lem}\label{l:last}
If $\delta = \delta_2$ then $(G,H,V) \not\in \mathcal{T}$. 
\end{lem}

\begin{proof}
By Lemma \ref{p:d4s3delta22}, we may assume $p \neq 2$, so $G=D_{14}$ (see \cite[Table 2]{Brundan}). Suppose $(G,H,V) \in \mathcal{T}$. The usual Borel analysis implies that one of the following holds:
\begin{itemize}\addtolength{\itemsep}{0.3\baselineskip}
\item[(i)] $a_{10}=1$ and $a_i=0$ for $i\in \{3,4,6,7,9,11,13,14\}$. 
\item[(ii)] $\{a_{13},a_{14}\}=\{1,2\}$ and $a_i=0$ for $i \in \{3,4,6,7,9,10,11\}$.
\item[(iii)] $\{a_{13},a_{14}\}=\{0,5\}$, $p \neq 3,5$ and $a_i=0$ for $i \in \{3,4,6,7,9,10,11\}$.
\item[(iv)] $\{a_6,a_7\}=\{0,2\}$ and $a_i=0$ for $i \in \{3,4,9,10,11,13,14\}$.
\item[(v)] $\{a_3,a_4\}=\{0,2\}$ and $a_i=0$ for $i \in \{6,7,9,10,11,13,14\}$.
\item[(vi)] $\{a_6,a_7\}=\{a_{13},a_{14}\}=\{0,1\}$ and $a_i=0$ for $i \in \{3,4,9,10,11\}$.
\item[(vii)] $\{a_3,a_4\}=\{a_{13},a_{14}\}=\{0,1\}$ and $a_i=0$ for $i \in \{6,7,9,10,11\}$.
\end{itemize}

If $a_{10}=1$ then $\nu= \lambda-\alpha_9-\alpha_{10}-\alpha_{11}$ is a weight of $V$ and by using the root restrictions given in Table \ref{t:d4root} we calculate that $\nu|_{X}=\lambda|_{X}-\beta_2+\beta_1=\mu_1-\beta_2+\beta_1$. However, Proposition \ref{p:notunderd4} implies that this is not a weight in any of the $V_i$, which is a contradiction. This eliminates case (i).

As in the proof of Lemma \ref{l:d42_2} (and also Lemma \ref{l:d412}), let us consider the 
parabolic subgroup $P_X=Q_XL_X$ of $X$ with $\Delta(L_X')=\{\beta_2\}$. Let $P=QL$ be the corresponding parabolic subgroup of $G$. We have $L'=L_1 \cdots L_r$ and $V/[V,Q]=M_1 \otimes \cdots \otimes M_r$, where each $L_i$ is simple with natural module $Y_i=W_{i-1}$ (where $W_i$ denotes the $i$-th $Q_X$-level of $W$), and each $M_i$ is a $p$-restricted irreducible $KL_i$-module. More precisely, we have $r=4$ and $L_1=A_1$, $L_2=A_2$, $L_3=A_5$ and $L_4=D_3$. 
We note that $Y_2|_{L_X'}$ is trivial and $Y_3|_{L_X'} = U \oplus U \oplus U$, where $U$ is the natural $2$-dimensional module for $L_X'=A_1$, and $Y_4|_{L_X'}$ is the direct sum of a $3$-dimensional trivial module and the irreducible $3$-dimensional $KL_X'$-module of highest weight $2\delta_2|_{L_X'}$. 
If we view $D_3={\rm SO}(Y_4)$ as an image of $A_3={\rm SL}(M)$, then $Y_4=\Lambda^2(M)$ and thus the action of $L_X'$ on $Y_4$ implies that a central extension of $L_X'$ acts on $M$ as the sum of two $2$-dimensional natural modules for $A_1$.

If $\{a_{13},a_{14}\}= \{1,2\}$ or $\{0,5\}$ then Lemma \ref{l:a1a3}(ii) implies that $M_4|_{L_X'}$ has more than six composition factors, which is a contradiction. This rules out cases (ii) and (iii) above. Similarly, Lemma \ref{l:a1a5}(iii) implies that $\{a_{6},a_{7}\}\neq \{0,2\}$, so case (iv) is also eliminated. 

To handle cases (vi) and (vii), we first show that $a_{12}=0$.  Suppose $\{a_6,a_7\}=\{a_{13},a_{14}\}=\{0,1\}$
and $a_{12} \neq 0$.
Then Lemmas \ref{l:a1a3} and \ref{l:a1a5} imply that $M_3|_{L_X'}$ and $M_4|_{L_X'}$ have more than two and three composition factors, respectively, which contradicts
the fact that $V/[V,Q]$ has at most six composition factors as a $KL_X'$-module.
Similarly, if $\{a_3,a_4\}=\{a_{13},a_{14}\}=\{0,1\}$ and $a_{12} \neq 0$ then
$M_2|_{L_X'}$ is a $3$-dimensional trivial $KL_X'$-module (since $Y_2|_{L_X'}$ is trivial), while Lemma \ref{l:a1a3} once again implies that $M_4|_{L_X'}$ has more than three composition factors. Again, this gives a contradiction and we conclude that $a_{12}=0$.

To complete the analysis of cases (vi) and (vii), consider the parabolic subgroup $P_X=Q_XL_X$ with $\Delta(L_X')=\{\beta_1,\beta_3,\beta_4\}$ and let $P=QL$ be the corresponding parabolic subgroup of $G$. Here   $L'=L_1L_2$ and $V/[V,Q]=M_1 \otimes M_2$, where $L_1=A_7$, $L_2=D_5$ and 
each $M_i$ is a $p$-restricted irreducible $KL_i$-module. 
Now $Y_1|_{L_X'}$ is the tensor product of three $2$-dimensional natural modules for $A_1$, so Lemma \ref{l:a1a1a1am} implies that $M_1|_{L_X'}$ has more than three composition factors.
In addition, $L_X'$ stabilizes a $1$-dimensional subspace of $Y_2$, so $Y_2|_{L_X'}$ is reducible. Therefore, if we consider the configuration $(L_X',L_2,M_2)$ then  \cite[Theorem 5.1]{Seitz2} implies that $M_2|_{L_X'}$ has at least
two composition factors, but this contradicts the fact that $V/[V,Q]$ has at most 
six composition factors as a $KL_X'$-module.

\par

Finally, it remains to deal with case (v), so $\{a_3,a_4\}=\{0,2\}$ and $a_i=0$ for all $i \in \{6,7,9,10,11,13,14\}$. Recall that 
$$\mu_1=\lambda|_{X}=c_1\delta_1+c_2\delta_2+c_3\delta_3+c_4\delta_4$$ 
and expressions for $\mu_2, \ldots, \mu_6$ are given in Lemma \ref{l:mu1d4s3}. Set $\mu = \l - \a_3-\a_4$, $\nu = \l - 2\a_3 - 2\a_4$ and note that $\mu$ and $\nu$ are weights of $V$, with $\mu|_{X} = \mu_1 + \b_1 - \b_4$ and $\nu|_{X} = \mu_1 +2\b_1-2\b_4$. 

Clearly, $\mu|_{X}$ is not under $\mu_1$, and Lemma \ref{l:mu1d4s3} quickly implies that $\mu|_{X}$ is not under $\mu_2$ nor $\mu_3$. In fact, by applying Lemma \ref{l:mu1d4s3} we deduce that one of the following holds:
\begin{itemize}\addtolength{\itemsep}{0.3\baselineskip}
\item[(a)] $\mu|_{X} \preccurlyeq \mu_4$ and $c_4 = c_1+2$;
\item[(b)] $\mu|_{X} \preccurlyeq \mu_5$ and $c_4 = c_1+2 = c_3+2$;
\item[(c)] $\mu|_{X} \preccurlyeq \mu_6$ and $c_4 = c_1+2 = c_3$.
\end{itemize}
For example, by Lemma \ref{l:mu1d4s3} we have
$$\mu_5 - \mu|_{X} = \left(\frac12(c_4-c_1)-1\right)\beta_1+\frac12(c_1-c_3)\beta_3+\left(\frac12(c_3-c_4)+1\right)\beta_4,$$
so $\mu|_{X} \preccurlyeq \mu_5$ if and only if $\frac12(c_4-c_1)-1$, $\frac12(c_1-c_3)$ and $\frac12(c_3-c_4)+1$ are all non-negative integers, which is true if and only if $c_4 = c_1+2 = c_3+2$ as in (b) above. Similarly, we find that $\nu|_{X}$ is not under $\mu_1, \mu_2$ nor $\mu_3$. Moreover, one of the following holds:
\begin{itemize}\addtolength{\itemsep}{0.3\baselineskip}
\item[(a)$^\prime$] $\nu|_{X} \preccurlyeq \mu_4$ and $c_4 = c_1+4$;
\item[(b)$^\prime$] $\nu|_{X} \preccurlyeq \mu_5$ and $c_4 = c_1+4 = c_3+4$;
\item[(c)$^\prime$] $\nu|_{X} \preccurlyeq \mu_6$ and $c_4 = c_1+4 = c_3$.
\end{itemize}

Clearly, conditions (a) -- (c) are incompatible with the conditions (a)$^\prime$ -- (c)$^\prime$, so $\mu|_{X}$ or $\nu|_{X}$ is not under any $\mu_i$ (for $1 \le i \le 6$), but this contradicts \eqref{e:vx2}. This final contradiction eliminates case (v), and the proof of the lemma is complete.
\end{proof}

\vs

This completes the proof of Theorem \ref{t:d4}. Moreover, in view of Theorems \ref{t:am}, \ref{T:DM5} and \ref{T:E6}, the proof of Theorems \ref{t:main} -- \ref{t:main2} is complete.

\chapter{Proof of Theorem \ref{T:MAIN3}}\label{s:thm5}

In this final section we establish Theorem \ref{T:MAIN3}. Let $G$ be a simple classical algebraic group of rank $n$ over an algebraically closed field $K$ of characteristic $p \ge 0$. Let $W$ be the natural $KG$-module and let $H$ be a positive-dimensional maximal subgroup of $G$. Let $\{\l_1, \ldots, \l_n\}$ be a set of fundamental dominant weights for $G$, labelled in the usual way. Recall that we want to determine the triples $(G,H,k)$, where $H$ acts irreducibly on all composition factors of the $KG$-module $\L^k(W)$ (and similarly for the symmetric powers $S^k(W)$). As in the statement of Theorem \ref{T:MAIN3}, for the exterior powers $\L^k(W)$ we will assume that $1<k<n$. Similarly, for $S^k(W)$ we assume that $k>1$, with the extra condition that $k<p$ if $p \neq 0$ (so that the $KG$-composition factor with highest weight $k\l_1$ is $p$-restricted). 

In order to prove Theorem \ref{T:MAIN3} we will consider all positive-dimensional subgroups $H$ of $G$, where $H$ either belongs to the collection $\mathcal{S}$, or $H$ is a maximal subgroup in one of the geometric subgroup collections $\C_i$ ($1 \le i \le 6$). 

Let $V$ denote the exterior power $\L^k(W)$, where $1<k<n$. We note that $V$ has a $KG$-composition factor with highest weight $\mu$, where either $\mu=\l_k$, or $G=D_n$, $k=n-1$ and $\mu=\l_{n-1}+\l_n$. If $G=A_n$ then $V$ is irreducible and the list of examples $(H,k)$ recorded in Table \ref{t:ext} is obtained by inspecting \cite[Table 1]{Seitz2} and \cite[Table 1]{BGT} (in the first table, the relevant cases are labelled ${\rm I}_{2}$ -- ${\rm I}_{12}$). In the same way (this time also using Theorem \ref{t:main}), if $G=B_n$ then it is easy to check that $H$ acts reducibly on $V_{G}(\l_{k})$. The same is true if $G=D_n$ and $k \le n-2$. 

Next suppose $G=D_n$ and $k=n-1$. Here $\L^k(W)$ has a $KG$-composition factor with highest weight $\l_{n-1}+\l_n$, and by arguing in the usual way we reduce to the case 
$(G,p)=(D_4,2)$ with $H=C_1^3.S_3$ a geometric subgroup in the $\mathcal{C}_4$ collection (see \cite[Theorem 1]{BGT}). In this case, $\L^3(W)$ has exactly two $KG$-composition factors, with highest weights $\l_3+\l_4$ (of dimension $48$) and $\l_1$ (the natural $8$-dimensional module), so this is one of the examples recorded in Table \ref{t:ext}.

To complete the analysis of exterior powers, let us assume $G=C_n$.  Here $V$ is reducible and the $KG$-composition factors have highest weights $\l_{k}, \l_{k-2}, \l_{k-4}, \ldots$ (with various multiplicities) -- see \cite[Chapter 8, Section 13.3]{Bou2}, for example. Again, the examples $(H,k)$ which arise in this case can be determined by combining Theorem \ref{t:main} with the information in Table 1 of \cite{Seitz2} and \cite{BGT} (in \cite[Table 1]{Seitz2}, the relevant cases are labelled ${\rm II}_{1}$ -- ${\rm II}_{9}$ and ${\rm S}_3$).

Finally, let us consider the $KG$-module $V=S^k(W)$; a very similar argument applies. First observe that $V$ has a $KG$-composition factor with highest weight $k\l_1$. As before, if $G=A_n$ then $V$ is irreducible and in the usual way we find that the only example is the case labelled ${\rm I}_{1}$ in \cite[Table 1]{Seitz2}, with $n=2l-1$ and $H=C_{l}$. In particular, the existence of this example implies that $V$ is also irreducible when $G=C_n$, and in this case, by inspecting the relevant tables, we deduce that $H$ is always reducible on $V$. Finally, let us assume $G$ is an orthogonal group. By considering $V_G(k\l_1)$ we quickly reduce to the case $k=2$. Here \cite[Theorem 5.1]{Lubeck} gives 
$\dim V_{B_n}(2\l_1) = 2n^2+3n-\e$ and $\dim V_{D_n}(2\l_1) = 2n^2+n-1-\e$, where 
$\e=1$ if $p$ divides $n$, otherwise $\e=0$. It follows that the $KG$-module $V$ has a unique nontrivial composition factor, so we only need to consider the irreducibility of $H$ on 
$V_G(2\l_1)$. The result follows in the usual manner (note that the relevant cases in \cite[Table 1]{Seitz2} are labelled ${\rm III}_{1}$, ${\rm S}_{1}$, ${\rm S}_{2}$ and ${\rm S}_{5}$).

\vs

We have now established that all pairs $(G,H)$ satisfying the hypotheses of Theorem \ref{T:MAIN3} appear in Tables \ref{t:ext} and \ref{t:sym}. To complete the proof, one observes that there are no inclusions among the various $H$ in a fixed $G$.

\vs

This completes the proof of Theorem \ref{T:MAIN3}.

\chapter*{Notation}

\label{p:notation}
\renewcommand{\arraystretch}{1.05}
\begin{tabular}{ll}
$\Phi(X)$ & root system of $X$ \\
$\Phi^+(X)$ & positive roots of $X$ \\
$\Delta(X)$ & base of the root system of $X$ \\
$\mathcal{L}(X)$ & Lie algebra of $X$ \\
$\mathcal{W}(X)$ & Weyl group of $X$ \\
$e(X)$ & maximum of the squares of the ratios of the lengths \\
 & of the roots in $\Phi(X)$ \\
$X.Y$ & an extension of $X$ by a group $Y$ \\
$X.n$ & an extension of $X$ by a cyclic group of order $n$ \\
$[W,Q_X^i]$ & $i$-th $Q_X$-commutator of $W$ \\
$W_i$ & $i$-th $Q_X$-level of $W$ \\
${\rm Isom}(W)$ & isometry group of a form on $W$ \\
$Cl(W)$ & ${\rm Isom}(W)'$; a simple classical type algebraic group with \\
& natural module $W$ \\
$S^k(W)$ & $k$-th symmetric power of $W$ \\
$\L^k(W)$ & $k$-th exterior power of $W$ \\
$T_k$ & $k$-dimensional torus \\
$W_X(\l)$ & Weyl module for $X$ with highest weight $\l$ \\
$V_X(\l)$ & irreducible module for $X$ with highest weight $\l$ \\
$\L(V)$ & set of weights of $V$ \\
$V_{\mu}$ & weight space of $\mu$ in $V$ \\
$m_V(\mu)$ & multiplicity of $\mu \in \L(V)$ \\ 
$\mu \preccurlyeq \eta$ & $\eta - \mu = \sum_{\b}c_{\b}\b$, $c_{\b} \in \mathbb{N}_0$ for all $\b \in \Delta(X)$ ($\mu,\eta \in \L(V)$) \\
$h_{\b}(c)$ & see p.\pageref{p:hbc} \\
& \\
\multicolumn{2}{l}{Specific notation used in the proof of the main theorems:} \\
& \\
$G$ & simply connected simple classical algebraic group \\
$H$ & almost simple positive-dimensional closed subgroup of $G$ \\
$\{\a_1, \ldots, \a_n\}$ & base of the root system of $G$ \\
$\{\l_1, \ldots, \l_n\}$ & fundamental dominant weights of $G$ \\
$\{\b_1, \ldots, \b_m\}$ & base of the root system of $H^0$ \\
$\{\delta_1, \ldots, \delta_m\}$ & fundamental dominant weights of $H^0$ \\
$\sum_{i=1}^na_i\l_i$ & highest weight of the $KG$-module $V=V_G(\l)$ \\
$\sum_{i=1}^mb_i\delta_i$ & highest weight of the $KH^0$-module $W=V_{H^0}(\delta)$ \\
$V|_{Y}$ & restriction of $V$ to a subgroup $Y \leqs G$ \\
$\l|_{Y}$ & restriction of $\l$ to a subtorus $T_Y$ of $Y$ \\
\end{tabular}
\renewcommand{\arraystretch}{1}

\backmatter

\bibliographystyle{amsalpha}

\end{document}